\documentclass[aos]{imsart}

\RequirePackage{amsthm,amsmath,amsfonts,amssymb, mathrsfs}
\RequirePackage[numbers]{natbib}
\RequirePackage{xfrac}
\RequirePackage[colorlinks,citecolor=blue,urlcolor=blue]{hyperref}
\RequirePackage[ruled,vlined]{algorithm2e}
\RequirePackage{graphicx}
\RequirePackage{cases}

\RequirePackage{pdflscape}

\hypersetup{
	colorlinks,
	citecolor=dukeblue,
	filecolor=black,
	linkcolor=dukeblue,
	urlcolor=dukeblue
}

\startlocaldefs
\numberwithin{equation}{section}
\theoremstyle{plain}

\newtheorem{theorem}{Theorem}[section]
\newtheorem{lemma}[theorem]{Lemma}
\newtheorem{proposition}[theorem]{Proposition}
\newtheorem{assumption}[theorem]{Assumption}
\newtheorem{remark}[theorem]{Remark}
\theoremstyle{remark}
\newtheorem{definition}[theorem]{Definition}

\def\*#1{\mathbf{#1}}
\newcommand{\diff}{\mathrm{d}} 

\endlocaldefs

\begin{document}

\begin{frontmatter}
\title{Turnpike in optimal control of PDEs, ResNets, and beyond}
\runtitle{Turnpike in optimal control of PDEs, ResNets, and beyond}

\begin{aug}
\author{\fnms{Borjan} \snm{Geshkovski}\ead[label=e1]{borjan@mit.edu}}
\address{Department of Mathematics\\
Massachusetts Institute of Technology\\
Cambridge MA, 02139 USA\\
\printead{e1}}
\author{\fnms{Enrique} \snm{Zuazua}\ead[label=e3]{enrique.zuazua@fau.de}}
\address{Chair in Dynamics, Control, and Numerics \\
Alexander von Humboldt-Professorship\\
Friedrich-Alexander-Universität Erlangen-Nürnberg\\
Cauerstrasse 11, 91052 Erlangen, Germany\\
\&\\
Chair of Computational Mathematics\\
Fundación Deusto\\
Av. de las Universidades 24 \\
48007 Bilbao, Basque Country, Spain\\
\&\\
Departamento de Matemáticas\\
Universidad Autónoma de Madrid\\ 
28049 Madrid, Spain\\
\printead{e3}}

\end{aug}

\begin{abstract}
The \emph{turnpike property} in contemporary macroeconomics asserts that if an economic planner seeks to move an economy from one level of capital to another, then the most efficient path, as long as the planner has enough time, is to rapidly move stock to a level close to the optimal stationary or constant path, then allow for capital to develop along that path until the desired term is nearly reached, at which point the stock ought to be moved to the final target.  
Motivated in part by its nature as a resource allocation strategy, over the past decade, the turnpike property has also been shown to hold for several classes of partial differential equations arising in mechanics. 
When formalized mathematically, the turnpike theory corroborates the insights from economics: for an optimal control problem set in a finite-time horizon, optimal controls and corresponding states, are close (often exponentially), during most of the time, except near the initial and final time, to the optimal control and corresponding state for the associated stationary optimal control problem. In particular, the former are mostly constant over time.
This fact provides a rigorous meaning to the asymptotic simplification that some optimal control problems appear to enjoy over long time intervals, allowing the consideration of the corresponding stationary problem for computing and applications.
We review a slice of the theory developed over the past decade --the controllability of the underlying system is an important ingredient, and can even be used to devise simple turnpike-like strategies which are nearly optimal--, and present several novel applications, including, among many others, the characterization of Hamilton-Jacobi-Bellman asymptotics, and stability estimates in deep learning via residual neural networks.
\end{abstract}

\vspace{0.5cm}
{\small
\emph{Dedicated to the memory of Roland Glowinski.}}



\end{frontmatter}

\tableofcontents

\section{Introduction} \label{sec: 1}

The field of \emph{control} provides the principles and methods used to design inputs which ensure that systems, arising in common physical, biological or social science applications, reach a desired configuration in some time, or maintain a desirable performance over time. The field, in its current form, may trace its origins to Norbert Wiener and his introduction of cybernetics \cite{wiener2019cybernetics}.
In most contemporary applications, control is applied to systems modeled by  ordinary or partial differential equations (ODEs or PDEs). For such systems, one may, in principle, design and use a variety of different controls which allow to steer or manipulate the state trajectory to one's choosing. Accordingly, in practice, for robustness, computing and production reasons, controls are sought to satisfy a certain optimality criterion. 
In other words, controls are found by minimizing some cost or maximizing some reward, which leads one to the field of \emph{optimal control} -- a classic of applied mathematics, treated in-depth and from several different lenses (\cite{lee1967foundations, lions1971optimal, hinze2008optimization, troltzsch2010optimal}).

An exemplifying application of optimal control of partial differential equations is that of \emph{flow control} (\cite{gad2007flow}). 
In aeronautics, fluid-flow is typically modeled by the pillars of fluid mechanics, namely the Navier-Stokes or Euler equations, while aeroelastic structural deformations are modeled by nonlinear elasticity systems or variants of beam equations. And in such contexts, typical control or design problems involve the optimal placement of actuators along a wing as to minimize vibrations, or optimal shape design of a wing so that drag is minimized (\cite{holmes2012turbulence, castro2007systematic, mohammadi2010applied}). 

An interesting and perpetual artifact appears in these practical applications: oftentimes, only the time-independent, stationary partial differential equations are considered by practitioners for control and/or design (\cite{jameson2010optimization, jameson1988aerodynamic}). This is done for obvious computational reasons -- a direct simulation of the time-dependent Navier-Stokes equations would be rather unfeasible and not wise for online design. 
Nonetheless, even if the underlying dynamics (in occurrence, the Navier-Stokes system) without control may stabilize to a steady state in large time, there are no guarantees that the stationary optimal control problem will reflect the features of solutions to the true, time-dependent problem. 
This marks the relevance of the role of the final time horizon in these contexts.

Long time horizons in optimal control of PDEs also occur in other related applications. An example is sonic-boom minimization, where the main interest is the design of aircraft which are sufficiently quiet to fly supersonically over land, which namely produces little to no acoustic disturbances for humans on ground level. In some settings, the mathematical formulation of the sonic-boom minimization problem can be seen as an inverse design or optimal control problem for the inviscid Burgers equation over long time horizons (\cite{alonso2012multidisciplinary, allahverdi2016numerical}). The stability and inversion properties of such problems are sensitive to the time horizon and analytical guarantees are thus required (\cite{liard2021initial, esteve2020inverse}). 
Similar considerations transfer to optimal control problems in climate science (\cite{nordhaus1992optimal, kellett2019feedback}), or data assimilation problems in meteorology and oceanography (\cite{ghil1991data}), in addition to problems transversing fields such as the computation of sensitivities, all of which fit within the framework discussed above.

\begin{figure}[h!]
\centering
\includegraphics[scale=1.1]{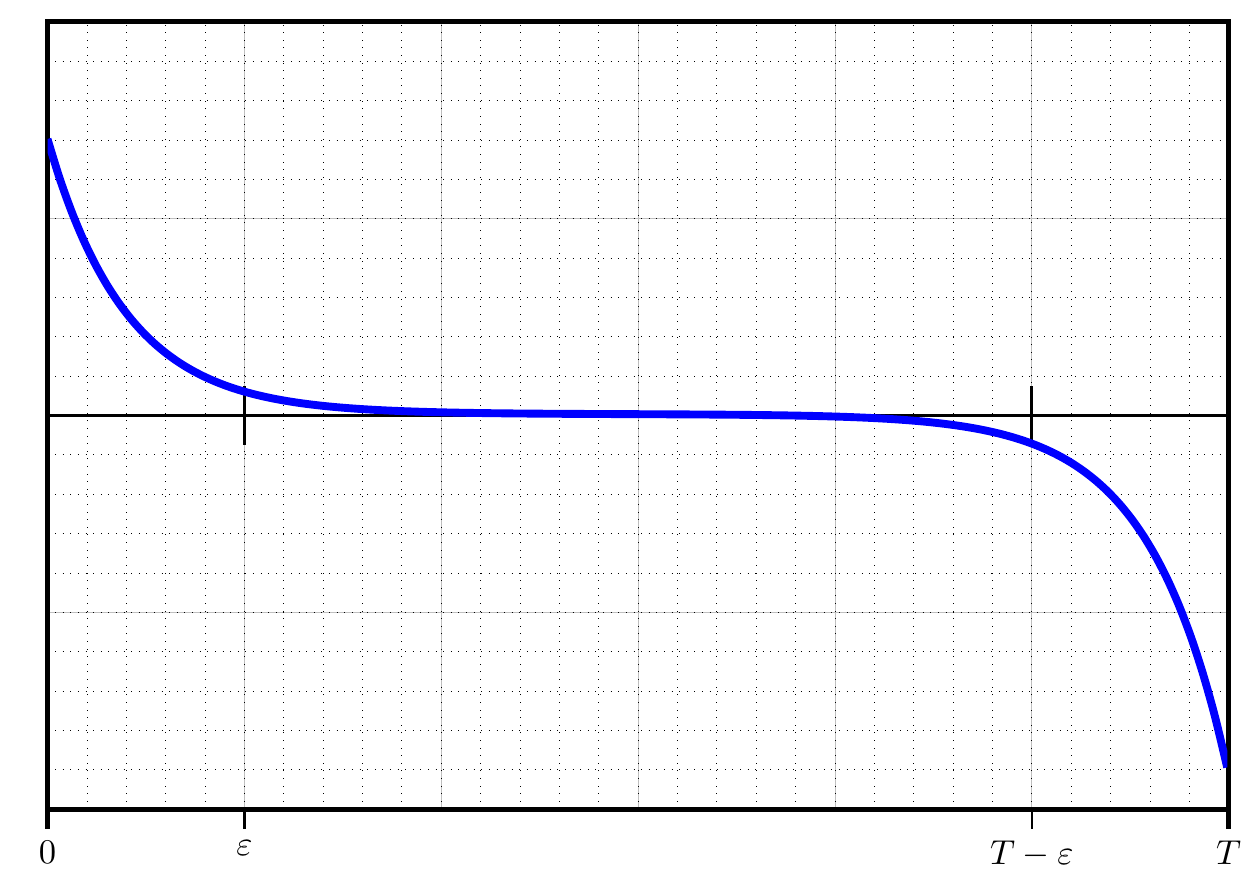}
\caption{{\textbf{The turnpike property} for optimal (with respect to some cost) controls $u(t)$ and corresponding states $y(t)$, solving $\dot{y}=f(y,u)$, over $t\in[0,T]$. The graph shows that $\|u(t)\|$ and $\|y(t)\|$ (blue) are near to $\|\overline{u}\|$ and $\|\overline{y}\|$ (black) respectively, during most of the time horizon, except for two small initial and final subintervals. In particular, the former are mostly constant.
Here $\overline{u}$ is the optimal (with respect to the corresponding time-independent cost) steady control and $\overline{y}$ corresponding optimal steady state satisfying $f(\overline{y},\overline{u})=0$. The latter are referred to as \emph{the turnpike}.}}
\label{fig: first}
\end{figure}

In view of the above discussions, one is expected to understand that the transition from time-dependent optimal control problems to the associated static optimal control problem is an issue which requires careful analysis and discussion.  And it is herein where the concept of \emph{turnpike} comes into play.
The \emph{turnpike property} reflects the fact that, for suitable optimal control problems set in a sufficiently large time horizon, any optimal solution thereof remains, during most of the time, close to the optimal solution of a corresponding stationary optimal control problem. This optimal stationary solution is referred to as \emph{the turnpike} – the name stems from the idea that a turnpike is the fastest route between two points which are far apart, even if it is not the most direct route.
In many of these cases, the turnpike property is quantified by an exponential estimate; typically, the optimal control-state pair $\left(u(t), y(t)\right)$ is $\mathcal{O}\left(e^{-\lambda t} + e^{-\lambda(T-t)}\right)$--close to the optimal stationary control-state pair $(\overline{u}, \overline{y})$, for $t\in[0,T]$ and for some rate $\lambda>0$ independent of $T$, provided $T\gg1$.

The denomination \emph{turnpike} was coined by three pre-eminent economists of the 20th century -- Paul Samuelson, Robert Solow, and Robert Dorfman -- in their seminal text \cite{dorfman1987linear}. We quote, verbatim, \cite[p. 331]{dorfman1987linear} (also found on the Wikipedia): 

\begin{center}
\begin{minipage}{0.925\textwidth}
\emph{
"
Thus in this unexpected way, we have found a real normative significance for steady growth -- not steady growth in general, but maximal von Neumann growth. 
It is, in a sense, the single most effective way for the system to grow, so that if we are planning long-run growth, no matter where we start, and where we desire to end up, it will pay in the intermediate stages to get into a growth phase of this kind. It is exactly like a turnpike paralleled by a network of minor roads. There is a fastest route between any two points; and if the origin and destination are close together and far from the turnpike, the best route may not touch the turnpike. But if the origin and destination are far enough apart, it will always pay to get on to the turnpike and cover distance at the best rate of travel, even if this means adding a little mileage at either end. The best intermediate capital configuration is one which will grow most rapidly, even if it is not the desired one, it is temporarily optimal.
"
}
\end{minipage}
\end{center}

\noindent
This quote and denomination are actually preceded by another work of Samuelson in 1948 (see \cite{samuelson1965catenary}), in which he shows that an efficient expanding economy would spend most of the time in the vicinity of a balanced equilibrium path (also called a \emph{von Neumann path}). Hence, as insinuated, these notions trace their origins back to an older work of John von Neumann \cite{von1937uber} (and even further back to \cite{ramsey1928mathematical}).
The turnpike theory had subsequently seen further development in the field of econometrics throughout the 1960s and 70s (\cite{mckenzie1976turnpike, haurie1976optimal}).
And yet, even-though the turnpike property appears to be a phenomenon which is implicitly used by many practitioners for different optimal control problems at different scales, a rigorous theory regarding its appearance for finite and infinite dimensional systems stemming from mechanics, distinguishing the relevant necessary or sufficient conditions, had been lacking for several decades, until recent developments in the parallel field of \emph{mean field games} (\cite{cardaliaguet2012long, cardaliaguet2013long}, see also Section \ref{sec: 14}).

It is important to ensure that such a theory is firmly established due to similar dissonances which arise between mathematics and applications. 
Recall the classical problem of \emph{control and discretize} versus \emph{discretize and control}, for instance. In the case of the wave equation with boundary control, the control and discretization processes do not commute (\cite{zuazua2005propagation}). Indeed, when one first spatially discretizes the wave equation, one obtains a linear finite dimensional control system, which can be shown to be controllable by a simple algebraic test, the so-called the \emph{Kalman rank condition}. In particular, this condition is entirely devoid of time dependence. However, it is well known and understood that the solutions of the wave equation manifest an oscillatory behavior, and in particular, cannot be controlled in an arbitrarily small time, as waves require a long time\footnote{The minimal time horizon can be characterized explicitly as a function of the support of the control, the velocity of propagation of the waves, and the domain where the waves propagate. See \cite{zuazua2005propagation} for a survey.} to travel across the domain. 
Similar conclusions also hold for shape optimization (\cite{hebrard2005spillover}) and inverse problems (\cite{ammari2011transient}) in wave propagation, or Bayesian inverse problems (\cite{stuart2010inverse}), among many other topics.
This way of thought should be extrapolated to the turnpike phenomenon -- one cannot simply drop time dependence and consider the static problem without a priori guarantees of proximity between both problems. 

Through this work, we aim to provide a concise review of the theory behind this transfer. We shall mostly focus on the heat and wave equations for clarity of the presentation, and provide comments throughout on results coming from finite dimensions, and extensions to PDEs of a different nature.
We shall begin by discussing the genesis and proofs of the turnpike property for optimal control problems subject to linear PDEs. An emphasis is put on some properties that both the cost functional and the underlying dynamics need to satisfy. Namely, we shall require a certain amount of \emph{controllability} or \emph{stabilizability} of the underlying system, since, even-though the turnpike may be clearly defined, one needs some innate mechanism to be able to reach such a stationary state. On another hand, the cost functional also needs to carry sufficient observation of the state over the time horizon -- we shall see that, even for dissipative/stable systems such as the heat equation, the turnpike property may fail if the functional doesn't track the trajectory over time. On another hand, the wave equation, which conserves energy and has oscillatory solutions, will satisfy the turnpike property when the state is tracked over time. This makes clear the need of tracking terms, which are quite natural from a practical point of view, as one implicitly wishes (namely, without necessarily imposing constraints) for the state trajectory to remain within some moderately sized box at all times.

\begin{figure}[h!]
\center
\includegraphics[scale=1.1]{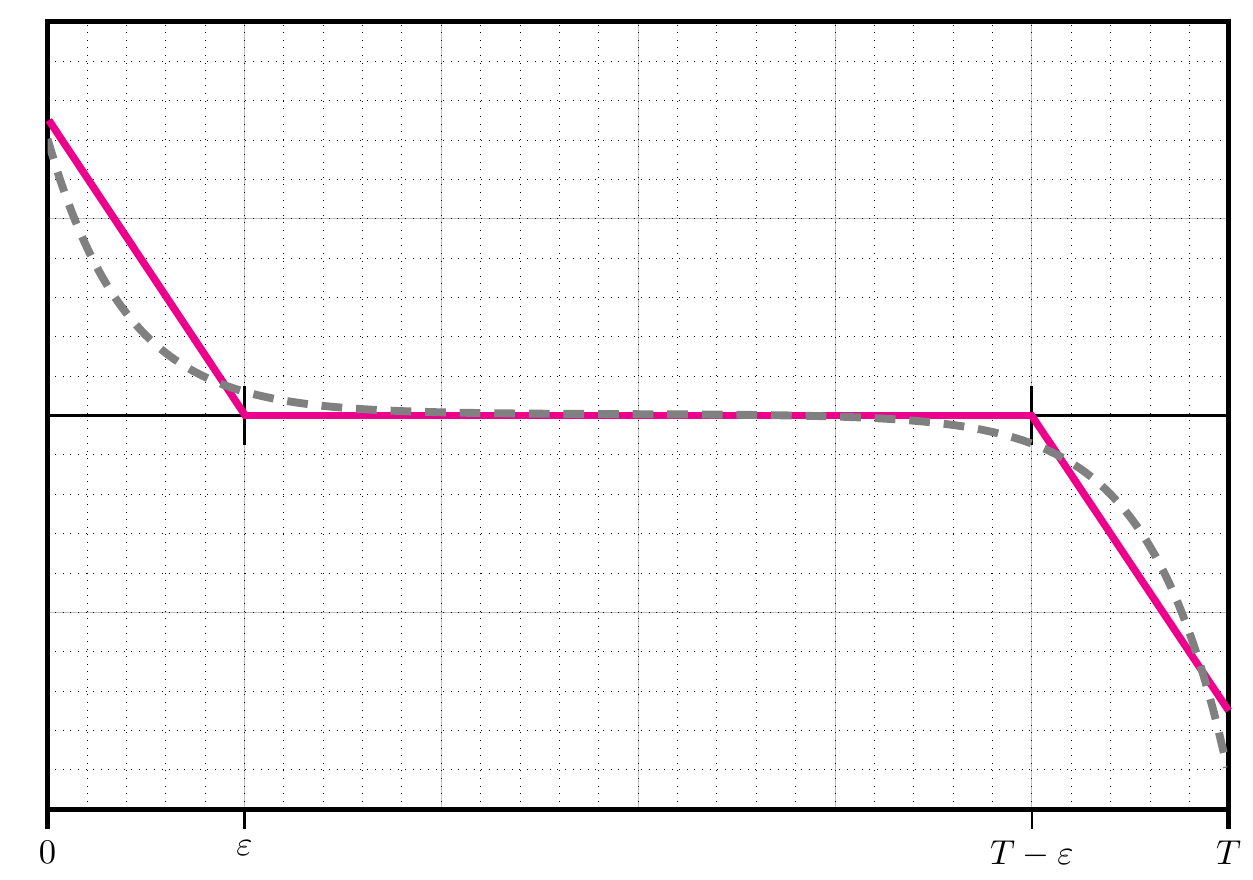}
\caption{The four-step (more precisely, an "initialization" step, plus three visible actions) {\bf quasi-turnpike strategy} (red) compared to the actual turnpike of Figure \ref{fig: first} (dashed gray). The quasi-turnpike strategy is nearly optimal.}
\label{fig: quasi.turnpike}
\end{figure}

\noindent
In problems for which controllability or stabilizability holds, even before proceeding with proofs of the turnpike property, one can devise a simple, yet illustrative, and almost-turnpike strategy in four simple steps:

\begin{enumerate}
\item[\textbf{1).}] Compute the turnpike $\overline{y}$ by solving the associated steady optimal control problem.

\item[\textbf{2).}] Use controllability to steer the trajectory $y(t)$ from $y^0$ to the computed turnpike $\overline{y}$ in time\footnote{In such strategies, the parameter $\varepsilon$ may be tuned/optimized with the goal of best approximating the optimal transient strategy -- one would expect that $\varepsilon$ should be large enough as to keep the controllability cost moderate, but not too large either, so as to  remain at the turnpike for as long as possible.} $t=\varepsilon$.

\item[\textbf{3).}] Keep $y(t)$ at the turnpike $\overline{y}$ for $t\in[\varepsilon,T-\varepsilon]$ by using the steady control.

\item[\textbf{4).}] Use the controllability to exit the turnpike with $y(t)$, starting from time $t=T-\varepsilon$, and reach the final target (if any) in time $t=T$.
\end{enumerate}

\noindent
We call such strategies \emph{quasi-turnpikes} -- they are actually commonly used in practical applications, and we shall often make use of such strategies as "test cases" in many proofs (see Figure \ref{fig: quasi.turnpike}), or as an initialization for iterative proofs (as presented in Section \ref{sec: 10} later on). 
Existing approaches (linearization-based, or tailored to the nonlinearity) for obtaining  results for nonlinear problems are also presented. 
In the nonlinear case, an emphasis is put on the non-uniqueness of global minimizers for the stationary optimal control problem, as counter-examples can be produced. This represents a serious warning for both theory and numerics in the nonlinear case, and raises several questions.
We shall also give several broad applications of the turnpike property, spanning  a priori guarantees for the design of efficient discretization algorithms for optimal control problems, long time asymptotics of Hamilton-Jacobi-Bellman equations, and stability properties in deep learning via residual neural networks, among others. Several open problems are sprinkled throughout the text.

\subsection{Outline} This paper is organized as follows.
\medskip

\noindent
\textbf{Section \ref{sec: 2}} is a brief mathematical introduction to the optimal control problems we shall consider in this work, namely minimizing quadratic functionals subject to linear (or nonlinear) PDEs, as well as a formal discussion regarding what ingredients such optimal control problems need to possess in view of exhibiting turnpike. 
\medskip

\noindent
\textbf{Part I} (sections 3--7) presents the methodology for proving turnpike for optimal control problems consisting of minimizing an appropriate quadratic functional subject to a linear PDE. 
This methodology always makes use of a study of the optimality system provided by the Pontryagin Maximum Principle (or simply, the Euler-Lagrange equations), which is a necessary and sufficient condition for optimality, and borrows tools from Riccati theory in the infinite-time horizon. Such ideas have been introduced in the works \cite{porretta2013long, porretta2016remarks, trelat2015turnpike}. 
Several other strategies are also discussed, e.g., those relying on \emph{dissipativity} in the sense of Willems (for which we follow \cite{trelat2018integral}), which brings the turnpike property closer to a Lyapunov method interpretation, or more direct techniques, as per the works \cite{grune2019sensitivity, grune2020exponential}, among others. 

\medskip
\noindent
\textbf{Part II} (sections 8--10) is an extension of the results presented in Part I to the case where the underlying equations are nonlinear. We present a couple of different strategies of proof, namely ones based on linearization of the optimality system (which require smallness assumptions on the target for the state, and on the initial data, as per \cite{trelat2015turnpike, porretta2016remarks}), as well as a new proof (\cite{esteve2020turnpike}), which avoids the use of the optimality system, combining quasi-turnpike and bootstrap arguments (but requires the targets to be steady-states). 

In the nonlinear case, without making specific assumptions on the target (as those above), no uniqueness of solutions for the optimality system may be guaranteed. Consequently, there may be solutions of the optimality system which are not optimal with respect to the cost to be minimized. 
In fact, we present a recent result (\cite{pighin2020nonuniqueness}) which provides a counter-example yielding the non-uniqueness of minimizers for optimal control problems subject to nonlinear elliptic PDEs. This raises several open problems.

\medskip
\noindent
\textbf{Part III} (sections 11-13) presents several direct applications of the turnpike property in numerical analysis and machine learning. For instance, as first observed by \cite{trelat2015turnpike}, the turnpike property can be used to provide an accurate initial guess for shooting problems in numerical optimization. The a priori knowledge of turnpike is also used for more efficient design of model predictive control (MPC) schemes (\cite{grune2019sensitivity}). In the finite-dimensional, linear quadratic optimal control setting, turnpike also provides a precise asymptotic decomposition for the unique viscosity solution of the associated Hamilton-Jacobi-Bellman equation \cite{esteve2020turnpike}. Finally, turnpike-like properties can also be shown to hold for supervised learning problems for residual neural networks, for which it guarantees exponential decay of the approximation error and stability estimates for the controls when the number of layers is increased (\cite{esteve2020large, faulwasser2021turnpikeML}). These, in turn, imply that the relevant information is concentrated in the beginning, and any layers beyond a certain stopping time/layer may be discarded safely.
Section \ref{sec: 14} indicates several topics related to the turnpike property worthy of interest but not treated in depth in this work.

\medskip
\noindent
\textbf{Part IV} concludes this paper, with a couple of major and intertwined open problems.

\subsection{Notation} 
We henceforth suppose that $\Omega\subset\mathbb{R}^d$ is a bounded and smooth domain.
We make use of standard Sobolev spaces -- we denote by $H^k(\Omega)$ Sobolev spaces of order $k\geqslant0$, namely $L^2(\Omega)$ functions with $k$ weak derivatives in $L^2(\Omega)$. Also, $H^1_0(\Omega)$ denotes the space of $H^1(\Omega)$ functions whose Dirichlet trace on the boundary $\partial\Omega$ vanishes. We recall that by the Poincaré inequality, $H^1_0(\Omega)$ is endowed with the norm $\|f\|_{H^1_0(\Omega)}:=\|\nabla f\|_{L^2(\Omega)}$. 
More details on the above concepts can be found in classic texts such as \cite{evans1998partial}.
Furthermore: $1_A$ denotes the characteristic function of a set $A$; $\text{meas}(A)$ denotes the Lebesgue measure of $A$; $\nabla$ and $\Delta$ denote the canonical spatial gradient and Laplacian on $\mathbb{R}^d$ respectively (we shall also use $\nabla_x$ and $\Delta_x$ to further emphasize the spatial differentiation where this may appear ambiguous).

\section{Genesis of the turnpike property} \label{sec: 2}

\subsection{An apparent lack of turnpike} 

Let us begin by considering a problem which will set the tone in what follows. 
We consider the linear, controlled heat equation 
\begin{equation} \label{eq: heat.equation}
\begin{cases}
\partial_t y - \Delta y = u1_{\omega} &\text{ in }(0,T)\times \Omega,\\
y = 0 &\text{ in } (0,T)\times\partial\Omega,\\
y_{|_{t=0}} = y^0 &\text{ in } \Omega.
\end{cases}
\end{equation}
In the above equation, $\Omega\subset\mathbb{R}^d$ is a bounded and smooth domain, $T>0$ is a given time horizon, $y^0\in L^2(\Omega)$ is an initial datum, $u=u(t,x)$ denotes the control actuating within an open and non-empty measurable subset $\omega\subset\Omega$, and $y=y(t,x)$ is the unknown state. 

It is now well-known that given an arbitrary initial datum $y^0\in L^2(\Omega)$, the heat equation \eqref{eq: heat.equation} is controllable to rest (\emph{null-controllable}) in any time $T>0$, and from any non-empty and open subset $\omega\subset\Omega$ (see \cite{lebeau1995controle, fursikov1996controllability}), in the sense that there exists a control $u\in L^2((0,T)\times\omega)$ such that the corresponding state $y\in C^0([0,T]; L^2(\Omega))\cap L^2(0,T; H^1_0(\Omega))$, designating the unique solution to \eqref{eq: heat.equation}, satisfies 
\begin{equation} \label{eq: null.controllability}
y(T,x) = 0 \hspace{1cm} \text{ for a.e. } x \in \Omega.
\end{equation} 
As a matter of fact, due to linearity and time invariance of the heat equation we consider herein, one can also ensure that the terminal zero state in \eqref{eq: null.controllability} can be replaced by any (controlled) steady state of \eqref{eq: heat.equation}. 
In view of the above fact, one can then be interested in finding controls which ensure the null-controllability of \eqref{eq: heat.equation} and which are of minimal norm, e.g., of minimal $L^2$--norm. 
Namely, one could look to solve the following optimal control problem: 
\begin{equation} \label{eq: minimal.norm.control.heat}
\inf_{\substack{u\in L^2((0,T)\times\omega)\\ y \text{ solves} \eqref{eq: heat.equation} \\ \eqref{eq: null.controllability} \text{ holds} }} \|u\|_{L^2((0,T)\times\omega)}^2.
\end{equation}

\begin{figure}[h!]
\includegraphics[scale=0.425]{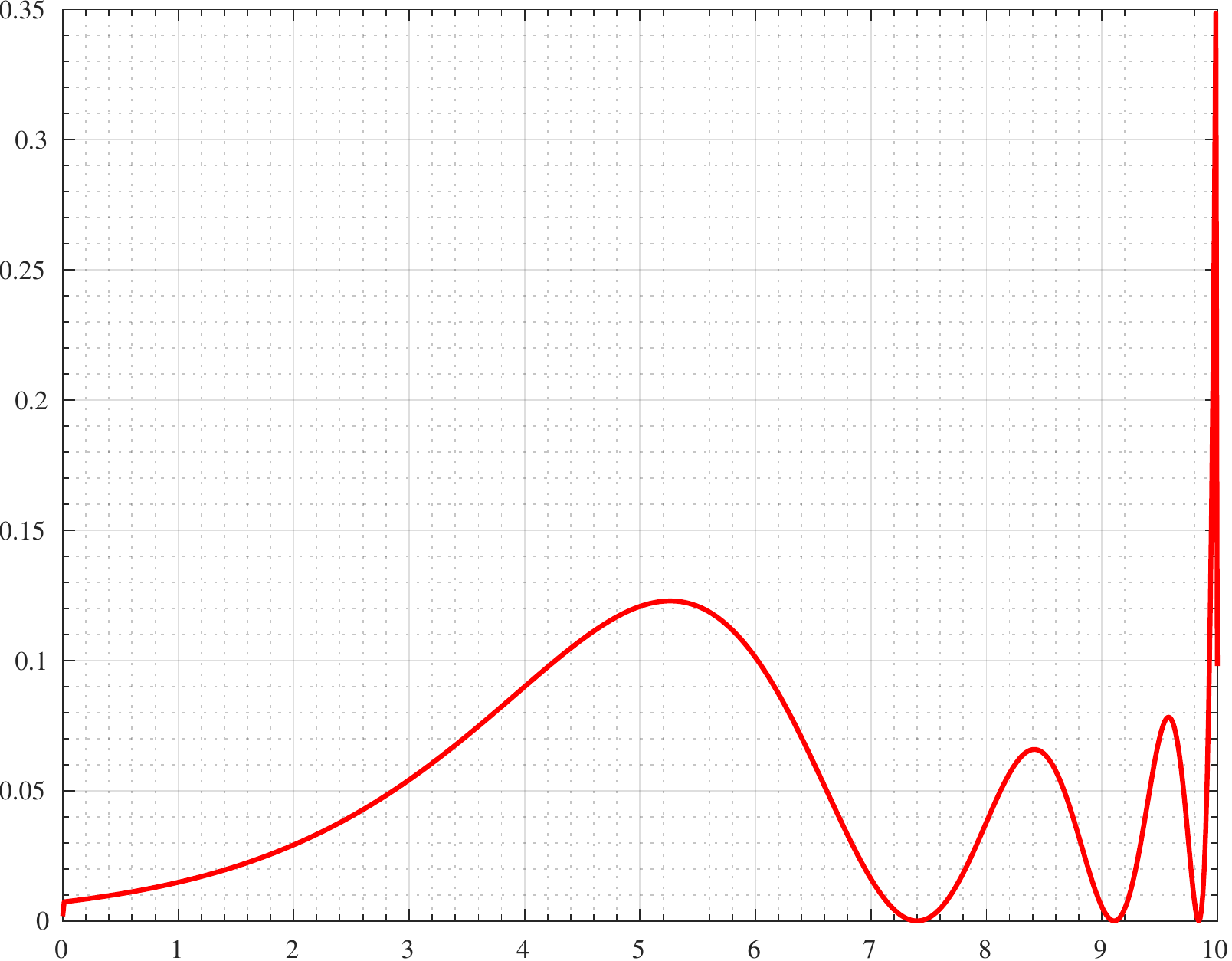}
\hspace{0.1cm}
\includegraphics[scale=0.425]{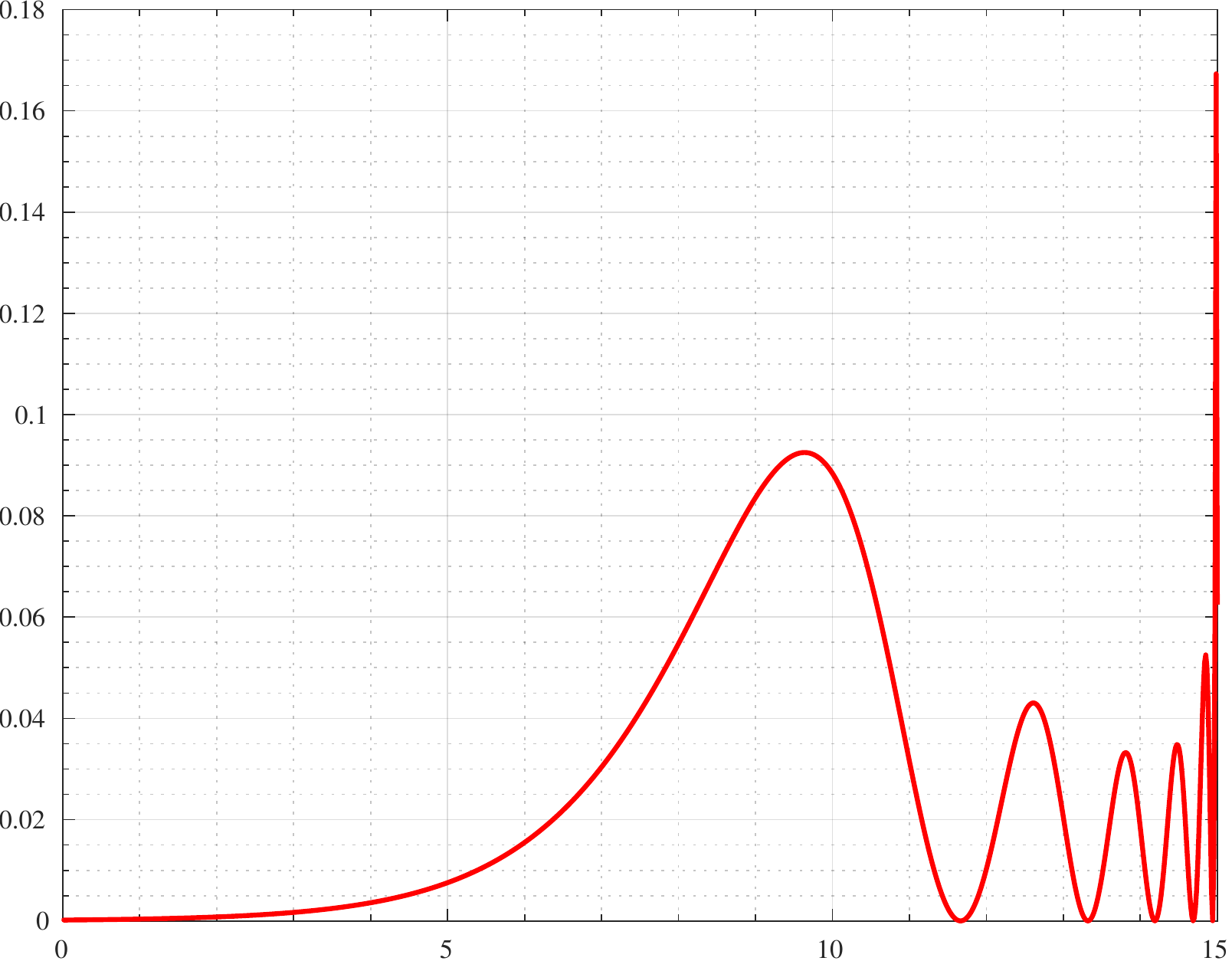}
\caption{The minimal $L^2((0,T)\times\omega)$--norm controls manifest a "lazy" behavior, as they activate and actuate only near $t=T$ in a singular manner. This is in accordance with the fact that they are characterized as the solutions of a specific backward adjoint heat equation. 
We display $t\mapsto \|u(t)\|_{L^2(\omega)}^2$ where $\omega=\left(0.25,0.75\right)^2$ and $\Omega=(0,1)^2$. The initial datum is  $y^0(x_1,x_2):=\sin(\pi x_1)\sin(\pi x_2)$.} 
 \label{fig: hum.heat}
\end{figure}

\noindent
Problem \eqref{eq: minimal.norm.control.heat} admits a unique solution by virtue of the direct method in the calculus of variations -- indeed, the set of admissible controls is a non-empty, closed linear subspace, due to the existence of at least one control ensuring \eqref{eq: null.controllability}, and the functional is coercive, continuous, and convex. 

On another hand, it is also well-known that the free solutions to \eqref{eq: heat.equation} possess a very strong dissipative mechanism, which ensures that they decay exponentially to $0$ as $t\to+\infty$ with rate $\lambda_1(\Omega)>0$
(the first eigenvalue of the Dirichlet Laplacian $-\Delta$). 
In view of this fact, one could be tempted to stipulate that the optimal controls and the controlled optimal solution behave similarly as well. This is however not the case -- see Figure \ref{fig: hum.heat} for an illustrative counterexample.  

To see why the minimal $L^2$--norm control ensuring \eqref{eq: null.controllability} does not satisfy a property of asymptotic simplification, we simply need to see how it is characterized by using the first order optimality conditions. This is the goal of the so-called \emph{Hilbert Uniqueness Method} (HUM, see \cite{lions1988exact, glowinski1995exact}). 
By convex duality, we know that minimal $L^2$--norm exact controls for the heat equation are in fact given by
\begin{equation}
u \equiv p1_\omega \hspace{1cm} \text{ a.e. in } (0,T)\times\omega,
\end{equation}
where $p=p(t,x)$ is the unique solution to the backward (\emph{adjoint}) heat equation
\begin{equation*}
\begin{cases}
\partial_t p + \Delta p = 0 &\text{ in }(0,T)\times\Omega,\\
p = 0 &\text{ in }(0,T)\times\partial\Omega,\\
p_{|_{t=T}} = p^T, &\text{ in }\Omega,
\end{cases}
\end{equation*}
associated to the datum $p^T\in \mathscr{H}$, which is the unique minimizer of the conjugate functional\footnote{Of course, the dual functional $\mathscr{J}^\star$ defined in \eqref{eq: J.star} also admits a unique minimizer by the direct method in the calculus of variations, as the coercivity of $\mathscr{J}^\star$, which is characterized by an \emph{observability inequality} for the adjoint system of the form $$\|p(0)\|_{L^2(\Omega)} \leqslant C(T,\omega)\|p\|_{L^2((0,T)\times\omega)}$$ for some $C(T,\omega)>0$ and for all $p_T\in L^2(\Omega)$, is equivalent to the controllability assumption \eqref{eq: null.controllability} for the forward one (\cite{lions1988exact}).}
\begin{equation} \label{eq: J.star}
\mathscr{J}^\star\left(p^T\right) := \frac12 \int_0^T \int_\omega |p(t,x)|^2 \diff x \diff t + \int_\Omega y^0(x) p(0,x) \diff x
\end{equation}
over the Hilbert space $\mathscr{H}$, which is defined as the completion of $C^\infty_{\mathrm{c}}(\Omega)$ with respect to the norm $\left\|p^T\right\|_{\mathscr{H}}:= \|p\|_{L^2((0,T)\times\omega)}$. 
In fact, this duality is due to the simple observation that 
\begin{equation*}
\int_\Omega y(T, x)p^T(x)\diff x = \int_0^T \int_\omega u(t,x)p(t,x)\diff x \diff t + \int_\Omega y^0(x)p(0,x)\diff x
\end{equation*}
holds for all $p^T \in L^2(\Omega)$.
Summarizing, the singular behavior of the optimal control $u$ near $t=T$ is due to the fact that the space $\mathscr{H}$ is very large -- due to the regularizing effect of the heat equation, any initial (at time $t=T$) datum $p^T$ of the backward heat equation with finite-order singularities away from the control set $\omega$ belongs to $\mathscr{H}$.  (See \cite{munch2010numerical} for a thorough presentation of this ill-posedness.)

The lack of asymptotic simplification is not solely due to the specific setting of the problem \eqref{eq: minimal.norm.control.heat}, and persists for more conventional optimal control problems for the heat equation, such as  
\begin{equation*} \label{eq: heat.control.final.time}
\inf_{\substack{u\in L^2((0,T)\times\omega) \\ y \text{ solves } \eqref{eq: heat.equation}}} \frac{\alpha}{2} \|y(T) - y_d\|_{L^2(\Omega)}^2 + \frac{1}{2} \|u\|_{L^2((0,T)\times\omega)}^2.
\end{equation*}
Here $y_d\in L^2(\Omega)$ denotes a prescribed target design, and $\alpha>0$ is a tunable regularization parameter. The above problem can again be shown to admit a unique minimizer by the direct method in the calculus of variations, this time without requiring any additional coercivity (observability) inequalities.  But then, looking at how the control is characterized, by using the Pontryagin Maximum Principle (or, equivalently, computing the Euler-Lagrange equations), one can see that there exists an adjoint state $p_T \in C^0([0,T]; L^2(\Omega))$ such that the optimal triple\footnote{Here and henceforth, we shall designate, by an underscore $T$, the dependence of a function with respect to $T$.} $(u_T,y_T,p_T)$ is the unique solution to the first-order optimality system 
\begin{equation*}
\begin{cases}
\partial_t y_T - \Delta y_T = u_T1_\omega &\text{ in }(0,T)\times\Omega,\\
\partial_t p_T + \Delta p_T = 0 &\text{ in } (0,T)\times\Omega,\\
y_T = p_T = 0 &\text{ in }(0,T)\times\partial\Omega,\\
{y_T}_{|_{t=0}} = y^0 &\text{ in }\Omega,\\
{p_T}_{|_{t=T}} = \alpha(y_T(T)-y_d) &\text{ in }\Omega,
\end{cases}
\end{equation*}
with 
\begin{equation*}
u_T\equiv p_T1_\omega \hspace{1cm} \text{ a.e. in } (0,T)\times\omega.
\end{equation*} 
Hence, one readily sees that the adjoint state $p_T$ and the state $y_T$ are only \emph{weakly} coupled, through the final condition, and the control is again given by the solution of the adjoint heat equation, hence similar conclusions hold as in the previous case (we provide more detail just below).

These artifacts are not unique to the (somewhat surprising) case of the heat equation; they are also present for analog optimal control problems for the perhaps more intuitive example of the wave equation  
\begin{equation} \label{eq: wave.equation}
\begin{cases}
\partial_t^2 y - \Delta y = u1_\omega &\text{ in }(0,T)\times\Omega,\\
y = 0 &\text{ in }(0,T)\times\partial\Omega,\\
(y, \partial_t y)_{|_{t=0}} = (y^0, y^1) &\text{ in }\Omega.
\end{cases}
\end{equation}
We recall that for any $(y^0, y^1) \in H^1_0(\Omega)\times L^2(\Omega)$, and for any control $u\in L^2((0,T)\times\omega)$, equation \eqref{eq: wave.equation} admits a unique finite-energy solution $y\in C^0([0,T]; H^1_0(\Omega))\cap C^1([0,T]; L^2(\Omega))$. 
Once again, when one considers a problem such as
\begin{equation} \label{eq: final.cost.wave}
\inf_{\substack{u\in L^2((0,T)\times\omega)\\ y \text{solves} \eqref{eq: wave.equation}}} \underbrace{\frac{1}{2}\|y(T)\|_{H^1_0(\Omega)}^2 +\frac{1}{2} \|\partial_t y(T)\|_{L^2(\Omega)}^2}_{:=\phi\big(y(T), \partial_ty(T)\big)} + \frac12 \|u\|^2_{L^2((0,T)\times\omega)},
\end{equation}
(where we took $y_d\equiv0$ for simplicity), the optimality system\footnote{As \eqref{eq: optimality.waves} is not a classical Cauchy problem, the existence of a unique solution to the above system is again due to the fact that the triple $(u_T,y_T,p_T)$ is optimal, hence follows from the Pontryagin Maximum Principle.} takes the form
\begin{equation} \label{eq: optimality.waves}
\begin{cases}
\partial_t^2 y_T - \Delta y_T = p_T1_\omega &\text{ in }(0,T)\times\Omega,\\
\partial_t^2 p_T - \Delta p_T = 0 &\text{ in }(0,T)\times\Omega,\\
y_T=p_T = 0 &\text{ in }(0,T)\times\partial\Omega,\\
(y,\partial_t y)_{|_{t=0}} = (y^0, y^1) &\text{ in }\Omega,\\
(p_T,\partial_t p_T)_{|_{t=T}} = (-\partial_t y_T(T),-\Delta y_T(T)) &\text{ in }\Omega.
\end{cases}
\end{equation}
Moreover, 
\begin{equation*}
u_T\equiv p_T 1_\omega \hspace{1cm} \text{ a.e. in } (0,T)\times\omega.
\end{equation*} 
Again, one sees that the forward state $y_T$ has no effect on the evolution of the adjoint state $p_T$. In other words, since $p_T$ solves a free wave equation, whose solutions conserve energy, $p_T$ will likely oscillate over the entire time interval (or even manifest a periodic pattern, as in the case $d=1$), which would entail the same conclusions for the control $u_T$, and would exclude the validity of the turnpike phenomenon. 

\subsection{The emergence of turnpike}

In view of the preceding discussion, one might ask if the turnpike property appears for optimal control problems for PDEs at all. 
To answer to these doubts, let us focus on the wave equation \eqref{eq: wave.equation}, and consider another staple problem of optimal control, namely the following linear quadratic (LQ) problem
\begin{equation} \label{eq: lq.turnpike.wave.func}
\inf_{\substack{u\in L^2((0,T)\times\omega)\\ y \text{ solves } \eqref{eq: wave.equation}}} \phi\big(y(T), \partial_t y(T)\big)+ \frac12 \int_0^T\|\nabla_x y(t)\|^2_{L^2(\Omega)}\diff t + \frac{1}{2}\int_0^T \|u(t)\|^2_{L^2(\omega)} \diff t.
\end{equation}
Here $\phi$ is defined as in \eqref{eq: final.cost.wave}, and one sees that a tracking term, which tracks the variations of $\nabla_x y(t)$ over all the time interval $(0,T)$, was added. 
The optimality system satisfied by an optimal triple $(u_T, y_T, p_T)$ (this time, for \eqref{eq: lq.turnpike.wave.func}) now reads
\begin{equation} \label{eq: optimality.system.wave}
\begin{cases}
\partial_t^2 y_T - \Delta y_T = p_T1_{\omega} &\text{ in }(0,T)\times\Omega,\\
\partial_t^2 p_T - \Delta p_T = \Delta y_T &\text{ in }(0,T)\times\Omega,\\
y_T = p_T = 0 &\text{ in }(0,T)\times\partial\Omega,\\
(y_T,\partial_t y_T)_{|_{t=0}} = (y^0,y^1) &\text{ in }\Omega,\\
(p_T,\partial_t p_T)_{|_{t=T}} = (-\partial_t y_T(T), -\Delta y_T(T)) &\text{ in }\Omega,
\end{cases}
\end{equation}
and, again, $u_T\equiv p_T1_\omega$.  

Now, both the forward and adjoint state are \emph{strongly} coupled due to the presence of the tracking term of $\nabla_x y_T(t)$ in \eqref{eq: lq.turnpike.wave.func}, which manifests itself as $\Delta y_T$ in \eqref{eq: optimality.system.wave}. 
It is precisely this coupling that will cause the occurrence of the turnpike property. 

Let us give a heuristic argument, following \cite{zuazua2017large}, to reinforce this claim. 
Assume that $\omega=\Omega$ in \eqref{eq: optimality.system.wave}, and let us ignore initial and terminal conditions. 
We write the solution $(y_T,p_T)$ in Fourier series as
\begin{equation*}
\begin{bmatrix}y_T(t,x)\\p_T(t,x)\end{bmatrix} = \sum_{j=1}^\infty \begin{bmatrix}\widehat{y}_{j}\\ \widehat{p}_{j} \end{bmatrix} e^{\mu_j t} \varphi_j(x) \hspace{1cm} \text{ for } (t,x) \in (0,T)\times\Omega,
\end{equation*}
for suitable frequencies $\mu_j$ and scalar coefficients $(\widehat{y}_{j}, \widehat{p}_{j})$; here $\{\varphi_j\}_{j=1}^\infty$ and $\{\lambda_j\}_{j=1}^\infty$ denote the orthonormal basis of eigenfunctions and corresponding eigenvalues of the Dirichlet Laplacian $-\Delta: H^1_0(\Omega)\to H^{-1}(\Omega)$, thus satisfying $-\Delta \varphi_j = \lambda_j \varphi_j$ in $\Omega$.
It is then readily seen that
\begin{equation*}
\begin{cases}
\widehat{y}_j\Big(\mu^2_j + \lambda_j\Big)&=\widehat{p}_j\\
\widehat{p}_j\Big(\mu^2_j + \lambda_j\Big)&=-\lambda_j \widehat{y}_j.
\end{cases}
\end{equation*}
In view of this, we have $\big(\mu^2_j+\lambda_j\big)^2+\lambda_j=0$ for $j\geqslant1$, 
clearly yielding four pairs of complex eigenvalues 
\begin{equation*}
\mu_j=\pm\sqrt{-\lambda_j\pm i\sqrt{\lambda_j}}, 
\end{equation*}
namely two pairs of conjugates -- two with strictly positive real parts and two with strictly negative ones, uniformly away from the imaginary axis $\{\Re(z)=0\}$ as $j\to+\infty$. 
This means that the solutions of the optimality system are constituted by the superposition of two time evolving components of oscillatory nature, one decaying exponentially as $t\to+\infty$ while the other grows exponentially.

\begin{figure}[h!]
\centering
\includegraphics[scale=0.42]{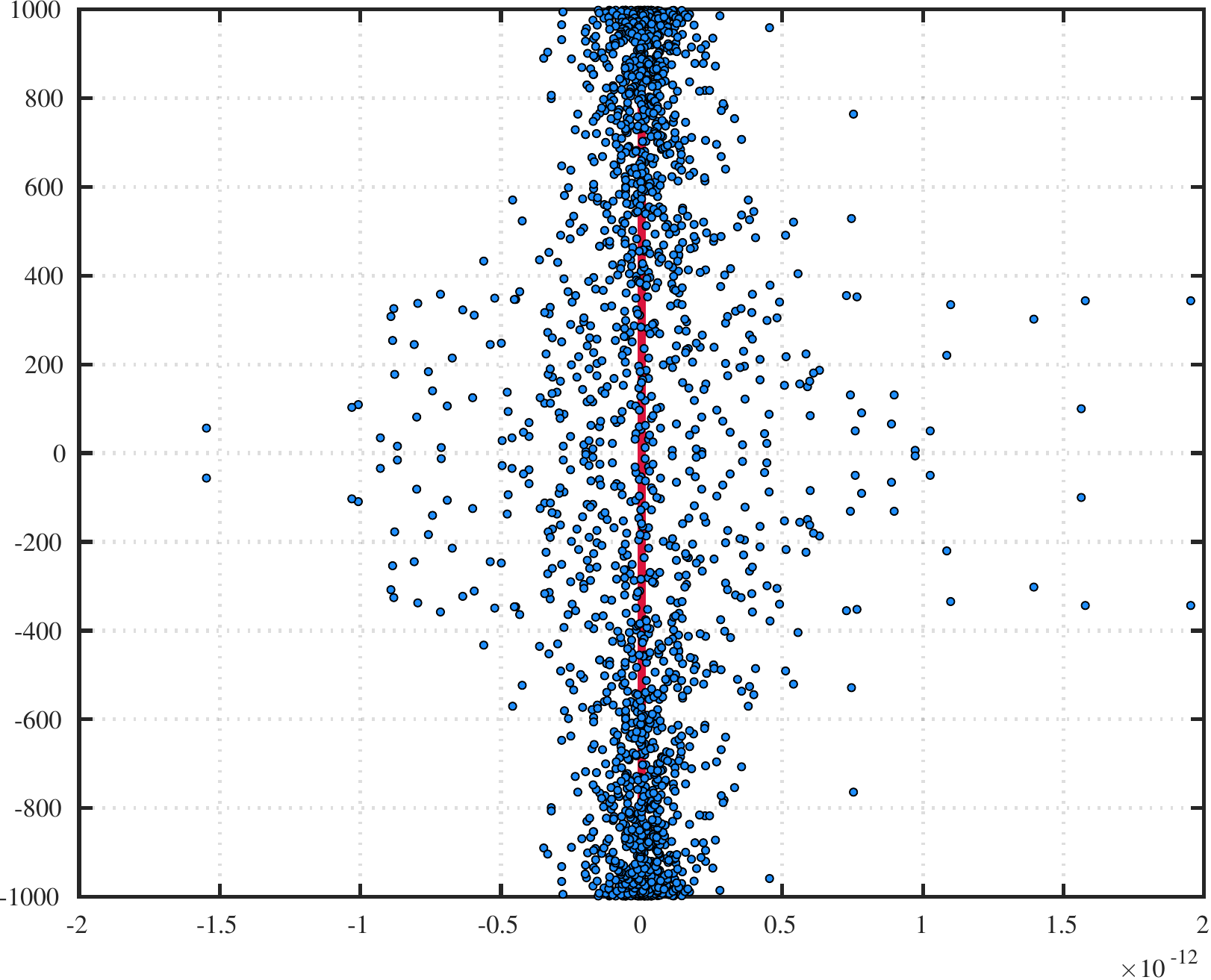}
\hspace{0.25cm}
\raisebox{2.1mm}{\includegraphics[scale=0.42]{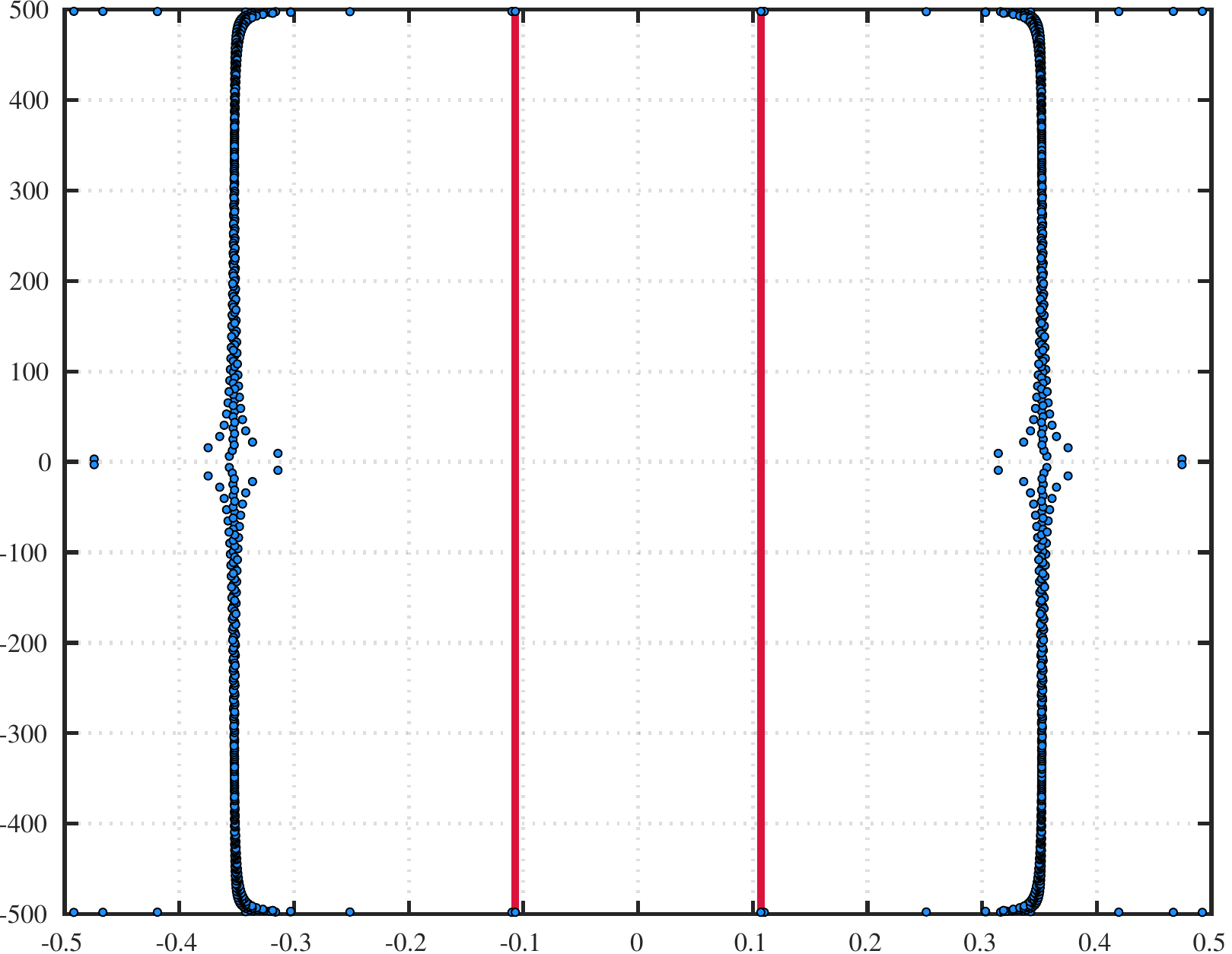}}
\caption{The spectrum (subset of $\mathbb{C}$) of the optimality system for the finite-difference wave equation (with $n=500$ points) with distributed control. The Kalman rank condition is satisfied.
(\emph{Left}) Spectrum of the optimality system without a tracking term. (\emph{Right}) Spectrum of the optimality system with an $L^2$-tracking term. We see a spectral gap for the real part of the eigenvalues in the presence of a tracking term; the symmetry indicates that those with negative real part entail a decay for the forward wave components, and those with a positive real part a decay for the backward ones. In the absence of a tracking term, the gap in the real part of the spectrum collapses to zero, entailing an unstable and unsteady nature of the optimality system. This finite-dimensional illustration is theoretically corroborated in Section \ref{sec: 5}.}
\label{fig: spectral.dic}
\end{figure}

This is contrary to the case without the tracking term for $\nabla_x y(t)$ (i.e., \eqref{eq: final.cost.wave}), where by writing the optimality system \eqref{eq: optimality.waves} in Fourier series as above (again, with $\omega=\Omega$), one sees that  
$\mu^2_j + \lambda_j = 0$ holds for $j\geqslant1$. 
Accordingly, $\mu_j$ would be purely imaginary, thus leading to solutions of purely oscillatory nature, in agreement with previous observations. In this case, in particular, the adjoint state $p_T$ and accordingly, the control $u_T$ will have a purely oscillatory behavior without never stabilizing around the optimal steady adjoint state $\overline{p}$ and control $\overline{u}$. 

While the above argument is solely heuristic, a similar diagonalization strategy for the optimality system may be used for a rigorous proof, as done in \cite{trelat2015turnpike} (see Section \ref{sec: 5} for more details).
We depict a numerical example of this spectral dichotomy for a finite-dimensional example in Figure \ref{fig: spectral.dic}. 

This behavior is compatible with the turnpike property, according to which the optimal solution $(y_T, u_T\equiv p_T1_\omega)$ should be close to the optimal steady state configuration $(\overline{y},\overline{u})$ during most of the time horizon of control $[0,T]$, when $T$ is large.
The optimal steady state configuration $(\overline{y}, \overline{u})$ is that in which solely the time is dropped, namely, $\overline{y}$ solves
\begin{equation} \label{eq: poisson.control}
\begin{cases}
-\Delta y = u1_\omega &\text{ in }\Omega,\\
y = 0 &\text{ on }\partial\Omega,
\end{cases}
\end{equation}
with $\overline{u}$ being the unique minimizer of 
\begin{equation*} 
\inf_{\substack{u \in L^2(\omega)\\ y \text{ solves} \eqref{eq: poisson.control}}}\frac12 \|\nabla y\|^2_{L^2(\Omega)} + \frac12 \|u\|^2_{L^2(\omega)}.
\end{equation*}
The above discussion leads us to conclude that the turnpike property does not hold for an optimal control problem simply because the underlying ordinary or partial differential equation has a dissipative and stabilizing (or controllable) character when time is large. On the contrary, depending on the cost to be minimized, turnpike may also hold for oscillatory systems such as the wave equation. 
The bottom line is that, to ensure turnpike, solely a controllability or stabilizability mechanism is needed for the underlying system and not necessarily a decay of the free dynamics; moreover, one requires sufficient coercivity of the cost functional with respect to the state of that system.

\part{Linear theory}

\section{The heat equation} \label{sec: 3}

We begin by presenting the theory for the linear heat equation with distributed controls. 
As seen in what precedes, even-though the heat equation is a dissipative and controllable system, for the turnpike property to appear, one also needs some coercivity (namely, observability) of the state in the functional to be minimized.
We will focus on a specific linear quadratic (LQ) problem for the heat equation with distributed control (i.e., \eqref{eq: heat.equation}) to avoid introducing too many assumptions and unnecessary technicalities in the proofs -- more general statements are given in subsequent sections. 
This framework will nonetheless contain most of the specific features and can readily be generalized. 

The full structure of the control problem we consider matters, in addition to penalizing the full state. For instance, the fact that the control enters in a distributed way, actuating within an open and non-empty subset $\omega\subset\Omega$, ensures the presence of a controllability mechanism which will promote the appearance of the turnpike property. 
The situation is different in the case where the control actuates at a nodal point of the Laplacian (i.e., one has $u(t)\delta_{x_0}$ instead of $u1_\omega$ in \eqref{eq: heat.equation}, with $x_0$ being a zero of an eigenfunction of the Laplacian), in which case controllability fails to hold.

We shall consider the following linear quadratic (LQ) optimal control problem
\begin{equation} \label{eq: pb.turnpike.heat}
\inf_{\substack{u \in L^2((0,T)\times\omega)\\ y \text{ solves } \eqref{eq: heat.equation}}} \frac12\int_0^T \|y(t)-y_d\|^2_{L^2(\omega_\circ)} \diff t + \frac12\int_0^T \|u(t)\|^2_{L^2(\omega)} \diff t,
\end{equation}
where $\omega_\circ\subset\Omega$ is open and non-empty\footnote{In this example, we are minimizing the discrepancy of the state $y(t)$ to the design target $y_d$ only in, possibly, a small subdomain $\omega_\circ$ of $\Omega$. Some PDEs (e.g. the wave equation) will require further geometric assumptions on $\omega_\circ$ for turnpike to be induced, as sufficient observation of the state inscribed within the functional itself will be needed (in addition to similar geometric assumptions on $\omega$ to ensure controllability).}, and $y_d\in L^2(\omega_\circ)$. 
We shall henceforth focus on running targets $y_d$ which are independent of time. This is rather natural as steady optimal control problems, used in applications, and described in the introduction, typically assume such a setup. 
But, in fact, many results and insights transfer to specific settings of time dependent targets (see Section \ref{sec: periodic.turnpike}).

The corresponding steady optimal control problem then reads
\begin{equation} \label{eq: optimal.steady.heat}
\inf_{\substack{u\in L^2(\omega) \\ y \text{ solves } \eqref{eq: steady.heat}}} \frac12\|y-y_d\|_{L^2(\omega_\circ)}^2 + \frac12\|u\|_{L^2(\omega)}^2,
\end{equation}
where the underlying PDE constraint is given by the linear controlled Poisson equation
\begin{equation} \label{eq: steady.heat}
\begin{cases}
-\Delta y = u1_\omega &\text{ in }\Omega,\\
y=0 &\text{ on }\partial\Omega.
\end{cases}
\end{equation}
By virtue of the direct method in the calculus of variations, one may easily show that problem \eqref{eq: optimal.steady.heat} admits a unique minimizer $\overline{u}\in L^2(\omega)$ and there exists a unique optimal steady state $\overline{y}\in H^2(\Omega)\cap H^1_0(\Omega)$, solution to \eqref{eq: steady.heat} corresponding to $\overline{u}$.
As discussed in preceding paragraphs, the turnpike property for the optimal pair $(u_T, y_T)$ solving \eqref{eq: pb.turnpike.heat} would mean that $(u_T,y_T)$ is "near" $(\overline{u}, \overline{y})$, namely the optimal solution to the steady problem \eqref{eq: optimal.steady.heat} (the \emph{turnpike}) during most of the time horizon $[0,T]$ with an exception of two boundary layers near $t=0$ and $t=T$.  This is illustrated by the following result, due to \cite{porretta2013long}.

\begin{theorem}[\cite{porretta2013long}] \label{thm: porretta.zuazua.1}
Let $y^0 \in L^2(\Omega)$ and $y_d\in L^2(\omega_\circ)$ be fixed. 
There exist a couple of constants $C>0$ and $\lambda>0$, independent of $y^0$ and $y_d$, such that for any large enough $T>0$, the unique solution $(u_T, y_T)$ to \eqref{eq: pb.turnpike.heat} satisfies
\begin{align}\label{eq: porretta.zuazua.estimate}
\|y_T(t) -\overline{y}\|_{L^2(\Omega)} &+ \|u_T(t) - \overline{u}\|_{L^2(\omega)}\nonumber\\
&\leqslant C\left(\left\|y^0-\overline{y}\right\|_{L^2(\Omega)}e^{-\lambda t} + \|\overline{p}\|_{L^2(\Omega)}e^{-\lambda(T-t)}\right),
\end{align}
for a.e. $t\in [0,T]$, where $(\overline{u}, \overline{y})$ denotes the unique solution to \eqref{eq: optimal.steady.heat}, and $\overline{p}\in H^1_0(\Omega)$ is the optimal steady adjoint state in \eqref{eq: steady.heat.optimality}. 
\end{theorem}

There exist (at least) a couple of ways to prove Theorem \ref{thm: porretta.zuazua.1}. 
Both of them rely on analyzing the decay properties of the corresponding optimality systems, found by computing the Euler-Lagrange equations at the optimal pairs $(u_T,y_T)$ and $(\overline{u},\overline{y})$ respectively. 
In the LQ case we present herein, these systems are necessary and sufficient conditions for optimality. 
As seen, for instance in \eqref{eq: heat.optimality.T}, the optimality system for the evolutionary problem is a coupled system, consisting of a forward heat equation for the state $y_T$, and a backward heat equation for the adjoint state $p_T$. Due to the coupling of states which evolve in different directions in time, it is not straightforward to obtain a full understanding of the decay properties of the system. 
\begin{itemize}
\item
In the original proof of \cite{porretta2013long}, which we present just below, one looks to uncouple the system by making use of some feedback operator --a rather classical procedure, described in \cite{lions1971optimal, lions1988exact} for instance. 
This feedback is constructed by making use of the Riccati operator from the associated infinite-time horizon problem (but without actually solving an infinite-dimensional Riccati equation), which is known to provide a feedback control ensuring exponential stability in infinite time. To take into account the final time horizon $T$, one cuts-off the preceding feedback by means of a corrector term, which will be shown to decay as $\mathcal{O}\left(e^{-(T-t)}\right)$ for $t\in[0,T]$.
\begin{figure}[h!]
\center
\includegraphics[scale=0.9]{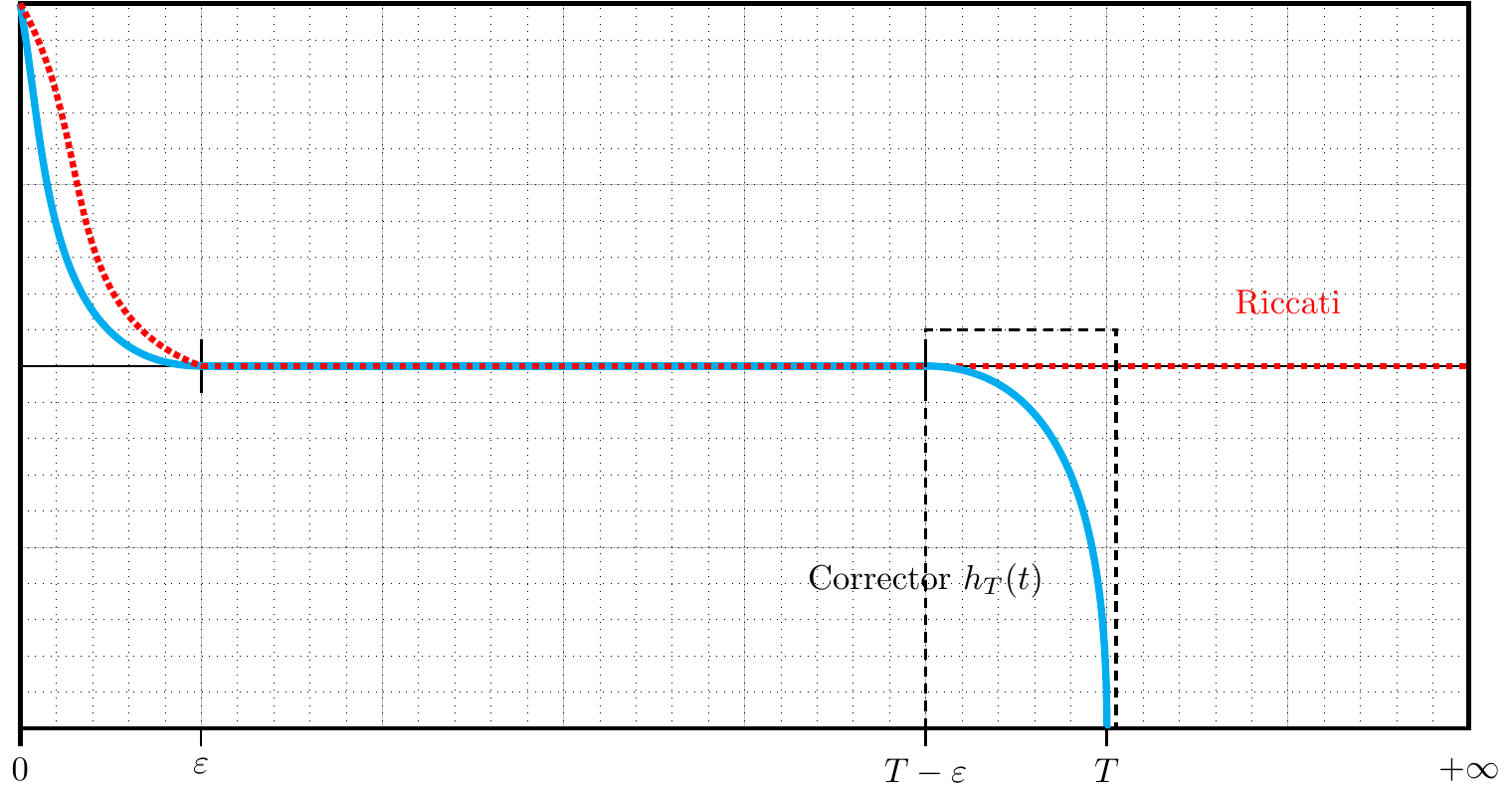}
\caption{The {\bf Riccati-inspired strategy}: we use the feedback given by the infinite-time horizon Riccati operator, and correct it near time $t=T$ by means of the "auxiliary" adjoint state $h_T(t)$.}
\end{figure}
\smallskip

\item An alternative strategy, introduced in \cite{trelat2015turnpike} (see also \cite{trelat2018steady}) which is especially transparent in the context of finite-dimensional linear control systems (discussed in section \ref{sec: 5}), consists in subtracting the evolutionary and stationary optimality system, and looking at the resulting system as a \emph{shooting problem}. The matrix appearing in this shooting problem can be diagonalized, again making use of the infinite-time horizon Riccati operator, resulting in an uncoupled system whose matrix is hyperbolic. Consequently, the first part of the state will decay forward in time, while the other will decay backward in time, yielding the double-arc exponential turnpike estimate.
\end{itemize}

\noindent
Before proceeding with further comments, we shall provide a sketch of the proof of \cite{porretta2013long}, namely following the first strategy, indicating the main steps.

\begin{proof}
Let us begin by writing down the first-order optimality systems for both the evolutionary and steady optimal control problems. They read, respectively, as 
\begin{equation} \label{eq: heat.optimality.T}
\begin{cases}
\partial_t y_T - \Delta y_T = p_T1_\omega &\text{ in }(0,T)\times\Omega,\\
\partial_t p_T + \Delta p_T = (y_T-y_d)1_{\omega_\circ} &\text{ in }(0,T)\times \Omega,\\
y_T=p_T=0 &\text{ in }(0,T)\times\partial\Omega,\\
{y_T}_{|_{t=0}} = y^0 &\text{ in }\Omega,\\
{p_T}_{|_{t=T}} = 0 &\text{ in }\Omega,
\end{cases}
\end{equation}
and
\begin{equation} \label{eq: steady.heat.optimality}
\begin{cases}
-\Delta \overline{y} = \overline{p}1_\omega &\text{ in }\Omega,\\
-\Delta \overline{p} = -(\overline{y}-y_d)1_{\omega_\circ} &\text{ in }\Omega,\\
\overline{y}=\overline{p}=0 &\text{ on }\partial\Omega.
\end{cases}
\end{equation}
Moreover, 
\begin{equation} \label{eq: uequivp}
u_T\equiv p_T1_\omega \hspace{1cm} \text{a.e. in } (0,T)\times\omega
\end{equation} 
and 
\begin{equation*}
\overline{u}\equiv \overline{p}1_\omega \hspace{1cm} \text{ a.e. in } \omega.
\end{equation*} 
Again, these systems can be found by either applying the Pontryagin Maximum Principle, or straightforwardly computing the Euler-Lagrange equations.
We now structure the proof in three steps.
\smallskip

\noindent
\textbf{Step 1. Riccati stability when $y_d\equiv0$.}
We shall begin by firstly considering the reference case in which $y_d\equiv 0$, and, unless otherwise stated, the triple $(u_T,y_T,p_T)$ refers specifically to this case.
We shall also denote by $\mathscr{J}_T^0(\cdot)$ the functional defined in \eqref{eq: pb.turnpike.heat} with $y_d\equiv0$.  

For $T>0$, we define the operator $\mathscr{E}(T): L^2(\Omega)\to L^2(\Omega)$ by
\begin{equation*} 
\mathscr{E}(T)y^0 := -p_T(0),
\end{equation*}
for $y^0 \in L^2(\Omega)$. 
Clearly, $\mathscr{E}(T)$ is linear. Now, by multiplying the first equation in \eqref{eq: heat.optimality.T} by $p_T$ and integrating over $(0,T)\times\Omega$, we derive the variational identities
\begin{align} 
\left\langle \mathscr{E}(T)y^0, y^0\right\rangle_{L^2(\Omega)} &= \int_0^T \|y_T(t)\|^2_{L^2(\omega_\circ)} \diff t + \int_0^T \|p_T(t)\|_{L^2(\omega)}^2\diff t \nonumber\\
&\stackrel{\eqref{eq: uequivp}}{=}\inf_{\substack{u\in L^2((0,T)\times\omega)\\ y\text{ solves } \eqref{eq: heat.equation}}}\mathscr{J}_T^0(u).
\label{eq: ET.variational}
\end{align} 
From \eqref{eq: ET.variational}, we may gather two crucial clues. 
\begin{itemize}
\item 
First of all, we see that $\mathscr{E}(T)$ is non-decreasing with respect to $T>0$. Indeed, for $t_1\leqslant t_2$, let $u_{t_1}$ and $u_{t_2}$ designate the minimizers of $\mathscr{J}_{t_1}^0$ and $\mathscr{J}_{t_2}^0$ respectively. 
From \eqref{eq: ET.variational} we see that
\begin{align*}
\left\langle \mathscr{E}(t_1)y^0, y^0\right\rangle_{L^2(\Omega)} = \mathscr{J}_{t_1}^0(u_{t_1})\leqslant\mathscr{J}_{t_1}^0(u_{t_2})&\leqslant\mathscr{J}_{t_2}^0(u_{t_2})\\
&=\left\langle \mathscr{E}(t_2)y^0, y^0\right\rangle_{L^2(\Omega)},
\end{align*}
as desired.
\smallskip
\item 
On another hand, from \eqref{eq: ET.variational}, we can also ensure that $\mathscr{E}(T)$ is bounded uniformly in $T>0$. 
Indeed, using the exponential stabilizability of the heat semigroup\footnote{Here, as a matter of fact, we use the exponential decay of the semigroup, but for more general settings in which exponential decay does not hold (e.g. some parabolic equations with lower order terms, the wave equation, and so on), exponential stabilizability by means of some feedback operator suffices (which in turn, is implied by controllability). This is addressed in Section \ref{sec: 4}.} in conjunction with Datko's theorem (\cite{datko1972uniform}), we may find that
\begin{align} \label{eq: some.energy}
\Bigg|\Big\langle &\,p_T(0),y^0\Big\rangle_{L^2(\Omega)}\Bigg|\leqslant\left\|p_T(0)\right\|_{L^2(\Omega)}\left\|y^0\right\|_{L^2(\Omega)}\nonumber\\ 
&\leqslant C_0\left(\int_0^T\|y_T(t)\|^2_{L^2(\omega_\circ)}\diff t+\int_0^T \|p_T(t)\|^2_{L^2(\omega)}\diff t\right)^{\sfrac12}\left\|y^0\right\|_{L^2(\Omega)}
\end{align}
holds for some constant $C_0>0$ independent of $T>0$ and $y^0$. Combining \eqref{eq: some.energy} with \eqref{eq: ET.variational} leads us to the desired conclusion.
For completeness, let us briefly sketch the proof of \eqref{eq: some.energy}. Fix an arbitrary $\psi^0\in L^2(\Omega)$ and consider
\begin{equation} \label{eq: heat.psi}
\begin{cases}
\partial_t \psi-\Delta \psi=0&\text{ in }(0,T)\times\Omega,\\
\psi=0&\text{ on }(0,T)\times\partial\Omega,\\
\psi_{|_{t=0}}=\psi^0&\text{ in } \Omega.
\end{cases}
\end{equation}
Multiplying the equation for $p_T$ in \eqref{eq: heat.optimality.T} by $\psi$ and integrating, and then using Cauchy-Schwarz, we find 
\begin{align} \label{eq: before.datko.psi}
\left|\Big\langle p_T(0),\psi^0\Big\rangle_{L^2(\Omega)}\right|&=\left|\int_0^T\int_{\Omega} y_T(t,x)1_{\omega_\circ} \psi(t,x)\diff x\diff t\right|\\
&\leqslant\left(\int_0^T\|y_T(t)\|_{L^2(\omega_\circ)}^2\diff t\right)^{\sfrac12}\left(\int_0^T\|\psi(t)\|_{L^2(\Omega)}^2\diff t\right)^{\sfrac12}\nonumber.
\end{align} 
By virtue of the exponential decay of solutions to \eqref{eq: heat.psi}, and Datko's theorem (\cite{datko1972uniform}), there exists a constant $C_0>0$, independent of $T$ and $y^0$, such that
\begin{equation} \label{eq: datko.psi}
\int_0^T\|\psi(t)\|_{L^2(\Omega)}^2\diff t\leqslant C_0\left\|\psi^0\right\|_{L^2(\Omega)}^2.
\end{equation}
Applying \eqref{eq: datko.psi} to \eqref{eq: before.datko.psi}, and choosing $\psi^0:=p_T(0)$, leads us to \eqref{eq: some.energy}.
\end{itemize}
The limit $\lim_{T\to+\infty}\langle \mathscr{E}(T)y^0,y^0\rangle_{L^2(\Omega)}$ thus exists, and is actually characterized in terms of the infinite-time horizon (the \emph{regulator}) problem, defining a limit operator $\mathscr{E}_\infty$.  Actually,  $\mathscr{E}_\infty: L^2(\Omega)\to L^2(\Omega)$ may be characterized as
\begin{equation*}
\mathscr{E}_\infty y^0 :=-p_\infty(0),
\end{equation*}
for $y^0\in L^2(\Omega)$, where, in this case, the pair $(y_\infty,p_\infty)$ solves the optimality system in an infinite-time horizon:
\begin{equation} \label{eq: infinite.time.horizon.heat}
\begin{cases}
\partial_t y_\infty - \Delta y_\infty = p_\infty1_\omega &\text{ in }(0,+\infty)\times\Omega,\\
\partial_t p_\infty + \Delta p_\infty = y_\infty 1_{\omega_\circ} &\text{ in }(0,+\infty)\times \Omega,\\
y_\infty=p_\infty=0 &\text{ in }(0,+\infty)\times\partial\Omega,\\
{y_\infty}_{|_{t=0}} = y^0 &\text{ in }\Omega,\\
p_\infty(t) \xrightarrow[L^2(\Omega)]{} 0 &\text{ as }t\to+\infty.
\end{cases}
\end{equation}
(We refer to \cite[Lemma 3.9]{porretta2013long} for the complete proof of this fact.)
Observe that by the semigroup property (namely, time invariance), we have
\begin{equation*}
p_\infty(t) =-\mathscr{E}_\infty y_\infty(t),
\end{equation*}
for $t\in(0,+\infty)$. Hence, the first equation in the infinite-time horizon problem \eqref{eq: infinite.time.horizon.heat} rewrites as $\partial_t y_\infty + M y_\infty = 0$ in $(0,+\infty)\times\Omega$, where 
\begin{equation*}
M:= -\Delta + \mathscr{E}_\infty 1_\omega.
\end{equation*}
In other words, the system \eqref{eq: infinite.time.horizon.heat} is now uncoupled.
Furthermore, it can be seen that $$f\mapsto-\langle\mathscr{E}_\infty f,f\rangle_{L^2(\Omega)}$$ 
is a Lyapunov functional for the first equation in \eqref{eq: infinite.time.horizon.heat}; indeed,
\begin{align*}
\frac{\diff}{\diff t}\big\langle-\mathscr{E}_\infty y_\infty(t),y_\infty(t)\big\rangle_{L^2(\Omega)} &= \frac{\diff}{\diff t}\big\langle p_\infty(t),y_\infty(t)\big\rangle_{L^2(\Omega)} \\ 
&= -\Big(\|y_\infty(t)\|_{L^2(\omega_\circ)}^2 + \|p_\infty(t)\|_{L^2(\omega)}^2\Big),
\end{align*}
for all $t\geqslant0$. From this, it can then rigorously be shown (again making use of Datko's theorem) that the operator $M:H^2(\Omega)\cap H^1_0(\Omega)\to L^2(\Omega)$ generates a strongly-continuous and exponentially stable semigroup $\{e^{tM}\}_{t\geqslant0}$ on $L^2(\Omega)$ -- namely, there exists $\lambda>0$ such that 
\begin{equation} \label{eq: exp.M}
\left\|e^{tM} y^0\right\|_{L^2(\Omega)} \leqslant e^{-\lambda t} \left\|y^0\right\|_{L^2(\Omega)},
\end{equation}
holds for all $t\geqslant 0$. Finally, it can furthermore be shown (we omit the proof, which can be found in \cite[Lemma 3.9]{porretta2013long})\footnote{We note that both of these conclusions are actually well-known facts, and in addition to the proof found in \cite[Lemma 3.9]{porretta2013long}, we refer the reader to \cite[Part IV, Chapter 4, Theorem 4.4, p. 241]{zabczyk2020mathematical}, and also to \cite[Sections 8-10]{lions1988exact} and the references therein.}, that there exists a constant $C_1>0$ (independent of $T$) such that
\begin{equation} \label{eq: estimate.operators}
\left\|\mathscr{E}(T)-\mathscr{E}_\infty\right\|_{\mathscr{L}(L^2(\Omega))} \leqslant C_1 e^{-\lambda T},
\end{equation}
holds for all $T\geqslant0$; here, $\lambda>0$ is the same as in \eqref{eq: exp.M}.
\smallskip

\noindent
\textbf{Step 2.  Uncoupling the optimality system with a correction near $t=T$.}
We now come back to the case $y_d\not\equiv0$. Note that when $y_d\equiv0$, and $T=+\infty$, we could readily uncouple the optimality system through the Riccati feedback operator $\mathscr{E}_\infty$. In the case $T<+\infty$, to match the terminal condition for the adjoint at $t=T$, we need to slightly correct this Riccati feedback.
To this end, let us define $h_T\in C^0([0,T]; L^2(\Omega))$ as the unique weak solution to the system
\begin{equation*} \label{eq: corrector.system}
\begin{cases}
-\partial_t h_T + \big(-\Delta + \mathscr{E}(T-t)1_\omega\big)h_T = 0 &\text{ in }(0,T)\times\Omega,\\
h_T = 0 &\text{ in }(0,T)\times\partial\Omega,\\
{h_T}_{|_{t=T}} = -\overline{p} &\text{ in }\Omega. 
\end{cases}
\end{equation*}
Note that, here, $\mathscr{E}(T-t)y^0:=-p_{T-t}(0)$, namely, is defined as in the first step, with $p_{T-t}$ designating the unique solution to the second equation in \eqref{eq: heat.optimality.T} set on $(0,T-t)$, with $y_d\equiv0$.
We introduce $h_T$ precisely in order to uncouple the optimality system: the key observation is that using judiciously the definition of $\mathscr{E}(t)$, one gathers
\begin{equation*}
\int_\Omega \big(p_T(t)-\overline{p}\big)f \diff x = \int_\Omega (y_T(t)-\overline{y})\big(\mathscr{E}(T-t)f\big)\diff x + \int_\Omega h_T(t)f \diff x,
\end{equation*}
for all $f\in L^2(\Omega)$. Whence, we see that the adjoint state $p_T$ can be represented by the affine feedback law
\begin{equation} \label{eq: affine.feedback}
p_T(t)-\overline{p} = \mathscr{E}(T-t)\big(y_T(t)-\overline{y}\big) + h_T(t),
\end{equation}
for $t\in[0,T]$.
One sees that $h_T(t)$ was designed to play the role of a corrector, taking care of the final arc near time $t=T$.
By using the above feedback, the optimality system \eqref{eq: heat.optimality.T} can then be uncoupled by seeing that the optimal trajectory $y_T$ satisfies
\begin{align*}
\partial_t y_T - \Delta y_T &= -\overline{p}1_\omega - \mathscr{E}(T-t)1_\omega\big(y_T(t)-\overline{y}\big)-h_T1_\omega 
\end{align*}
in $(0,T)\times\Omega$, the above identity being interpreted in the weak sense.
\smallskip

\noindent
\textbf{Step 3. Energy estimates for the uncoupled system.}
Let us now set $\zeta(t):=y_T(t)-\overline{y}$; since $\overline{y}$ solves the first equation in \eqref{eq: steady.heat.optimality}, by the Duhamel formula one finds
\begin{equation*}
\zeta(t)=e^{tM}\Big(y^0-\overline{y}\Big)+\int_0^t e^{(t-s)M}\Big(\mathcal{K}(s)\zeta(s)-h_T(s)1_\omega\Big)\diff s,
\end{equation*}
where $\mathcal{K}(s):=\Big(\mathscr{E}_\infty-\mathscr{E}(T-s)\Big)1_\omega$.
By the Duhamel formula once again,
\begin{equation*}
h(t) = -e^{(T-t)M}\overline{p} + \int_{t}^{T} e^{(t-s)M}\mathcal{K}(s)h_T(s)\diff s,
\end{equation*}
where the identity is understood in the $L^2(\Omega)$--sense. 
By using Gr\"onwall's lemma along with \eqref{eq: estimate.operators} and \eqref{eq: exp.M}, one finds
\begin{equation} \label{eq: estimate.h}
\|h_T(t)\|_{L^2(\Omega)} \leqslant C_1 e^{-\lambda(T-t)} \|\overline{p}\|_{L^2(\Omega)},
\end{equation}
for $t\in[0,T]$.  Using Gr\"onwall's lemma once more, along with \eqref{eq: exp.M}, \eqref{eq: estimate.operators}, and \eqref{eq: estimate.h} to $\zeta(t)$ leads us to 
\begin{equation} \label{eq: turnpike.zeta}
\|\zeta(t)\|_{L^2(\Omega)}\leqslant C_2\left(\left\|y^0 - \overline{y}\right\|_{L^2(\Omega)} e^{-\lambda t} + \|\overline{p}\|_{L^2(\Omega)} e^{-\lambda(T-t)}\right)
\end{equation}
for any $t\in[0,T]$. Here the constant $C_2>0$ is clearly independent of $T$, but also independent of the choice of initial data and running target $y_d$. 
This yields the desired turnpike property for $\zeta(t):=y_T(t)-\overline{y}$. Taking advantage of the affine feedback law \eqref{eq: affine.feedback} once again, using \eqref{eq: turnpike.zeta}, the uniform-in-$T$ boundedness of $\mathscr{E}(T)$, as well as \eqref{eq: estimate.h}, we also find 
\begin{equation*}
\|p_T(t)-\overline{p}\|_{L^2(\Omega)}\leqslant C_3\left(\left\|y^0 - \overline{y}\right\|_{L^2(\Omega)} e^{-\lambda t} + \|\overline{p}\|_{L^2(\Omega)} e^{-\lambda(T-t)}\right)
\end{equation*}
for $t\in[0,T]$, and for some possibly larger constant $C_3>0$, independent of $T, y^0$ and $y_d$. As $u_T\equiv p_T1_\omega$ and $\overline{u}\equiv\overline{p}1_\omega$, we may conclude.
\end{proof}

\begin{remark}[The decay rate $\lambda$] Reading the proof, one notes that the decay rate $\lambda>0$ appearing in the turnpike estimate is in fact explicit. It is precisely given as the exponential decay rate for the system
\begin{equation*}
\begin{cases}
\partial_t y + (-\Delta  + \mathscr{E}_\infty1_\omega) y = 0 &\text{ in }(0,+\infty)\times\Omega,\\
y_{|_{t=0}}=y^0 &\text{ in }\Omega.
\end{cases}
\end{equation*}
Namely, $\lambda$ corresponds to the spectral abscissa of the operator $-\Delta+\mathscr{E}_\infty1_\omega$.
This is also seen in the strategy of \cite{trelat2015turnpike} (see Section \ref{sec: 5}).
\end{remark}

\begin{remark}[Feedback law \& turnpike for the adjoint] Once again by reading the proof, one garners further information than what is stated in the theorem. First of all, we note that the optimal control $u_T$ is given by an affine feedback law of the form 
\begin{equation*}
u_T(t) = \Big(\overline{p}+\mathscr{E}(T-t)(y_T(t)-\overline{y})+h_T(t)\Big)1_\omega \hspace{1cm} \text{ for } t\in(0,T).
\end{equation*}
On another hand, the turnpike property also holds for the adjoint state $p_T(t)$ and corresponding stationary adjoint state $\overline{p}$: 
\begin{equation*}
\|p_T(t)-\overline{p}\|_{L^2(\Omega)}\leqslant C\left(\left\|y^0-\overline{y}\right\|_{L^2(\Omega)}e^{-\lambda t}+\|\overline{p}\|_{L^2(\Omega)}e^{-\lambda(T-t)}\right)
\end{equation*}
for $t\in[0,T]$. 
\end{remark}

\begin{remark}[Pay-off at time $T$] \label{remark: pay.off}
Let us stress that the turnpike estimate would take a more "symmetric" form if the adjoint state $p_T$ had a different data prescribed at time $t=T$. 
To achieve such a goal, one could consider a cost functional which contains an additional pay-off at the final time, such as, for instance
\begin{equation*}
\mathscr{J}_T(u) :=  \langle p^T, y(T)\rangle_{L^2(\Omega)} + \frac12 \int_0^T \|y(t)-y_d\|_{L^2(\omega_\circ)}^2 \diff t + \frac12 \int_0^T  \|u(t)\|_{L^2(\omega)}^2 \diff t
\end{equation*}
for some $p^T\in L^2(\Omega)$. 
In this case, the adjoint state, by writing the optimality system, would have to satisfy $p_T(T) = p^T$, and the above proof applies without any change except that now the corrector term $h_T$ will take a different final condition (equal to $p^T-\overline{p}$) and the estimate would become
\begin{align*}
&\left\|y_T(t)-\overline{y}\right\|_{L^2(\Omega)} + \|u_T(t)-\overline{u}\|_{L^2(\Omega)} \nonumber\\
&\quad\leqslant C\left(\left\|y^0 - \overline{y}\right\|_{L^2(\Omega)} e^{-\lambda t} + \left\|p^T-\overline{p}\right\|_{L^2(\Omega)} e^{-\lambda(T-t)}\right)
\end{align*}
for all $t\in[0,T]$.  A more general pay-off $\phi(y(T))$ instead of $\langle p^T, y(T)\rangle_{L^2(\Omega)}$ can also be considered in the definition of $\mathscr{J}_T$ just above (assuming it is, for example, Fréchet differentiable on $L^2(\Omega)$, convex, and bounded from below), and one would then change the terminal condition for the adjoint state: one would have $p_T(T)=\nabla\phi(y_T(T))$, where the gradient is interpreted as the one found by the Fréchet derivative and subsequently the Riesz representation theorem.
Of course, for a more general payoff, the symmetry with respect to the data in the turnpike estimate just above would not be replicated.
\end{remark}

\begin{remark}[Reference for the control] 
The turnpike result remains the same if the control $u(t)$ in the functional defined in \eqref{eq: pb.turnpike.heat} tracks a given reference $u_d\in L^2(\omega)$, namely, if one minimizes
\begin{equation*}
\mathscr{J}_T(u):= \frac12 \int_0^T \|y(t)-y_d\|^2_{L^2(\omega_\circ)}\diff t + \frac12\int_0^T \|u(t)-u_d\|^2_{L^2(\omega)}\diff t,
\end{equation*}
instead of the functional defined in \eqref{eq: pb.turnpike.heat}. The proof remains identical, with the only differences being the definition of the optimal control wherein one also accounts for $u_d$, namely $u_T(t)\equiv u_d+p_T1_\omega$ (and similarly for the steady control $\overline{u}$), and thus also the addition of $u_d$ as a source in the equation for the forward state $y_T$ (and similarly for the steady state $\overline{y}$).
\end{remark}

We delay further comments after generalizing the above result to a wider array of evolution equations. This is done in what follows.

\section{General evolution equations} \label{sec: 4}

The linear heat equation enjoys several properties which play a role in the proof  just above. These namely include the fact that the heat semigroup $\left\{e^{-t\Delta}\right\}_{t\geqslant0}$ is \emph{exponentially stable}, and that the heat equation is \emph{observable} from any open and non-void subset $\omega_\circ\subset\Omega$. 
One may thus be lead to think that turnpike only holds for such dissipative systems. This is not the case -- as we shall see, it will suffice for the system to be solely \emph{stabilizable} by means of some feedback law. And for the latter, controllability suffices. 
This is in agreement with common sense. Indeed, if the system under consideration is stabilizable, the optimal control will actually stabilize the system. The controlled system will therefore behave as an exponentially decaying system. Once the system enters this stable regime, the turnpike property will be manifested.

It is thus worthwhile to see under what conditions the turnpike property holds for  general partial differential equations and cost functionals. We shall see that the same result holds for significantly more general evolution equations -- for instance, hyperbolic equations --, with boundary controls and boundary observations in the tracking terms.

\subsection{The transport equation as a motivating example}

To motivate the appearance of turnpike for hyperbolic equations, let us illustrate the validity of the turnpike property for perhaps the simplest such equation imaginable: the linear transport equation. 

We consider
\begin{equation} \label{eq: transport}
\begin{cases}
\partial_t y + \partial_x y = 0 &\text{ in }(0,T)\times(0,1),\\
y(t,0) = u(t) &\text{ in }(0,T),\\
y(0,x)= y^0(x) &\text{ in }(0,1),
\end{cases}
\end{equation}
and the natural LQ problem
\begin{equation} \label{eq: LQ.transport}
\inf_{\substack{u\in L^2(0,T) \\ y \text{ solves} \eqref{eq: transport}}} \frac12\int_0^T\int_0^1 |y(t,x)-y_d(x)|^2 \diff x \diff t + \frac12\int_0^T|u(t)|^2\diff t.
\end{equation}
Here $y_d\in L^2(0,1)$ is a given running target. Given $y^0\in L^2(0,1)$ and $u\in L^2(0,T)$, \eqref{eq: transport} admits a unique weak solution $y\in C^0([0,T]; L^2(0,1))$ (see \cite[Section 2.1.1]{coron2007control} for the appropriate notion of weak solution). 

One may look to replicate the Riccati-inspired proof presented in the context of the heat equation -- to this end, we can first write the optimality system for an optimal pair $(u_T,y_T)$ for \eqref{eq: LQ.transport} -- \eqref{eq: transport}, which reads
\begin{equation} \label{eq: transport.optimality.system}
\begin{cases}
\partial_t y_T+\partial_x y_T = 0 &\text{ in }(0,T)\times(0,1),\\
\partial_t p_T+\partial_x p_T=y_T-y_d &\text{ in }(0,T)\times(0,1),\\
y_T(t,0)=p_T(t,0) &\text{ in }(0,T),\\
p_T(t,1)=0 &\text{ in }(0,T),\\
y_T(0,x)=y^0(x) &\text{ in }(0,1),\\
p_T(T,x)=0 &\text{ in } (0,1),
\end{cases}
\end{equation}
with 
\begin{equation*}
u_T(t)=p_T(t,0) \hspace{1cm} \text{ for } t\in[0,T].
\end{equation*}
But, for the transport equation \eqref{eq: transport}, the turnpike property can actually be derived by explicit calculations. 

Let us corroborate this claim.
Since solutions to \eqref{eq: transport} are constant along characteristics, one readily sees that $y$ takes the form
\begin{equation} \label{eq: explicit.transport}
y(t,x) = 
\begin{cases}
y^0(x-t) &\text{ for }t\leqslant x,\\
u(t-x) &\text{ for }t\geqslant x.
\end{cases}
\end{equation}
Because of this formula, for a given and fixed datum $y^0\in L^2(0,1)$, we can see that \eqref{eq: LQ.transport} is actually equivalent to the unconstrained quadratic problem 
\begin{equation*}
\inf_{u\in L^2(0,T)} \underbrace{\frac12 \int_0^T \int_0^1 |u(t-x)-y_d(x)|^2 1_{\{x\leqslant t\}} \diff x \diff t + \frac12 \int_0^T |u(t)|^2\diff t}_{:=\mathscr{J}_T(u)}.
\end{equation*}
Since $\mathscr{J}_T$ is strictly convex, continuous, and coercive, it admits a unique minimizer $u_T$, which is also a solution to \eqref{eq: LQ.transport}. 
We compute the G\^ateaux derivative of $\mathscr{J}_T$ at $u_T$ in any direction $v\in L^2(0,T)$ to find that
\begin{equation*}
\int_0^1\int_0^T \Big(u_T(t-x)-y_d(x)\Big)v(t-x)1_{\{t\geqslant x\}}\diff t\diff x +\int_0^T u_T(t)v(t)\diff t = 0.
\end{equation*}
The change of variable $t-x=\tau$ yields
\begin{equation*}
\int_0^T \int_0^1 \Big(u_T(\tau)-y_d(x)\Big) v(\tau)1_{(0,T-x)}(\tau) \diff \tau \diff x +\int_0^T u_T(t) v(t)\diff t =0.
\end{equation*}
Another change of variable in the indicator function above leads us to
\begin{equation*}
2u_T(\tau)-\int_0^1 y_d(x)1_{(0,T-\tau)}(x)\diff x = 0,
\end{equation*}
for a.e. $\tau\in(0,T)$. We then clearly see that
\begin{equation} \label{eq: u.transport}
u_T(t) = \frac12 \int_0^{\min\{1, T-t\}} y_d(x)\diff x
\end{equation}
for a.e. $t\in(0,T)$. And in view of \eqref{eq: explicit.transport}, we also find
\begin{align}
\int_0^1 y_T(t,x)\diff x &= \left(\int_0^t u_T(\tau)\diff \tau + \int_0^{1-t} y^0(\zeta)\diff \zeta\right)1_{\{t\leqslant1\}}\nonumber\\
&\quad +\left(\int_{t-1}^t u_T(\tau)\diff \tau\right) 1_{\{t\geqslant 1\}}\nonumber\\
&=\left(\frac12\int_0^t \int_0^{\min\{1,T-\tau\}} y_d(\zeta)\diff \zeta \diff \tau + \int_0^{1-t} y^0(\zeta)\diff \zeta\right)1_{\{t\leqslant1\}}\nonumber\\
&\quad+\left(\frac12 \int_0^1 y_d(\zeta)\diff \zeta\right) 1_{\{1<t<T-1\}}\nonumber \\
&\quad + \left(\frac12 \int_{t-1}^t \int_0^{T-\tau} y_d(\zeta)\diff\zeta \diff \tau\right) 1_{\{t\geqslant T-1\}}\label{eq: y.transport}
\end{align}
for all $t\in[0,T]$.
The above characterizations clearly indicate an exact turnpike-like pattern, as, for instance, we see that the (mass of the) optimal state $y_T(t)$ is stationary at $\frac12\int_0^1 y_d(x)\diff x$ over the time interval $(1,T-1)$. Furthermore, this pattern actually emerges rather rapidly, namely when $T>2$ only. 
This is also visible in the numerical experiments shown in Figure \ref{fig: transport}. 

To be able to conclude and consider this as a turnpike phenomenon, we need to ensure that the optimal steady control-state pair is precisely given by 
$$(\overline{u},\overline{y})=\left(\frac12 \int_0^1 y_d(x)\diff x, \frac12 \int_0^1 y_d(x)\diff x\right).$$ 
To this end, we consider the steady problem corresponding to \eqref{eq: LQ.transport}, which reads
\begin{equation} \label{eq: LQ.transport.steady}
\inf_{\substack{u\in\mathbb{R}\\ \partial_x y=0 \text{ in } (0,1)\\ y(0)=u}}\frac12\int_0^1 |y(x)-y_d(x)|^2 \diff x + \frac12|u|^2.
\end{equation}
One readily sees that the constraints in \eqref{eq: LQ.transport.steady} yield $y\equiv u$, and so \eqref{eq: LQ.transport.steady} is actually an unconstrained minimization problem on $\mathbb{R}$:
\begin{equation} \label{eq: 4.5}
\inf_{u\in\mathbb{R}}\frac12\int_0^1 |u-y_d(x)|^2\diff x + \frac12|u|^2.
\end{equation}
It is readily seen that the unique solution to \eqref{eq: 4.5} is $\overline{u}\equiv\frac{1}{2}\int_0^1 y_d(x)\diff x$, and as $\overline{u}\equiv\overline{y}$, we deduce a turnpike property for the optimal evolutionary pair $(u_T,y_T)$ to $(\overline{u},\overline{y})$.
 
This simple example indicates that the turnpike property may also appear for hyperbolic equations. We provide more examples and a general setup in Sections 4.2--4.3.

\begin{remark}[Compatible norms] 
It is important to note that the above derivation, and subsequent result, rely on the fact that the state and the control are penalized in \emph{compatible topologies} (here, $L^2(0,T)$ for the boundary control, and consequently, $L^2((0,T)\times(0,1))$ for the state). 
The computations are then explicit due to the choice of these topologies, but, in essence, the result is inherently due to the possibility of exponentially stabilizing the system through a feedback operator defined on the energy space. 
The bottom line is that there should be a compatibility in the topologies being penalized for the control and the state, due to conservation of regularity. 
This is clearly seen in the optimality system \eqref{eq: transport.optimality.system}. Roughly speaking, if solely the $H^{-1}(0,T)$-norm of the boundary control $u(t)$ is penalized, then there would be a mismatch of regularity between the state $y_T$ and the adjoint state $p_T$ through the boundary condition at $x=0$. 
The same artifact appears in the context of the wave equation, and is discussed later on.
\end{remark}

\begin{figure}[h!] 
\center
\includegraphics[scale=0.425]{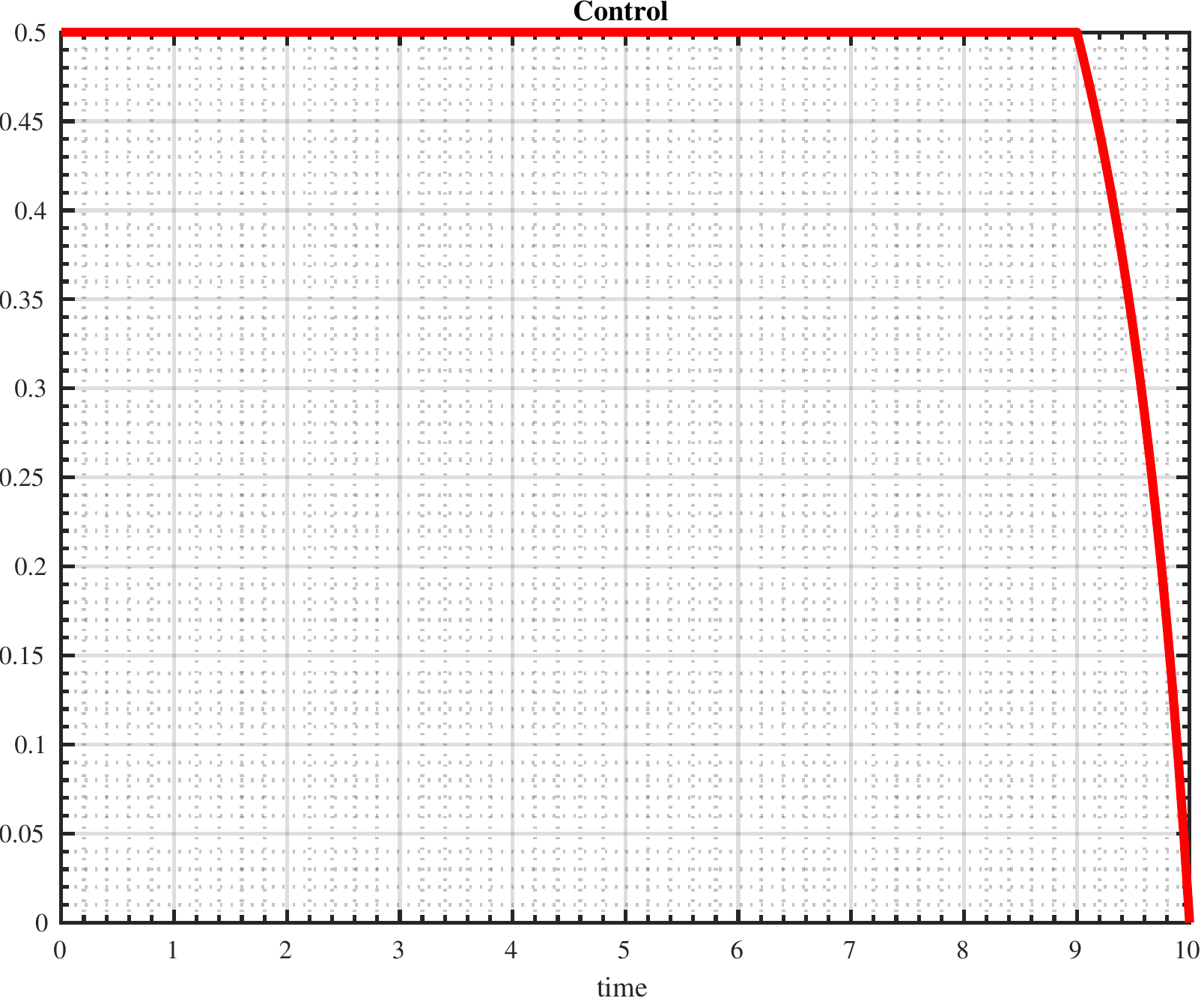}
\hspace{0.2cm}
\includegraphics[scale=0.425]{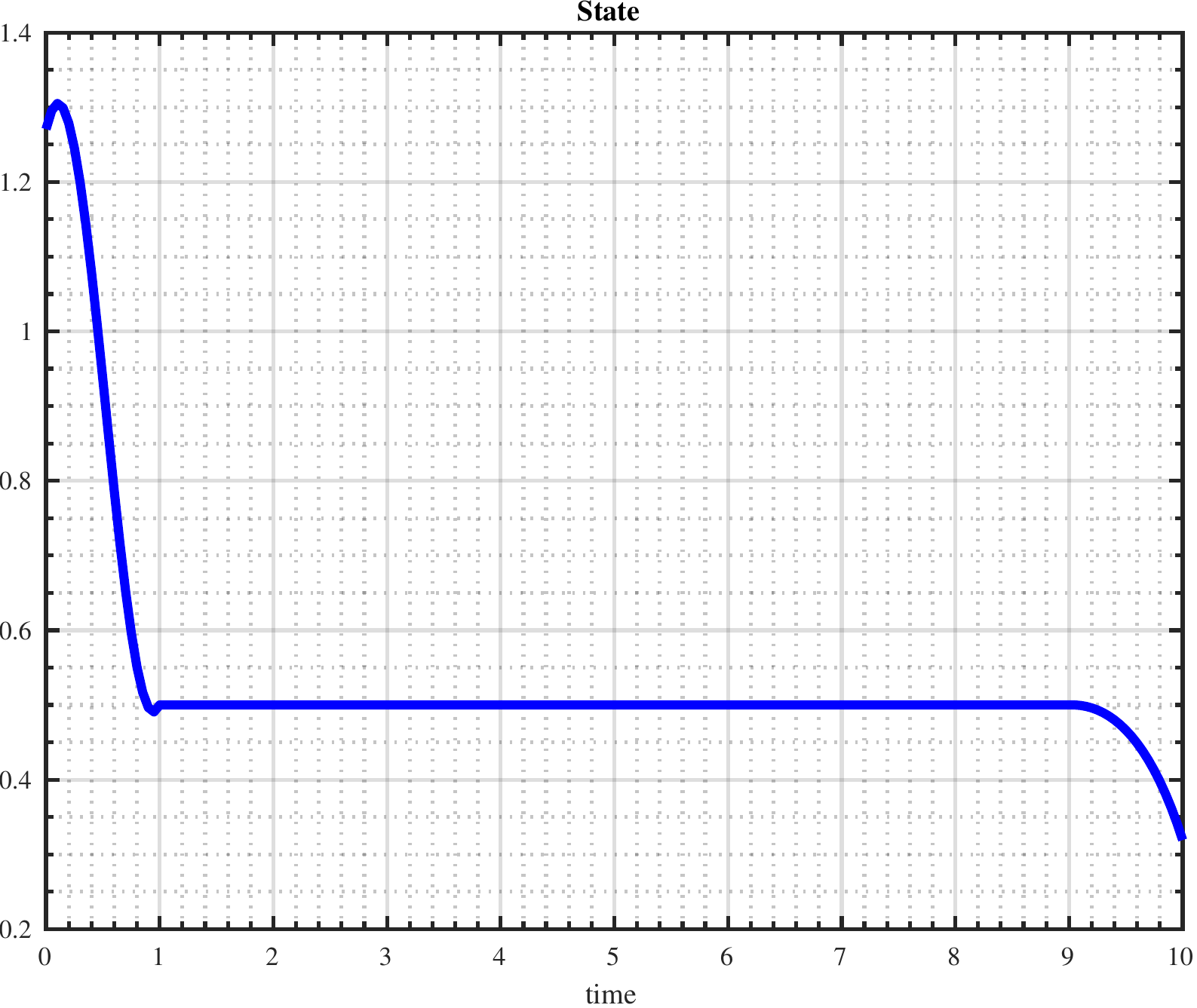}
\caption{A numerical visualization of the optimal control $u_T$ solving \eqref{eq: LQ.transport} (\emph{left}), and mass of the corresponding solution $y_T$ to \eqref{eq: transport} (\emph{right}), for $T=10$, $y^0(x)=\sin(\pi x)$, and $y_d\equiv1$. (\emph{Left}) We see that the optimal control $u_T(t)$ is constant equal to the turnpike $\frac12\int_0^1 y_d=\frac12$ for $t\leqslant T-1$ and reaches $0$ at time $t=T$, as per \eqref{eq: u.transport}. (\emph{Right}) We also see that the mass of the optimal state $y_T(t)$ splits in three stages: it descends to the turnpike $\frac12$ in time $t=1$, stays at the turnpike until time $t=T-1$, and then exits, as per \eqref{eq: y.transport}.}
\label{fig: transport}
\end{figure}

\begin{figure}[h!]
\center
\includegraphics[scale=0.45]{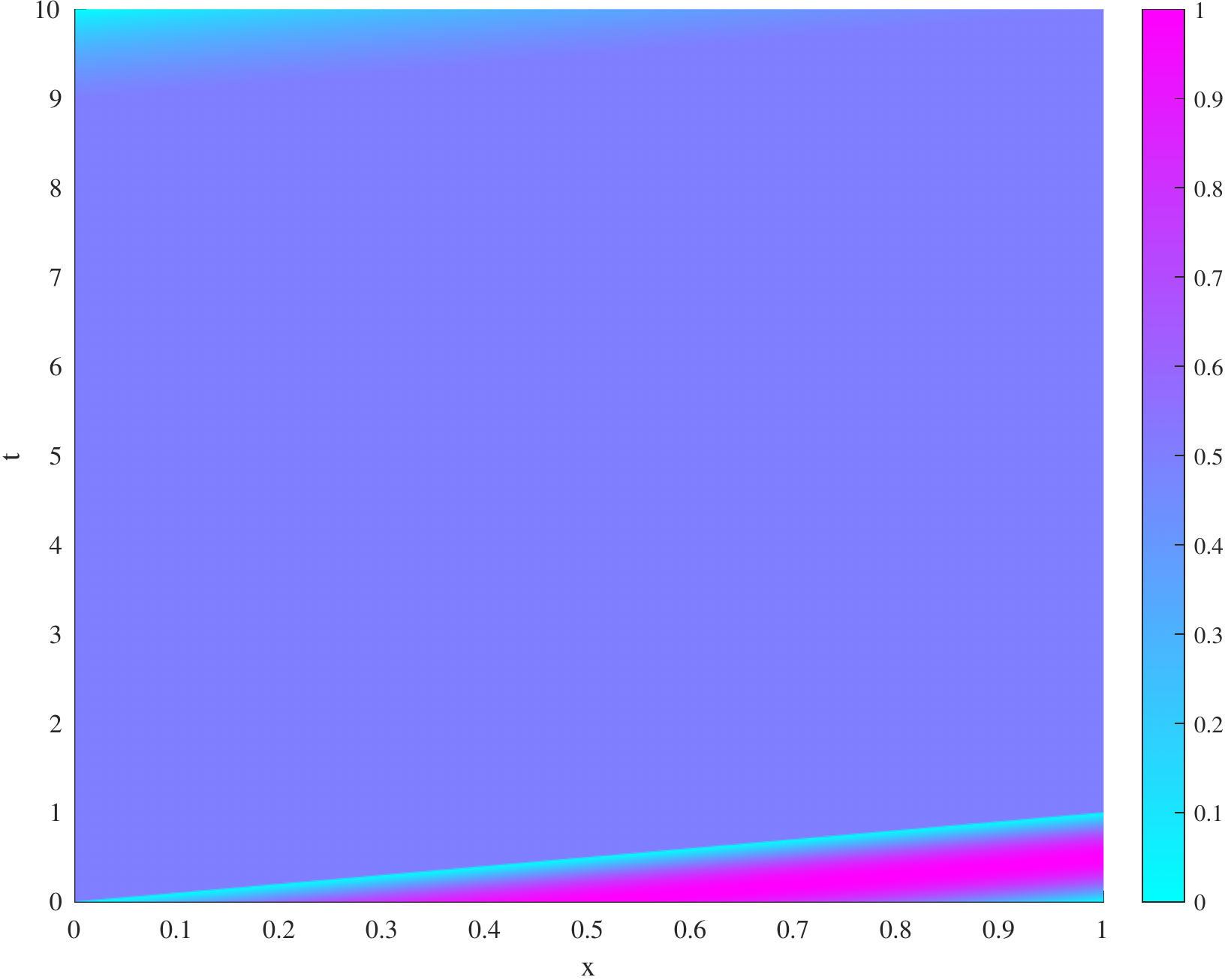}
\caption{A surface plot of the optimal state $y_T(t,x)$ solution to \eqref{eq: transport}, further showing the turnpike phenomenon: the initial datum is transported until time $t=1$, beyond which the state stays at the turnpike equal to $\frac12$ until time $t=T-1$, when it exits per \eqref{eq: y.transport}.}
\end{figure}

\subsection{First-order in time (parabolic) equations}

We shall consider general evolution equations written as abstract first-order systems, with a main focus on \emph{parabolic equations}. 
While the wave equation may also fit in this setting, there is a difficulty in defining a general functional setting for such differing kinds of problems, as the wave equation conserves the regularity of the initial datum, unlike the gain of regularity typically encountered in parabolic equations. The specific proof of turnpike however, and the structural hypotheses on the dynamics, control, and observation operators, are identical in both cases. We thus postpone the specific case of the wave equation (and natural generalizations thereof) to the subsequent section.
The presentation will require elementary knowledge of semigroup theory and functional analysis; we refer the reader to \cite{tucsnak2009observation} for all the needed details. 

Let us henceforth suppose that we are given a couple of Hilbert spaces $X$ and $\mathscr{H}$ such that
\begin{equation*}
X\hookrightarrow\mathscr{H}\hookrightarrow X'
\end{equation*}
 (with dense embeddings), where the pivot space $\mathscr{H}$ is identified with its dual $\mathscr{H}'$.  This is a Gelfand triple, the canonical example thereof of course being $X:=H^1_0(\Omega)$, $\mathscr{H}:=L^2(\Omega)$, with $X' = H^{-1}(\Omega)$. 
 We shall focus on linear, first order control systems, written in a canonical form
\begin{equation} \label{eq: abstract.system}
\begin{cases}
\partial_t y = Ay+Bu &\text{ in }(0,T),\\
y_{|_{t=0}} = y^0.
\end{cases}
\end{equation}
Here, 
\begin{itemize}
\item $A:\mathfrak{D}(A)\to\mathscr{H}$ is closed and densely defined, with $A\in\mathscr{L}(X,X')$; we also suppose that $-A+\alpha\text{Id}$ is coercive for some $\alpha>0$, in the sense that there exist a couple of constants $(\alpha,\beta)\in(0,+\infty)^2$ such that
\begin{equation} \label{eq: A.coercive}
\Big\langle(-A+\alpha\text{Id})f, f\Big\rangle_{\mathscr{H}} \geqslant\beta\|f\|_X^2,
\end{equation}
holds for all $f\in\mathfrak{D}(A)$. 
The above hypothesis entails that $A$ generates a strongly continuous semigroup $\{e^{tA}\}_{t\geqslant0}$ on $\mathscr{H}$ (see \cite[Chapter 3, pp. 100--105]{lions1971optimal}). 
We need not assume that $A$ is symmetric.
This is an inherently "parabolic" hypothesis, as it is mostly valid in cases where the principal part of the operator $-A$ is self-adjoint, and, consequently, the bilinear form inferred from the principal part of $-A$ is equivalent to the norm $X$-norm.
We suppose that $\mathfrak{D}(A)$ is also dense in $X$, so \eqref{eq: A.coercive} also holds for all $f\in X$, modulo replacing the inner product in $\mathscr{H}$ by the duality bracket between $X'$ and $X$.
\smallskip
\item 
Let us also note that, for ensuring the generation of a strongly continuous semigroup, one may also simply assume $A$ being $m$-dissipative in the sense of \cite{cazenave1998introduction}: 
\begin{equation*}
\|f-\lambda Af\|_{\mathscr{H}}\geqslant\|f\|_{\mathscr{H}}
\end{equation*}
holds for all $\lambda>0$ and $f\in\mathfrak{D}(A)$, and, moreover, the equation $(\text{Id}-\lambda A)f=g$ admits a solution $f\in\mathfrak{D}(A)$ for any $g\in \mathscr{H}$. 
We shall actually use \eqref{eq: A.coercive} to also guarantee the existence of solutions to the steady optimal control problem (see Remark \ref{rem: steady.well.posed}).  
\smallskip
\item 
On the other hand, the control operator is $B\in\mathscr{L}(\mathscr{U},\mathscr{H})$, where $\mathscr{U}$ is another Hilbert space.
A feasible scenario is having $\mathscr{H}=L^2(\Omega)$ and $\mathscr{U}=L^2(\omega)$, with $\omega\subset\Omega$ open and non-empty, namely the typical distributed control setting as considered in \eqref{eq: heat.equation}. We comment on systems involving boundary controls in Remark \ref{rem: boundary.control}; the framework and results can be adapted by making use of transposition and duality arguments. These are solely technical considerations, and do not carry significant conceptual differences to the strategy for proving turnpike we have presented in the context of the heat equation with distributed control.  
\end{itemize}

\begin{remark}[Examples of \eqref{eq: A.coercive}]
The coercivity inequality \eqref{eq: A.coercive} is not only satisfied by the Dirichlet Laplacian (with $\alpha=0$ and $\beta=1$, where $X=H^1_0(\Omega)$, $\mathscr{H}=L^2(\Omega)$ and $\mathfrak{D}(A)=H^2(\Omega)\cap H^1_0(\Omega)$), but also by the Neumann Laplacian (with $\alpha=\beta=1$, where $X=H^1(\Omega)$, $\mathscr{H}=L^2(\Omega)$, and $\mathfrak{D}(A)=\{y\in H^2(\Omega)\,\bigm|\,\partial_n y=0 \text{ on } \partial\Omega\}$), and also for more general elliptic operators involving lower order perturbations.
\end{remark}

By virtue of these assumptions on $A$ and $B$, for any $u\in L^2(0,T;\mathscr{U})$ and $y^0\in\mathscr{H}$, the abstract system \eqref{eq: abstract.system} is well posed, in the sense that there exists a unique weak solution\footnote{In fact, one has stronger information in that, moreover, $\partial_t y\in L^2(0,T; X')$.} $y\in C^0([0,T]; \mathscr{H})\cap L^2(0,T; X)$ (see \cite[Chapter 3, pp. 100--105]{lions1971optimal}, and also \cite{tucsnak2009observation} for a primer on semigroup theory in control). 

We shall henceforth consider the following optimal control problem
\begin{equation}  \label{eq: abstract.heat.ocp}
\inf_{\substack{u\in L^2(0,T; \mathscr{U})\\ y \text{ solves } \eqref{eq: abstract.system}}} \frac12 \int_0^T \|Cy(t)-y_d\|^2_{\mathscr{H}} \diff t + \frac12\int_0^T \|u(t)\|_{\mathscr{U}}^2\diff t.
\end{equation}
In the above problem, $C\in\mathscr{L}(\mathscr{H})$ is\footnote{Henceforth, whenever we use $C$ to denote the observation operator, we shall use lowercase letters (e.g. $c$) to denote constants in various estimates.} a given observation operator, whereas $y_d\in\mathscr{H}$. 
In the specific example of \eqref{eq: pb.turnpike.heat} for instance, we had $Cy=y|_{\omega_\circ}$ and $\mathscr{H}=L^2(\Omega)$, with $\omega_\circ\subset\Omega$. But as we shall see in what proceeds, the definition of $C$ can be relaxed to take into account scenarios which are of practical relevance, such as boundary observation via Neumann traces.
Final pay-offs may also be considered, under similarly moderate assumptions (convex, Fréchet differentiable, bounded from below). Again, these are solely technical adaptations, so we omit them to avoid even more cumbersome notation. 

The optimal control problem \eqref{eq: abstract.heat.ocp} again admits a unique solution by the direct method in the calculus of variations. The steady problem corresponding to \eqref{eq: abstract.heat.ocp} reads as
\begin{equation} \label{eq: static.abstract.ocp}
\inf_{\substack{(u, y)\in\mathscr{U}\times X \\ Ay+Bu=0}} \frac12 \|Cy-y_d\|_{\mathscr{H}}^2 + \frac12 \|u\|_{\mathscr{U}}^2;
\end{equation}
\eqref{eq: static.abstract.ocp} also admits a unique optimal solution $(\overline{u}, \overline{y})$, but we postpone the brief argument\footnote{We do note however that it is relevant to optimize over pairs $(u,y)$ over the manifold $\{Ay+Bu=0\}\subset\mathscr{U}\times X$, as opposed to optimizing solely over $u$ with $y$ satisfying the equation $Ay+Bu=0$. Both are equivalent whenever $-A$ is invertible, since in this case, for any given $u$ there exists a unique solution $y$ to $Ay+Bu=0$. Herein we consider a more general scenario, to account for cases such as the Neumann Laplacian.} to Remark \ref{rem: steady.well.posed}.

Since the proof of Theorem \ref{thm: porretta.zuazua.1} consist in studying the decay properties of the optimality system, in this new abstract framework, we will also need to ensure that the forward equation for the state, as well as the backward equation for the adjoint state, possess a stabilization mechanism.
To this end, we will make the following two natural assumptions. 

Beforehand, we recall that an operator semigroup $\{\mathcal{T}(t)\}_{t\geqslant0}$ on a Hilbert space $\mathscr{H}$ is called \emph{exponentially stable} if there exist a couple of constants $c\geqslant1$ and $\lambda>0$ such that 
\begin{equation*}
\|\mathcal{T}(t)\|_{\mathscr{L}(\mathscr{H})} \leqslant c\,e^{-\lambda t}
\end{equation*}
holds for all $t\geqslant 0$. 

\begin{assumption}[Stabilizability] \label{ass: 1.2}
We suppose that there exists a feedback operator $K\in\mathscr{L}(\mathscr{H},\mathscr{U})$ such that the semigroup\footnote{Note that since $BK\in\mathscr{L}(\mathscr{H})$, as a bounded perturbation of $A$, the operator $A+BK$ also generates a strongly continuous semigroup on $\mathscr{H}$ (see \cite[Section 2.11]{tucsnak2009observation}).} $\left\{e^{t(A+BK)}\right\}_{t\geqslant0}$ on $\mathscr{H}$ is exponentially stable. Equivalently,
\begin{equation} \label{eq: stabilizability.ineq}
\sup_{\|y^0\|_{\mathscr{H}}\leqslant1}\int_0^{+\infty}\left\|e^{t(A+BK)}y^0\right\|_{\mathscr{H}}^2<+\infty
\end{equation}
holds.
\end{assumption}

When the above assumption holds true, we say that $(A,B)$ is \emph{exponentially stabilizable}.
The equivalence stated in Assumption \ref{ass: 1.2} is due to \cite{datko1972uniform} (see also \cite[Corollary 6.1.14]{tucsnak2009observation}). By virtue of \eqref{eq: stabilizability.ineq}, one readily sees that there exists a constant $c>0$ such that for all $T>0$ and $y^0\in\mathscr{H}$, the unique solution $y$ to 
\begin{equation*}
\begin{cases}
\partial_t y = (A+BK)y &\text{ in }(0,T),\\
y_{|_{t=0}}=y^0
\end{cases}
\end{equation*}
satisfies
\begin{equation}\label{eq: implied.by.stabilizability}
\int_0^T \|y(t)\|_{\mathscr{H}}^2\diff t \leqslant c\left\|y^0\right\|_{\mathscr{H}}^2.
\end{equation}
Here, it is critical to emphasize that the constant $c>0$ is independent of $T$.

\begin{assumption}[Detectability] \label{ass: 1.1} 
We suppose that there exists a feedback operator $K\in\mathscr{L}(\mathscr{H})$ such that the semigroup $\left\{e^{t(A^*+C^*K)}\right\}_{t\geqslant0}$ on $\mathscr{H}$ is exponentially stable. Equivalently, 
\begin{equation} \label{eq: detectability.ineq}
\sup_{\left\|p^T\right\|_{\mathscr{H}}\leqslant1}\int_0^{+\infty}\left\|e^{t(A^*+C^*K)}p^T\right\|_{\mathscr{H}}^2<+\infty
\end{equation}
holds.
\end{assumption}

In such a case, we say that the pair $(A, C)$ is assumed to be \emph{exponentially detectable}. And similarly as before, \eqref{eq: detectability.ineq} implies that there exists a constant $c>0$ such that for all $T>0$ and $p^T\in\mathscr{H}$, the unique solution $p$ to
 \begin{equation} \label{eq: visavis.adjoint}
 \begin{cases}
 -\partial_t p =(A^*+C^*K)p &\text{ in }(0,T),\\
 p_{|_{t=T}} = p^T
 \end{cases}
 \end{equation}
 satisfies
 \begin{equation} \label{eq: implied.by.detectability}
 \int_0^T\|p(t)\|_{\mathscr{H}}^2 \diff t\leqslant c\left\|p^T\right\|_{\mathscr{H}}^2.
 \end{equation}
Once again, as for \eqref{eq: implied.by.stabilizability}, we emphasize that the constant $c>0$ is independent of $T$. We also note, vis-à-vis \eqref{eq: visavis.adjoint}, that both the forward and the adjoint equation are posed in the same Hilbert space $\mathscr{H}$, which is identified with its dual. This artifact is in line with the assumptions we had made on the structure of the underlying system and the governing operator, and are typical for parabolic equations.   
 
 Inequalities \eqref{eq: implied.by.stabilizability} and \eqref{eq: implied.by.detectability} are then used in proving that $M:=-A+BB^*\mathscr{E}_\infty$ generates an exponentially stable semigroup on $\mathscr{H}$, and that $\mathscr{E}(T)$ converges exponentially to $\mathscr{E}_\infty$ (both defined as in the proof of Theorem \ref{thm: porretta.zuazua.1}); these two properties are cornerstones of the proof. We refer to Remark \ref{rem: assumptions.turnpike} for more details on how these assumptions are used to derive weaker observability inequalities (and consequently, some kind of unique continuation properties) which appear in the proof, as well as how they may be derived from stronger, but more intuitive assumptions such as controllability and observability.
 
  Taking stock of the above conditions, we may state the following generalization of Theorem \ref{thm: porretta.zuazua.1} -- namely, an exponential turnpike property for the solutions to \eqref{eq: abstract.heat.ocp}. 

\begin{theorem}[\cite{porretta2013long}] \label{thm: porretta.zuazua.2}
 Suppose $y^0\in \mathscr{H}$ and $y_d\in\mathscr{H}$ are fixed. Under Assumptions \ref{ass: 1.2} and \ref{ass: 1.1},  there exist a couple of constants $c>0$ and $\lambda>0$, independent of $y^0$ and $y_d$, such that for any large enough $T>0$, the unique solution $(u_T,y_T)$ to \eqref{eq: abstract.heat.ocp} satisfies
\begin{align}
\|y_T(t)-\overline{y}\|_{\mathscr{H}} &+ \|u_T(t)-\overline{u}\|_{\mathscr{U}}\nonumber \\
&\leqslant c\left(\left\|y^0-\overline{y}\right\|_{\mathscr{H}}e^{-\lambda t} + \|\overline{p}\|_{\mathscr{H}}e^{-\lambda(T-t)}\right)
\end{align}
for a.e. $t\in[0,T]$, where $(\overline{u}, \overline{y})$ denotes the unique solution to \eqref{eq: static.abstract.ocp}, and $\overline{p}\in X$ is the optimal steady adjoint state.
\end{theorem}

\begin{proof} 
The proof follows the same lines as that for the heat equation (Theorem \ref{thm: porretta.zuazua.1}), and may be found in \cite{porretta2013long}. First, one may readily write the optimality systems for the time-dependent and steady optimal control problems. They read, respectively, as
\begin{equation}
\begin{cases}
\partial_t y_T=Ay_T+B\iota_u B^*p_T &\text{ in } (0,T),\\
-\partial_t p_T=A^*p_T-C^*(Cy_T-y_d) &\text{ in } (0,T),\\
{y_T}_{|_{t=0}}=y^0,\\
{p_T}_{|_{t=T}}=0 
\end{cases}
\end{equation}
and
\begin{equation}
\begin{cases}
-A\overline{y}=B\iota_u B^*\overline{p},\\
A^*\overline{p}=C^*(C\overline{y}-y_d),
\end{cases}
\end{equation}
with $u_T\equiv \iota_u B^*p_T$ and $\overline{u}\equiv\iota_u B^*\overline{p}$.
Here $\iota_u:\mathscr{U}'\to\mathscr{U}$ is the natural injection of the dual $\mathscr{U}'$ in the Hilbert space $\mathscr{U}$. 
Assumptions \ref{ass: 1.1} and \ref{ass: 1.2} are then used (see also Remark \ref{rem: assumptions.turnpike}) to ensure the convergence of $\mathscr{E}(T)$ to $\mathscr{E}$ in the reference case $y_d\equiv0$ (defined as in the proof of Theorem \ref{thm: porretta.zuazua.1}), exponentially, with rate $\lambda>0$; this is shown precisely in \cite[Section 3.2]{porretta2013long}.
 \end{proof}
 
There are several examples to which one can apply the above theorem. Let us name a few to illustrate the wide spectrum of applications they encompass.
\begin{enumerate}
\item[1.] \emph{Advection-diffusion equations.} We may consider a more general setting to the linear heat equation with constant coefficients we presented in what precedes, namely an advection-diffusion equation with distributed control and observation, where
\begin{equation*}
Ay:= -\nabla \cdot (a(x)\nabla y) + c(x)y + b(x)\cdot \nabla y.
\end{equation*}
The coefficients are assumed as follows: $a\in L^\infty(\Omega;\mathbb{R}^{d\times d})$ is such that 
$$\alpha_1 \text{Id}\leqslant a(x)\leqslant \alpha_2 \text{Id}$$ 
for some $\alpha_1,\alpha_2>0$ and for a.e. $x\in\Omega$, while $c\in L^\infty(\Omega)$ and $b\in L^\infty(\Omega)^d$. Accordingly,  $A$ is an elliptic operator. 
Let $Bu=u1_{\omega}$ and $Cy=y|_{\omega_\circ}$, with $\mathscr{U}=L^2(\omega)$; both $\omega,\omega_\circ\subset\Omega$ are open and non-empty.
Setting $X=H^1_0(\Omega)$, $\mathscr{H}=L^2(\Omega)$, we see that $A$ satisfies the coercivity requirements stated in what precedes. 
Should $\|b\|_{L^\infty}$ and/or $\|c\|_{L^\infty}$ be large, then $A$ might not generate an exponentially stable semigroup. 
Yet $(A,B)$ and $(A^*,C^*)$ are stabilizable due to the presence of some control (through $B$ and $C^*$), and the turnpike property then holds. This is another example of a system which may be unstable in the absence of control, but can then be stabilized through the action of a control. In occurrence, this is also sufficient for  the turnpike property to be manifested.
\smallskip
\item[2.] \emph{Stokes equations.} Similarly, the result applies for \emph{systems} of equations, such as the linear Stokes equations with Dirichlet boundary conditions on a bounded and smooth domain $\Omega\subset\mathbb{R}^2$:
\begin{equation*}
\begin{cases}
\partial_t \*y - \Delta \*y = -\nabla p + \*u1_\omega &\text{ in }(0,T)\times\Omega,\\
\nabla \cdot \*y = 0 &\text{ in }(0,T)\times\Omega,\\
\*y = 0 &\text{ in }(0,T)\times\partial\Omega,\\
\*y_{|_{t=0}} = \*y^0 &\text{ in }\Omega.
\end{cases}
\end{equation*}
Here $\*y=(y_1,y_2)$ and $\*u=(u_1,u_2)$. The functional setting is only slightly more delicate in this case. The underlying Hilbert state space $\mathscr{H}$ is defined as 
\begin{equation*}
\mathscr{H}:=\left\{\*y\in L^2(\Omega; \mathbb{R}^2)\, \Bigm|\, \nabla\cdot\*y=0,\, \*y|_{\partial\Omega} \cdot \nu =0\right\}
\end{equation*}
In the definition of $\mathscr{H}$, $\nu\in\mathbb{R}^2$ denotes the outward unit normal to $\partial\Omega$. We then set $X=H^1_{\text{div}}(\Omega)$, where 
$$H^1_{\text{div}}(\Omega):=\left\{\*y\in H^1_0(\Omega; \mathbb{R}^d)\, \Bigm|\, \nabla\cdot\*y=0\right\}.$$
We can define the Stokes operator $A:\mathfrak{D}(A)\to\mathscr{H}$ as $A=-\mathcal{P}\Delta$, with domain $$\mathfrak{D}(A)=\left\{\*y\in H^1_{\text{div}}(\Omega)\, \bigm|\, A\*y\in\mathscr{H}\right\}.$$ 
Here $-\Delta$ denotes the Dirichlet Laplacian on $\Omega$, while $$\mathcal{P}: L^2(\Omega; \mathbb{R}^2)=\mathscr{H}\oplus\mathscr{H}^\perp\to\mathscr{H}$$ 
is the Leray projector.
The operator $A$ is self-adjoint, and exponentially stabilizable (\cite{fernandez2004local}), hence previous considerations apply.
Linear convective potentials (as in the first item) may also be added; this allows one to see the framework as linearized Navier-Stokes.
\smallskip
\item[3.] Many further examples can be fit in this framework, including several classes of degenerate linear parabolic equations (\cite{cannarsa2013null, gueye2014exact, geshkovski2020null}), evolution equations for the fractional Laplacian with Dirichlet boundary conditions (\cite{warma2020exponential, macia2021observability}), and so on.
\end{enumerate}

\subsection{The wave equation}

The linear wave equation
\begin{equation} \label{eq: wave.eq.turnpike}
\begin{cases}
\partial_t^2 y - \Delta y = u1_\omega &\text{ in }(0,T)\times\Omega,\\
y=0 &\text{ in }(0,T)\times\partial\Omega,\\
(y,\partial_t y)_{|_{t=0}} = (y^0, y^1) &\text{ in }\Omega,
\end{cases}
\end{equation}
may also fit in the setting of the result presented above. This is done in greater depth in \cite{zuazua2017large}.
As \eqref{eq: wave.eq.turnpike} is a second-order system, the state is $(y,\partial_t y)$. Therefore, some adaptations are needed in terms of the functional setting, but the proof of turnpike follows precisely the same arguments. We shall avoid abstractions in this part, and state the result specific to \eqref{eq: wave.eq.turnpike}. 
In other words, the turnpike property does also hold for appropriate optimal control problems for the wave equation \eqref{eq: wave.eq.turnpike}, and under appropriate assumptions on the control domain $\omega$.

We may consider
\begin{equation} \label{eq: ocp.wave.eq}
\inf_{\substack{u\in L^2((0,T)\times\omega)\\ y\text{ solves } \eqref{eq: wave.eq.turnpike}}} \frac12\int_0^T\left\|y(t)-y_d\right\|^2_{H^1_0(\Omega)} \diff t + \frac12 \int_0^T \|u(t)\|_{L^2(\omega)}^2 \diff t.
\end{equation}
Note that we are not only penalizing $\nabla_x y(t,x)$ in \eqref{eq: ocp.wave.eq}, but we do so over the entire domain $\Omega$ (instead of an open and non-empty subdomain $\omega_\circ$). We discuss both of these considerations in Remark \ref{rem: considerations} -- the latter one is actually not necessary, but renders the presentation simpler.

The corresponding steady system is the same as the one for the heat equation, namely \eqref{eq: steady.heat}. The steady optimal control problem then reads
\begin{equation} \label{eq: steady.ocp.wave.eq}
\inf_{\substack{u\in L^2(\omega)\\ y\text{ solves } \eqref{eq: steady.heat}}} \frac12\left\|y-y_d\right\|^2_{H^1_0(\Omega)}+\frac12 \|u\|_{L^2(\omega)}^2.
\end{equation}
Theorem \ref{thm: porretta.zuazua.2} applies to \eqref{eq: ocp.wave.eq} under the assumption that $\omega\subset\Omega$ satisfies the \emph{Geometric Control Condition} (GCC). This condition roughly asserts that all the rays of geometric optics in $\Omega$, reflected according to the Descartes-Snell law on the boundary, enter the domain $\omega$ in some finite, uniform time (see the seminal work \cite{bardos1992sharp}).

The following result then holds.

\begin{theorem}[\cite{zuazua2017large}]
Suppose that $\omega\subset\Omega$ is open, non-empty, and satisfies GCC. Let $(y^0,y^1)\in H^1_0(\Omega)\times L^2(\Omega)$ and $y_d\in H^1_0(\Omega)$ be fixed. 
There exist a couple of constants $C>0$ and $\lambda>0$, independent of $y^0$ and $y_d$, such that for any $T>0$ large enough, the unique solution $(u_T, y_T)$ to \eqref{eq: ocp.wave.eq} satisfies
\begin{align}\label{eq: porretta.zuazua.estimate}
\|y_T(t)&-\overline{y}\|_{H^1_0(\Omega)}+\|\partial_t y_T(t)\|_{L^2(\Omega)}+\|u_T(t) - \overline{u}\|_{L^2(\omega)}\nonumber\\
&\leqslant C\left(\left\|\Big(y^0-\overline{y}, y^1\Big)\right\|_{H^1_0(\Omega)\times L^2(\Omega)}e^{-\lambda t} + \|\overline{p}\|_{L^2(\Omega)}e^{-\lambda(T-t)}\right),
\end{align}
for a.e. $t\in [0,T]$, where $(\overline{u}, \overline{y})$ denotes the unique solution to \eqref{eq: steady.ocp.wave.eq}, and $\overline{p}\in L^2(\Omega)$ is the optimal steady adjoint state. 
\end{theorem}

\begin{proof}
The proof follows precisely the same lines as that for the heat equation, and we only provide a sketch thereof. Let us focus on Step 1 per the proof of Theorem \ref{thm: porretta.zuazua.1}, in which $y_d\equiv0$. 
We consider the transient optimality system
\begin{equation} \label{eq: optimality.wave.eq}
\begin{cases}
\partial_t^2 y_T-\Delta y_T = p_T1_\omega &\text{ in }(0,T)\times\Omega,\\
\partial_t^2 p_T-\Delta p_T = \Delta y_T &\text{ in }(0,T)\times\Omega,\\
y_T=p_T=0 &\text{ on }(0,T)\times\partial\Omega,\\
(y_T,\partial_t y_T)_{|_{t=0}}=(y^0,y^1) &\text{ in }\Omega,\\
(p_T,\partial_t p_T)_{|_{t=T}}=(0,0) &\text{ in }\Omega.
\end{cases}
\end{equation} 
Of course, once again, $u_T\equiv p_T1_\omega$.
For $T>0$, we can define the operator $$\mathscr{E}(T): H^1_0(\Omega)\times L^2(\Omega)\to H^{-1}(\Omega)\times L^2(\Omega)$$ as
\begin{equation*}
\mathscr{E}(T)\left(y^0,y^1\right):=\left(-\partial_t p_T(0), p_T(0)\right),
\end{equation*}
and see that 
\begin{align}
\Big\langle\mathscr{E}(T)\left(y^0,y^1\right), \left(y^0,y^1\right)\Big\rangle 
&=\int_0^T \|y(t)\|_{H^1_0(\Omega)}^2\diff t + \int_0^T \|p(t)\|^2_{L^2(\omega)}\diff t\nonumber\\
&=\inf_{\substack{u\in L^2((0,T)\times\omega)\\ y \text{ solves} \eqref{eq: wave.eq.turnpike}}}\mathscr{J}_T^0(u). \label{eq: wave.riccati}
\end{align}
Here, $\mathscr{J}_T^0$ denotes the functional defined in \eqref{eq: ocp.wave.eq} with $y_d\equiv0$, and $\langle\cdot,\cdot\rangle$ denotes the duality bracket between $H^{-1}(\Omega)\times L^2(\Omega)$ and $H^1_0(\Omega)\times L^2(\Omega)$. This characterization then implies that $\mathscr{E}$ is monotonically increasing with $T>0$. 

To derive similar conclusions as for the heat equation, we seek to use the stabilizability assumptions in \eqref{eq: wave.riccati} to show that $\mathscr{E}(T)$ is bounded uniformly with respect to $T>0$, from which point on, an exponential convergence to the regulator operator $\mathscr{E}_\infty$ can be derived.

It is well-known (\cite{bardos1992sharp, burq1997condition}) that GCC for $\omega$ is a sharp sufficient (and almost necessary) condition for the observability  of the adjoint wave equation. Namely, for any $T>T_{\min}(\omega,\Omega)>0$ (excluding the trivial case $\omega=\Omega$, in which $T_{\min}=0$), there exists a constant $C>0$, depending on $\omega,\Omega$ and $T$, such that for any pair of initial data $(p_0, p_1)\in L^2(\Omega)\times H^{-1}(\Omega)$, the corresponding solution $p$ to the adjoint wave equation
\begin{equation*}
\begin{cases}
\partial_t^2 p-\Delta p=0 &\text{ in }(0,T)\times\Omega,\\
p=0 &\text{ on }(0,T)\times\partial\Omega,\\
(p,\partial_t p)_{|_{t=T}}=(p_0, p_1) &\text{ in }\Omega,
\end{cases}
\end{equation*}
satisfies\footnote{At this point, we furthermore see that the stabilizability of the state equation, and the detectability of the adjoint one, must take place in the appropriate dual space. And this is linked precisely to the notion of having well-balanced norms penalized in the cost functional.}
\begin{equation}\label{eq: obs.ervedoza}
\int_\Omega |p(0,x)|^2\diff x + \|\partial_t p(0,\cdot)\|_{H^{-1}(\Omega)}^2 \leqslant C\int_0^T \int_{\omega} |p(t,x)|^2\diff t\diff x.
\end{equation}
The observability inequality \eqref{eq: obs.ervedoza} then yields\footnote{The observability inequality can actually be used to build more general feedback operators which ensure the exponential decay of the energy for the associated closed-loop wave system at any rate $\mu>0$ (\cite{komornik1997rapid}).} the stabilizability of the forward wave equation, in the sense that there exist $c\geqslant1$ and $\mu>0$, independent of the solution $y$ to the damped wave equation
\begin{equation*}
\begin{cases}
\partial_t^2 y-\Delta y+1_{\omega}\partial_t y=0 &\text{ in }(0,+\infty)\times\Omega,\\
y=0 &\text{ in }(0,+\infty)\times\partial\Omega,
\end{cases}
\end{equation*}
such that 
\begin{equation} \label{eq: energy.def}
E(y(t)):=\int_\Omega |\partial_t y(t,x)|^2\diff x + \int_\Omega |\nabla y(t,x)|^2 \diff x\leqslant c e^{-\mu t} E(y(0))
\end{equation}
holds for all $t\geqslant0$. The stabilizability of the underlying dynamics (combined with the equivalent characterization through Datko's theorem \cite{datko1972uniform}) is precisely the ingredient used in ensuring the uniform boundedness of $\mathscr{E}(T)$ with respect to $T$ (see \cite{porretta2013long}), and ultimately the exponential convergence of $\mathscr{E}(T)$ to $\mathscr{E}_\infty$, which allows to uncouple the optimality system, just as done for the heat equation in the previous section.
\end{proof}

\begin{remark}[Observation in the functional]\label{rem: considerations}
To prove turnpike, we need to ensure that the cost functional allows to recover
enough information on the state $(y_T,\partial_t y_T)$. This is in agreement with our discussions in preceding sections, in which we indicated the relevance of controllability/stabilizability for the turnpike phenomena to emerge.
\begin{enumerate}
\item[1.]
It is for this reason that we penalize $\|\nabla_x y(t)\|^2_{L^2(\Omega)}$, instead of solely $\|y(t)\|_{L^2(\Omega)}$ over $(0,T)$. (Actually, $\|\partial_t y(t)\|_{L^2(\Omega)}$ would also suffice, due to the equipartition of energy for the wave equation, according to which, modulo a compact remainder, the time-averages of $\|\partial_t y(t)\|_{L^2(\Omega)}$ and $\|\nabla_x y(t)\|_{L^2(\Omega)}$ are equivalent.) 
Indeed, if we were to solely penalize $\|y(t)\|_{L^2(\Omega)}$, there would already an apparent mismatch in the optimality system, which in such a case, would read as
\begin{equation}
\begin{cases}
\partial_t^2 y_T-\Delta y_T = p_T1_\omega &\text{ in }(0,T)\times\Omega,\\
\partial_t^2 p_T-\Delta p_T = y_T &\text{ in }(0,T)\times\Omega.
\end{cases}
\end{equation} 
We see that here the right hand side term of the adjoint equation is $y_T(t)\in H^1_0(\Omega)$, unlike in \eqref{eq: optimality.wave.eq}, where $\Delta y_T(t)\in H^{-1}(\Omega)$, which is the correct regularity for the source term to ensure that $p_T(t)\in L^2(\Omega)$. 
This mismatch results in the fact that the Riccati feedback operator $\mathscr{E}(T)$ cannot be ensured to converge exponentially to the regulator $\mathscr{E}_\infty$. 
\smallskip

\item[2.] In the cost functional, we had penalized $\nabla_x y(t, x)$ everywhere in $\Omega$, instead of solely within an open and non-empty subdomain $\omega_\circ\subset\Omega$. This was done solely to simplify the presentation, as one would need to localize the observation within $\omega_\circ$ through a cut-off function, whenever $\partial\omega_\circ\not\subset\partial\Omega$.
One could consider, for instance, a cut-off $\chi\in C^\infty_c(\mathbb{R}^d)$, with $\chi\equiv1$ in an appropriate compact subset $K\Subset\omega_\circ$, and $\chi\equiv0$ in a neighborhood of $\partial\omega_\circ$, as well as in $\Omega\setminus\omega_\circ$, and rather, minimize the functional
\begin{equation*}
\mathscr{J}_T(u):=\frac12\int_0^T\int_{\omega_\circ}\left|\chi\nabla\big(y(t,x)-y_d(x)\big)\right|^2 \diff x \diff t+\frac12\int_0^T\int_\omega |u(t,x)|^2\diff x\diff t.
\end{equation*} 
Just as assumed for $\omega$, for turnpike to hold, one needs to suppose that $\text{supp}(\chi)\subset\omega_\circ$ satisfies GCC, as to ensure the presence of an exponentially stabilizing mechanism with a damping localized through the cut-off $\chi$.
\end{enumerate}
\end{remark}

\begin{remark}[Further second-order examples] The above turnpike result may also be applied to other second-order systems.
\begin{enumerate}
\item[1.] More general settings for wave equations can also be considered (as done in \cite{zuazua2017large, trelat2018steady, grune2020exponential}); for instance, Neumann boundary conditions, or taking into account the medium heterogeneities through variable coefficients
\begin{equation*}
\rho(x)\partial_t^2 y - \nabla \cdot \Big(\sigma(x)\nabla y\Big) = u1_\omega \hspace{1cm} \text{ in } (0,T)\times\Omega,
\end{equation*}
where $\sigma,\rho$ are smooth up to the boundary, with $\sigma(x)>0$ and $\rho(x)>0$ for $x\in\overline{\Omega}$. 
\smallskip
\item[2.] One may also replace the Dirichlet Laplacian $-\Delta$ by the biharmonic operator $\Delta^2$ (and adapt the boundary conditions appropriately) -- this gives rise to the Euler-Bernouilli beam equation. Since the latter is controllable and observable (in \emph{any} positive time $T$ -- see \cite[Appendix 1]{lions1988controlabilite}, \cite[Proposition 7.5.7]{tucsnak2009observation}), the turnpike property also holds in this case.  
\end{enumerate}
\end{remark}

\begin{remark}[Lack of GCC] \label{rem: gcc}
\begin{enumerate}
\item[1.] 
If $\omega$ does not satisfy GCC, then one can ensure at least logarithmic decay for the smooth solution to the damped wave equation, namely logarithmic stabilizability for the wave equation, in the sense that the smooth solution $y$ to 
\begin{equation*}
\begin{cases}
\partial_t^2 y - \Delta y + 1_\omega \partial_t y = 0 &\text{ in }\mathbb{R}\times\Omega,\\
y=0 &\text{ in }\mathbb{R}\times\partial\Omega,\\
(y,\partial_t y)_{|_{t=0}} = (y^0, y^1) &\text{ in }\Omega,
\end{cases}
\end{equation*}
satisfy (recall the definition of the energy $E$ in \eqref{eq: energy.def})
\begin{align*}
E(y(t))\leqslant \frac{C_0}{(\log(2+t))^2} \left(\left\|y^0\right\|_{H^2(\Omega)}^2 + \left\|y^1\right\|_{H^1(\Omega)}^2\right).
\end{align*}
(See \cite{lebeau1996equation}.)
Proceeding by duality as done in \cite[Lemma 4.5]{porretta2013long}, one can then only ensure an estimate of the form
\begin{align} \label{eq: weaker.est}
&\|p(0)\|_{H^{-1}(\Omega)}^2 + \|\partial_t p(0)\|_{(H^1_0\cap H^2(\Omega))'}^2 \\
&\,\leqslant \frac{C\,T}{(\log(T+2))^2} \left(\left\|p^T\right\|_{L^2(\Omega)}^2 +\|p\|_{L^2((0,T)\times\omega)}^2+\|f\|_{L^2(0,T;H^{-1}(\Omega))}^2\right)\nonumber
\end{align}
for some $C>0$ independent of $T>0$, and for any $p_T\in L^2(\Omega)$, $f\in L^2(0,T; H^{-1}(\Omega))$, and the corresponding solution $p$ to
\begin{equation}
\begin{cases}
\partial_t^2 p-\Delta p = f &\text{ in }(0,T)\times\Omega,\\
p=0 &\text{ on }(0,T)\times\partial\Omega,\\
(p,\partial_t p)_{|_{t=T}}=(p^T,0)&\text{ in }\Omega.
\end{cases}
\end{equation}
This is a significantly weaker estimate than
\begin{align} \label{eq: gcc.1}
&\|p(0)\|^2_{L^2(\Omega)} + \|\partial_t p(0)\|^2_{H^{-1}(\Omega)} \\
&\,\leqslant C\left(\left\|p^T\right\|_{L^2(\Omega)}^2 +\|p\|_{L^2((0,T)\times\omega)}^2+\|f\|_{L^2(0,T;H^{-1}(\Omega))}^2\right), \nonumber
\end{align}
which holds for some $C>0$ independent of $T$ when $\omega$ satisfies GCC, both in terms of the topology\footnote{This topology is in fact very weak, as $(H^1_0(\Omega)\cap H^2(\Omega))'$ is not even a space of distributions, since $C^\infty_c(\Omega)$ is not dense in $H^1_0(\Omega)\cap H^2(\Omega)$. This space is nonetheless well suited to the study of evolution equations governed by the Laplacian, as it's simply the dual space of its domain, which can be characterized by Fourier expansion in the orthobasis of eigenfunctions.} for which it holds, and the fact that upper bound in \eqref{eq: weaker.est} will grow with $T$.
Inequality \eqref{eq: gcc.1} is implied by \eqref{eq: obs.ervedoza}, and is specifically used to prove the exponential decay of the Riccati feedback operator $\mathscr{E}(T)$ to $\mathscr{E}_\infty$. 
\medskip

\item[2.]
When GCC doesn't hold, one can still obtain an inkling of a turnpike property (albeit not an exponential turnpike property). More specifically, in \cite{porretta2013long}, for data $y^0\in H^2(\Omega)\cap H^1_0(\Omega)$ and $y^1\in H^1(\Omega)$, the authors show that 
\begin{align*}
&\frac{1}{T}\left(\int_0^T \|y_T(t)-\overline{y}\|_{L^2(\omega_\circ)}^2\diff t + \int_0^T \|u(t)-\overline{u}\|_{L^2(\omega)}^2 \diff t\right)\\
&\quad\leqslant \frac{C}{(\log(2+T))^2}\left(\left\|y^0\right\|_{H^2(\Omega)}^2 +\left\|y^1\right\|_{H^1(\Omega)}^2+\|\overline{p}\|_{H^2(\Omega)}^2\right).
\end{align*}
This is an \emph{integral turnpike} property, indicating the convergence, when $T\to+\infty$, of time averages of optimal evolutionary pairs to the corresponding optimal steady pair.
We refer to \cite{han2021slow} for recent results in the context of wave equations on planar graphs. The lack of exponential stabilizability is also typical in this context (\cite{dager2006wave, valein2009stabilization}). We also refer to \cite{gugat2019turnpike} for a direct strategy for proving integral turnpike properties tailored to first-order, linear hyperbolic systems.
\end{enumerate}
\end{remark} 

\subsection{Discussion}
\begin{remark}[On Assumptions \ref{ass: 1.2} and \ref{ass: 1.1}] \label{rem: assumptions.turnpike}
Let us make some observations regarding the assumptions, in particular, relating them with more familiar and easy-to-check controllability and observability properties, following \cite{porretta2013long}.
\begin{itemize}
\item We begin by noting that detectability for $(A,C)$ implies the existence of a constant $c>0$ such that for every $y\in C^0([0,T];\mathscr{H})$, $f\in L^2(0,T;X')$ and $y^0\in\mathscr{H}$ such that
\begin{equation*}
\begin{cases}
\partial_t y =Ay+ f &\text{ in }(0,T),\\
y_{|_{t=0}}=y^0,
\end{cases}
\end{equation*}
the inequality 
\begin{equation} \label{eq: weak.observability}
\|y(T)\|_{\mathscr{H}}^2 \leqslant c\left(\left\|y^0\right\|_{\mathscr{H}}^2+\int_0^T \|f(t)\|_{X'}^2 \diff t + \int_0^T\|Cy(t)\|_{\mathscr{H}}^2\diff t \right)
\end{equation}
holds for all $T>0$. 
This inequality is clearly satisfied (even with $C\equiv0$) whenever $-A$ is coercive on $X$, namely, $\langle -Af, f\rangle_{\mathscr{H}}\geqslant \beta\|f\|_X^2$ for some $\beta>0$ and all $f\in \mathfrak{D}(A)$, by straightforward energy estimates. Otherwise, the contribution of $Cy(t)$ is non-negligable, and the fulfillment of \eqref{eq: weak.observability} requires an effective interaction of the operator $C$ and the dynamics generated by $A$.
An analog result can be obtained for the adjoint system by making use of the stabilizability assumption (see \cite[Hypothesis 3.3]{porretta2013long}). We refer to \cite[Lemma 3.5]{porretta2013long} for a proof.
\smallskip 

\item In fact, in \cite{porretta2013long}, only \eqref{eq: weak.observability} is assumed, contrary to assuming the exponential detectability hypothesis. Analogously, a similar hypothesis is assumed for the adjoint system, which is then implied by the exponential stabilizability assumption we make here. This is done for simplicity of the presentation.
\smallskip

\item We also note that \eqref{eq: weak.observability} holds whenever a stronger estimate of the form
\begin{equation} \label{eq: whatever}
\|y(\tau)\|_{\mathscr{H}}^2 \leqslant c_\tau \left(\int_0^{\tau} \|f(t)\|_{\mathscr{H}}^2 \diff t + \int_0^{\tau} \|Cy(t)\|_{\mathscr{H}}^2\diff t\right)
\end{equation}
holds for some $c_\tau>0$ and for all $y$ such that $\partial_t y =Ay+f$ in $(0,\tau)$. 
To see this, one invokes \eqref{eq: whatever} over $(T-\tau,T)$ to obtain 
\begin{equation*}
\|y(T)\|_{\mathscr{H}}^2 \leqslant c_\tau \left(\int_{T-\tau}^{T} \|f(t)\|_{\mathscr{H}}^2 \diff t + \int_{T-\tau}^{T} \|Cy(t)\|_{\mathscr{H}}^2\diff t\right),
\end{equation*}
and so \eqref{eq: weak.observability} holds for $T\geqslant\tau$. The local well-posedness of the equation implies \eqref{eq: weak.observability} for $T\leqslant\tau$.
On another hand, by superposition, estimate \eqref{eq: whatever} holds if and only if the observability inequality 
\begin{equation} \label{eq: obs.ineq.str}
\left\|e^{\tau A}y^0\right\|_{\mathscr{H}}^2 \leqslant c_{\tau}\int_0^{\tau} \left\|Ce^{tA}y^0\right\|_{\mathscr{H}}^2 \diff t
\end{equation}
holds for all $\tau>0$, $y^0\in\mathscr{H}$, and for some $c_\tau>0$ depending only on $\tau, A$ and $C$. 
In other words, observability in the sense of \eqref{eq: obs.ineq.str} suffices for ensuring estimate \eqref{eq: weak.observability}. 
Note that an observability inequality such as \eqref{eq: obs.ineq.str} for $(A,C)$, which is actually equivalent to the null controllability of $(A^*,C^*)$, also implies the exponential detectability for $(A,C)$ (which, we recall, means that $(A^*,C^*)$ is exponentially stabilizable). 
Analogous conclusions hold for the stabilizability for $(A,B)$.
Both of these implications are part of the same, namely, the well-known fact that null-controllability implies exponential stabilizability (see \cite{haraux1989remarque} in the context of the wave equation, and \cite[Theorem 3.3, pp. 227]{tucsnak2009observation} for the general setting). See \cite{trelat2019characterization} for further details regarding these characterizations.
\smallskip

\item In the finite-dimensional case (in which $A\in\mathbb{R}^{d\times d}$, $B\in\mathbb{R}^{d\times m}$ and $C\in\mathbb{R}^{s\times d}$), stabilizability and detectability are not only sufficient, but also necessary for having exponential turnpike (see \cite[Theorem A.3]{esteve2020turnpikea}). The necessity of these assumptions in the PDE context is also likely, but has not been demonstrated in full generality to our knowledge.
\end{itemize}
\end{remark}

\begin{remark}[Existence of steady minimizers] \label{rem: steady.well.posed}
To ensure the existence and uniqueness of minimizers to $\mathscr{J}_s$ defined in \eqref{eq: static.abstract.ocp}, namely solutions to the latter, one would again look to apply the direct method in the calculus of variations. However, due to the fact that we are now optimizing over pairs $(u,y)$, coercivity of $\mathscr{J}_s$ with respect to $y$ in the norm of $X$ is also needed. And said coercivity follows from \eqref{eq: weak.observability}. To see as to why this is the case, we note that \eqref{eq: weak.observability} implies that there exists a constant $c_1>0$ such that
\begin{equation} \label{eq: stationary.coercivity}
\|y\|_X^2 \leqslant c_1\Big(\|Ay\|_{X'}^2 + \|Cy\|_{\mathscr{H}}^2\Big)
\end{equation}
holds for all $y\in X$. 
Indeed, applying \eqref{eq: weak.observability} (which is implied by Assumption \ref{ass: 1.1}, per the previous remark) to $\zeta(t):=ty$ for an arbitrary $y\in X$, we get
\begin{equation*}
T^2\|y\|_{\mathscr{H}}^2 \leqslant 2c \frac{T^3}{3} \Big(\|Ay\|_{X'}^2+\|Cy\|_{\mathscr{H}}^2\Big) + 2c T\|y\|_{X'}^2.
\end{equation*}
By virtue of $\mathscr{H}\hookrightarrow X'$, and choosing $T\gg c$, we find 
\begin{equation*}
\|y\|_{\mathscr{H}}^2 \leqslant c_0\Big(\|Ay\|_{X'}^2 + \|Cy\|_{\mathscr{H}}^2\Big).
\end{equation*}
The conclusion then follows by adding $\langle -Ay, y\rangle_{X',X}$ on both sides of the estimate, and using the coercivity assumption on $-A$ and $A\in\mathscr{L}(X,X')$. Note that from \eqref{eq: stationary.coercivity}, one readily sees that the functional $\mathscr{J}_s(u,y)$ defined in \eqref{eq: static.abstract.ocp} is coercive with respect to $(u,v)$ in the $\mathscr{U}\times X$--norm. This, combined with the strict convexity of the problem allows to apply the direct method and derive existence and uniqueness of solutions to \eqref{eq: static.abstract.ocp}.
\end{remark}

\begin{remark}[Boundary control] \label{rem: boundary.control}
In the context of boundary control, for instance, when $y(t,x)=u(t,x)1_{\Gamma}$ on $(0,T)\times\partial\Omega$ where $\Gamma\subset\partial\Omega$ is open and non-empty, instead of having distributed controls of the form $u1_{\omega_\circ}$ as in \eqref{eq: heat.equation}, the turnpike property as stated above still holds. It is however not a direct consequence of Theorem \ref{thm: porretta.zuazua.2}, which assumed that $B\in\mathscr{L}(\mathscr{U},X')$, a hypothesis which is not satisfied by trace operators. 
In the context of the simple heat equation \eqref{eq: heat.equation}, the proof can be adapted by making use of a prudent lifting of the trace, albeit at the cost of additional technicalities. 
In the abstract setting of Theorem \ref{thm: porretta.zuazua.2}, the proof requires introducing the concept of \emph{admissible} control operators $B$ (see \cite{tucsnak2009observation}). We merely stated the result in the context of bounded control operators to avoid many unnecessary technical details.
The proof of the turnpike property for such control operators may be found in \cite{trelat2018steady} and in \cite{grune2020exponential}.
\end{remark}

\begin{remark}[Tracking boundary observations] 
In many of the examples we mentioned, the observation operator $C$ is a bounded linear operator on $\mathscr{H}$. For example, this is usually the case when we can observe the state $y$ within an arbitrarily small, open subset $\omega_\circ\subset\Omega$, in which case, $Cy=y|_{\omega_\circ}$ and $\mathscr{H}=L^2(\Omega)$.
However, in applications stemming from geophysics and tomography, among many others, it is natural to think of a regression problem in which only boundary measurements of the state are tracked. Namely, one could imagine having an observation operator given by the Neumann trace, say, on the entire boundary $\partial\Omega$:
\begin{equation*}
Cy=\partial_\nu y, \hspace{1cm} \text{ for } y\in X.
\end{equation*}
In this case, the operator $C$ is not bounded on $\mathscr{H}$, or even from $X$ to $\mathscr{H}$; rather, it is defined on a domain $\mathfrak{D}(C)$ which is dense in $X$, and its range is typically a subset of some other Hilbert space $\mathscr{V}$. But this does not a priori allow to consider the adjoint $C^*$ as an operator $C^*\in\mathscr{L}(\mathscr{V}', X')$, which would allow us, given the optimal steady state $\overline{y}\in X$, to define the steady adjoint state $\overline{p}\in X$ as satisfying
\begin{equation} \label{eq: AstarCstar}
A^*\overline{p} = C^*\iota(C\overline{y}-y_d),
\end{equation}
where $\iota: \mathscr{V}\to\mathscr{V}'$ is the natural injection of $\mathscr{V}$ in its dual $\mathscr{V}'$.

A remedy for this issue is to define the adjoint state $\overline{p}$ through a transposition argument. We focus on the stationary adjoint state -- the evolution problem follows a similar argument. 
Let us assume that there exists some functional space $\mathscr{W}\subset X$ such that $C\in\mathscr{L}(\mathscr{W},\mathscr{V})$ and, simultaneously, such that $A\in\mathscr{L}(\mathscr{W},\mathscr{H})$ is invertible.  
In this case, the adjoint state $\overline{p}\in\mathscr{H}$ can be defined, instead of \eqref{eq: AstarCstar}, by solving the equation
\begin{equation*}
\langle \overline{p}, A\varphi\rangle_{\mathscr{H}} = \langle C\overline{y}-y_d, C\varphi\rangle_{\mathscr{V}}, \hspace{1cm} \text{ for all } \varphi\in \mathscr{W}.
\end{equation*}
For example, in the case of Neumann trace observation, and working with the Dirichlet Laplacian and distributed controls, one would have $X=H^1_0(\Omega)$ and $\mathscr{W}=H^2(\Omega)\cap H^1_0(\Omega)$. This transposition argument only slightly changes the proof of turnpike, namely the definition of the optimality system (see \cite{porretta2013long}).
\end{remark}

\section{A diagonalization strategy} \label{sec: 5}

The proof of turnpike presented in what precedes can be slightly tweaked to obtain a version which may be seen as even more illustrative. 
In the finite dimensional case, this variation relies on essentially diagonalizing the optimality system, leading, as before, to an uncoupled system for which the asymptotics are transparent. This is done by noting that the optimality system can be written as a shooting problem governed by a matrix which, under the Kalman rank condition, is hyperbolic, namely has eigenvalues with non-zero real part. Presented in \cite{trelat2015turnpike} for the finite-dimensional LQ case (and actually for nonlinear problems by linearization and smallness, as discussed in Part 2), the strategy has also been extended in \cite{trelat2018steady} to the PDE setting. 

For the sake of clarity, let us sketch the idea of this strategy in the finite dimensional case. The PDE setting can be dealt with in a similar way, albeit with some minor technical changes, as done so in \cite{trelat2018steady}. (Furthermore, the controllability assumption entailed by the Kalman rank condition can be relaxed to a stabilizability assumption, as seen in the latter paper.) 
We consider systems of the form
\begin{equation} \label{eq: finite.dim.sys}
\begin{cases}
\dot{y} = Ay + Bu &\text{ in }(0,T),\\
y(0) = y^0,
\end{cases}
\end{equation}
where now $A\in\mathbb{R}^{d\times d}(\mathbb{R})$ and $B\in\mathbb{R}^{d\times m}(\mathbb{R})$, with $d,m\geqslant 1$ (and, typically, $d>m$). 
We now consider the following optimal control problem (which, can be made slightly more general, but we avoid doing so, for simplicity): 
\begin{equation} \label{eq: finite.dim.ocp}
\inf_{\substack{u\in L^2(0,T;\mathbb{R}^m)\\ y \text{solves} \eqref{eq: finite.dim.sys}}} \frac12\int_0^T \|y(t)-y_d\|^2 \diff t + \frac12 \int_0^T \|u(t)\|^2\diff t.
\end{equation}
Here, $y_d\in\mathbb{R}^d$ is given. The existence and uniqueness of a solution to \eqref{eq: finite.dim.ocp} requires no specific assumptions on $A$ or $B$, unlike for  turnpike, as seen just below.

The turnpike property then naturally also holds for the unique optimal pair $(u_T,y_T)$ solving \eqref{eq: finite.dim.ocp}, under similar stabilizability and detectability assumptions. 
We shall assume a stronger property on the dynamics. Namely, we suppose that the Kalman rank condition
\begin{equation*}
\text{rank}\left(\left[B\, AB\, \ldots\, A^{d-1}B\right]\right) = d
\end{equation*}
holds. The corresponding steady optimal control problem reads as 
\begin{equation*} \label{eq: steady.ocp.fin.dim}
\inf_{\substack{(u,y)\in\mathbb{R}^m\times \mathbb{R}^d\\ Ay+Bu=0}} \frac12 \|y-y_d\|^2 + \frac12 \|u\|^2.
\end{equation*}
Problem \eqref{eq: steady.ocp.fin.dim} admits a unique solution, since $\text{ker}(A^*)\cap\text{ker}(B^*)=\{0\}$ by virtue of the Kalman rank condition. 
Then, writing the optimality systems for both the optimal time-dependent triple $(u_T, y_T, p_T)$ and the steady triple $(\overline{u},\overline{y},\overline{p})$, 
where $u_T\equiv B^*p_T$ and $\overline{u}\equiv B^*\overline{p}$, and setting 
\begin{equation*}
\delta y(t):=y_T(t) - \overline{y}, \hspace{1cm} \delta p(t):=p_T(t)-\overline{p},
\end{equation*}
we see that $\delta y(t)$ and $\delta p(t)$ satisfy
\begin{equation*}
\begin{cases}
\delta\dot{y}(t) = A\delta y(t) + BB^*\delta p(t) &\text{ in }(0,T),\\
\delta \dot{p}(t) = \delta y(t) - A^*\delta p(t) &\text{ in }(0,T),\\
\delta y(0) = y^0 - \overline{y},\\
\delta p(T) = -\overline{p}.
\end{cases}
\end{equation*}
But, by setting $\*z:=\left[\delta y^\top, \delta p^\top\right]^\top$, this system can then be seen as a shooting problem for the linear differential system
\begin{equation*}
\dot{\*z}(t) = \mathfrak{H} \*z(t), \quad \text{ in } (0,T),
\end{equation*}
where the matrix $\mathfrak{H}\in\mathbb{R}^{2d\times 2d}(\mathbb{R})$ (designating a \emph{Hamiltonian} matrix) is given by
\begin{equation} \label{eq: def.ham}
\mathfrak{H}:=
\begin{bmatrix}
A & BB^*\\
\text{Id} & -A^*
\end{bmatrix},
\end{equation}
and for which a part of the initial and final data are imposed. The shooting problem consists in determining the initial condition $\delta p(0)$ for which $\*z(t)$, starting at $\*z(0)=\left[(y^0-\overline{y})^\top,\delta p(0)^\top\right]$, satisfies $\delta p(T)=-\overline{p}$. 
The critical observation is that, under the Kalman rank condition, the matrix $\mathfrak{H}$ is \emph{hyperbolic}, namely 

\begin{lemma} The matrix $\mathfrak{H}$ in \eqref{eq: def.ham} is hyperbolic, in the sense that if $\lambda\in\mathbb{C}$ is an eigenvalue of $\mathfrak{H}$, then $\mathrm{Re}(\mathfrak{H})\neq0$. Moreover, if $\lambda$ is an eigenvalue of $\mathfrak{H}$, then so is $-\lambda$.
\end{lemma}

This is precisely what we have seen in Figure \ref{fig: spectral.dic}: the coupling in the optimality system, stemming from the tracking term, instills a stabilizing and symmetric structure.
The proof of this lemma is in fact quite important in the general strategy, so we sketch it.

\begin{proof} Let $\mathscr{E}_-$ (resp. $\mathscr{E}_+$) be the symmetric negative definite matrix (resp., the symmetric positive definite matrix) solution of the algebraic Riccati equation (see \cite{vinter2010optimal}):
\begin{equation*}
XA + A^*X + XBB^*X - \text{Id} = 0.
\end{equation*}
Note that uniqueness of solutions follows from the controllability assumption. 
Setting 
\begin{equation*}
P=\begin{bmatrix}
\text{Id} & \text{Id}\\
\mathscr{E}_- & \mathscr{E}_+
\end{bmatrix}
\end{equation*}
we see that the matrix $P$ is invertible, and in fact
\begin{equation*}
P^{-1} \mathfrak{H} P = 
\begin{bmatrix}
A+BB^*\mathscr{E}_- & 0 \\ 
0 & A+ BB^*\mathscr{E}_+
\end{bmatrix}.
\end{equation*}
Now the fact that the matrix $A+BB^*\mathscr{E}_-$ has (complex) eigenvalues with negative real parts is a known property of algebraic Riccati theory, due to the fact that $(A,B)$ satisfies the Kalman rank condition (\cite{vinter2010optimal}). On another hand, subtracting the Riccati equations satisfied by $\mathscr{E}_+$ and $\mathscr{E}_-$, we find
\begin{equation*}
(\mathscr{E}_+-\mathscr{E}_-)(A+BB^*\mathscr{E}_+) + (A+BB^*\mathscr{E}_-)^*(\mathscr{E}_+-\mathscr{E}_-) = 0.
\end{equation*}
Since $\mathscr{E}_+-\mathscr{E}_-$ is invertible, it follows that the eigenvalues of $A+BB^*\mathscr{E}_+$ are the negative of those of $A+BB^*\mathscr{E}_-$. This concludes the proof.
\end{proof}

The above proof motivates working in a different coordinate system in view of understanding the turnpike asymptotics. In fact, the proof allows to diagonalize $M$ in a rather appropriate way.
We consider the change of variable
\begin{equation*}
\*z(t) = 
\begin{bmatrix}
\text{Id} & \text{Id}\\
\mathscr{E}_- & \mathscr{E}_+
\end{bmatrix}\*x(t),
\end{equation*}
to then find that $\*x(t)$ satisfies
\begin{equation*}
\dot{\*x}(t) = \begin{bmatrix}
A+BB^*\mathscr{E}_- & 0 \\ 
0 & A+ BB^*\mathscr{E}_+
\end{bmatrix}\*x(t).
\end{equation*}
But now the above system, consisting of $2d$ equations, is purely hyperbolic, namely it is governed by a matrix with eigenvalues with non-zero real part, and is also symmetric. Thus, the first $d$ equations represent a contracting system forward in time, and the last $d$ ones represent a contracting system backward in time. 
To be more precise, setting $\*x(t) = [\zeta(t), \eta(t)]$, we find that
\begin{equation*}
\begin{cases}
\dot{\zeta}(t) = (A+BB^*\mathscr{E}_-)\zeta(t) &\text{ in }(0,T),\\
\dot{\eta}(t) = (A+BB^*\mathscr{E}_+)\eta(t) &\text{ in }(0,T).
\end{cases}
\end{equation*}
And since all the eigenvalues of $A+BB^*\mathscr{E}_-$ have negative real parts, while the ones of $A+BB^*\mathscr{E}_+$ are the negative of those of $A+BB^*\mathscr{E}_-$, it follows that 
\begin{equation*}
\|\zeta(t)\| \leqslant c\|\zeta(0)\| e^{-\lambda t}, \qquad \|\eta(t)\|\leqslant c\|\eta(T)\| e^{-\lambda(T-t)} 
\end{equation*}
for some $c>0$ independent of $(\zeta,\eta)$ and $T$, and for every $t\in[0,T]$, where $\lambda$ is the spectral abscissa of the matrix $A+BB^*\mathscr{E}_-$, namely
\begin{equation*}
\lambda = -\max\left\{\Re(\mu)\,\Bigm|\, \mu \in \text{spec}(A+BB^*\mathscr{E}_-)\right\}>0.
\end{equation*}
We thus recover the same decay rate as the one obtained via the strategy presented in what precedes. We refer to \cite{trelat2015turnpike, trelat2018steady} for technical details.

\section{Dissipativity and measure turnpike} \label{sec: 6} Up to now, we only focused on a characterization of the turnpike property by means of a double-arc exponential decay estimate: when $T\gg1$, point-wise, the optimal triple is $\mathcal{O}\left(e^{-t} + e^{-(T-t)}\right)$ for all $t\in[0,T]$. 
And we refer to such an estimate as the \emph{exponential turnpike property}.
There exist, however, weaker notions and characterizations, which warrant some attention, in particular due to a breadth of existing techniques, and the possibility of including state and control constraints. One of them is the so called \emph{measure turnpike property}, which states that for all $\varepsilon>0$, the measure of the set of times $t\in[0,T]$ for which $\|y_T(t)-\overline{y}\|_{\mathscr{H}} + \|u_T(t)-\overline{u}\|_{\mathscr{U}}$ is larger than $\varepsilon$, is not "too big". It is noteworthy that in some settings, a sufficient condition for this property to hold can be seen as an extension of Lyapunov's second method. This is the so called \emph{dissipativity of systems}, in the sense of Willems \cite{willems1972dissipative} (see \cite{faulwasser2017turnpike} for a contemporary treatment). 
In other words, the study of dissipativity can be seen, to a certain regard, as a Lyapunov-akin strategy (namely, an extension of Lyapunov to an open-loop setting) to proving the turnpike property.

Let us provide some more details to this discussion in the context of PDEs, for which we follow \cite{trelat2018integral}. (In fact, in \cite{trelat2018integral}, the results are stated and proven for more general nonlinear systems, but the theory being local around a steady pair, we focus on the linear case here.) 
We borrow the notations from previous sections, and consider 
\begin{equation} \label{eq: trelat.ocp..}
\inf_{\substack{u \in L^2(0,T; \mathscr{U})\\ y \text{ solves} \eqref{eq: trelat.eq..}}} \underbrace{\int_0^T f^0(y(t), u(t)) \diff t}_{:=\mathscr{J}_T(u)},
\end{equation}
where
\begin{equation} \label{eq: trelat.eq..}
\begin{cases}
\partial_t y =Ay+ Bu &\text{ in }(0,T),\\
y_{|_{t=0}} = y^0.
\end{cases}
\end{equation}
Here, we assume that $f^0\in C^0(X\times\mathscr{U};\mathbb{R})$ is bounded from below, convex, and coercive with respect to the $X\times\mathscr{U}$--norm; once again, $A$ is supposed to generate a continuous semigroup on $\mathscr{H}$, and $B\in\mathscr{L}(\mathscr{U},\mathscr{H})$. Accordingly, as before, for $T>0$, \eqref{eq: trelat.eq..} admits a unique (mild) solution $y\in C^0([0,T]; \mathscr{H})$ for data $y^0\in\mathscr{H}$ and $u\in L^2(0,T; \mathscr{U})$, whereas, due to continuity, convexity, and coercivity, \eqref{eq: trelat.ocp..} can be shown to admit a minimizer by the direct method in the calculus of variations. 
The corresponding steady optimal control problem then reads
\begin{equation} \label{eq: static.ocp.trelat..}
\inf_{\substack{(u,y)\in\mathscr{U}\times X\\ Ay+Bu=0}} f^0(y,u).
\end{equation}
We denote $\mathscr{J}_s(u):=f^0(y,u)$. We shall assume that \eqref{eq: static.ocp.trelat..} admits a solution (see \cite{trelat2018integral} for more details).

We shall distinguish pairs $(u,y)$ which are optimal and admissible for \eqref{eq: trelat.ocp..}. Namely, we say that the pair $(u,y)\in L^2(0,T; \mathscr{U})\times C^0([0,T];\mathscr{H})$ is \emph{admissible} for \eqref{eq: trelat.ocp..} if $\partial_t y =Ay+ Bu$ for $t\in(0,T)$. 
We say that the pair $(u,y)$ is \emph{optimal} if, in addition to being admissible, $\mathscr{J}_T(u)\leqslant\mathscr{J}_T(v)$ for all functions $v\in L^2(0,T; \mathscr{U})$, and $y(0)=y^0$ hold. 
In particular, any optimal steady pair $(\overline{y}, \overline{u})$ for \eqref{eq: static.ocp.trelat..} is also admissible for \eqref{eq: trelat.ocp..}.

We may begin by defining the relevant notions of dissipativity.

\begin{definition}[Dissipativity] 
Let $T>0$. 
We say that \eqref{eq: trelat.ocp..} is \emph{dissipative} at an optimal steady pair $(\overline{u}, \overline{y})$ solving \eqref{eq: static.ocp.trelat..}, if there exists a \emph{storage function} $\*S: \mathscr{H}\to\mathbb{R}$, locally bounded and bounded from below, such that for any $T>0$, the inequality
\begin{equation*}
\*S(y(\tau)) - \*S(y(0))\leqslant\int_0^\tau \Big(f^0(y(t),u(t)) - f^0(\overline{y},\overline{u})\Big)\diff t
\end{equation*}
holds for any $\tau\in[0,T]$ and for any optimal pair $(u, y)$ solution to \eqref{eq: trelat.ocp..}. 

We say that \eqref{eq: trelat.ocp..} is \emph{strictly dissipative} at an optimal steady pair $(\overline{u}, \overline{y})$ solving \eqref{eq: static.ocp.trelat..}, if there exists a nonnegative function $\alpha\in C^0([0,+\infty))$, with $\alpha$ strictly increasing and\footnote{Such functions $\alpha$ are said to be of \emph{class} $\mathcal{K}$.} $\alpha(0)=0$, and a storage function $\*S: \mathscr{H}\to\mathbb{R}$, locally bounded and bounded from below, such that for any $T>0$, the inequality
\begin{align*}
\*S(y(\tau))-\*S(y(0)) &\leqslant \int_0^\tau \Big(f^0(y(t),u(t)) - f^0(\overline{y},\overline{u})\Big)\diff t \\
&\quad- \int_0^\tau \alpha\left(\Big\|\big(y(t)-\overline{y}, u(t)-\overline{u}\big)\Big\|_{\mathscr{H}\times\mathscr{U}}\right) \diff t
\end{align*}
holds for any $\tau\in[0,T]$  and for any optimal pair $(u, y)$ solution to \eqref{eq: trelat.ocp..}. 
\end{definition}

Let us provide some comments regarding the above definitions. 
The function $$\omega(y, u):= f^0(y,u) - f^0(\overline{y},\overline{u}),$$ with respect to which dissipativity is defined, is usually referred to as the \emph{supply rate function}.
We then note that for dissipativity to hold, it suffices to find a $C^1$, non-negative function $\*S$ satisfying 
\begin{equation*}
\frac{\diff}{\diff t} \*S(y(t)) \leqslant \omega(y(t),u(t))
\end{equation*}
for all $t\in[0,T]$ along optimal pairs $(y,u)$. This makes the storage function $\*S$  akin to a Lyapunov functional, the difference being the presence of the supply rate $\omega$, which accounts for the energy input in the system due to the presence of an open-loop control $u(t)$. (Recall that the Lyapunov stability method applies to systems without inputs: $\dot{y}(t)=f(y(t))$.) The supply rate indicates, in some sense, the total external energy added to the system at time $t$. And so, there can be no internal "creation of energy", rather, only internal dissipation of energy.  
Strict dissipativity entails a stronger differential inequality, of the form
\begin{equation*}
\frac{\diff}{\diff t} \*S(y(t)) \leqslant \omega(y(t),u(t))-\alpha\left(\Big\|\big(y(t)-\overline{y}, u(t)-\overline{u}\big)\Big\|_{\mathscr{H}\times\mathscr{U}}\right)
\end{equation*}
for all $t\in[0,T]$. While sufficient and illustrative, this is not a necessary assumption as looking for a differentiable storage function is rather restrictive. In fact, as discussed in \cite{trelat2018integral}, the value function for \eqref{eq: trelat.ocp..} is always a storage function, but is not differentiable for many optimal control problems.
The following theorem holds.

\begin{theorem}[\cite{trelat2018integral}]  Suppose that there exists some constant $M>0$ such that for any $T>0$, any optimal pair $(u_T, y_T)$ for \eqref{eq: trelat.ocp..} is such that 
\begin{equation*}
\|y_T(t)\|_{\mathscr{H}}+\|u_T(t)\|_{\mathscr{U}}\leqslant M
\end{equation*}
for a.e. $t\in[0,T]$. Let $(\overline{u}, \overline{y})$ be \emph{some} solution to \eqref{eq: static.ocp.trelat..}. 
\begin{enumerate}
\item[1.] Suppose furthermore that \eqref{eq: trelat.ocp..} is dissipative at $(\overline{u}, \overline{y})$.
Then 
\begin{equation*}
\frac{\mathscr{J}_T}{T} = \mathscr{J}_s + \mathcal{O}\left(\frac{1}{T}\right) \hspace{1cm} \text{ as } T\to+\infty.
\end{equation*}
\item[2.] Suppose furthermore that \eqref{eq: trelat.ocp..} is strictly dissipative at $(\overline{u}, \overline{y})$. 
Then for every $\varepsilon>0$, there exists $\kappa(\varepsilon)>0$ such that 
\begin{equation*}
\text{meas}\left(\left\{t\in[0,T]\, \Biggm|\, \Big\|\big(y_T(t)-\overline{y}, u_T(t)-\overline{u}\big)\Big\|_{\mathscr{H}\times\mathscr{U}}>\varepsilon\right\}\right)\leqslant \kappa(\varepsilon)
\end{equation*}
holds for all $T>0$, where $(u_T,y_T)$ is an optimal pair for \eqref{eq: trelat.ocp..}.
\end{enumerate}
\end{theorem}

The first result in the above theorem is usually referred to as the \emph{integral turnpike property} -- time averages of the functional converge to the stationary functional as $T\to+\infty$. This ergodic-like pattern is a relatively weak property and can be proven by means of a variety of techniques (mainly energy estimates; see however \cite{mazari2020quantitative} for a proof by means of so-called quantitative inequalities in the presence of state constraints).
On the other hand, the second property is referred to as the \emph{measure turnpike property}, and states that the measure of the set of times where an optimal control and state pair are away from some optimal steady control and state pair is not "too big". 

This being said, the constant $\kappa(\varepsilon)$ is of the form $\sfrac{C}{\alpha(\varepsilon)}$ for some constant $C$ independent of $\varepsilon$ and $T$. Since $\alpha$ is increasing and $\alpha(0)=0$, we see that as $\varepsilon$ goes to $0$, the upper bound for the measure of the set grows. Furthermore, it is not apparent specifically where the time instances at which the discrepancies of the time-depending pairs to the steady pairs are small, are located. In comparison, the exponential turnpike property provides the exact distribution of these time instances.
Note that the exponential turnpike property implies both of the above statements.
 And while it is not always clear how to find a storage function $\*S$ for PDEs beyond  LQ problems, wherein sufficient conditions are known for the exponential turnpike property to hold, we do refer the reader to \cite[Section 4]{trelat2018integral}, where the authors devise a clever duality method for finding a storage function. We refer to \cite{brogliato2007dissipative} for further insights regarding sufficient conditions for storage functions -- we emphasize that this is a delicate question in general.
 
 \begin{remark}[Constraints]
Note that in \cite{trelat2018integral}, the authors impose constraints on the admissible pairs $(u, y)$ within the optimal control problem, namely, that $(u(t),y(t))$ lie in a compact subset of $\mathscr{U}\times\mathscr{H}$ for a.e. $t\in[0,T]$. In particular, this would mean that the assumption in the statement is satisfied. 
\end{remark}
 
\begin{remark}[Enhancing measure turnpike]
Strict dissipativity is a rather strong assumption for ensuring the measure turnpike property, which, as said above, is rather weak when compared to the exponential turnpike property. But actually, under the assumption of strict dissipativity, in \cite{trelat2020linear} it is shown that for almost every $s\in(0,1)$, $y_T(sT)\to\overline{y} $ and $u_T(sT)\to\overline{u}$ as $T\to+\infty$, which is a significantly stronger result. 
In fact, it can be said that, in some sense, the turnpike property is engraved within the notion of strict dissipativity.
 \end{remark}

\begin{proof} 
We focus on proving the measure turnpike property only. 
Let $T>0$ and let $(u_T, y_T)$ be any optimal pair for \eqref{eq: trelat.ocp..}. 
For $\varepsilon>0$, let us denote
\begin{equation*}
\mathscr{Q}_{\varepsilon,T} := \left\{t\in[0,T]\,\Biggm|\, \Big\|\big(y_T(t)-\overline{y}, u_T(t)-\overline{u}\big)\Big\|_{\mathscr{H}\times\mathscr{U}}>\varepsilon\right\}.
\end{equation*}
We readily see that
\begin{equation} \label{eq: meas.Qeps}
\text{meas}(\mathscr{Q}_{\varepsilon,T}) = \int_0^T 1_{\mathscr{Q}_{\varepsilon,T}} \diff t = \frac{1}{\alpha(\varepsilon)} \int_0^T \alpha(\varepsilon)1_{\mathscr{Q}_{\varepsilon,T}} \diff t.
\end{equation}
Since $\alpha$ is a non-decreasing function, we find that
\begin{equation} \label{eq: est.meas.1}
\frac{1}{\alpha(\varepsilon)} \int_0^T \alpha(\varepsilon)1_{\mathscr{Q}_{\varepsilon,T}} \diff t \leqslant \frac{1}{\alpha(\varepsilon)}\int_0^T \alpha\left(\Big\|\big(y(t)-\overline{y}, u(t)-\overline{u}\big)\Big\|_{\mathscr{H}\times\mathscr{U}}\right) \diff t.
\end{equation}
On another hand, by strict dissipativity, we have
\begin{align}
\frac{1}{\alpha(\varepsilon)}\int_0^T \alpha\left(\Big\|\big(y(t)-\overline{y}, u(t)-\overline{u}\big)\Big\|_{\mathscr{H}\times\mathscr{U}}\right) \diff t &\leqslant \mathscr{J}_T(u_T)-T\mathscr{J}_s(\overline{u})\nonumber \\
&\, + \*S(y_T(0)) - \*S(y_T(T)) \label{eq: before.last}.
\end{align}
And then, using the fact that an optimal steady pair $(\overline{u}, \overline{y})$ is admissible for \eqref{eq: trelat.ocp..}, we also find
\begin{equation} \label{eq: after.before.last}
\mathscr{J}_T(u_T)\leqslant \mathscr{J}_T(\overline{u})=\int_0^T f^0(\overline{u},\overline{y})\diff t = T\mathscr{J}_s(\overline{u}).
\end{equation}
Plugging \eqref{eq: after.before.last} in \eqref{eq: before.last}, we find 
\begin{equation*}
\frac{1}{\alpha(\varepsilon)}\int_0^T \alpha\left(\Big\|\big(y(t)-\overline{y}, u(t)-\overline{u}\big)\Big\|_{\mathscr{H}\times\mathscr{U}}\right) \diff t\leqslant \*S(y_T(0)) - \*S(y_T(T)).
\end{equation*}
Now since $y_T(t)$ is bounded by assumption, and $\*S$ is locally bounded, 
there exists a constant $C=C(M)>0$ (depending only on $M$, and independent of $T$ and $\varepsilon$) such that $|\*S(y)|\leqslant C$ for all $y\in\mathscr{H}$. And so, we find that
\begin{equation} \label{eq: est.meas.2}
\frac{1}{\alpha(\varepsilon)}\int_0^T \alpha\left(\Big\|\big(y(t)-\overline{y}, u(t)-\overline{u}\big)\Big\|_{\mathscr{H}\times\mathscr{U}}\right) \diff t\leqslant 2C.
\end{equation}
Whence, combining \eqref{eq: meas.Qeps}, \eqref{eq: est.meas.1}, and \eqref{eq: est.meas.2}, we deduce that
\begin{equation*}
\text{meas}(\mathscr{Q}_{\varepsilon,T}) \leqslant \frac{2C}{\alpha(\varepsilon)},
\end{equation*}
as desired.
\end{proof}

The theory of dissipativity has been applied for obtaining turnpike results for discrete-time, finite-dimensional systems, as well as LQ problems in finite dimensions (\cite{damm2014exponential, grune2016relation, grune2018strict, grune2018turnpike, grune2017relation, grune2021relation, faulwasser2019towards, faulwasser2021dissipativity, gugat2021turnpike}). 
The results in these works are mostly \emph{measure} turnpike properties (or \emph{cardinal} turnpike, in the discrete-time setting), for continuous-time and discrete-time respectively, and can be enhanced to exponential turnpike under stabilizability and detectability assumptions.
We refer also the reader to the survey \cite{faulwasser2020turnpike} for an in-depth overlook and bibliography of this theory, which we only touched upon. 

\section{Beyond} \label{sec: 7}

We tried to provide an all-encompassing review of existing results regarding turnpike for LQ problems for partial differential equations. There are several topics that we did not present in great depth, and related open problems. 

\subsection{Third proof of exponential turnpike}
The strategies we presented in what precedes are, of course, not definitive in the linear turnpike theory. In particular, in \cite{grune2020exponential, grune2019sensitivity},  the authors derive the exponential turnpike property for LQ problems for abstract linear PDEs (see \cite{grune2021abstract} for an extension to semilinear parabolic PDEs by linearization and smallness, a strategy presented in the subsequent section), written in the canonical form $\dot{y}=Ay+Bu$, under the same stabilizability and detectability assumptions for $(A,B,C)$ we made above. 
Note that the framework of these papers accounts for possibly unbounded control operators $B$; these need only be assumed admissible, thus covering boundary control systems (see \cite{tucsnak2009observation} for more detail on these notions).
In these works, the authors write the entire optimality systems as a linear operator equation in Bochner spaces; for the time-dependent optimal triple $(u_T,y_T,p_T)$ for instance, one has 
\begin{equation*}
\underbrace{
\begin{bmatrix}
C^*C &-\partial_t - A^*\\
0 & \Psi_T\\
\partial_t - A & -BB^*\\
\Psi_0 & 0
\end{bmatrix}
}_{:=M_T}
\begin{bmatrix}
y_T\\
p_T
\end{bmatrix}
=
\begin{bmatrix}
C^*Cy_d\\
0\\
0\\
y^0
\end{bmatrix}.
\end{equation*}
Here, $C^0([0,T];\mathscr{H})\ni\Psi_t y:=y(t)\in \mathscr{H}$ for $t\in[0,T]$. 
Defining the perturbation variables $\delta y(t):=y_T(t)-\overline{y}$ and $\delta p(t)=p_T(t)-\overline{p}$, one then finds
\begin{equation*}
M_T\begin{bmatrix}\delta y\\ \delta p\end{bmatrix} = \begin{bmatrix}0\\-\overline{p}\\0\\y^0-\overline{y}\end{bmatrix}.
\end{equation*} 
The authors can then, using mostly energy estimates, first prove an estimate of the form
\begin{align*}
&\|\delta y(t)\|_{\mathscr{H}} + \|\delta p(t)\|_{\mathscr{H}}\\ 
&\leqslant c\left\|M_T^{-1}\right\|_{\mathscr{L}(L^2(0,T;\mathscr{H}), C^0([0,T];\mathscr{H}))}\left(e^{-\lambda t} + e^{-\lambda(T-t)}\right)\Big(\|\overline{p}\|_{\mathscr{H}} + \left\|y^0-\overline{y}\right\|_{\mathscr{H}}\Big).
\end{align*}
The stabilizability and detectability assumptions are then used to prove that $M_T^{-1}$ is uniformly bounded with respect to $T$, from which the exponential turnpike property follows. Besides the upper bound, the decay rate $\lambda>0$ also depends on $M_T^{-1}$. Hence the uniform bound on this inverse is needed for a uniform upper bound and a uniform decay rate.

\smallskip

\subsection{Turnpike in optimal shape design}
In \cite{lance2020shape} (see also \cite{trelat2018optimal} for related results), measure turnpike has been shown also for shape optimization problems of the form 
\begin{equation*}
\inf_{\substack{\omega(\cdot)\in\mathscr{U}_\gamma\\ y\text{ solves} \eqref{eq: shape.design.pb}}} \frac{1}{T}\int_0^T \|y(t)-y_d\|^2_{L^2(\Omega)}\diff t,
\end{equation*}
where
\begin{equation} \label{eq: shape.design.pb}
\begin{cases}
\partial_t y - \Delta y = 1_{\omega(t)} &\text{ in }(0,T)\times\Omega,\\
y=0 &\text{ in }(0,T)\times\partial\Omega,\\
y_{|_{t=0}} = y^0 &\text{ in }\Omega.
\end{cases}
\end{equation}
Here $\mathscr{U}_\gamma:=\{\omega\subset\Omega\,\bigm|\,\text{meas}(\omega)\leqslant\gamma\text{ meas}(\Omega)\}$ denotes the set of admissible shapes, for a given $\gamma\in(0,1)$. The setup of this problem is quite in the spirit of the original problem regarding Navier-Stokes shape design discussed in the introduction. 

In \cite{lance2020shape}, the authors convexify the problem by relaxation (namely, by considering the convex closure of $\mathscr{U}_\gamma$ in the $L^\infty$ weak-* topology, which roughly translates to replacing $1_{\omega(t)}$ by a bounded potential $a(t)\in[0,1]$ with mass $\leqslant\gamma\text{ meas}(\Omega)$), and make use of techniques inspired by the calculus of variations to prove measure and integral turnpike properties for the optimal shapes. 
A proof of the exponential turnpike property remains an open problem. All in all, a complete theory of turnpike for shape optimization problems has not been established as of yet.

\subsection{Unsteady turnpike} \label{sec: periodic.turnpike}

Finally, let us comment on the fundamental notion of turnpike we deal with in this work. The definition of the turnpike property we had considered entails a proximity of  time-dependent optimal strategies to \emph{the} associated steady ones. But this definition does not paint the whole picture. First of all, it could happen that the turnpike is \emph{not} unique (as it is the case in some nonlinear problems, as seen in the subsequent section). 
Sometimes, the turnpike may not even be of a steady nature. 

The latter can even occur in linear problems. Let us corroborate with some more details. An example is the artifact which appears whenever one considers \emph{time-dependent, periodic} running targets $y_d(t)$ in the LQ problem, as noted in \cite{samuelson1976periodic, rapaport2004turnpike, zanon2016periodic} for finite-dimensional systems. Suppose for instance that $y_d\in C^0([0,+\infty);\mathscr{H})$ is periodic of period $\pi_\bullet>0$, namely, 
\begin{equation*}
y_d(t+\pi_\bullet) = y_d(t), \hspace{1cm} \text{ for } t>0.
\end{equation*}
One can then consider the standard LQ problem for one's favorite linear PDE, written in the canonical form $\partial_t y=Ay+Bu$. And under similar stabilizability and detectability assumptions for the underlying PDE dynamics $(A,B)$ and observation operator $C$, the authors in \cite{trelat2018steady}. show that the exponential turnpike property holds, where now the turnpike is given by the unique triple $(u_\pi, y_\pi, p_\pi)$ solving 
\begin{equation*}
\begin{cases}
\partial_t y_\pi = Ay_\pi + BB^*p_\pi &\text{ in }(0,\pi_\bullet),\\
-\partial_t p_\pi =A^*p_\pi -C^*C(y_\pi - y_d) &\text{ in }(0,\pi_\bullet),\\
{y_\pi}_{|_{t=0}} = {y_\pi}_{|t=\pi_\bullet},\\
{p_\pi}_{|_{t=0}} = {p_\pi}_{|t=\pi_\bullet},
\end{cases}
\end{equation*}
with 
\begin{equation*}
u_\pi(t)\equiv B^*p_\pi(t)  \hspace{1cm} \text{ for a.e. } t\in[0,\pi_\bullet].
\end{equation*} 
To our knowledge, the taxonomy of different turnpikes which could occur depending on the choice of functional and underlying dynamics has not yet been proposed or established.
As a general principle, the turnpike can be any trajectory or any invariant set of the system. For instance, in \cite{trelat2020linear}, the turnpike is a monotonically increasing trajectory, exemplified in practical applications by the motion of a medium (400m) distance runner (\cite{aftalion2021pace}). In such applications, the velocity of the runner is essentially constant from beginning to end, but the position of the runner evolves in a monotonic fashion. A more complete picture on the structure and reasons behind this artifact may be found in \cite{pighin2020theturnpike}, relying on the Kalman decomposition, which ensures that the exponential turnpike property is inherent to the observable components of the underlying system.
For infinite-dimensional systems, these issues have not been thoroughly explored, and merit further attention. Monotonic trajectories of this kind are a hallmark for systems arising in fluid mechanics -- they can be essentially laminar. We are not aware if such turnpike questions have been studied in a controlled scenario.

\begin{figure}[h!]
\centering
\includegraphics[scale=1.2]{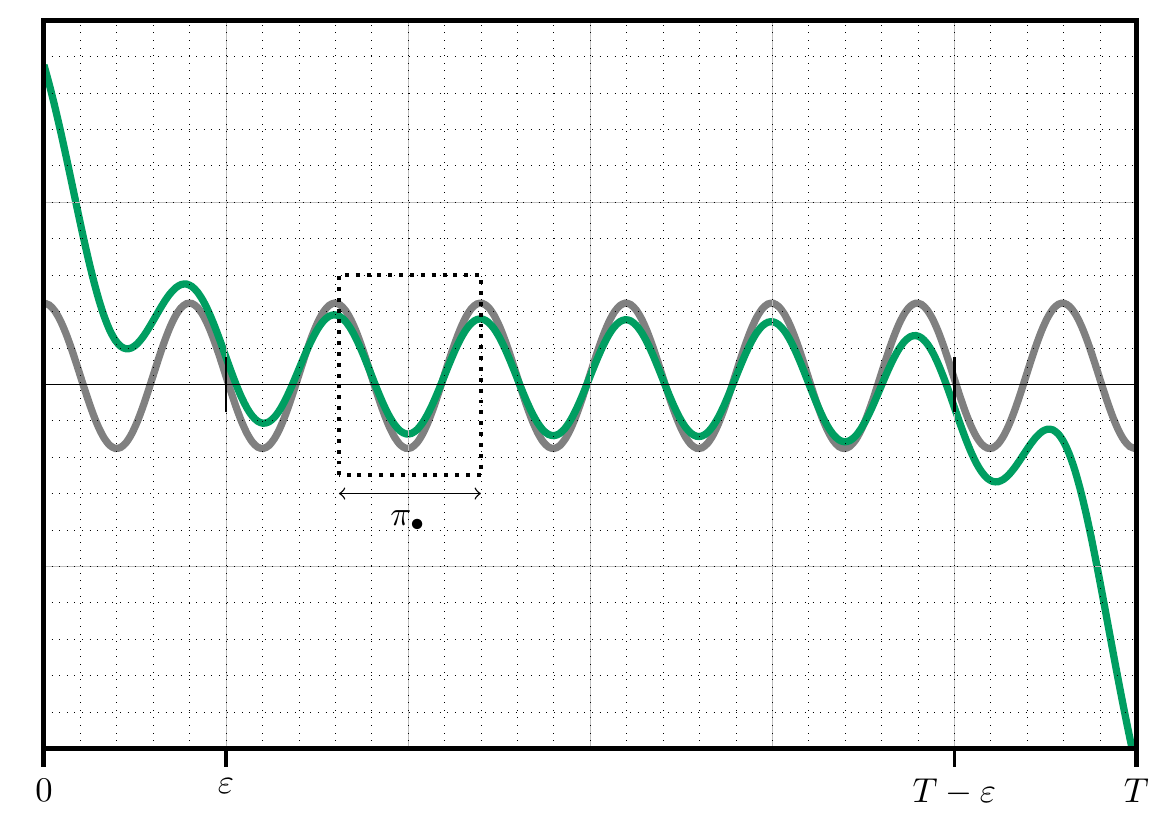}
\caption{The {\bf periodic turnpike property}: the (norm of the) optimal trajectory $y_T(t)$ (green) stays exponentially close to the periodic turnpike trajectory $y_\pi(t)$ of period $\pi_\bullet$ (gray).}
\end{figure}
\subsection{Even further in the finite-dimensional case}

An extensive theory regarding cardinal and measure-like turnpike properties for finite-dimensional discrete and continuous time systems has been developed independently by Zaslavski in a series of works (see \cite{zaslavski2005turnpike, zaslavski2007turnpike, zaslavski2015turnpike} and the references therein).
In the finite dimensional case, dynamical systems techniques based on stable manifold theory have also been used and developed for proving the exponential turnpike property (\cite{sakamoto2019turnpike}). Further links with systems theory are established in \cite{heiland2020classical}, and additional direct strategies for proving exponential turnpike properties for finite dimensional systems may be found in \cite{lou2019turnpike}.

\part{Nonlinear theory}

The case in which the underlying constraint in the optimal control problem is a nonlinear PDE is rather different. 
It requires a case-by-case study, and one cannot expect to provide a turnpike theory without any smallness assumptions encompassing all PDE systems. 
This is, of course, a problem which transcends many other fields and topics, not just optimal control of PDEs.
We emphasize that this theory is far from mature, and many open problems persist, even in some relatively simple cases.

\section{Linearization and smallness} \label{sec: 8}

One can expect, as done for a variety of different control concepts, to transfer the linear results to a nonlinear setting by means of linearization and fixed point arguments, provided some smallness assumptions on the data. 
In fact, we can first prove a turnpike property for the optimality system, under the condition that the initial and final states, for the forward and adjoint state respectively, are close enough to the stationary forward and dual state, respectively. 
We will also see that this result applies to global minimizers of the cost functional, at least in the case that the running target $y_d$ is small enough. 
In such a case, one can also ensure that the turnpike is unique. 

We thus query the validity of the turnpike property for the semilinear heat equation
\begin{equation} \label{eq: semilinear.heat}
\begin{cases}
\partial_t y - \Delta y + f(y) = u1_{\omega} &\text{ in }(0,T)\times\Omega,\\
y= 0 &\text{ in }(0,T)\times\partial\Omega,\\
y_{|_{t=0}} = y^0 &\text{ in }\Omega,
\end{cases}
\end{equation}
where $y^0\in L^2(\Omega)$. 
Here and in what follows, we assume that 
\begin{equation*}
f\in C^2(\mathbb{R}), \hspace{0.25cm} \text{ with } \hspace{0.25cm} f'\geqslant 0 \hspace{0.15cm} \text{ and } \hspace{0.15cm} f(0)=0.
\end{equation*}
A canonical example is the cubic nonlinearity $f(y)=y^3$ in dimensions $d\leqslant3$. 
Under these assumptions, system \eqref{eq: semilinear.heat} is well-posed, in the sense that given any $y^0\in L^2(\Omega)$ and $u\in L^2((0,T)\times\omega)$, there exists a unique  solution $y\in C^0([0,T]; L^2(\Omega))\cap L^2(0,T; H^1_0(\Omega))$. Such a result can be shown by employing a fixed point argument, making use of the dissipative nature of the nonlinearity to obtain global results (see \cite[Appendix B]{pighin2020turnpike} and the references therein). 

We may thus consider the following optimal control problem
\begin{equation} \label{eq: ocp.heat.semilinear}
\inf_{\substack{u\in L^2((0,T)\times\omega) \\ y \text{ solves } \eqref{eq: semilinear.heat}}} \underbrace{\phi(y(T)) + \frac12 \int_0^T \|y(t)-y_d\|_{L^2(\omega_\circ)}^2 \diff t + \frac12 \int_0^T \|u(t)\|_{L^2(\omega)}^2 \diff t}_{:=\mathscr{J}_T(u)},
\end{equation}
where $y_d\in L^2(\omega_\circ)$, with 
$$\phi(y(T)):=\left\langle p^T, y(T)\right\rangle_{L^2(\Omega)}$$ 
for a given and fixed $p^T\in L^2(\Omega)$.
The corresponding steady optimal control problem consists in solving 
\begin{equation} \label{eq: semilinear.steady.problem}
\inf_{\substack{u \in L^2(\omega)\\ y \text{ solves} \eqref{eq: semilinear.poisson}}} \|y-y_d\|_{L^2(\omega_\circ)}^2 + \|u\|_{L^2(\omega)}^2,
\end{equation}
where the underlying PDE constraint is given by the semilinear controlled Poisson equation
\begin{equation} \label{eq: semilinear.poisson}
\begin{cases}
-\Delta y + f(y) = u1_\omega &\text{ in }\Omega,\\
y = 0 &\text{ on }\partial\Omega.
\end{cases}
\end{equation}
In both cases, one can also ensure the existence of solutions (minimizers) by the direct method in the calculus of variations. Uniqueness can only be guaranteed under smallness assumptions on the target $y_d$; this will be a major plotline in what follows.
We can also readily write the corresponding optimality systems for the evolutionary triple $(u_T, y_T, p_T)$ and the steady one $(\overline{u}, \overline{y}, \overline{p})$.
As per \cite[Chapter 1]{ito2008lagrange}, the optimality systems read as
\begin{equation} \label{eq: semilinear.heat.optimality}
\begin{cases}
\partial_t y_T - \Delta y_T + f(y_T) = p_T1_\omega &\text{ in }(0,T)\times\Omega,\\
\partial_t p_T +\Delta p_T - f'(y_T)p_T = (y_T-y_d)1_{\omega_\circ} &\text{ in }(0,T)\times\Omega,\\
y_T=p_T = 0 &\text{ in }(0,T)\times\partial\Omega,\\
{y_T}_{|_{t=0}}= y^0 &\text{ in }\Omega,\\
{p_T}_{|_{t=T}} = p^T &\text{ in }\Omega,
\end{cases}
\end{equation}
as well as
\begin{equation} \label{eq: semilinear.poisson.optimality}
\begin{cases}
-\Delta \overline{y} + f(\overline{y}) = \overline{p}1_\omega &\text{ in }\Omega,\\
-\Delta \overline{p} + f'(\overline{y})\overline{p} = -(\overline{y}-y_d)1_{\omega_\circ} &\text{ in }\Omega,\\
\overline{y} = \overline{p} = 0 &\text{ on }\partial\Omega.
\end{cases}
\end{equation}
Of course, once again, $u_T\equiv p_T1_\omega$ and $\overline{u}\equiv \overline{p}1_\omega$. 
But, due to the nonlinearity of the problems under consideration, the methods presented for the linear theory cannot be applied directly. 
Hence, we can seek to test a local theory around a given steady state optimal control-state pair. 
To this end, we define the perturbation variables
\begin{equation*}
\delta y := y_T - \overline{y}, \hspace{1cm} \delta p:= p_T - \overline{p}.
\end{equation*}
Then, $(\delta y, \delta p)$ would satisfy
\begin{equation} \label{eq: 20}
\begin{cases}
\partial_t \delta y - \Delta \delta y + \mathfrak{f}(\delta y) = \delta p 1_\omega &\text{ in }(0,T)\times\Omega,\\
\partial_t \delta p + \Delta \delta p - \mathfrak{g}(\delta y, \delta p) = \delta y1_{\omega_\circ} &\text{ in }(0,T)\times\Omega,\\
\delta y = \delta p = 0 &\text{ in }(0,T)\times\partial\Omega,\\
\delta y_{|_{t=0}} = \delta y^0 &\text{ in }\Omega,\\
\delta p_{|_{t=T}} = \delta p^T &\text{ in }\Omega,
\end{cases}
\end{equation}
where $\delta y^0 := y^0-\overline{y}$ and $\delta p^T:= p^T-\overline{p}$, and moreover, the nonlinearities are
\begin{align*}
\mathfrak{f}(\delta y) &:= f(\overline{y}+\delta y) - f(\overline{y}), \\ 
\mathfrak{g}(\delta y, \delta p) &:= f'(\overline{y}+\delta y)(\overline{p}+\delta p) - f'(\overline{y})\overline{p}.
\end{align*}
Since our aim is to build a pair $(u_T,y_T)$ fulfilling the turnpike property, namely such that $(u_T,y_T) \sim (\overline{u}, \overline{y})$ in the sense of the previous section, in the $(\delta y, \delta p)$ coordinates, this is equivalent to ensuring $(\delta y, \delta p) \sim (0,0)$. 
Furthermore, since the latter are perturbation variables, it is natural to look at the linearized version of the nonlinear system \eqref{eq: 20}. Namely, we would look at the first order Taylor expansion of $\mathfrak{f}$ and $\mathfrak{g}$ near $(\delta y, \delta p) = (0, 0)$, which, due to the form of these nonlinearities, corresponds to 
\begin{equation} \label{eq: linearized.optimality.system}
\begin{cases}
\partial_t \delta y - \Delta \delta y + f'(\overline{y})\delta y = \delta p 1_\omega &\text{ in }(0,T)\times\Omega,\\
\partial_t \delta p + \Delta \delta p - f'(\overline{y})\delta p = \delta y1_{\omega_\circ} + f''(\overline{y}) \overline{p} \delta y &\text{ in }(0,T)\times\Omega,\\
\delta y = \delta p = 0 &\text{ in }(0,T)\times\Omega,\\
\delta y_{|_{t=0}} = \delta y^0 &\text{ in }\Omega,\\
\delta p_{|_{t=T}} = \delta p^T &\text{ in }\Omega.
\end{cases}
\end{equation}
The following result can then be shown to hold. 

\begin{theorem}[\cite{porretta2016remarks}] \label{thm: porretta.zuazua.3}
Suppose that $f\in C^2(\mathbb{R})$, $f'\geqslant0$ and $d\leqslant3$. Let $(\overline{y},\overline{p})\in (H^1_0(\Omega)\cap L^\infty(\Omega))^2$ be some solution\footnote{Again, such a solution exists due to the fact that \eqref{eq: semilinear.poisson.optimality} is the Euler-Lagrange equation for a minimization problem.} to the optimality system \eqref{eq: semilinear.poisson.optimality}.
Suppose that there exist $C>0$ and $\lambda>0$ such that for any $T>0$ and $(\delta y^0, \delta p^T)\in L^\infty(\Omega)\times L^\infty(\Omega)$, the unique solution $(\delta y, \delta p)$ to \eqref{eq: linearized.optimality.system} satisfies the turnpike property
\begin{equation*}
\|\delta y(t)\|_{L^\infty(\Omega)} + \|\delta p(t)\|_{L^\infty(\Omega)} \leqslant C\left(e^{-\lambda t} + e^{-\lambda(T-t)}\right)
\end{equation*}
for all $t\in [0,T]$. Then, there exists some $\varepsilon>0$ (independent of $T$) such that for all data $(y^0, p^T) \in L^\infty(\Omega)\times L^\infty(\Omega)$ satisfying 
\begin{equation*}
\left\|y^0-\overline{y}\right\|_{L^\infty(\Omega)} + \left\|p^T-\overline{p}\right\|_{L^\infty(\Omega)} \leqslant \varepsilon,
\end{equation*}
there exists a solution $(y_T,p_T)$ to the optimality system \eqref{eq: semilinear.heat.optimality} which satisfies
\begin{equation*}
\left\|y_T(t)-\overline{y}\right\|_{L^\infty(\Omega)} + \left\|p_T(t)-\overline{p}\right\|_{L^\infty(\Omega)} \leqslant C\left(e^{-\lambda t} + e^{-\lambda(T-t)}\right)
\end{equation*}
for all $t\in [0,T]$. 
\end{theorem}

\begin{remark}
The assumption $d\leqslant3$ is not essential -- should we be working with power-type nonlinearities of the form $f(s)=|s|^{p-1}s$, in which case the appropriate functional space would be $H^1_0(\Omega)\cap L^{p+1}(\Omega)$. 
\end{remark}

\begin{proof} The proof may be found in \cite{porretta2016remarks}.
\end{proof}

The above theorem states that there exists a solution to the nonlinear optimality system \eqref{eq: semilinear.heat.optimality} for which the turnpike property holds. 
Thus, the result does not have the nature we expect -- in other words, it does not  apply to the minimizers of the functional $\mathscr{J}_T$ under consideration in \eqref{eq: ocp.heat.semilinear}. 
Moreover, the statement \emph{assumes} that the turnpike property is satisfied by the linearized optimality system. 
As we shall see just below, this theorem applies at least when the target $y_d$ is small enough, in the sense that, first of all, the steady optimal control problem has a unique minimizer, which is also small, and, second of all, \eqref{eq: linearized.optimality.system} satisfies the exponential turnpike property.
In this special case, the minimizer of the parabolic optimal control problem also turns out to be unique, and thus coincides with the solution of the optimality system.

\subsection{Small targets}
Let us now consider the particular case where both the target $y_d$ and the initial datum $y^0$ are small in $L^2$, and where $\omega_\circ=\Omega$ (we comment on this assumption in Remark \ref{rem: carleman.domain}). 
In this case, one can actually show that 1). the optimal pair for the steady-state problem is unique, and 2). the linearized optimality system satisfies the turnpike property. 
These conditions would thus ensure that the turnpike property is satisfied by the linearized optimality system \eqref{eq: linearized.optimality.system}, thus ensuring the validity of the hypothesis in Theorem \ref{thm: porretta.zuazua.3}. 

To see why item 2). in the above discussion would hold, denoting 
\begin{equation*}
\varrho(x) := 1-f''(\overline{y}(x))\overline{p}(x) \hspace{1cm} \text{ for } x\in\Omega,
\end{equation*}
a clever observation is that \eqref{eq: linearized.optimality.system} is an optimality system for the LQ problem (thus, a necessary and sufficient condition)
\begin{equation} \label{eq: func.rho}
\inf_{\substack{v\in L^2((0,T)\times\omega)\\ \zeta \text{ solves} \eqref{eq: linearized.semilinear.heat}}} \frac12 \int_0^T \int_\Omega \varrho(x) \zeta(t,x)^2 \diff x \diff t + \frac12 \int_0^T \int_\omega v(t,x)^2 \diff x \diff t,
\end{equation}
where the underlying PDE is 
\begin{equation} \label{eq: linearized.semilinear.heat}
\begin{cases}
\partial_t \zeta - \Delta \zeta + f'(\overline{y})\zeta = v1_\omega &\text{ in }(0,T)\times\Omega,\\
\zeta=0 &\text{ in }(0,T)\times\partial\Omega,\\
\zeta_{|_{t=0}} = \delta y^0 &\text{ in }\Omega.
\end{cases}
\end{equation}
And turnpike holds for the above LQ problem whenever there is some $\delta>0$ such that 
\begin{equation*}
\varrho(x)\geqslant\delta>0 \hspace{1cm} \text{ for } x\in\Omega.
\end{equation*}
We shall see just below that this can be ensured precisely if $\|y_d\|_{L^2(\Omega)}$ is small enough, as this would entail that both $\overline{y}
$ and $\overline{p}$ are small in $L^\infty(\Omega)$.
Let us briefly sketch as to why such smallness assumptions would yield the uniqueness of steady minimizers per 1), and, all the while, $\varrho(x)>0$, as desired.

\begin{itemize}
\item \emph{Steady functional is strictly convex for small controls.}
We claim that the functional $\mathscr{J}_s$, defined in the steady optimal control problem \eqref{eq: semilinear.steady.problem}, is strictly convex whenever the control input $\overline{u}\in L^2(\omega)$ is small enough in $L^p(\Omega)$, for some $p>\sfrac{d}{2}$.
Let us support this claim by showing that the Hessian is positive definite for such controls. 
Following \cite[Proposition 2.3]{casas2002second, casas2002seconda}, we find
\begin{equation} \label{eq: casas.identity}
\mathscr{J}''_s(\overline{u})v_1 v_2 = \int_{\Omega} \eta_{v_1} \eta_{v_2}\diff x + \int_\omega v_1 v_2\diff x - \int_\Omega f''(\overline{y}) \overline{p}\, \eta_{v_1} \eta_{v_2}\diff x,
\end{equation}
for any $\overline{u}\in L^2(\omega)$, where $\overline{y}\in H^1_0(\Omega)$ denotes the corresponding solution to \eqref{eq: semilinear.poisson}, $\overline{p}\in H^1_0(\Omega)$ is the adjoint steady state, solution to
\begin{equation*}
\begin{cases}
-\Delta \overline{p} + f'(\overline{y}) \overline{p} = \overline{y}-y_d &\text{ in }\Omega,\\
\overline{p}=0 &\text{ on }\partial\Omega,
\end{cases}
\end{equation*}
while $\eta_{v_j}\in H^1_0(\Omega)$ are the solutions of the linearized steady equation in the directions $v_j\in L^2(\omega)$, namely
\begin{equation*}
\begin{cases}
-\Delta \eta_{v_j} + f'(\overline{y}) \eta_{v_j} = v_j 1_\omega &\text{ in }\Omega,\\
\eta_{v_j} = 0 &\text{ on }\partial\Omega.
\end{cases}
\end{equation*} 
Whenever $\overline{u}$ is small enough in $L^p(\Omega)$ for some $p>\sfrac{d}{2}$, 
in conjunction with the monotonicity of $f$, puts us in the framework of classic elliptic regularity which ensures that $\overline{y}$ is small in $H^1_0(\Omega)\cap L^\infty(\Omega)$.
In turn, the same can be said of $\overline{p}$, namely $\overline{p}$  is small in $H^1_0(\Omega)\cap L^\infty(\Omega)$, due to the fact that $\overline{y} \in L^\infty(\Omega)$ and $f'\geqslant0$. Therefore, since
\begin{equation} \label{eq: elliptic.eta.v}
\|\eta_v\|_{H^1_0(\Omega)}\leqslant C_1 \|v\|_{L^2(\omega)},
\end{equation}
for some $C_1=C_1(\overline{y}, f)>0$ by Lax-Milgram, when $v_1=v_2=v$, the term
\begin{equation*}
-\int_\Omega f''(\overline{y})\overline{p}\,\eta_{v}^2\diff x
\end{equation*}
in \eqref{eq: casas.identity} can be absorbed by 
\begin{equation*}
\int_\omega v^2\diff x
\end{equation*}
thanks to the fact that $\overline{y}\in L^\infty(\Omega)$ and the smallness of $\overline{p}$ in $L^\infty(\Omega)$. 
Indeed, by using the Poincaré inequality in \eqref{eq: elliptic.eta.v}, and since $f\in C^2(\mathbb{R})$ and $\overline{y}\in L^\infty(\Omega)$, we find
\begin{align} \label{eq: coercivity.Js}
\mathscr{J}''_s(\overline{u})vv &\geqslant \int_{\Omega} \eta_v^2 \diff x + \int_\omega v^2\diff x - C\Big(f, \|\overline{y}\|_{L^\infty(\Omega)}\Big) \|\overline{p}\|_{L^\infty(\Omega)} \int_\Omega \eta_v^2\diff x  \nonumber\\
&\geqslant \int_{\Omega} \eta_v^2 \diff x + \Big(1-C_2\|\overline{p}\|_{L^\infty(\Omega)}\Big) \int_\omega v^2 \diff x,
\end{align}
for some $C_2=C_2(f, \overline{y}, \Omega)>0$. Taking $\|\overline{p}\|_{L^\infty(\Omega)}$ small enough renders the lower bound in \eqref{eq: coercivity.Js} strictly positive. Hence, $\mathscr{J}_s(\overline{u})$ is strictly convex whenever $\overline{u}$ is taken small enough in $L^p(\Omega)$ for some $p>\sfrac{d}{2}$. 
\smallskip

\item \emph{Steady optima $(\overline{u}, \overline{y})$ live in a ball of radius $\|y_d\|_{L^2(\Omega)}$.}
We observe, by comparing the steady functional evaluated at the minimizer with that at $0$, that any steady minimizer $(\overline{u}, \overline{y})$ satisfies
\begin{equation} \label{eq: steady.inequality.elementary}
\|\overline{y}-y_d\|_{L^2(\Omega)}^2 + \|\overline{u}\|_{L^2(\omega)}^2 \leqslant \|y_d\|_{L^2(\Omega)}^2.
\end{equation}
So assuming that the target $y_d$ is small enough in $L^2(\Omega)$ would ensure the smallness of the optimal control $\overline{u}$ in $L^2(\omega)$. And since $d\leqslant3$, we have that $\overline{u}\in L^p(\omega)$ for some $p>\sfrac{d}{2}$. Therefore, $\overline{u}$ is small in $L^p(\omega)$ whenever $y_d$ is small in $L^2(\Omega)$.
\smallskip

\item\emph{Uniqueness of minimizers for small targets.} From the previous 2 items, we gather that 1). whenever $y_d$ is small enough in $L^2(\Omega)$, any minimizer $\overline{u}$ of $\mathscr{J}_s$ is small enough in $L^p(\omega)$ for some $p>\sfrac{d}{2}$, and 2). the functional $\mathscr{J}_s$ is strictly convex over the set of controls which are small enough in $L^p(\omega)$ for some $p>\sfrac{d}{2}$.
This thus ensures the uniqueness of minimizers of $\mathscr{J}_s$ whenever $\|y_d\|_{L^2(\Omega)}$ is small enough. 
\end{itemize}

\noindent
Hence, in this case, the theorem stated before can be enhanced to read as follows\footnote{The smallness condition on the target may manifest in slightly different ways depending on the nature of the nonlinearity. For instance, in the context of the Navier-Stokes system (quadratic nonlinearity), the smallness condition involves the discrepancy between the target $y_d$ and the turnpike $\overline{y}$ \cite{zamorano2018turnpike}.}. 

\begin{theorem}[\cite{porretta2016remarks}] \label{thm: porretta.zuazua.nonlinear.2} 
Suppose that $f\in C^2(\mathbb{R})$ with $f'\geqslant0$ and $d\leqslant3$. Suppose that $\omega_\circ=\Omega$. Then, there exists some $\varepsilon>0$ and $\lambda>0$ such that for all $T>0$, for all $y_d\in L^2(\Omega)$, and for all data $(y^0, p^T) \in L^\infty(\Omega)\times L^\infty(\Omega)$ satisfying 
\begin{equation*}
\|y_d\|_{L^2(\Omega)}+ \left\|y^0-\overline{y}\right\|_{L^\infty(\Omega)} + \left\|p^T-\overline{p}\right\|_{L^\infty(\Omega)} \leqslant \varepsilon,
\end{equation*}
there exists a solution to the optimality system \eqref{eq: semilinear.heat.optimality} which satisfies
\begin{equation*}
\left\|y_T(t)-\overline{y}\right\|_{L^\infty(\Omega)} + \left\|p_T(t)-\overline{p}\right\|_{L^\infty(\Omega)} \leqslant C\left(e^{-\lambda t} + e^{-\lambda(T-t)}\right)
\end{equation*}
for all $t\in [0,T]$, where $(\overline{y}, \overline{p})$ are the unique solutions to \eqref{eq: semilinear.poisson.optimality}. 
\end{theorem}

\begin{remark}[Localized observations] 
\label{rem: carleman.domain}
Note that in the above derivation, we had assumed $\omega_\circ=\Omega$. If $\omega_\circ\subsetneq\Omega$, then the weight $\varrho$ appearing in the optimality system, and hence in \eqref{eq: func.rho}, will rather read as
\begin{equation*}
\varrho(x):=1_{\omega_\circ}(x)-f''(\overline{y}(x))\overline{p}(x) \hspace{1cm} \text{ for } x\in\Omega.
\end{equation*}
In particular, the smallness of the optimal steady state and adjoint state are not sufficient to prescribe the sign of $\varrho$, since there is no reason to say, a priori, that these states will be supported within $\omega_\circ$. The question when $\omega_\circ\subsetneq\Omega$ then boils down to ensuring that the functional in \eqref{eq: func.rho} admits a minimizer, even if the weight $\varrho$ might change sign. We believe that, due to the smallness of $\overline{y}$ and $\overline{p}$ in $L^\infty(\Omega)$, the set where $\varrho$ may be negative can, in some sense, be "absorbed" and rendered negligible. But this point requires further rigorous analysis, which could call for the use of Carleman inequalities.
\end{remark}

In both of the aforementioned results, the turnpike property is satisfied by one solution of the optimality system \eqref{eq: semilinear.heat.optimality}. Since the functional $\mathscr{J}_T$ in the optimal control problem \eqref{eq: ocp.heat.semilinear} may be not convex, we cannot directly assert that such a solution of the optimality system is the unique minimizer (optimal control) for \eqref{eq: ocp.heat.semilinear}.
Whether the turnpike property actually holds for the optima under smallness conditions on the initial datum is not indicated in the above statements.
An answer to this question is provided in the recent work \cite{pighin2020turnpike}, in which it is shown that this is indeed the case for the optimal control-state pair $(u_T, y_T)$ for \eqref{eq: ocp.heat.semilinear}. 

Let us henceforth suppose that $p^T\equiv0$ in \eqref{eq: ocp.heat.semilinear}, namely we work with $\phi\equiv0$.
We begin by stating and proving the following result, which ensures that under appropriate smallness assumptions on $y^0$ and $y_d$, the functional $\mathscr{J}_T$ defined in the optimal control problem \eqref{eq: ocp.heat.semilinear} admits a unique minimizer $u_T$. This would then imply the turnpike property for the solution to \eqref{eq: ocp.heat.semilinear}, namely the unique minimizer to $\mathscr{J}_T$.  

\begin{proposition}[\cite{pighin2020turnpike}] \label{lem: uniqueness.smallness}
There exists a $\delta>0$ such that for any $T>0$, and for any $y^0\in L^\infty(\Omega)$ and $y_d\in L^\infty(\Omega)$ satisfying 
\begin{equation*}
\left\|y^0\right\|_{L^\infty(\Omega)}+\|y_d\|_{L^\infty(\Omega)} \leqslant \delta,
\end{equation*}
the problem \eqref{eq: ocp.heat.semilinear} admits a unique solution $u_T\in L^2((0,T)\times\omega)$. Moreover, $$u_T \in L^\infty((0,T)\times\omega).$$ 
\end{proposition}

The proof of the above proposition in turn is based on the following lemma, also found in \cite{pighin2020turnpike}, which ensures that for bounded initial data $y^0$ and target $y_d$, any optimal control $u_T$ and corresponding state trajectory $y_T$ are uniformly bounded with respect to $T$. Namely, 

\begin{lemma}[\cite{pighin2020turnpike}] \label{lem: strange.lemma.pig}
Let $y^0\in L^\infty(\Omega)$ and $y_d\in L^\infty(\Omega)$ be fixed. Then, there exists a constant 
\begin{equation*}
\mathfrak{C}=\mathfrak{C}\left(\left\|y^0\right\|_{L^\infty(\Omega)} + \|y_d\|_{L^\infty(\omega_\circ)}\right)>0
\end{equation*}
such that for any $T>0$, any optimal solution $u_T\in L^2((0,T)\times\omega)$ to \eqref{eq: ocp.heat.semilinear} satisfies 
$$u_T\in L^\infty((0,T)\times\omega)$$ as well as
\begin{align} \label{eq: misterious.estimate.pig}
\|u_T\|_{L^\infty((0,T)\times\omega)} &+ \|y_T\|_{L^\infty((0,T)\times\Omega)}\nonumber \\
&\leqslant \mathfrak{C} \left(\left\|y^0\right\|_{L^\infty(\Omega)} + \|y_d\|_{L^\infty(\omega_\circ)}\right), 
\end{align}
where $y_T$ denotes the optimal state trajectory, unique solution to \eqref{eq: semilinear.heat}. Moreover, $\zeta\mapsto\mathfrak{C}(\zeta)$ is non-decreasing as a function from $[0,+\infty)$ to $(0,+\infty)$, with $\mathfrak{C}(0)=0$.
\end{lemma}

We omit the proof, which is rather technical and relies on the construction of suboptimal, quasi-turnpike controls -- further details may be found in the cited paper. We rather focus on proving Proposition \ref{lem: uniqueness.smallness}.

\begin{proof}  First of all, an optimal control $u_T$ is bounded in space-time by \cite[Lemma 2.1]{pighin2020turnpike} (see also the lemma just above). Let us thus focus on proving the uniqueness of minimizers to $\mathscr{J}_T$. The proof is roughly an adaptation of the discussion preceding Theorem \ref{thm: porretta.zuazua.nonlinear.2} to the evolutionary setting. 

Consider $y^0\in L^\infty(\Omega)$ and $y_d\in L^\infty(\Omega)$ such that 
\begin{equation} \label{eq: less.than.1}
\|y^0\|_{L^\infty(\Omega)}+\|y_d\|_{L^\infty(\Omega)}\leqslant 1.
\end{equation}
 We introduce a critical ball in $L^\infty((0,T)\times\omega)$ by:
\begin{equation*}
\mathfrak{B}:= \left\{u\,\Biggm|\, \|u\|_{L^\infty((0,T)\times\omega)} \leqslant \mathfrak{C} \left(\left\|y^0\right\|_{L^\infty(\Omega)}+\|y_d\|_{L^\infty(\omega_\circ)}\right)\right\}
\end{equation*}
where $\mathfrak{C}(\cdot)>0$ appears in \eqref{eq: misterious.estimate.pig}\footnote{The constant $\mathfrak{C}$ will be independent of any smallness constant $1\gg\delta>0$ chosen in what follows, due to the fact that it is non-decreasing with respect to $\|y^0\|_{L^\infty(\Omega)}+\|y_d\|_{L^\infty(\omega_\circ)}$ (as per Lemma \ref{lem: strange.lemma.pig}), and $\|y^0\|_{L^\infty(\Omega)}+\|y_d\|_{L^\infty(\omega_\circ)}\leqslant1$.}.
We look to prove the strict convexity of the functional $\mathscr{J}_T$ appearing in \eqref{eq: ocp.heat.semilinear} in $\mathfrak{B}$. 
To this end, we prove that its Hessian is positive definite. We proceed in doing so by noting that, as in the elliptic case, the second order G\^ateaux derivative of $\mathscr{J}_T$ at some $u\in L^\infty((0,T)\times\omega)$ in a direction $v\in L^\infty((0,T)\times\omega)$, reads as
\begin{equation*}
\mathscr{J}''_T(u) v v = \int_0^T \int_{\Omega} \eta_v^2 \diff x \diff t + \int_0^T \int_\omega v^2 \diff x \diff t  - \int_0^T \int_\Omega f''(y)\varphi\, \eta_v^2 \diff x \diff t,
\end{equation*}
where $y$ solves \eqref{eq: semilinear.heat} with control $u$, $\eta_v$ solves the linearized forward system
\begin{equation*}
\begin{cases}
\partial_t \eta - \Delta \eta + f'(y) \eta = v1_\omega &\text{ in }(0,T)\times\Omega,\\
\eta=0 &\text{ on }(0,T)\times\partial\Omega,\\
\eta_{|_{t=0}} = 0 &\text{ in }\Omega,
\end{cases}
\end{equation*}
whereas $\varphi$ solves the linearized adjoint system
\begin{equation} \label{eq: linearized.adjoint.heat.pig}
\begin{cases}
-\partial_t \varphi - \Delta \varphi + f'(y) \varphi = y-y_d &\text{ in }(0,T)\times\Omega,\\
\varphi=0 &\text{ in }(0,T)\times\partial\Omega,\\
\varphi_{|_{t=T}} = 0 &\text{ in }\Omega.
\end{cases}
\end{equation}
Since $f'\geqslant0$, by standard energy estimates for the linear heat equation it follows that
\begin{equation} \label{eq: pighin.1}
\|\eta_v\|_{L^2((0,T)\times\Omega)} \leqslant C_0(\Omega, f) \|v\|_{L^2((0,T)\times\omega)}
\end{equation}
for some constant $C_0(\Omega, f)>0$ independent of $T$. 
Now let $u\in \mathfrak{B}$. 
{\color{black}By a crafted comparison argument (see \cite{pighin2020turnpike}) applied to \eqref{eq: semilinear.heat} and \eqref{eq: linearized.adjoint.heat.pig}, we may find
\begin{equation} \label{eq: pighin.2}
\|y\|_{L^\infty((0,T)\times\Omega)} + \|\varphi\|_{L^\infty((0,T)\times\Omega)} \leqslant C_1 \left(\left\|y^0\right\|_{L^\infty(\Omega)} + \|y_d\|_{L^\infty(\Omega)}\right)
\end{equation}
for some $C_1(\Omega)>0$ independent of $T$.}
Hence, using estimates \eqref{eq: pighin.1} and \eqref{eq: pighin.2}, it follows that
\begin{equation*}
\int_0^T \int_\Omega \left|f''(y)\varphi\right| \eta^2_v \diff x \diff t \leqslant C_2 \left(\left\|y^0\right\|_{L^\infty(\Omega)}+\|y_d\|_{L^\infty(\Omega)}\right) \int_0^T \int_\omega v^2 \diff x \diff t,
\end{equation*}
where $C_2 = C_2(\Omega, f, f'')>0$ is independent of $T$, and we have used \eqref{eq: less.than.1}. Therefore,
\begin{align*}
\mathscr{J}_T''(u)vv &\geqslant \int_0^T \int_{\omega_\circ} \eta_v^2 \diff x \diff t \\
&\quad+ \left(1-C_2\left(\left\|y^0\right\|_{L^\infty(\Omega)}+\|y_d\|_{L^\infty(\Omega)}\right)\right)\int_0^T \int_\omega v^2 \diff x \diff t. 
\end{align*}
Hence, if $\|y^0\|_{L^\infty(\Omega)}+\|y_d\|_{L^\infty(\Omega)}\leqslant\delta$ for some $0<\delta\ll1$ small enough, we can ensure that
\begin{equation*}
\mathscr{J}''_T(u)vv \geqslant \frac12 \int_0^T \int_\omega v^2 \diff x \diff t
\end{equation*}
holds for any $v\in L^\infty((0,T)\times\omega)$. Consequently, $\mathscr{J}_T$ is strictly convex in the ball $\mathfrak{B}$. 
In other words, $\mathscr{J}_T$ has a unique minimizer in the ball $\mathfrak{B}$. 
Now by Lemma \ref{lem: strange.lemma.pig}, whenever 
$$\|y^0\|_{L^\infty(\Omega)}+\|y_d\|_{L^\infty(\Omega)}\leqslant\delta\ll1,$$ 
then clearly any given optimal control $u_T$, namely any minimizer of $\mathscr{J}_T$, is an element of $\mathfrak{B}$. Thus $u_T$ is a global minimizer to $\mathscr{J}_T$.
\end{proof}

Both of the above results then lead to the following turnpike result for the semilinear heat equation, without any smallness assumptions on the initial data.

\begin{theorem}[\cite{pighin2020turnpike}] \label{thm: pighin.whatever.it.is}  Let $\varepsilon>0$ be fixed. Then, there exists $r_\varepsilon>0$ such that for every $T>0$, $y_0\in L^\infty(\Omega)$, and for every $y_d\in L^\infty(\Omega)$ satisfying 
\begin{equation*}
\|y_d\|_{L^\infty(\Omega)} \leqslant r_{\varepsilon},
\end{equation*}
any $u_T$ solution (global minimizer) to \eqref{eq:  ocp.heat.semilinear} and corresponding solution $y_T$ to \eqref{eq: semilinear.heat} satisfy
\begin{equation} \label{eq: assymetric.turnpike}
 \|y_T(t)-\overline{y}\|_{L^\infty(\Omega)} +\|u_T(t)-\overline{u}\|_{L^\infty(\omega)} \leqslant C_{\varepsilon} e^{-\lambda t} + \varepsilon e^{-\lambda(T-t)}
\end{equation}
for all $t\in[0,T]$, for some constants $C_{\varepsilon}=C\left(\varepsilon,\Omega, \omega,y^0\right)>0$ and $\lambda>0$ independent of $T$. 
\end{theorem}

\begin{remark} Let us make a couple of comments regarding Theorem \ref{thm: pighin.whatever.it.is}.
 
\begin{itemize}
\item Here, once again, $(\overline{u}, \overline{y})$ denotes the unique solution to \eqref{eq: semilinear.steady.problem} -- uniqueness follows precisely from the smallness of the target $y_d$, as discussed in what precedes.
We also note that the decay rate $\lambda>0$ is the same as in the previous statements.
\smallskip 

\item Let us also comment briefly on the asymmetric nature of estimate \eqref{eq: assymetric.turnpike}. Note that the statement of the above theorem does not require any smallness assumptions on the initial datum $y^0$.
Moreover, the constant multiplying the final arc $e^{-\lambda(T-t)}$ may be made arbitrarily small, whereas the constant multiplying the initial arc $e^{-\lambda t}$ will be large whenever $\varepsilon\ll1$. This is in part due to \eqref{eq: steady.inequality.elementary} and the smallness of $\|y_d\|_{L^\infty(\Omega)}$, as namely $\overline{y}\sim y_d \sim 0$ which yields the smallness of the final arc, when the state $y_T(t)$ leaves the turnpike $\overline{y}$ to match the final condition for the adjoint state. On the other hand, the initial condition $y^0$ can be arbitrarily large, which is reflected by the constant $C_\varepsilon$. 
\end{itemize}
\end{remark}

\subsection{Large targets, weaker turnpike}

When the target $y_d$ is taken to be arbitrarily large, one can still obtain asymptotic simplification (turnpike) results for the semilinear heat equation, albeit with a significantly weaker rate of convergence when $T\to+\infty$.

\begin{theorem}[\cite{pighin2020turnpike}] Let $y^0\in L^\infty(\Omega)$ and $y_d\in L^\infty(\omega_\circ)$ be fixed. 
Then 
\begin{equation*}
\frac{1}{T}\inf_{\substack{u_T\in L^2((0,T)\times\Omega)\\ y_T\text{ solves} \eqref{eq: semilinear.heat}}}\mathscr{J}_T(u_T)\xrightarrow{T\to+\infty} \inf_{\substack{\overline{u}\in L^2(\Omega)\\\overline{y}\text{ solves} \eqref{eq: semilinear.poisson}}}\mathscr{J}_s(\overline{u}).
\end{equation*}
Suppose in addition that $y^0\in L^\infty(\Omega)\cap H^1_0(\Omega)$. Then, moreover,
\begin{equation*}
\|\partial_t y_T\|_{L^2((0,T)\times\Omega)} \leqslant C
\end{equation*}
holds for some constant $C>0$ independent of $T$, where $y_T$ denotes the unique solution to \eqref{eq: semilinear.heat} corresponding to any control $u_T$ optimal for $\mathscr{J}_T$.
\end{theorem}

\begin{proof}
We shall only provide a sketch of the main ideas. The details may be found in \cite{pighin2020turnpike}.
The proof relies on the following elements.
\begin{enumerate}
\item[1.] First of all, solely using the equation satisfied by $y_T$ and integration by parts, one can find that
\begin{align} \label{eq: pighin.incomprehensible}
\mathscr{J}_T(u_T)&=\int_0^T\mathscr{J}_s\Big(-\Delta_x y_T(t,\cdot)+f(y_T(t,\cdot))\Big)\diff t\nonumber \\
&+ \frac12 \int_0^T\int_\Omega |\partial_t y(t,x)|^2\diff t\diff x \\
&+ \frac12\int_\Omega\Big\{\big|\nabla y(T,x)\big|^2-\big|\nabla y^0(x)\big|^2\nonumber\\
&\hspace{1.5cm} + 2F\big(y_T(T,x)\big) - 2F\big(y^0(x)\big)\Big\}\diff x\nonumber
\end{align}
holds, where 
\begin{equation*}
F(z):=\int_0^z f(x)\diff x
\end{equation*}
designates the anti-derivative of $f$. (We note that in \cite{pighin2020turnpike}, the author also assumes that $\omega=\Omega$ precisely in this step, but this is not needed.)
\smallskip

\item[2.]
Using \eqref{eq: pighin.incomprehensible} and the nondecreasing character of $f$, one may then find
\begin{equation} \label{eq: 8.21}
\left|\inf_{\substack{u_T\in L^2((0,T)\times\Omega)\\ y_T\text{ solves} \eqref{eq: semilinear.heat}}}\mathscr{J}_T(u_T)-T\inf_{\substack{\overline{u}\in L^2(\Omega)\\\overline{y}\text{ solves} \eqref{eq: semilinear.poisson}}}\mathscr{J}_s(\overline{u})\right|\leqslant C
\end{equation}
for some $C>0$ independent of $T$. Indeed, \eqref{eq: pighin.incomprehensible} would yield 
\begin{align*}
\inf_{\substack{u_T\in L^2((0,T)\times\Omega)\\ y_T\text{ solves} \eqref{eq: semilinear.heat}}}\mathscr{J}_T(u_T)&\geqslant  T\inf_{\substack{\overline{u}\in L^2(\Omega)\\\overline{y}\text{ solves} \eqref{eq: semilinear.poisson}}}\mathscr{J}_s(\overline{u})\\
&\quad-\left(\frac12\int_\Omega \big|\nabla_x y^0(x)\big|^2 \diff x+\int_\Omega 2F\big(y^0(x)\big)\diff x\right),
\end{align*}
whereas obtaining an appropriate reversed estimate to derive \eqref{eq: 8.21} is in principle significantly simpler (see \cite[Appendix D]{pighin2020turnpike}).  
\end{enumerate}
\end{proof}

\noindent
To conclude this section, we also comment on what happens when the controls are assumed to be \emph{time-independent}. Following \cite{porretta2016remarks}, we consider
\begin{equation} \label{eq: p.semilinear.heat}
\begin{cases}
\partial_t y - \Delta y + |y|^{p-1}y = u(x)1_\omega &\text{ in }(0,T)\times\Omega,\\
y=0 &\text{ in }(0,T)\times\partial\Omega,\\
y_{|_{t=0}} = y^0 &\text{ in }\Omega,
\end{cases}
\end{equation}
with $p>1$, and we consider 
\begin{equation} \label{eq: p.ocp}
\inf_{\substack{u\in L^2(\omega)\\ y \text{ solves} \eqref{eq: p.semilinear.heat}}}  \frac12 \int_0^T \|y(t)-y_d\|^2_{L^2(\omega_{\circ})}\diff t + \frac{T}{2}\|u\|_{L^2(\omega)}^2.
\end{equation}
The corresponding steady analog reads as
\begin{equation}  \label{eq: p.ocp.steady}
\inf_{\substack{u\in L^2(\omega)\\ y \text{ solves} \eqref{eq: p.poisson}}} \frac12\|y-y_d\|_{L^2(\omega_\circ)}^2 + \frac12\|u\|_{L^2(\omega_\circ)}^2,
\end{equation}
where the underlying PDE constraint is the semilinear Poisson equation
\begin{equation} \label{eq: p.poisson}
\begin{cases}
-\Delta y + |y|^{p-1}y = u1_\omega &\text{ in }\Omega,\\
y= 0 &\text{ on }\partial\Omega.
\end{cases}
\end{equation}
The following result can be shown by making use of $\Gamma$-convergence arguments, the dissipativity of the semilinear heat equation, and taking advantage of the fact that the controls under consideration are independent of $t$.

\begin{theorem}[\cite{porretta2016remarks}]
Let $\{u_T\}_{T>0} \subset L^2(\omega)$ be a family of optimal controls for \eqref{eq: p.ocp}. Then, this family is relatively compact in $L^2(\omega)$ and any accumulation point $\overline{u}$ as $T\to+\infty$ is a solution of the steady problem \eqref{eq: p.ocp.steady}.
\end{theorem}

Analogous results have been shown for the Navier-Stokes equations set in $\Omega\subset\mathbb{R}^2$ in \cite{zamorano2018turnpike}, and also for shape optimization problems for the linear heat equation in \cite{allaire2010long}, with both works again exploiting the same $\Gamma$-convergence arguments, for time-independent controls. The proof in this setting is significantly simpler, and vis-à-vis \cite{allaire2010long} in particular, as mentioned at different points in the text, a (quantitative) extension of this result to the setting of time-dependent controls remain open.

We note that the uniqueness of the optimal control is not guaranteed neither for the time-dependent problem, nor for the steady one. This is due to the lack of convexity of the functionals under minimization, which stems from the nonlinear character of the state equations. Thus the statement above refers necessarily to the accumulation points of the family $\{u_T\}_{T>0}$ and its inclusion within the set of steady state controls.

\section{A warning regarding non-uniqueness} \label{sec: 9}

As discussed in what precedes, whenever the running target $y_d$ is taken arbitrarily large in the context of semilinear problems, uniqueness of minimizers to the corresponding steady optimal control problem cannot be guaranteed.
Such a lack of uniqueness would namely mean that the turnpike is not clearly or uniquely defined. 
As for most nonlinear optimization problems, non-uniqueness of minimizers may be stipulated due to possible lack of convexity. 
The latter could stem from the nonlinear nature of the control to state map. 
But this is, of course, not a sufficient argument to ensure such a fact.
Nevertheless, the recent work \cite{pighin2020nonuniqueness} shows that non-uniqueness of minimizers may indeed occur for the steady optimal control problem. This is demonstrated by designing a specific running target, which is large, of course.
We present these results and insights in what follows.

To stay close to the original material, let us consider the boundary control problem. 
The techniques and results apply for both boundary control and distributed control systems. 
We consider
\begin{equation} \label{eq: ocp.nonuniqueness}
\inf_{\substack{u\in L^\infty(\mathbb{S}^{d-1})\\ y \text{ solves} \eqref{eq: bc}}} \underbrace{\frac12\int_{B_1} |y-y_d|^2 \diff x + \frac12\int_{\mathbb{S}^{d-1}} u^2 \diff \sigma(x)}_{:=\mathscr{J}(u)},
\end{equation}
where the underlying constraint is the following Poisson equation with boundary control
\begin{equation} \label{eq: bc}
\begin{cases}
-\Delta y + f(y) = 0 &\text{ in }B_1\\
y=u &\text{ on }\mathbb{S}^{d-1}.
\end{cases}
\end{equation}
Here $B_1:=\{|x|\leqslant1\}$ denotes the unit ball in $\mathbb{R}^d$, $\mathbb{S}^{d-1}:=\{|x|=1\}$ denotes the unit sphere (representing the domain's boundary), $\diff \sigma$ represents the Lebesgue surface measure, and we take $d\leqslant 3$. Balls with arbitrary, positive radii, may also be considered.  

The following non-uniqueness theorem can be shown to hold.

\begin{theorem}[\cite{pighin2020nonuniqueness}] Suppose $f\in C^1(\mathbb{R})\cap C^2(\mathbb{R}\setminus\{0\})$, with $f'\geqslant0$, $f(0)=0$ and 
\begin{equation*}
f''(y) \neq 0 \hspace{1cm} \text{ for all } y\neq 0.
\end{equation*}
Then there exists a target $y_d\in L^\infty(B_1)$ such that the functional $\mathscr{J}$ defined in \eqref{eq: ocp.nonuniqueness} admits \emph{at least two} global minimizers.
\end{theorem}

The proof relies on a rather clever idea, in which the author distinguishes the cases of non-constant and constant controls. If the optimal control is not constant, then by choosing a radial target $y_d$ and using the radial symmetry of the domain, the minimizer can be rotated by an orthogonal matrix to obtain a second, different minimizer. 
If the optimal control is a constant, then one can select a special target, which is the sum of two characteristic functions, of carefully constructed domains. For that target, the functional will be shown to admit at least two local minimizers: one in $(-\infty,0]$ and another in $[0,+\infty)$. A careful bisection argument then yields a couple of distinguished global minimizers in these sets.  

\begin{proof} We shall sketch the steps of the proof. Before proceeding, we define 
the control-to-state map $\Phi(u)=y$, where $y$ solves \eqref{eq: bc} with control $u$. We then define
\begin{equation*}
\mathscr{I}(u,y_d) :=  \frac12 \int_{B_1} |\Phi(u)|^2 \diff x + \frac12 \int_{\mathbb{S}^{d-1}} u^2 \diff \sigma(x) -\int_{B_1} \Phi(u)y_d \diff x.
\end{equation*}
It is readily seen that for any $y_d\in L^\infty(B_1)$, 
\begin{equation*}
\mathscr{J}(\cdot) = \mathscr{I}(\cdot, y_d)+\frac12\|y_d\|_{L^2(B_1)}^2.
\end{equation*} 
Hence, for a fixed target $y_d$, minimizing $\mathscr{I}(\cdot, y_d)$ is equivalent to minimizing $\mathscr{J}$.
Such a change of coordinates is rather convenient because $\mathscr{I}(0,y_d)=0$ for any target $y_d$. 

\smallskip
\noindent
\textit{1).} \textbf{Non-constant controls.} 
First suppose that for some radial target $y_d(x)=\mathfrak{g}(\|x\|)$, the optimal control $u$ is not constant. Then, it can be shown (see \cite[Lemma A.5]{pighin2020nonuniqueness}) that there exists an orthogonal matrix $\mathbf{M}\in \mathbb{R}^{d\times d}(\mathbb{R})$ such that $u\circ \mathbf{M}\neq u$. Now, one can show that
\begin{align*}
&\mathscr{I}\left(u\circ \mathbf{M}, y_d\right) \\ 
&\quad= \frac12 \int_{B_1} |\Phi(u\circ \mathbf{M})|^2 \diff x + \frac12 \int_{\mathbb{S}^{d-1}} |u\circ \mathbf{M}|^2 \diff \sigma(x) -\int_{B_1} \Phi(u\circ \mathbf{M}) y_d \diff x \\
&\quad= \frac12 \int_{B_1} |\Phi(u)|^2 \diff z + \frac12 \int_{\mathbb{S}^{d-1}} u^2 \diff \sigma(z) -\int_{B_1} \Phi(u)y_d \diff z \\
&= \mathscr{I}(u,y_d),
\end{align*}
by making use of the change of variable $z=\mathbf{M}x$. (This is nothing else but exploiting an invariance with respect to rotations.)
Then, $u$ and $u\circ \mathbf{M}$ are two different global minimizers for $\mathscr{I}(\cdot, y_d)$. 
This concludes the proof whenever $u$ is not a constant.

\smallskip
\noindent
\textit{2).} \textbf{Constant controls.} 
In view of the previous step, we may henceforth focus on constant controls.
\begin{enumerate}
\item[1.] We first look to construct a special target $y_d\in L^\infty(B_1)$ such that there exist a couple of constant controls $u_-<0<u_+$ for which 
\begin{equation} \label{eq: I<0}
\mathscr{I}(u_\pm, y_d)<0.
\end{equation}
To this end, let us first fix an arbitrary $y_d$, and let $u_-<0<u_+$ be given. We try to characterize $u_\pm$ by means of a simpler sufficient condition. By definition, we first observe that \eqref{eq: I<0} holds if and only if
\begin{align} \label{eq: inequalities.pighin.1}
\int_{B_1} \Phi(u_-)y_d\diff x &> \frac{d\,\text{meas}(B_1)}{2}|u_-|^2 + \frac12 \int_{B_1} |\Phi(u_-)|^2 \diff x,\\
\int_{B_1} \Phi(u_+)y_d\diff x &> \frac{d\,\text{meas}(B_1)}{2}|u_+|^2 + \frac12 \int_{B_1} |\Phi(u_+)|^2 \diff x,\label{eq: inequalities.pighin.2}
\end{align}
both hold; here $\text{meas}(B_1)$ designates the volume of the unit ball $B_1$.
Now let us suppose that the target $y_d$ takes the specific form 
\begin{equation*}
y_d := \varsigma_1 1_{\omega_1}+ \varsigma_2 1_{\omega_2},
\end{equation*}
where $(\varsigma_1,\varsigma_2)\in\mathbb{R}^2$ are scalars to be found, and $\omega_1, \omega_2$ are two non-empty subsets of $B_1$ such that
\begin{equation} \label{eq: omega1omega2}
\omega_1 \cup \omega_2 \subseteq B_1, \hspace{0.35cm} \omega_1 \cap\omega_2 = \varnothing, \hspace{0.35cm} \text{meas}(\omega_1 \cup\omega_2)=\text{meas}(B_1).
\end{equation}
Now we see that a simple sufficient condition to ensure that the inequalities \eqref{eq: inequalities.pighin.1}--\eqref{eq: inequalities.pighin.2} are satisfied is 
to ensure the following linear system of algebraic equations is solvable for $(\varsigma_1,\varsigma_2)$:
\begin{equation} \label{eq: linear.system.pighin}
\begin{cases}
\varsigma_1 \int_{\omega_1}\Phi(u_-)\diff x + \varsigma_2 \int_{\omega_2}\Phi(u_-)\diff x &= c_+,\\
\varsigma_1 \int_{\omega_1}\Phi(u_+)\diff x + \varsigma_2 \int_{\omega_2}\Phi(u_+)\diff x &= c_-,
\end{cases}
\end{equation}
where 
\begin{equation*}
c_\pm := \frac{d\,\text{meas}(B_1)}{2} |u_\pm|^2 + \frac12 \int_{B_1} |\Phi(u_\pm)|^2\diff x +1.
\end{equation*}
A cornerstone of the proof is precisely the invertibility of the matrix 
\begin{equation*}
\mathfrak{A}:=\begin{bmatrix}
\displaystyle \int_{\omega_1} \Phi(u_-)\diff x & \displaystyle\int_{\omega_2} \Phi(u_-)\diff x\\
\displaystyle\int_{\omega_1} \Phi(u_+)\diff x & \displaystyle\int_{\omega_2} \Phi(u_+)\diff x
\end{bmatrix}
\end{equation*}
appearing in \eqref{eq: linear.system.pighin}. 
\medskip

\item[2.] This is ensured by \cite[Lemma 3.4]{pighin2020nonuniqueness}, which states precisely what we assumed in the previous item; namely, that there exist a couple of constant controls $u_-<0<u_+$ and a couple of subsets $\omega_1,\omega_2$ of $B_1$ satisfying \eqref{eq: omega1omega2} such that the matrix $\mathfrak{A}$ is invertible.
This result is a cornerstone of the proof, and we refer to the original reference for the technical proof.

Since the matrix is invertible, $c_\pm$ are then clearly defined, and we can solve the linear system \eqref{eq: linear.system.pighin} to find $(\varsigma_1,\varsigma_2)$, for which the corresponding target $y_d$ is such that \eqref{eq: I<0} holds.
\medskip

\item[3.] By \cite[Lemma 3.2]{pighin2020nonuniqueness}, since $y_d$ is now fixed, there exist $u_1\leqslant0$ and $u_2\geqslant0$, with $u_1\neq u_2$, such that
\begin{equation} \label{eq: constraints.pighinpig}
\mathscr{I}(u_1,y_d) = \inf_{v\in(-\infty,0]} \mathscr{I}(v,y_d), \hspace{1cm} \mathscr{I}(u_2,y_d)=\inf_{v\in[0,+\infty)} \mathscr{I}(v,y_d).
\end{equation}
Then, by \eqref{eq: I<0}, we also have that
\begin{equation*}
\mathscr{I}(u_1,y_d)\leqslant\mathscr{I}(u_-,y_d)<0=\mathscr{I}(0,y_d),
\end{equation*}
as well as
\begin{equation*}
\mathscr{I}(u_2,y_d)\leqslant\mathscr{I}(u_+,y_d)<0=\mathscr{I}(0,y_d).
\end{equation*}
Therefore, necessarily, $u_1<0$ and $u_2>0$, whence the constraints in \eqref{eq: constraints.pighinpig} are not saturated. In other words, and we see that $\mathscr{I}(\cdot,y_d)$ has at least two local minimizers on $\mathbb{R}$. 
Furthermore, when evaluated at these local minimizers which have different signs, $\mathscr{I}(\cdot,y_d)$ is strictly negative. We may conclude by means of a bisection argument (\cite[Lemma 3.3]{pighin2020nonuniqueness}), which states that there exists a target $y_d^\star\in L^\infty(B_1)$ such that
\begin{equation*}
\inf_{v\in(-\infty,0)}\mathscr{I}(v,y_d^\star)=\inf_{v\in(0,+\infty)}\mathscr{I}(v,y_d^\star).
\end{equation*}
Since, as seen above, $\mathscr{I}(\cdot,y_d^\star)$ has a global minimizer in both $(-\infty,0)$ and $(0,+\infty)$ by \cite[Lemma 3.2]{pighin2020nonuniqueness}, we conclude that each of these minimizers, which are distinguished and have different signs, are global on $\mathbb{R}$. This concludes the proof. 
\end{enumerate}
\end{proof}

\begin{figure}[h!]
\center
\includegraphics[scale=0.8]{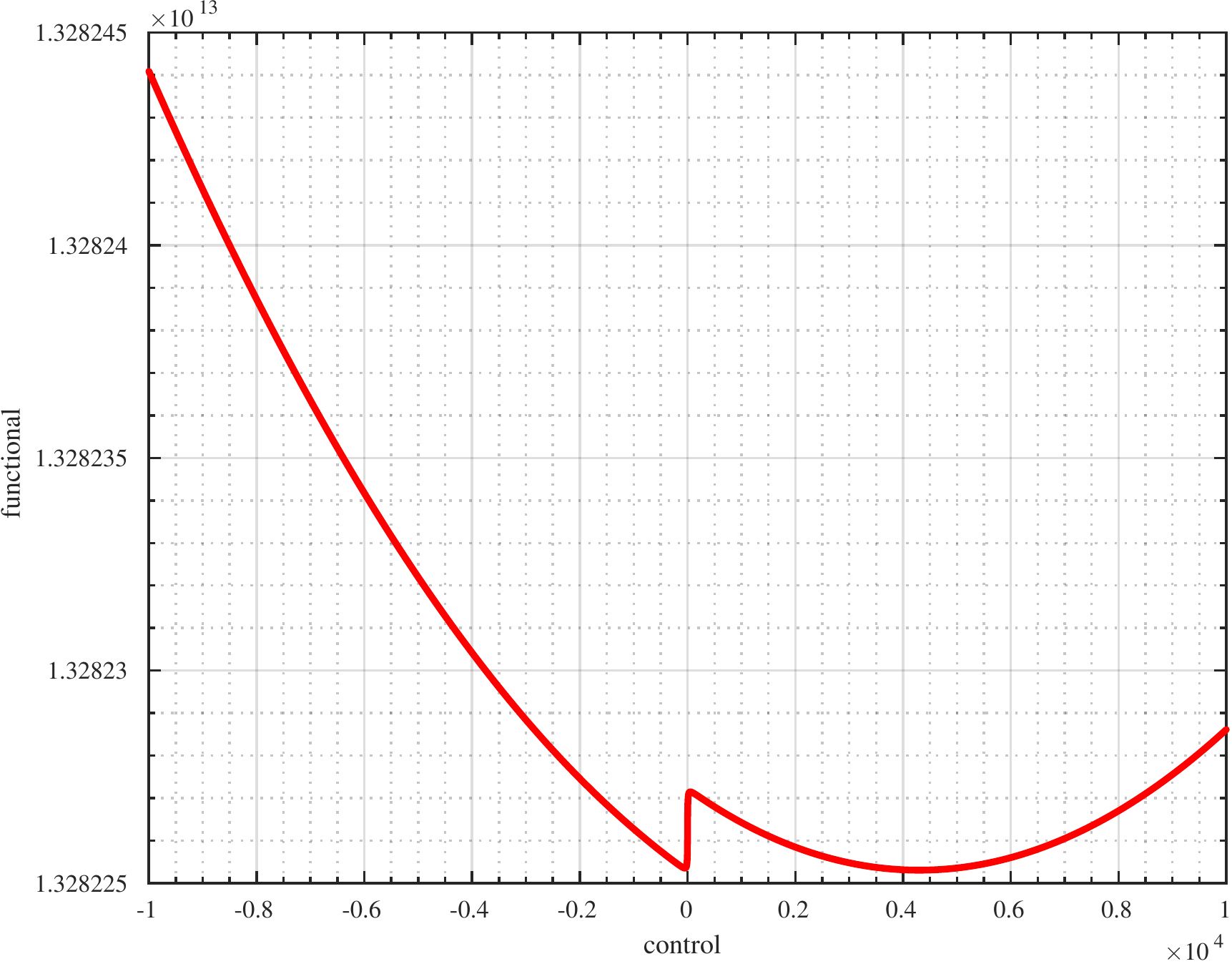}
\caption{The proof above indicates that one need only look the functional $\mathscr{J}$ restricted to constant controls on $\mathbb{R}$. We display the functional $\mathscr{J}$ restricted on $\mathbb{R}$, for the cubic Poisson equation set on $(0,1)$, and a target $y_d:=410000\cdot 1_{\left(0,\frac14\right)\cup\left(\frac34,1\right)}-10300000\cdot1_{\left(\frac14, \frac34\right)}$. 
The numerics consist in a finite-difference discretization of the Laplacian, and a fixed-point type algorithm with relaxation for the cubic nonlinearity.
We clearly see the appearance of two global minimizers, which are $u_1\approx-50$ and $u_2\approx 4298$; each defines a well designating a basin of attraction. Determining which one is the turnpike is an open problem -- a possibility is the well at which the linearized optimality system's Hamiltonian matrix has the largest spectral abscissa. See Section \ref{sec: open.problems} for a discussion.
}
\label{fig: dario}
\end{figure}

The above non-uniqueness result has since been extended to more abstract settings by means of techniques using convexity properties of \emph{Chebychev sets} (\cite{christof2020nonuniqueness}), and has also been explored for finite-dimensional control systems in \cite{trelat2020linear}. We further discuss the latter in Section \ref{sec: open.problems}.

\section{Large, well-adapted targets} \label{sec: 10}

In a couple of recent works (\cite{esteve2020turnpike, esteve2020large}) motivated by applications in machine learning, a new tailored, nonlinear strategy has been conceived for dealing with large targets which are steady states of the underlying equation. This strategy differs significantly from those presented before, as it bypasses analyzing the linearized optimality system, and thus allows for systems with (globally Lipschitz) nonlinearities which are non-smooth (e.g., the ReLU, $x\mapsto\max\{x,0\}$, typically encountered in machine learning applications -- see Section \ref{sec: 13}).  

We shall assume the setting of the semilinear wave equation, but the strategy can readily be adapted to more general, semilinear, exactly controllable systems (e.g.  ODEs, or conservative systems). 
Actually, time-irreversible equations (which are typically only controllable to steady states and trajectories), such as the semilinear heat equation, may also be considered under certain assumptions on the functional to be minimized -- we postpone these cases for a later discussion (see Section \ref{sec: discussion.nonlinearity}).

\subsection{Setup}

Let us consider 
\begin{equation} \label{eq: borjan.ocp}
\inf_{\substack{u\in L^2((0,T)\times\omega)\\ y:=(\zeta,\partial_t \zeta)\\ \zeta \text{ solves } \eqref{eq: abstract.sys.large.1}}}\underbrace{\phi(y(T))+\int_0^T\|y(t)-\overline{y}\|^2_{H^1_0(\Omega)\times L^2(\Omega)}\diff t +\int_0^T\|u(t)\|_{L^2(\omega)}^2\diff t}_{\mathscr{J}_T(u)},
\end{equation}
where 
\begin{equation}\label{eq: abstract.sys.large.1}
\begin{cases}
\partial_t^2 \zeta -\Delta\zeta + f(\zeta) = u1_\omega &\text{ in }(0,T)\times\Omega,\\
\zeta = 0 &\text{ on }(0,T)\times\partial\Omega,\\
(\zeta,\partial_t \zeta)_{|_{t=0}} = (\zeta^0, \zeta^1) &\text{ in }\Omega.
\end{cases}
\end{equation}
Note that in \eqref{eq: borjan.ocp}, $y:=(\zeta,\partial_t\zeta)$ -- we are penalizing the energy norm of the full state $(\zeta(t),\partial_t\zeta(t))$ of the wave system, unlike what we had done in the linear setting (the $L^2$ norm of either $\nabla_x\zeta(t)$ or $\partial_t\zeta(t)$ was sufficient). This particular consideration is an artifact of the strategy of proof -- more details may be found in Remark \ref{rem: full.state}.
Just as before, $\omega\subset\Omega$ is open and non-empty, satisfying additional geometrical assumptions specified later on.
Furthermore, we assume that
\begin{equation} \label{eq: f.lip}
f\in\text{Lip}(\mathbb{R}).
\end{equation} 
In particular, no smoothness assumptions are made on $f$, in which case the methodology based on linearizing the optimality system, presented in preceding sections, is not applicable (one needs $f\in C^2(\mathbb{R})$). 

It will be rather more convenient to work with the wave equation as a first-order system; 
we may of course rewrite \eqref{eq: abstract.sys.large} solely in terms of $y:=(\zeta,\partial_t\zeta)$ as 
\begin{equation} \label{eq: abstract.sys.large}
\begin{cases}
\partial_t y = Ay + \mathfrak{f}(y) + Bu &\text{ in }(0,T),\\
y_{|_{t=0}} =y^0,
\end{cases}
\end{equation}
where the operators $A$ and $B$, and nonlinearity $\mathfrak{f}$, are defined as
\begin{equation} \label{eq: def.A.wave}
A:=\begin{bmatrix}0&\text{Id}\\\Delta&0\end{bmatrix}, \hspace{1cm} \mathfrak{D}(A):= \big(H^2(\Omega)\cap H^1_0(\Omega)\big)\times H^1_0(\Omega),
\end{equation}
as well as
\begin{equation} \label{eq: def.f.wave}
Bu:=\begin{bmatrix}0\\u1_\omega\end{bmatrix}, \hspace{1cm} \mathfrak{f}(y):=\begin{bmatrix}0\\-f(\zeta)\end{bmatrix}.
\end{equation}
Clearly, the PDE constraint in \eqref{eq: borjan.ocp} may also equivalently be changed to \eqref{eq: abstract.sys.large}. The first-order formulation \eqref{eq: abstract.sys.large} can then be used to generalize the presentation to a broader class of systems by adequately adapting the functional framework, as long as the core assumptions, presented in what will follow, are maintained.

Let us henceforth denote 
\begin{equation*}
\mathscr{H}:=H^1_0(\Omega)\times L^2(\Omega).
\end{equation*}
We may recall that $A$ generates a strongly continuous semigroup $\left\{e^{tA}\right\}_{t\geqslant0}$ on $\mathscr{H}$, which is also conservative, in the sense that
\begin{equation} \label{eq: conservation}
\left\|e^{tA}\right\|_{\mathscr{L}(\mathscr{H})}=1
\end{equation}
for all $t$. Thus, by virtue of a Banach fixed point argument, given any $u\in L^2((0,T)\times\omega)$ and $y^0\in\mathscr{H}$, \eqref{eq: abstract.sys.large} admits\footnote{This well-posedness result applies to more general nonlinearities $f$ and does not need assuming \eqref{eq: f.lip} (see \cite[Section 12]{evans1998partial}). Rather, assuming \eqref{eq: f.lip} suffices for bounding $\sup_{t\in[0,T]}\|y(t)-\overline{y}\|_{\mathscr{H}}$ by means of the $L^2(0,T;\mathscr{H})$ norms of $y-\overline{y}$ and $u$, plus the norm of $y^0-\overline{y}$, which we use repeatedly in the proof. Assumption \eqref{eq: f.lip} might not be necessary and is of a technical nature -- see Remark \ref{rem: f.lip} for further details.} a unique finite-energy solution $y\in C^0([0,T];\mathscr{H})$. 

We shall make a specific assumption on the target $\overline{y}\in\mathscr{H}$: we  suppose that it is an uncontrolled steady-state of \eqref{eq: abstract.sys.large}, namely 
\begin{equation} \label{eq: kernel.eq}
A\overline{y} + \mathfrak{f}(\overline{y}) = 0.
\end{equation}
We will, however, not make any smallness assumptions on $\overline{y}$.
Due to the form of $A$ and $\mathfrak{f}$ in \eqref{eq: def.A.wave} and \eqref{eq: def.f.wave} respectively, selecting $\overline{y}$ as such amounts to saying that $\overline{y}=(\overline{\zeta}, 0)$, with 
\begin{equation*}
\begin{cases}
-\Delta\overline{\zeta}+f(\overline{\zeta}) = 0 &\text{ in }\Omega,\\
\overline{\zeta} = 0 &\text{ on }\partial\Omega.
\end{cases}
\end{equation*}
Finally, the final cost $\phi:\mathscr{H}\to\mathbb{R}$ is solely assumed continuous, convex, and bounded from below (say by $0$ for simplicity).

Clearly, as before, \eqref{eq: borjan.ocp} admits a (not necessarily unique) solution by the direct method in the calculus of variations.
On another hand, due to the specific choice of $\overline{y}$ in \eqref{eq: kernel.eq}, we see that the optimal steady solution, namely the unique pair $(u_s,y_s)$ solving
\begin{equation*}
\inf_{\substack{(u,y)\in L^2(\omega)\times\mathscr{H}\\ Ay+\mathfrak{f}(y)+Bu=0}}\underbrace{\|y-\overline{y}\|_{\mathscr{H}}^2 + \|u\|_{L^2(\omega)}^2}_{:=\mathscr{J}_s(u,y)},
\end{equation*}
is precisely $(u_s,y_s)\equiv(0, \overline{y})$; indeed, 
\begin{equation*}
\mathscr{J}(0,\overline{y})=0=\inf_{\substack{(u,y)\in L^2(\omega)\times\mathscr{H}\\ Ay+\mathfrak{f}(y)+Bu=0}}\mathscr{J}_s(u,y) = \min_{\substack{(u,y)\in L^2(\omega)\times\mathscr{H}\\ Ay+\mathfrak{f}(y)+Bu=0}}\mathscr{J}_s(u,y).
\end{equation*} 
In other words, the steady optimal control problem admits a unique solution, given by $(0, \overline{y})$. 

Finally, we shall assume that \eqref{eq: abstract.sys.large} is exactly controllable in some time $T_0>0$. 
More precisely, we make the following hypothesis.

\begin{assumption}[Linear control cost] \label{ass: linear.cost} We suppose that

\begin{enumerate}
\item[1).]
There exists a time $T_0>0$ such that \eqref{eq: abstract.sys.large} is exactly-controllable in time $T_0$. Namely, for any data $(y^0, y^1)\in\mathscr{H}\times\mathscr{H}$, there exists a control $u\in L^2((0,T_0)\times\omega)$ such that the unique solution $y$ to \eqref{eq: abstract.sys.large} set on $(0,T_0)$ satisfies $y(0)=y^0$ and $y(T_0)=y^1$.  
\smallskip

\item[2).] There exists some $r>0$ and some constant $C(T_0)>0$ such that 
\begin{equation} \label{eq: 10.3}
\inf_{\substack{u\in L^2((0,T_0)\times\omega)\\y(0)=y^0\\y(T_0)=\overline{y}}}\|u\|^2_{L^2((0,T_0)\times\omega)} \leqslant C(T_0)\left\|y^0-\overline{y}\right\|_{\mathscr{H}}^2,
\end{equation}
and
\begin{equation} \label{eq: 10.4}
\inf_{\substack{u\in L^2((0,T_0)\times\omega)\\y(0)=\overline{y}\\y(T_0)=y^1}}\|u\|^2_{L^2((0,T_0)\times\omega)} \leqslant C(T_0)\left\|y^1-\overline{y}\right\|_{\mathscr{H}}^2,
\end{equation}
for every $y^0, y^1\in\mathfrak{B}_r(\overline{y})$, where
\begin{equation*}
\mathfrak{B}_r(\overline{y}):=\left\{z\in\mathscr{H}\,\Bigm|\,\|z-\overline{y}\|_{\mathscr{H}}\leqslant r\right\}.
\end{equation*}
\end{enumerate}
\end{assumption}
 
In the particular case of the semilinear wave equation, the above assumption is satisfied\footnote{A subtile point regarding possible extensions to non-globally Lipschitz nonlinearities is that the controllability time $T_0$ may depend on the initial datum $y^0$ (see \cite{joly2014note}). But in the big picture of turnpike this is not necessarily an issue, since we are looking at $T\gg1$.} when, in addition to \eqref{eq: f.lip}, $f(0)=0$ (this is generally needed for ensuring \eqref{eq: 10.3} -- \eqref{eq: 10.4}), and under specific geometric assumptions on $\omega\subset\Omega$, which needs to satisfy a slightly stronger condition than just GCC, namely the so-called multiplier condition (see \cite{zuazua1993exact, zhang2004exact, joly2014note}, and \cite[Remark 10]{esteve2020turnpike}, as well as the setting presented in the latter paper).
We state this as an explicit hypothesis in order to render transparent the needed elements for generalizing the strategy to other systems\footnote{In the finite-dimensional setting, due to the time-reversible nature, this assumption is not restrictive and holds when 1). the associated linear system satisfies the Kalman rank condition, and 2). a fixed-point argument can be performed to transfer the results from the linear system to the nonlinear one. The latter typically requires that the Lipschitz constant of $\mathfrak{f}$ is small enough, so to transfer the controllability of the linear system to the nonlinear one by a small perturbation argument through linearization or a fixed point argument. In the infinite-dimensional setting, the assumption holds more generally when we are working with skew-adjoint operators ($A^*=-A$), generating strongly-continuous and conservative groups of operators, for distributed control problems with appropriate geometric conditions on the control domain $\omega$. The canonical example satisfying this property is of course the wave equation, but there is also the Euler-Bernoulli beam equation, among others. In this setting, globally Lipschitz nonlinear perturbations can be included without further smallness conditions, since they constitute lower-order perturbations of the PDE. The multitude of examples is also one of the reasons why we write the wave equation as a first-order system.}. 

\begin{remark}[Global exact controllability]
We emphasize that, in Assumption \ref{ass: linear.cost}, we assume that exact controllability holds \emph{globally}, namely, without any smallness assumptions on the target $\overline{y}$ or the initial datum $y^0$. On the other hand, we suppose that the cost of controllability is bounded as in \eqref{eq: 10.3} -- \eqref{eq: 10.4} only for initial data $y^0$ in a possibly small ball $\mathfrak{B}_r(\overline{y})$ centered at $\overline{y}$ -- while it might appear structurally restrictive, this is a local assumption. As we shall see later on, the latter is not an impediment to having a turnpike result which holds without any smallness assumptions whatsoever\footnote{Albeit if exact controllability holds only for small initial data, then the turnpike result will also inherit these restrictions and hold locally. This is the case for instance for blowing-up wave equations of the form $\partial_t^2\zeta-\partial_x^2\zeta+\zeta^3=u1_\omega$. Although small amplitude initial data can be controlled to zero by a perturbation argument, the finite velocity of propagation yields that large solutions might blow-up in the finite-time, and this regardless of what the control $u$ is, thus making the controllability of large initial data impossible.}.
\end{remark}

 The following theorem then holds.
 
 \begin{theorem}[\cite{esteve2020turnpike}] \label{thm: borjan.thm}
Let $y^0\in\mathscr{H}$, and let $\overline{y}$ be as in \eqref{eq: kernel.eq}. 
There exists $T^*>0$, and constants $C>0$ and $\lambda>0$, such that for all $T\geqslant T^*$, any solution $u_T$ to \eqref{eq: borjan.ocp}, and the associated unique solution $y_T$ to \eqref{eq: abstract.sys.large}, are such that 
\begin{equation*}
\|u_T\|_{L^2((0,T)\times\omega)}\leqslant C
\end{equation*}
and 
\begin{equation*}
\|y_T(t)-\overline{y}\|_{\mathscr{H}}\leqslant C\Big(e^{-\lambda t} + e^{-\lambda(T-t)}\Big)
\end{equation*}
holds for all $t\in[0,T]$. 
\end{theorem}

We note that, among other things, the result does not contain an exponential turnpike estimate for $u_T(t)$, or precise characterizations of the constants $C$ and $\lambda$ -- both are hallmarks of the linear theory, and we comment on this in Remark \ref{rem: turnpike.u} and Remark \ref{rem: constants.nonlinearity} respectively. 
All in all, a complete discussion regarding the assumptions and possible extensions of Theorem \ref{thm: borjan.thm} may also be found in Section \ref{sec: discussion.nonlinearity}.

\subsection{Sketch of the proof of Theorem \ref{thm: borjan.thm}}  \label{sec: sketch.proof} Solely for simplicity of the subsequent sketch, let us suppose that estimates \eqref{eq: 10.3} -- \eqref{eq: 10.4} hold \emph{globally}, namely that $\mathfrak{B}_r(\overline{y})=\mathscr{H}$. 
The entire strategy can roughly be summarized as in Figure \ref{fig: bootstrap}. 
Through a repetitive use of the quasi-turnpike principle, and an interpolation inequality tied to the Lipschitz character of the underlying system, we may inductively decrease the radius of the tubular neighborhood where $y_T(t)$ is localized by looking over shrinking time intervals. We corroborate with more detail.

\begin{figure}[h!]
\center
\includegraphics[scale=0.8]{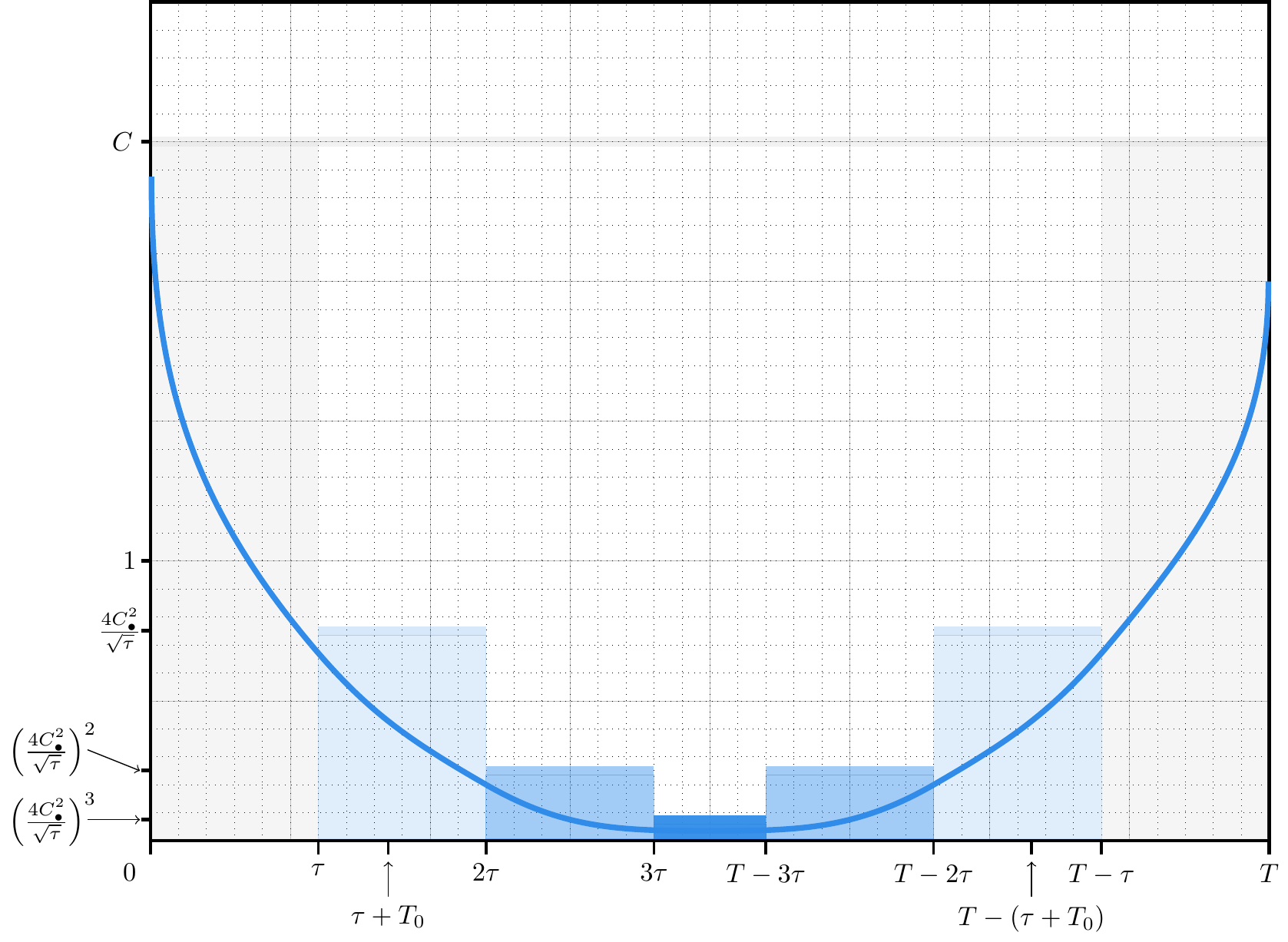}
\caption{Our strategy for showing turnpike for $\|y_T(t)-\overline{y}\|_{\mathscr{H}}$ (blue curve) is as follows. We first show that $\|y_T(t)-\overline{y}\|_{\mathscr{H}}$ is bounded by some possibly large constant $C>0$ independent of $T>0$, over the entire interval $[0,T]$, for $T>0$ large enough. Then, for some sufficiently large $\tau>0$ independent of $T$, in a "symmetrical staircase" fashion, we zoom in over successively smaller subintervals $[n\tau,T-n\tau]$ by induction over all $n\geqslant1$ such that $T-2n\tau\geqslant2T_0$ (an upper bound on $n$, which guarantees that controllability may be used in two disjoint subintervals of $[n\tau,T-n\tau]$ to construct a quasi-turnpike). 
And in each such subinterval, we exponentially decrease an upper bound of the form $\gamma:=\sfrac{4C_\bullet^2}{\sqrt{\tau}}<1$, independent of $T$. Here $C_\bullet>0$ is some constant slightly larger than $C$. In other words, we inductively decrease the radius of the tubular neighborhood where $y_T(t)$ is localized, by shrinking the time intervals where $t$ lies.}
\label{fig: bootstrap}
\end{figure}

\begin{itemize}
\item 
The first tool in our arsenal will be the following inequality for solutions to \eqref{eq: abstract.sys.large}: there exists a constant $C_1>0$, which is (crucially) independent of $T$, such that
\begin{equation} \label{eq: estimate.no.number}
\sup_{t\in[0,T]}\|y(t)-\overline{y}\|_{\mathscr{H}}\leqslant C_1\Big(\|y(0)-\overline{y}\|_{\mathscr{H}}+\|y-\overline{y}\|_{L^2(0,T;\mathscr{H})}+\|u\|_{L^2((0,T)\times\omega)}\Big)
\end{equation}
holds for any, not necessarily optimal $u$, and corresponding solution $y(t)$ to \eqref{eq: abstract.sys.large}. 
The assumption $f\in\text{Lip}(\mathbb{R})$ set in \eqref{eq: f.lip} is used precisely here, as it suffices for proving \eqref{eq: estimate.no.number}.
We refer the reader to Lemma \ref{lem: poincare.sobolev}.
\medskip 

\item Regarding problem \eqref{eq: borjan.ocp}: we first show (Lemma \ref{lem: uniform.bounds}) that there exists $C_2>0$, also independent of $T$, such that
\begin{equation} \label{eq: JT.bounded}
\mathscr{J}_T(u_T)\leqslant C_2
\end{equation}
holds for all $T\geqslant T_0$, where $T_0>0$ is the controllability time.
As the target $\overline{y}$ is a steady state as in \eqref{eq: kernel.eq}, estimate \eqref{eq: JT.bounded} can be shown easily, and done by using the quasi-turnpike principle presented in the introduction.
When used in conjunction with \eqref{eq: estimate.no.number}, estimate \eqref{eq: JT.bounded} yields
\begin{equation} \label{eq: unif.bound.explain}
\sup_{t\in[0,T]}\|y(t)-\overline{y}\|_{\mathscr{H}}^2+\mathscr{J}_T(u_T)\leqslant C_3^2
\end{equation}
for some constant $C_3>0$, depending on $C_1,C_2$ (precisely the constants from \eqref{eq: estimate.no.number} and \eqref{eq: JT.bounded} respectively) and $\left\|y^0-\overline{y}\right\|_{\mathscr{H}}$, but independent of $T\geqslant T_0$.
\medskip

\item Estimate \eqref{eq: unif.bound.explain} is a cornerstone of the subsequent arguments, containing a couple of crucial clues. 
First among these two clues is that the exponential turnpike can immediately be derived on intervals whose length is independent of $T$. Indeed, for $t\in[0,\tau+T_0]$ for instance, from \eqref{eq: unif.bound.explain} one gathers that
\begin{equation} \label{eq: equationstar}
\|y_T(t)-\overline{y}\|_{\mathscr{H}}\leqslant C_3e^{\lambda t} e^{-\lambda t}\leqslant C_3e^{\lambda(\tau+T_0)}\Big(e^{-\lambda t} + e^{-\lambda(T-t)}\Big)
\end{equation}
holds for any $\lambda>0$ (the specific $\lambda$ appearing in Theorem \ref{thm: borjan.thm} will then be fully determined at the end of the proof). A similar computation can then be repeated for $t\in[T-(\tau+T_0),T]$. Herein, one already notes that $T$ needs to be chosen sufficiently large, namely, 
\begin{equation}\label{eq: T>}
T>2(\tau+T_0),
\end{equation}
where $\tau>0$ is a free parameter, chosen large enough later on (with the slight caveat of increasing the constant $C_3 e^{\lambda(\tau+T_0)}$ in \eqref{eq: equationstar}).
\medskip

\item And so, turnpike only needs to be shown for $t\in[\tau+T_0, T-(\tau+T_0)]$. 
To this end, we invoke the second clue that \eqref{eq: unif.bound.explain} provides: 
there must exist $\tau_1\in[0,\tau)$ and $\tau_2\in (T-\tau,T]$ such that
\begin{equation} \label{eq: 10.10}
\|y_T(\tau_j)-\overline{y}\|_{\mathscr{H}}\leqslant\frac{\|y_T-\overline{y}\|_{L^2(0,T; \mathscr{H})}}{\sqrt{\tau}}\stackrel{\eqref{eq: unif.bound.explain}}{\leqslant} \frac{C_3}{\sqrt{\tau}}.
\end{equation}
(If not, one readily derives a contradiction.) Here, $C_3>0$ is the constant appearing in \eqref{eq: estimate.no.number}.
As $\tau$ will be chosen at least larger than $C_3^2$ just below, this estimate motivates localizing the entire problem in $[\tau_1,\tau_2]$ in view of sharpening the pointwise estimate of \eqref{eq: unif.bound.explain}.
And so, restricting $u_T$ to the subinterval $[\tau_1,\tau_2]$, one sees that it is a solution to
\begin{equation} \label{eq: aux.ocp}
\inf_{\substack{u\in L^2((\tau_1,\tau_2)\times\omega)\\ \partial_t y=Ay+\mathfrak{f}(y)+Bu \text{ in } (\tau_1,\tau_2)\\y(\tau_1)=y_T(\tau_1)\\y(\tau_2)=y_T(\tau_2)} } \int_{\tau_1}^{\tau_2}\|y(t)-\overline{y}\|^2_{\mathscr{H}}\diff t + \int_{\tau_1}^{\tau_2} \|u(t)\|^2_{L^2(\omega)}\diff t.
\end{equation}
(This can be seen as some kind of dynamic programming principle, and is readily shown by arguing by contradiction.)
We then show that there exists some constant\footnote{Estimate \eqref{eq: 10.12} actually holds with some constant $C_4>0$ independent of $T,\tau_1,\tau_2$ and $\tau$ (the proof follows the lines of that of \eqref{eq: unif.bound.explain}, employing the quasi-turnpike principle), and, in principle, there is no reason to guarantee that $C_4\geqslant C_3$ initially. But we may simply take $C_\bullet:=\max\{C_3,C_4\}$ so that $C_\bullet\geqslant C_3$, and we do so, so that subsequent bounds are simpler to write.} $C_\bullet\geqslant C_3$, independent of $T,\tau_1,\tau_2$ and $\tau$, such that
\begin{equation} \label{eq: 10.12}
\|y_T(t)-\overline{y}\|_{\mathscr{H}}\leqslant C_\bullet\Big(\|y_T(\tau_1)-\overline{y}\|_{\mathscr{H}} + \|y_T(\tau_2)-\overline{y}\|_{\mathscr{H}}\Big)
\end{equation}
holds for all $t\in[\tau_1,\tau_2]$. In view of \eqref{eq: estimate.no.number}, such an estimate would follow should we bound the functional minimized in \eqref{eq: aux.ocp} by means of the right-hand-side in \eqref{eq: 10.12}. The latter can indeed be shown by arguing through the quasi-turnpike principle (see Figure \ref{fig: quasi.2}).
Estimate \eqref{eq: 10.12} combined with \eqref{eq: 10.10} yields
\begin{equation} \label{eq: 934}
\|y_T(t)-\overline{y}\|_{\mathscr{H}}\leqslant\frac{2C_3\cdot C_\bullet}{\sqrt{\tau}}\leqslant\frac12\cdot\frac{4C_\bullet^2}{\sqrt{\tau}}
\end{equation}
for all $t\in[\tau_1,\tau_2]$, and thus also for all $t\in[\tau,T-\tau]$. We henceforth fix 
\begin{equation*}
\tau>16C_\bullet^4; 
\end{equation*}
estimate \eqref{eq: 934} thus yields a contraction. 
The entire argument which precedes can then be repeated by induction on even smaller sub-intervals $[n\tau,T-n\tau]$ for all integers $n\geqslant1$ which satisfy $T-2n\tau\geqslant2T_0$ (this is an upper bound on $n$, in order to be able to repeat the quasi-turnpike argument of Figure \ref{fig: quasi.2}) to obtain 
\begin{equation} \label{eq: bootstrap.ineq}
\sup_{t\in[n\tau,T-n\tau]}\|y_T(t)-\overline{y}\|_{\mathscr{H}}\leqslant\frac12\left(\frac{4C_\bullet^2}{\sqrt{\tau}}\right)^n.
\end{equation}
We may rewrite \eqref{eq: bootstrap.ineq} as 
\begin{align} \label{eq: please.leave.me.alone}
\sup_{t\in[n\tau,T-n\tau]}\|y_T(t)-\overline{y}\|_{\mathscr{H}}&\leqslant\frac12\left(\frac{4C_\bullet^2}{\sqrt{\tau}}\right)^n \nonumber\\ 
&= \frac12\exp\left(-n\log\left(\frac{\sqrt{\tau}}{4C_\bullet^2}\right)\right),
\end{align}
and since $\tau>16C_\bullet^4$, the double-arc exponential estimate will readily follow by a judicious choice of $n$, with $n\geqslant1$ and $T-2n\tau\geqslant2T_0$. (See \eqref{eq: def.n} for the exact choice of $n$, as well as Remark \ref{rem: constants.nonlinearity} for the form of the constants $C$ and $\lambda$, which arise directly from \eqref{eq: def.n} applied to \eqref{eq: please.leave.me.alone}.) 
\end{itemize}

\begin{figure}[h!]
\center
\includegraphics[scale=0.575]{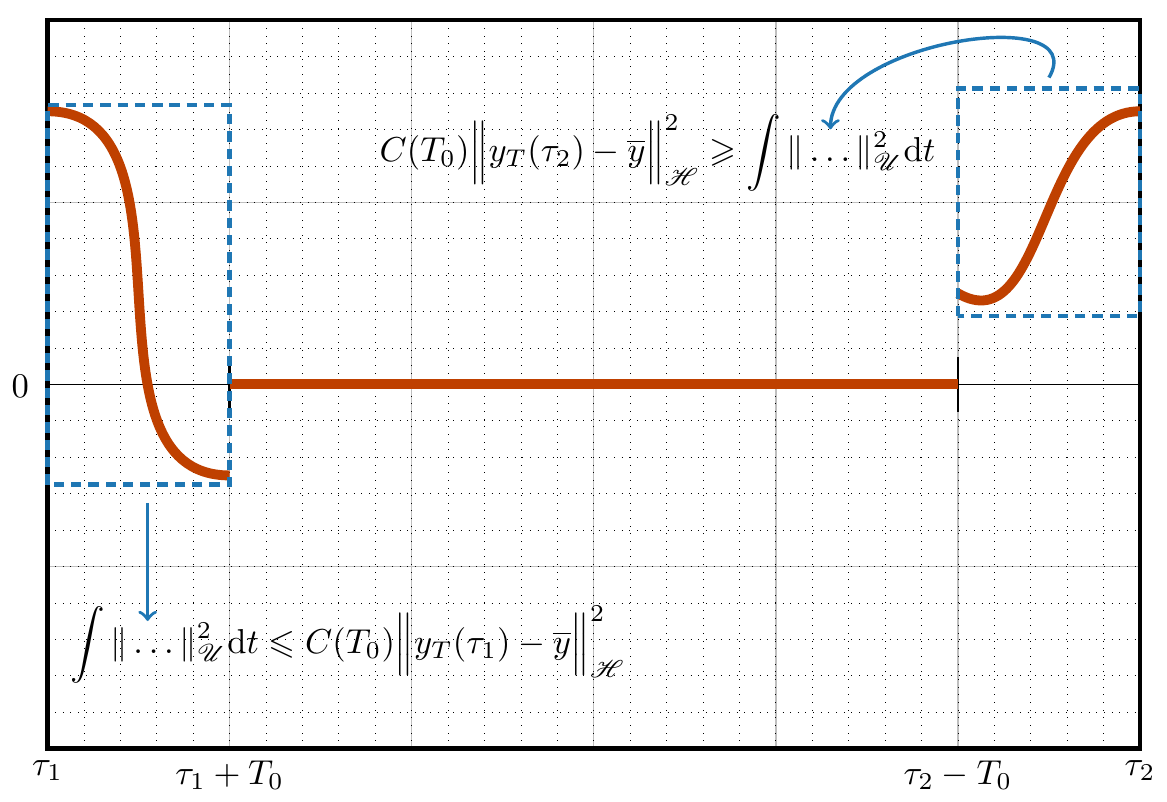}
\includegraphics[scale=0.575]{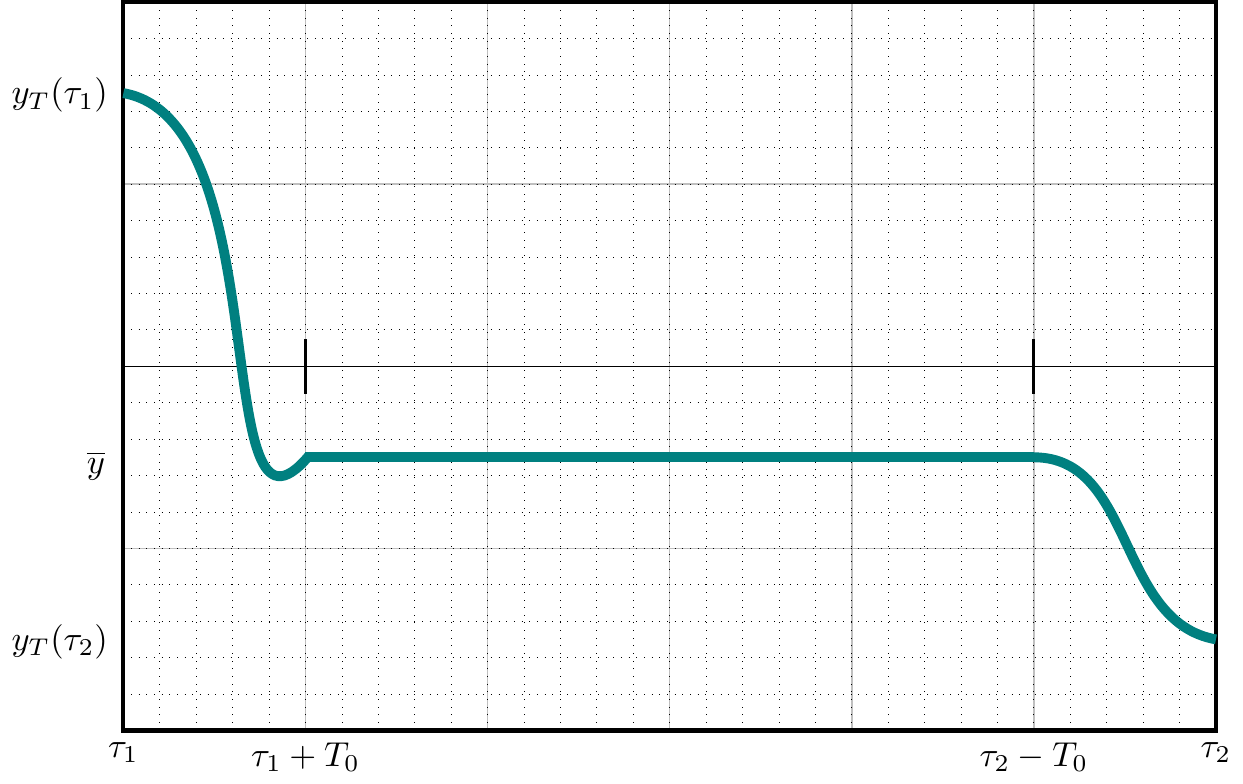}
\caption{(\emph{Left}) We construct a quasi-turnpike control $u^{\text{aux}}(t)=u^11_{[\tau_1,\tau_1+T_0]}(t)+u^21_{[\tau_2-T_0,\tau_2]}(t)$ for $t\in[0,T]$, where $u^1$ controls the state from $y_T(\tau_1)$ to $\overline{y}$ in time $\tau_1+T_0$, while $u^2$ controls from $\overline{y}$ (starting in time $t=\tau_2-T_0$) to $y_T(\tau_2)$ in time $t=\tau_2$ (\emph{right}). 
Using \eqref{eq: 10.10} and Assumption \ref{ass: linear.cost}, we may ensure that $\|u^j\|_{L^2}\leqslant C(T_0)\|y_T(\tau_j)-\overline{y}\|_{\mathscr{H}}$, which combined with the Gr\"onwall inequality for estimating the tracking terms, and the suboptimality of $u^{\text{aux}}$, yields \eqref{eq: 10.12}.}
\label{fig: quasi.2}
\end{figure}

\noindent
Note that at no point in the above steps did we make use of the optimality system, nor explicitly linearize the system. This in turn allowed us to avoid assuming $C^2$-nonlinearities, and smallness assumptions on the initial data $y^0$ or the target $\overline{y}$, which are needed if one proceeds by linearization of the optimality system as in \cite{trelat2015turnpike, porretta2016remarks}. 

\subsection{Discussion}  \label{sec: discussion.nonlinearity}

\begin{remark}[The constants $C$ and $\lambda$] \label{rem: constants.nonlinearity}

From the proof (presented below) and also \eqref{eq: bootstrap.ineq}, one can gather that the constants $C>0$ and $\lambda>0$ appearing in the exponential estimate of Theorem \ref{thm: borjan.thm} are explicit (albeit rather compound).  
We recall that, being given $r>0$ (defined in Assumption \ref{ass: linear.cost}) and $T_0>0$, one selects $\tau$ large enough (at least strictly larger than $16C_\bullet^4$, where $C_\bullet=C_\bullet(r,T_0)>0$ is the constant appearing in \eqref{eq: bootstrap.ineq}), and subsequently, takes $T\geqslant 2(T_0+\tau):=T^*$. 
Then,

\begin{itemize}
\item
The decay rate $\lambda>0$ is given by
	\begin{equation*}
	\lambda:=\frac{\log\left(\frac{\sqrt{\tau}}{4C_\bullet^2}\right)}{\tau+T_0}=\frac{\log\left(\frac{\tau}{16C_\bullet^4}\right)}{T^*}.
	\end{equation*}
	In particular, $\lambda$ also depends on the radius $r>0$ defined in Assumption \ref{ass: linear.cost} through $C_\bullet$ (in an increasing and exponential manner, due to underlying Gr\"onwall inequalities). As a matter of fact, should $\mathfrak{B}_r(\overline{y})=\mathscr{H}$ (namely, the estimates \eqref{eq: 10.3} -- \eqref{eq: 10.4} on the control cost hold for any initial datum), then we can select $r:=\left\|y^0-\overline{y}\right\|_{\mathscr{H}}$ in the proof, which already insinuates a dependence of the decay rate on the initial datum, quite unlike what was encountered in the linear case $\dot{y}=Ay+Bu$, where $\lambda$ solely depends on $A$ and $B$.  
	\smallskip
	
	\item
	On the other hand, from \eqref{eq: equationstar} and \eqref{eq: please.leave.me.alone} (along with the discussion regarding \eqref{eq: please.leave.me.alone}), we gather that the constant $C>0$ takes the form
	\begin{equation*}
	C:=\max\left\{C_3 e^{\lambda(\tau+T_0)},\frac{\sqrt{\tau}}{4C_\bullet^2}\right\},
	\end{equation*} 
	with $C_3>0$ stemming from \eqref{eq: unif.bound.explain}. Furthermore, we may also deduce that, roughly, 
	\begin{equation*}
	C_3^2=\phi(\overline{y})+C(T_0, f, \omega)\left\|y^0-\overline{y}\right\|_{\mathscr{H}}^2. 
	\end{equation*}
	Thus, the linear dependence with respect to $\Big(\left\|y^0-\overline{y}\right\|_{\mathscr{H}}, \|\overline{p}\|_{\mathscr{H}}\Big)$ of the LQ case is not quite maintained through the strategy presented in this section.
	\end{itemize}
\end{remark}

\begin{remark}[Time-irreversible equations] \label{rem: semi.heat}
\begin{itemize}
\item
When $\phi\equiv0$ in \eqref{eq: borjan.ocp}, or $\phi(\overline{y})=0$ (recall that $\phi\geqslant0$), one can repeat the proof above by iterating solely forward in time (namely, consider intervals of the form $[n\tau,T]$ in the induction argument) and show an estimate of the form 
\begin{equation} \label{eq: whatever.2}
\|y_T(t)-\overline{y}\|_{\mathscr{H}}\leqslant Ce^{-\lambda t}. 
\end{equation}
Here, we do not see the final arc near $t=T$ since the turnpike $\overline{y}$ is a zero of the final cost $\phi$, i.e. $\overline{y}\in\{\phi=0\}$.
In some sense, with \eqref{eq: whatever.2} we are recovering a nonlinear extension of well-known linear Riccati theory without making use of the optimality system. 
The result is however not trivial (in the sense that it is not a direct consequence of the controllability), since the underlying dynamics are nonlinear, and the stabilizing control is found by minimizing a (tractable) functional.
Furthermore, in this case, assuming \eqref{eq: 10.4} is not necessary, as solely \eqref{eq: 10.3} suffices. Similarly, solely controllability to the steady state $\overline{y}$ suffices.
This result is also provided and detailed in \cite{esteve2020turnpike}.
\smallskip

\item When \eqref{eq: 10.4} is not needed (suppose, for simplicity, that $\phi\equiv0$ in view of the above discussion), we see that at no point does one need to assume that the semigroup is conservative (i.e. \eqref{eq: conservation}). 
Thus, for problems of the form 
\begin{equation*}
\inf_{\substack{u\in L^2((0,T)\times\omega)\\ y \text{ solves} \eqref{eq: abstract.sys.large}}} \int_0^T \|y(t)-\overline{y}\|_{\mathscr{H}}^2\diff t + \int_0^T\|u(t)\|_{L^2(\omega)}^2\diff t,
\end{equation*}
where $\mathscr{H}=L^2(\Omega)$, with $A=\Delta$, $Bu=u1_\omega$ and $\mathfrak{f}=f$ in \eqref{eq: abstract.sys.large}, assuming only \eqref{eq: 10.3} (and not \eqref{eq: 10.4}, \eqref{eq: conservation}), one can ensure that $\|y_T(t)-\overline{y}\|_{\mathscr{H}}\leqslant Ce^{-\lambda t}$ by slightly adapting the proof presented above. This ensures the validity of the strategy also for the semilinear heat equation with a globally Lipschitz nonlinearity.
\smallskip

\item Having $\phi(\overline{y})\neq0$ and assuming \eqref{eq: 10.4} is precisely an obstacle for applying the strategy to time-irreversible systems such as the (semilinear) heat equation.
Reading the proof, one sees that the target $y^{1}$ will manifest itself roughly as a trajectory snapshot of the form $y_T(T-n\tau)$ (e.g., in \eqref{eq: aux.ocp}), so exact controllability to this reference point would also hold for the semilinear heat equation. The issue is rather ensuring the estimate \eqref{eq: 10.4}, which is used in the process of obtaining \eqref{eq: 10.12}. 
Indeed, $y_T(T-n\tau):=\hat{y}(T_\bullet)$ is an instance of a trajectory $\hat{y}$, which comes along with its own control $\hat{u}$, one would have $\|u-\hat{u}\|_{L^2((0,T_\bullet)\times\omega)}\leqslant C(T_\bullet)\|\overline{y}-\hat{y}(T_\bullet)\|_{\mathscr{H}}$ for a minimal $L^2$-norm control (see \cite[Lemma 8.3]{pighin2018controllability} and the references therein). Such an estimate will not suffice, since then one cannot provide a bound of the minimal $L^2$-norm control $u$ solely in terms of $\|\overline{y}-\hat{y}(T_\bullet)\|_{\mathscr{H}}$. 
\end{itemize}
\end{remark}

\begin{remark}[Exponential estimate for $u_T$] \label{rem: turnpike.u}
Due to the fact that the proof does not make use of the optimality system and linearization (to avoid smoothness assumptions on $f$, and smallness assumptions on $y^0$ and in particular on $\overline{y}$), $u_T$ is not characterized through the adjoint state $p_T$, and thus only an integral turnpike property/estimate rather than an exponential one for $u_T$ is guaranteed. There is however a case, presented in \cite{esteve2020turnpike}, in which exponential turnpike can be ensured. 
If $\phi\equiv0$, as discussed in the above remark, one can ensure that 
\begin{equation*}
\|y_T(t)-\overline{y}\|_{\mathscr{H}}\leqslant Ce^{-\lambda t}.
\end{equation*}
But if moreover $\mathscr{H}=\mathbb{R}^d$ and the underlying ODE is of driftless control-affine form: 
\begin{equation*}
\dot{y}(t)=\sum_{j=1}^m u_j(t) f_j(y(t)), 
\end{equation*}
with $f_j\in\mathrm{Lip}(\mathbb{R}^d;\mathbb{R}^d)$, then 
\begin{equation*}
\|u_T(t)\|\leqslant Ce^{-\lambda t}
\end{equation*}
for $t\in[0,T]$ also holds. The proof of this fact makes crucial use of the homogeneity properties that driftless systems enjoy, which allows one to construct suboptimal controls by simple scalings and show an estimate of the form
\begin{equation*}
\int_{t}^{t+h}\|u_T(t)\|^2\diff t \leqslant 2\int_{t}^{t+h} \|y_T(t)-\overline{y}\|^2\diff t
\end{equation*}
for $h\ll1$ and $t\in[0,T)$. The Lebesgue differentiation theorem would then yield the desired conclusion. 
\end{remark}

\begin{remark}[State penalty] \label{rem: full.state}
Note that in \eqref{eq: borjan.ocp} we are penalizing the energy norm of the full state (which, in the case of the wave equation, is $y(t):=(\zeta(t),\partial_t\zeta(t))$). 
We do this due to the fact that the energy norm of the full state appears on the right-hand side in the interpolation estimate \eqref{eq: estimate.no.number} (see also Lemma \ref{lem: poincare.sobolev}). Indeed, since the strategy consists in showing that the functional $\mathscr{J}_T$ evaluated at an optimal pair is bounded (through the quasi-turnpike principle), and then using this information to ensure a pointwise bound of the state $y(t)$ through \eqref{eq: estimate.no.number}, we need to ensure that the functional $\mathscr{J}_T$ contains all of the terms appearing in the right-hand side of the estimate in \eqref{eq: estimate.no.number}.

For the semilinear wave equation, it is plausible that this restriction can be relaxed, in the sense that one penalizes solely the kinetic or potential energy of the waves (as in the LQ case), by taking advantage of the equipartition of energy principle. This adaptation, however, does not appear trivial, we leave it open for future work. 
\end{remark}

\begin{remark}[On the nonlinearity $f$] \label{rem: f.lip}
The globally Lipschitz character of $f$ is used precisely in \eqref{eq: estimate.no.number} (namely Lemma \ref{lem: poincare.sobolev}). 
This may be solely a technical assumption, which is, however, not necessarily trivial to overcome at a first glance. 
One could stipulate that the strategy should also apply to equations with superlinear nonlinearities (which preserve the controllability mechanism of the linear dynamics), contrary to solely globally Lipschitz ones. 
In essence, the adaptation boils down to obtaining an estimate akin to \eqref{eq: estimate.no.number} for such systems. 
\begin{itemize}
\item
To illustrate the issues which may arise, let us first provide a simple proof of \eqref{eq: estimate.no.number} in the finite-dimensional case:
\begin{equation*} \label{eq: first.equation}
\dot{y}(t)=Ay(t)+\mathfrak{f}(y(t))+u(t)  \hspace{1cm} \text{ in } (0,T),
\end{equation*}
where $A\in\mathbb{R}^{d\times d}$, and $\mathfrak{f}\in\mathrm{Lip}(\mathbb{R}^d;\mathbb{R}^d)$. Suppose that $\overline{y}\in\mathbb{R}^d$ is some non-trivial steady state, with null control. 
We see that $\zeta(t):=y(t)-\overline{y}$ solves
\begin{equation} \label{eq: second.equation}
\dot{\zeta}(t) = A\zeta(t)+ g(\zeta(t)) + u(t) \hspace{1cm} \text{ in } (0,T),
\end{equation}
where $g$ is again globally Lipschitz. Clearly
\begin{equation} \label{eq: third.equation}
|\zeta(t)|^2 - |\zeta(0)|^2 = 2\int_0^t \zeta(s)\cdot\dot{\zeta}(s)\diff s
\end{equation}
for $t\in[0,T]$. But then, by the Cauchy-Schwarz and Young inequalities,
\begin{equation} \label{eq: fourth.equation}
\int_0^t \zeta(s)\cdot\dot{\zeta}(s)\diff s\leqslant \int_0^t |\zeta(s)|^2\diff s + \frac{1}{4}\int_0^t |\dot{\zeta}(s)|^2\diff s.
\end{equation}
Finally, directly using \eqref{eq: second.equation} and the Lipschitz character of $g$, one finds
\begin{equation} \label{eq: fifth.equation}
|\dot{\zeta}(s)|\leqslant C(A,g)|\zeta(s)|+|u(s)|.
\end{equation}
Putting \eqref{eq: third.equation}, \eqref{eq: fourth.equation} and \eqref{eq: fifth.equation} together, one derives \eqref{eq: estimate.no.number}.
\smallskip

\item
A canonical superlinear nonlinearity for which, oftentimes, controllability is preserved from the linear dynamics (in both finite and infinite dimensions) is the cubic nonlinearity. Let us thus consider
\begin{equation*}
\dot{y}=Ay-|y|^2 y+u \hspace{1cm} \text{ in } (0,T).
\end{equation*}
Suppose $\overline{y}\in\mathbb{R}^d$ is a non-trivial steady state, with zero control. 
We look to repeat the same arguments as in what precedes. Starting from \eqref{eq: third.equation}, we see that
\begin{align*}
\int_0^t (y(s)-\overline{y})\dot{y}(s)\diff s = \int_0^t (y(s)-\overline{y})\cdot\big(Ay(s)-|y(s)|^2 y(s)+u(s)\big)\diff s.
\end{align*}
If one applies the Cauchy-Schwarz and Young inequalities as in \eqref{eq: fourth.equation}, and uses the fact that $A\overline{y}-|\overline{y}|^2\overline{y}=0$, then inevitably the term
\begin{equation} \label{eq: cubic.problem}
\int_0^t \Big|\big|y(s)\big|^2y(s)-\big|\overline{y}\big|^2\overline{y}\Big|^2 \diff s
\end{equation}
appears. Recall that in \eqref{eq: estimate.no.number}, the norms appearing in the upper bound are precisely those minimized in the cost functional (in occurrence, $\|y-\overline{y}\|_{L^2(0,T;\mathbb{R}^d)}^2+\|u\|_{L^2(0,T;\mathbb{R}^d)}^2$). Hence, to derive \eqref{eq: estimate.no.number}, we would like to roughly absorb \eqref{eq: cubic.problem} by the quantity we minimize in the cost functional.
But the validity of such a claim is far from obvious to our understanding. If the derivative $\dot{y}(t)$ is also tracked in the functional, the elementary Sobolev embedding $H^1(0,T)\hookrightarrow C^0([0,T])$ could perhaps be of use. But this too remains to be analyzed with more rigor, as the constant which appears from applying the Sobolev inequality would depend on $T$ (albeit explicitly). Similar issues persist in the PDE case (for both heat and wave equations).
\end{itemize}
At any rate, should one be able to prove that an estimate such as \eqref{eq: estimate.no.number} holds for, say, locally Lipschitz-only nonlinearities for which blow-up is avoided, then the strategy should also be applicable to such settings. 
\end{remark}

\begin{remark}[Linear cost assumption]
\begin{itemize}
\item While we suppose that the underlying system \eqref{eq: abstract.sys.large} is controllable for arbitrarily large data, through \eqref{eq: 10.3} -- \eqref{eq: 10.4} we only assume that the cost of control is proportionate to the distance from the chosen steady state $\overline{y}$ in some, possibly arbitrarily small ball around this steady state. This assumption is oftentimes satisfied by semilinear systems which are shown to be controllable by looking at an associated linear problem combined with a fixed point theorem of some form (with possibly under various smallness assumptions on the Lipschitz constant for finite dimensional systems, see, e.g., \cite{carmichael1985fixed, seidman1987invariance}, and \cite{zhang2000exact} for further references). Such conclusions hold, for instance, for the semilinear wave equation (with $f(0)=0$).
\smallskip

\item
We cannot ensure the validity of estimates \eqref{eq: 10.3} -- \eqref{eq: 10.4} for driftless systems: $$\dot{y}(t)=\sum_{j=1}^m u_j(t)f_j(y(t)),$$ when $m<d$. This is due to the so-called ball-box theorem in sub-Riemannian geometry (\cite{agrachev2019comprehensive}), for smooth vector fields $f_1,\ldots,f_m$. This theorem states the following. Suppose that the vector fields $f_1,\ldots,f_m$ satisfy the H\"ormander condition, namely that the iterated Lie brackets of these vector fields at any point span $\mathbb{R}^d$. Let us henceforth denote 
$$\bigtriangleup^1(x):=\mathrm{span}\{f_1(x),\ldots, f_m(x)\}$$ 
for $x\in\mathbb{R}^d$, and then iterate as 
$$\bigtriangleup^{k+1}:=\bigtriangleup^1+[\bigtriangleup^{k},\bigtriangleup^1]$$ 
for $k\geqslant1$. Then, by virtue of the H\"ormander condition, there exists some $\kappa\geqslant1$ such that $\bigtriangleup^\kappa(x)=\mathbb{R}^d$ for all $x$. Furthermore, by the ball-box theorem, for $y^0$ close enough to $y^1$, an estimate of the form 
$$\|y^0-y^1\|\lesssim d_{\mathrm{SR}}(y^0,y^1)\lesssim\|y^0-y^1\|^{\sfrac{1}{\kappa}}$$ 
holds, where $d_{\mathrm{SR}}(y^0,y^1)$ is the sub-Riemannian distance of $y^0$ to $y^1$, equal (modulo a scalar multiple depending on $T_0$) to the infimum defined in \eqref{eq: 10.3} -- \eqref{eq: 10.4}. 
	Herein, one sees that if $m\geqslant d$, it may happen to find at least $d$ among $m$ vector fields which are linearly independent, thus ensuring that $\kappa=1$, as desired; this is quite simply impossible when $m<d$. 
	 This exact constraint is also encountered in \cite[Theorem 5.1]{esteve2020large}, where the estimates \eqref{eq: 10.3} -- \eqref{eq: 10.4} are shown to hold for $m\geqslant d$ in the driftless setting. 
	A clearer picture regarding this issue is also needed for general control-affine systems beyond those for which linearization techniques might not apply. We refer to \cite{jean2015complexity, prandi2014holder} for developments in this direction.
\end{itemize}
\end{remark}

\begin{remark}[Controlled steady states] 
One can also consider more general controlled steady states $\overline{y}$ as targets in $\mathscr{J}_T$; focusing on the semilinear wave equation, we could take $\overline{y}$ such that
\begin{equation*}
A\overline{y}+\mathfrak{f}(\overline{y})+B\overline{u}=0
\end{equation*}
for a given $\overline{u}\in L^2(\omega)$, provided the functional $\mathscr{J}_T$ is modified accordingly, namely rather consider
\begin{equation*}
\mathscr{J}_T(u):=\phi(y(T))+\int_0^T\|y(t)-\overline{y}\|^2_{\mathscr{H}}\diff t+\int_0^T \|u(t)-\overline{u}\|_{L^2(\omega)}^2\diff t.
\end{equation*}
The cornerstone of the above strategy lies in using the controllability assumption to construct sub-optimal controls which annul the running cost beyond time $T_0>0$. In the presence of a target for the control, given a controllability control $u^1$ steering $y^1$ to $\overline{y}$ in time $T_0$, one could simply construct a quasi-turnpike control by setting $u^{\text{aux}}(t)=u^11_{[0,T_0]}+\overline{u}1_{[T_0,T]}$ for $t\in[0,T]$, and the strategy would remain the same.
\end{remark}

\subsection{Proof of Theorem \ref{thm: borjan.thm}}

We proceed with (most details of) the proof. We focus on providing a transparent presentation, and for the full technical details, we refer the reader to \cite{esteve2020turnpike}. The proof requires a couple of preliminary lemmas. 
We begin with

\begin{lemma}[Poincaré-Sobolev-type inequality] \label{lem: poincare.sobolev}
Let $y^0\in\mathscr{H}$, and let $\overline{y}$ be as in \eqref{eq: kernel.eq}. 
There exists a constant $C>0$ (depending on $y^0,\overline{y}, f$) such that for any $T>0$ and $u\in L^2((0,T)\times\omega)$, the unique solution $y$ to \eqref{eq: abstract.sys.large} is such that
\begin{equation*}
\|y(t)-\overline{y}\|_{\mathscr{H}}\leqslant C\Big(\left\|y^0-\overline{y}\right\|_{\mathscr{H}}+\|y-\overline{y}\|_{L^2(0,T;\mathscr{H})}+\|u\|_{L^2((0,T)\times\omega)}\Big)
\end{equation*}
holds for all $t\in[0,T]$.
\end{lemma}

The proof of the Lemma in the infinite-dimensional setting is slightly more complicated than the elementary argument presented in Remark \ref{rem: f.lip}, due to the presence of an unbounded operator $A$, so we provide some more detail.

\begin{proof}
We readily see that $\psi(t):=y(t)-\overline{y}$ is the unique solution to 
\begin{equation*}
\begin{cases}
\partial_t \psi - A\psi + \mathfrak{f}(\psi+\overline{y})-\mathfrak{f}(\overline{y})=Bu &\text{ in } (0,T),\\
\psi_{|_{t=0}}=y^0-\overline{y},
\end{cases}
\end{equation*}
and has the Duhamel formula characterization
\begin{equation*}
y(t)-\overline{y}=e^{tA}\left(y^0-\overline{y}\right)+\int_0^t e^{(t-s)A}Bu(s)\diff s -\int_0^t e^{(t-s)A}\Big(\mathfrak{f}(y(s))-\mathfrak{f}(\overline{y})\Big)\diff s.
\end{equation*}
Let us first suppose that $t\leqslant 1$. As $\left\|e^{tA}\right\|_{\mathscr{L}(\mathscr{H})}=1$, using solely the Lipschitz character of $\mathfrak{f}$ (through that of $f$) and the fact that $B$ is bounded, we find
\begin{align*}
\|y(t)-\overline{y}\|_{\mathscr{H}}&\leqslant\left\|y^0-\overline{y}\right\|_{\mathscr{H}}\\
&\quad+ C(B,f)\left(\int_0^t \|u(s)\|_{L^2(\omega)}\diff s + \int_0^t\|y(s)-\overline{y}\|_{\mathscr{H}}\diff s\right).
\end{align*}
Using Gr\"onwall's lemma, and the Cauchy-Schwarz inequality (as $t\leqslant1$) yield the conclusion. 
Now suppose that $t\in(1,T]$. We claim that there exists $t^*\in(t-1,t]$ such that
\begin{equation} \label{eq: 10.6}
\|y(t^*)-\overline{y}\|_{\mathscr{H}}\leqslant\|y-\overline{y}\|_{L^2(0,T;\mathscr{H})}.
\end{equation}
This can indeed readily be shown by arguing by contradiction. By writing the Duhamel formula for $y(t)-\overline{y}$ in $[t^*,t]$, namely
\begin{equation*}
y(t)-\overline{y} = e^{tA}\left(y^0-\overline{y}\right)+\int_{t^*}^t e^{(t-s)A}Bu(s)\diff s - \int_{t^*}^t e^{(t-s)A}\Big(\mathfrak{f}(y(s))-\mathfrak{f}(\overline{y})\Big)\diff s,
\end{equation*}
we see that, just as before,
\begin{align*}
\|y(t)-\overline{y}\|_{\mathscr{H}}&\leqslant\left\|y(t^*)-\overline{y}\right\|_{\mathscr{H}}\\
&\quad+ C(B,f)\left(\int_{t^*}^t \|u(s)\|_{L^2(\omega)}\diff s + \int_{t^*}^t\|y(s)-\overline{y}\|_{\mathscr{H}}\diff s\right).
\end{align*}
Now, using Gr\"onwall's lemma, the Cauchy-Schwarz inequality (as $t-t^*\leqslant 1$), and \eqref{eq: 10.6}, we may conclude the proof.
\end{proof}

\begin{lemma}[Uniform bounds] \label{lem: uniform.bounds}
Suppose $y^0\in\mathscr{H}$. Let $T>T_0$, where $T_0$ denotes the (minimal) controllability time for \eqref{eq: abstract.sys.large}, and let $u_T$ be any global minimizer to $\mathscr{J}_T$, with $y_T$ denoting the associated state, unique solution to \eqref{eq: abstract.sys.large}. Then there exists a constant $C>0$, independent of $T$, such that
\begin{equation*}
\mathscr{J}_T(u_T)+\|y_T(t)-\overline{y}\|_{\mathscr{H}}^2\leqslant C.
\end{equation*} 
holds for all $t\in[0,T]$.
\end{lemma}

\begin{proof}  By controllability, there exists some control $u^1$ (independent of $T$) such that the solution $y^1$ to \eqref{eq: abstract.sys.large} on $[0,T_0]$ satisfies $y^1(T_0)=\overline{y}$. 
We now set 
\begin{equation*}
u^{\text{aux}}(t):=u^1(t)1_{[0,T_0]}(t) \hspace{1cm} \text{ for } t\in[0,T],
\end{equation*} 
and let $y^{\text{aux}}$ be the associated solution to \eqref{eq: abstract.sys.large}. Clearly, 
\begin{equation*}
y^{\text{aux}}(t)\equiv\overline{y} \hspace{1cm} \text{ for } t\in[T_0,T].
\end{equation*}
Since $u_T$ is a minimizer of $\mathscr{J}_T$, we find 
\begin{equation*}
\mathscr{J}_T(u_T)\leqslant\mathscr{J}_T\left(u^{\text{aux}}\right)=\phi(\overline{y})+\int_0^{T_0}\left\|y^1(t)-\overline{y}\right\|_{\mathscr{H}}^2\diff t + \int_0^{T_0}\left\|u^1(t)\right\|^2_{L^2(\omega)}\diff t.
\end{equation*} 
As the right-hand-side is a constant independent of $T$, the result follows by applying Lemma \ref{lem: poincare.sobolev}.
\end{proof}

We may now provide the proof of Theorem \ref{thm: borjan.thm}.

\begin{proof} The uniform bound on optimal controls $u_T$ follows by Lemma \ref{lem: uniform.bounds}. We thus focus on proving the exponential turnpike estimate for the optimal state $y_T$. 
Let $T_0>0$ be the (minimal) controllability time of \eqref{eq: abstract.sys.large}.
Let $C_1>0$ denote the constant appearing in Lemma \ref{lem: uniform.bounds}. Let 
\begin{equation*}
\tau>0 
\end{equation*}
be a fixed degree of freedom and to be chosen later, and suppose 
\begin{equation*}
T\geqslant 2(\tau+T_0). 
\end{equation*}
We shall distinguish two cases.
\smallskip

\noindent
\textbf{Case 1. Should $t\in[0,\tau+T_0]\cup[T-(\tau+T_0),T]$.} In this case, the length of each time interval where $t$ lies is independent of $T$, and so the exponential turnpike estimate follows simply by Lemma \ref{lem: uniform.bounds}. Indeed, by Lemma \ref{lem: uniform.bounds}, for any $\lambda>0$, we have
\begin{align*}
\|y_T(t)-\overline{y}\|_{\mathscr{H}}&\leqslant C_1 e^{\lambda t} e^{-\lambda t}\leqslant C_1 e^{\lambda(\tau+T_0)}\Big(e^{-\lambda t}+e^{-\lambda(T-t)}\Big)
\end{align*}
for $t\in[0,\tau+T_0]$, and similarly
\begin{align*}
\|y_T(t)-\overline{y}\|_{\mathscr{H}}&\leqslant C_1 e^{\lambda(T-t)} e^{-\lambda(T-t)}\\&\leqslant C_1 e^{\lambda(\tau+T_0)}\Big(e^{-\lambda t}+e^{-\lambda(T-t)}\Big)
\end{align*}
for $t\in[T-(\tau+T_0),T]$. This yields the desired conclusion in the union of these time intervals; since $\lambda>0$ is arbitrary in both of the above estimates, the final rate $\lambda>0$ appearing in the statement of the theorem will be derived from the second case.
\smallskip

\noindent
\textbf{Case 2. Should $t\in(\tau+T_0,T-(\tau+T_0))$.} This is more delicate. 
The main clue will be to actually prove an estimate of the form
\begin{equation} \label{eq: rough.estimate}
\sup_{t\in[n\tau,T-n\tau]}\|y_T(t)-\overline{y}\|_{\mathscr{H}}\leqslant\left(\frac{4C_\bullet}{\sqrt{\tau}}\right)^n
\end{equation}
for some constant $C_\bullet>0$ independent of $T$ and $\tau$, and for any integer $n$ such that
\begin{equation*}
1\leqslant n\leqslant\frac{1}{\tau}\left(\frac{T}{2}-T_0\right). 
\end{equation*}
Indeed, suppose that estimate \eqref{eq: rough.estimate} holds. 
Then for any fixed but otherwise arbitrary $t\in(\tau+T_0,T-(\tau+T_0))$, one sets
\begin{equation} \label{eq: def.n}
n(t):=\min\left\{\left\lfloor\frac{t}{\tau+T_0}\right\rfloor,\left\lfloor\frac{T-t}{\tau+T_0}\right\rfloor\right\}.
\end{equation}
Clearly $n(t)\geqslant1$, as well as 
\begin{equation*}
n(t)\leqslant\frac{1}{\tau}\left(\frac{T}{2}-T_0\right)
\end{equation*}
(the latter is more tricky, and needs noting that $s\mapsto\frac{s-2T_0}{s}$ is nondecreasing), and finally, 
\begin{equation*}
n(t)\tau\leqslant t\leqslant T-n(t)\tau.
\end{equation*} 
Thus, for such fixed $t$, if we select $\tau>16C_\bullet^4$ in \eqref{eq: rough.estimate}, we find
\begin{align*}
\|y_T(t)-\overline{y}\|_{\mathscr{H}}&\leqslant\exp\left(-n(t)\log\left(\frac{\sqrt{\tau}}{4C_\bullet^2}\right)\right)\\
&\leqslant\frac{\sqrt{\tau}}{4C_\bullet^2}\left(\exp\left(-\frac{\log\left(\frac{\sqrt{\tau}}{4C_\bullet^2}\right)}{\tau+T_0}t\right)+\exp\left(-\frac{\log\left(\frac{\sqrt{\tau}}{4C_\bullet^2}\right)}{\tau+T_0}(T-t)\right)\right),
\end{align*}
where in the last estimate we have used the fact that either $n(t)\geqslant\frac{t}{\tau+T_0}-1$ or $n(t)\geqslant\frac{T-t}{\tau+T_0}-1$ must hold by definition of $n(t)$. 
This is the desired exponential turnpike estimate, with decay rate 
\begin{equation*}
\lambda:=\frac{\log\left(\frac{\sqrt{\tau}}{4C_\bullet}\right)}{\tau+T_0}>0,
\end{equation*}
and taking Case 1 into account, the constant $C>0$ appearing in the statement of the theorem takes the form
\begin{equation*}
C:=\max\left\{C_1e^{\lambda(\tau+T_0)},\frac{\sqrt{\tau}}{4C_\bullet^2}\right\}.
\end{equation*}
Thus, our task reduces to proving an estimate of the form \eqref{eq: rough.estimate}.
We shall proceed by induction.

\begin{enumerate}
\item[\textbf{1.}] Since $T\geqslant 2(\tau+T_0)$ and thus $\tau\leqslant\sfrac{T}{2}$, we can readily show\footnote{This estimate is true for any $\psi\in C^0([0,T];\mathscr{H})$, and can be shown by an indirect argument: if \begin{equation*}\|\psi(t)\|_{\mathscr{H}}>\frac{\|\psi\|_{L^2(0,T;\mathscr{H})}}{\sqrt{\tau}}\end{equation*} for all $t\in[0,\tau)$ or all $t\in(T-\tau,T]$, the integrating over $[0,T]$ one readily finds a condradiction.} that there exist a couple of time instances $\tau_1\in[0,\tau)$ and $\tau_2\in(T-\tau,T]$ such that 
\begin{equation*}
\|y_T(\tau_j)-\overline{y}\|_{\mathscr{H}}\leqslant\frac{\|y_T-\overline{y}\|_{L^2(0,T;\mathscr{H})}}{\sqrt{\tau}}
\end{equation*}
for $j\in\{1,2\}$. This estimate, combined with Lemma \ref{lem: uniform.bounds}, then yields
\begin{equation} \label{eq: conjunction}
\|y_T(\tau_i)-\overline{y}\|_{\mathscr{H}}\leqslant\frac{C_1}{\sqrt{\tau}}.
\end{equation}
Here, the constant $C_1>0$ stems from Lemma \ref{lem: uniform.bounds}, and is independent of both $T$ and $\tau$.
We shall now restrict our analysis onto $[\tau_1,\tau_2]$ and extrapolate onto the strict subset $[\tau,T-\tau]$. It can be seen\footnote{Can be shown by arguing by contradiction.} that $u_{T_{|_{[\tau_1,\tau_2]}}}$ is a solution to 
\begin{equation} \label{eq: ocp.mit}
\inf_{\substack{u\in L^2((\tau_1,\tau_2)\times\omega)\\\partial_t y=Ay+\mathfrak{f}(y)+Bu \text{ in } (\tau_1,\tau_2)\\ y(\tau_1)=y_T(\tau_1)\\ y(\tau_2)=y_T(\tau_2)}} \int_{\tau_1}^{\tau_2}\|y(t)-\overline{y}\|_{\mathscr{H}}^2 + \int_{\tau_1}^{\tau_2}\|u(t)\|_{L^2(\omega)}^2\diff t.
\end{equation}
For this optimal control problem, arguing as in Lemma \ref{lem: uniform.bounds}, whilst using the fact $\tau_2-\tau_1>2T_0$, as well as Assumption \ref{ass: linear.cost} in conjunction with \eqref{eq: conjunction} for 
\begin{equation} \label{eq: tau.r}
\tau\geqslant\frac{C_1^2}{r^2}
\end{equation}
(namely $\tau$ such that $\sfrac{C_1}{\sqrt{\tau}}\leqslant r$ in \eqref{eq: conjunction}), one can show that there exists a constant $C_2>0$, depending on $r$ and $C_1$, but otherwise independent of $T, \tau,\tau_1$ and $\tau_2$, such that
\begin{equation} \label{eq: 10.22}
\|y_T(t)-\overline{y}\|_{\mathscr{H}}\leqslant C_2\Big(\|y_T(\tau_1)-\overline{y}\|_{\mathscr{H}} + \|y_T(\tau_2)-\overline{y}\|_{\mathscr{H}}\Big)
\end{equation}
holds for all $t\in[\tau_1,\tau_2]$. 
The proof of \eqref{eq: 10.22} (precisely illustrated in Figure \ref{fig: quasi.2}, see also \cite[Lemma 5.2]{esteve2020turnpike}) relies on constructing a quasi-turnpike control (which is suboptimal for the functional in \eqref{eq: ocp.mit}) steering $y_T(t)$ to $\overline{y}$ in time $\tau_1+T_0$ by controllability, then staying at the steady state $\overline{y}$ until time $\tau_2-T_0$ by using no control whatsoever, and finally exiting $\overline{y}$ to reach $y_T(\tau_2)$ in time $\tau_2$, again by controllability. (See Figure \ref{fig: quasi.2}.) 
This quasi-turnpike control will be bounded precisely by the right-hand-side in \eqref{eq: 10.22} through the linear cost assumption estimates \eqref{eq: 10.3} -- \eqref{eq: 10.4}, and the same can then be said for the state tracking terms by using a Gr\"onwall inequality argument.
Setting 
\begin{equation*}
C_\bullet:=\max\{C_1,C_2\}>0, 
\end{equation*}
we find 
\begin{equation*}
\|y_T(t)-\overline{y}\|_{\mathscr{H}}\leqslant \frac{1}{2}\left(\frac{4C_\bullet^2}{\sqrt{\tau}}\right)
\end{equation*}
for all $t\in[\tau_1,\tau_2]$, and thus also for all $t\in[\tau,T-\tau]$. This proves the desired estimate \eqref{eq: rough.estimate} for $n=1$. 
\smallskip

\item[\textbf{2.}] We now bootstrap the above argument. We shall show that for any integer $n$ satisfying 
\begin{equation*}
1\leqslant n\leqslant\frac{1}{\tau}\left(\frac{T}{2}-T_0\right),
\end{equation*} 
one has
\begin{equation} \label{eq: 5.24}
\sup_{t\in[n\tau,T-n\tau]}\|y_T(t)-\overline{y}\|_{\mathscr{H}}\leqslant\frac{1}{2}\left(\frac{4C_\bullet^2}{\sqrt{\tau}}\right)^n.
\end{equation}
The parameter $n$ is chosen as such to ensure $T-2n\tau\geqslant 2T_0$, in view of repeating the argument of the case $n=1$ (which requires constructing a quasi-turnpike control by using controllability in two disjoint intervals of length $T_0$, hence the factor $2T_0$). By induction, we suppose that \eqref{eq: 5.24} holds for some $n$, and we aim to show heredity at stage $n+1$.
To this end, suppose that 
\begin{equation*}
n+1\leqslant\frac{1}{\tau}\left(\frac{T}{2}-T_0\right). 
\end{equation*}
Then we clearly have 
$$\tau\leqslant\frac{T-2n\tau}{2}.$$ 
And since $T-2n\tau\geqslant 2T_0$, it can be seen that $u_T|_{[n\tau,T-n\tau]}$ is a solution to
\begin{equation*}
\inf_{\substack{u\in L^2((n\tau,T-n\tau)\times\omega)\\\partial_t y=Ay+\mathfrak{f}(y)+Bu \text{ in } (n\tau,T-n\tau)\\ y(n\tau)=y_T(n\tau)\\ y(T-n\tau)=y_T(T-n\tau)}} \int_{n\tau}^{T-n\tau}\|y(t)-\overline{y}\|_{\mathscr{H}}^2 + \int_{n\tau}^{T-n\tau}\|u(t)\|_{L^2(\omega)}^2\diff t.
\end{equation*}
Arguing as before, we may again find time instances $t_1\in[n\tau,(n+1)\tau)$ and $t_2\in(T-(n+1)\tau,T-n\tau]$ such that
\begin{align*}
\|y_T(t_i)-\overline{y}\|_{\mathscr{H}}&\leqslant\frac{\|y_T-\overline{y}\|_{L^2(n\tau,T-n\tau; \mathscr{H})}}{\sqrt{\tau}}\nonumber\\
&\leqslant \frac{C_2}{\sqrt{\tau}}\Big(\|y_T(n\tau)-\overline{y}\|_{\mathscr{H}}+\|y_T(T-n\tau)-\overline{y}\|_{\mathscr{H}}\Big).
\end{align*}
Here, $C_2>0$ is precisely the same constant as in \eqref{eq: 10.22}.
We may use the induction hypothesis \eqref{eq: 5.24} to deduce that
\begin{equation} \label{eq: as.above}
\|y_T(t_i)-\overline{y}\|_{\mathscr{H}}\leqslant\frac{C_2}{\sqrt{\tau}}\left(\frac{4C_\bullet^2}{\sqrt{\tau}}\right)^n
\end{equation}
We need to recover a power of $C_2$ in the estimate to conclude, so we repeat the same argument on $[t_1,t_2]$. Since $t_2-t_1\geqslant2T_0$, and since $u_T|_{[t_1,t_2]}$ is a solution to 
\begin{equation*}
\inf_{\substack{u\in L^2((t_1,t_2)\times\omega)\\\partial_t y=Ay+\mathfrak{f}(y)+Bu \text{ in } (t_1,t_2)\\ y(t_1)=y_T(t_1)\\ y(t_2)=y_T(t_2)}} \int_{t_1}^{t_2}\|y(t)-\overline{y}\|_{\mathscr{H}}^2 + \int_{t_1}^{t_2}\|u(t)\|_{L^2(\omega)}^2\diff t.
\end{equation*}
using \eqref{eq: as.above} and $C_\bullet\geqslant C_2$, we may deduce that
\begin{align*}
\|y_T(t)-\overline{y}\|_{\mathscr{H}}&\leqslant C_2\Big(\|y_T(t_1)-\overline{y}\|_{\mathscr{H}}+\|y(t_2)-\overline{y}\|_{\mathscr{H}}\Big)\\
&\leqslant\frac{1}{2}\frac{4C_\bullet^2}{\sqrt{\tau}}\left(\frac{4C_\bullet^2}{\sqrt{\tau}}\right)^n
\end{align*}
for $t\in[t_1,t_2]$. This estimate also holds for $t\in[(n+1)\tau,T-(n+1)\tau]$, as desired, thus concluding the proof of \eqref{eq: 5.24} (namely \eqref{eq: rough.estimate}) should $\tau>\sfrac{C_1^2}{r^2}$. 
\end{enumerate}
In view of \eqref{eq: tau.r}, and to then ensure that $\sfrac{\sqrt{\tau}}{4C_\bullet^2}<1$ in \eqref{eq: rough.estimate}, we need to select
\begin{equation*} \label{eq: def.tau}
\tau>16C_\bullet^4+\frac{C_1^2}{r^2}.
\end{equation*}
This concludes the proof.
\end{proof}

\part{Applications}

\section{Initializing optimization algorithms} \label{sec: 11}

\subsection{Background}
One of the first practical applications of the concrete mathematical developments in turnpike theory was given in \cite{trelat2015turnpike}. In addition to proving a local turnpike property for nonlinear optimal control problems, the authors also provide an efficient way for initializing numerical methods for optimal control.

There exist, in essence, two kinds of approaches for the numerical resolution of continuous-time optimal control problems: \emph{direct} and \emph{indirect} methods (\cite{von1992direct, benzi2005numerical, betts2010practical, trelat2005controle, trelat2012optimal}). 
Direct methods consist of discretizing both the state and the control, so as to reduce the optimal control problem to a constrained optimization problem in finite dimension -- this is known as the \emph{discretize then control/optimize} paradigm. 
Indirect methods on the other hand consist of solving numerically the boundary value problem derived from the application of the Pontryagin maximum principle (which is a \emph{control/optimize then discretize} paradigm).
Both methods are known to suffer from issues regarding initialization. Yet, the  knowledge that turnpike is valid for the underlying optimal control problem can be used as a prior for constructing the initial point.
Following \cite{trelat2015turnpike}, we briefly present how this can be done for indirect methods in particular, where turnpike provides a clever insight.
For direct methods, the turnpike property also provides significant speedup and accuracy in the optimization scheme, in which one merely initializes the method precisely at the turnpike. 

\subsection{Setting}
In the context of optimal PDE control, one generally first semi-discretizes the PDE in the spatial (generally, any non-time) variable with care, and considers some quadrature formula for the cost functional. 
Let us thus focus on optimal ODE control, to have a clearer picture of the turnpike insights. We consider a generic optimal control problem of the form
\begin{equation} \label{eq: trelat.ocp.shooting}
\inf_{\substack{u\in L^2(0,T;\mathbb{R}^m)\\ y\text{ solves} \eqref{eq: trelat.ode}}} \int_0^T f^0(y(t),u(t))\diff t,
\end{equation}
where the underlying ODE constraint is
\begin{equation} \label{eq: trelat.ode}
\begin{cases}
\dot{y}(t) = f(y(t),u(t)) &\text{ in }(0,T),\\
y(0) = y^0\in\mathbb{R}^d.
\end{cases}
\end{equation}
We avoid final conditions or pay-offs for simplicity; the running cost $f^0$ in \eqref{eq: trelat.ocp.shooting} is assumed to satisfy necessary convexity, coercivity and continuity assumptions for ensuring (at least) the existence of solutions. 
One applies the Pontryagin Maximum Principle for an optimal pair $(u, y)$ to find the existence of an adjoint state $p$ such that 
\begin{equation} \label{eq: 11.3}
\begin{cases}
\dot{y}(t) = \partial_p\*H(y(t), p(t), u(t)) &\text{ in } (0,T),\\
\dot{p}(t) = -\partial_y\*H(y(t), p(t), u(t)) &\text{ in } (0,T),\\
y(0) = y^0, \\
p(T) = 0,
\end{cases}
\end{equation}
with $u(t)$ being found by solving
\begin{equation} \label{eq: 11.4}
\partial_u \*H(y(t),p(t), u(t)) = 0 \hspace{1cm} \text{ for } t\in(0,T).
\end{equation}
Recall that the Hamiltonian $\*H$ is given by $\*H(y,p,u) = p\cdot f(y,u) + f^0(y,u)$.
One notes that $\partial_p\*H(y,p,u) = f(y,u)$; moreover, in the LQ setting, one clearly recovers $u=B^*p$. 

\subsection{Shooting method}
If we assume that equation \eqref{eq: 11.4} gives an explicit representation for $u$ in terms of $(y,p)$ (as is the case, for instance, when the Hamiltonian is a power-like nonlinearity in $u$, as is typical in most cases), we see the optimality system \eqref{eq: 11.3} as a shooting problem: setting $\*z:=(y,p)$, due to \eqref{eq: 11.4}, one writes the first two equations in \eqref{eq: 11.3} as
\begin{equation} \label{eq: 11.5}
\dot{\*z}(t) = \*F(\*z(t)) \hspace{1cm} \text{ for } t\in(0,T),
\end{equation}
and the latter two as 
\begin{equation} \label{eq: 11.6}
\*G(\*z(0),\*z(T))=0.
\end{equation}
In the classic setting of the shooting method, one somehow initializes the datum $\*z_0\in\mathbb{R}^{2d}$, and finds the solution $\*z(t;\*z_0)$ to
\begin{equation} \label{eq: 11.6.5}
\begin{cases}
\dot{\*z}(t) = \*F(\*z(t)) &\text{ in }(0,T),\\
\*z(0) = \*z_0.
\end{cases}
\end{equation}
With this, \eqref{eq: 11.5} -- \eqref{eq: 11.6} is equivalent to finding a $\*z_0\in\mathbb{R}^{2d}$ such that
\begin{equation} \label{eq: 11.7}
\*R(\*z_0) := \*G(\*z(0; \*z_0),\*z(T;\*z_0)) = 0.
\end{equation}
Only the $d$ last components of $\*z_0$ are unknown, as $\*z_0=(y^0, p(0))$ and $y^0\in\mathbb{R}^d$ is fixed.  
We are thus only finding the roots of equation \eqref{eq: 11.7} over $\mathbb{R}^d$.
This is usually done by means of a Newton method, combined with one's favorite numerical integration method for solving \eqref{eq: 11.6.5}.

\subsubsection{The turnpike property as a blueprint} 

Due to the small domain of convergence of the Newton method, it is hard to initialize such a method. Many remedies exist for specific cases. (See \cite{trelat2012optimal}.) 
All things considered, in order to guarantee an inkling of convergence, one needs to provide an adequate initialization of $\*z_0$. 
Note that
\begin{itemize}
\item
The proximity entailed by turnpike cannot be used directly to ensure the convergence of the shooting method described above, if implemented in the usual way. Indeed, this is due to the fact that one knows the solution over $[\varepsilon, T-\varepsilon]$ for some $\varepsilon>0$, but not at the terminal points $t=0$ and $t=T$.
\smallskip 
\item The natural idea is then to modify the usual implementation of the shooting method and to initialize it at some arbitrary point of $[\varepsilon, T-\varepsilon]$, where we know that $\*z(t)$ will be exponentially near the turnpike. For instance, we select $t=\sfrac{T}{2}$. 
\end{itemize}

\noindent
This leads to the following variant suggested by \cite{trelat2015turnpike}, which has been shown to be quite effective by means of several numerical experiments. 
\medskip

\begin{algorithm}[H]
\SetAlgoLined
Unknown is $\*z_0\in \mathbb{R}^{2d}$, designating the value $\*z\left(\frac{T}{2}\right)$\;
Initialization: $\*z_0=(\overline{y}, \overline{p})$, where $(\overline{y},\overline{p})$ is the optimal steady pair\;
Then iterate
\begin{enumerate}
\item Integrate \eqref{eq: 11.5} backwards in time over $\left[0,\sfrac{T}{2}\right]$ to get a value of $\*z(0; \*z_0)$;
\item Integrate \eqref{eq: 11.5} forwards in time over $\left[\sfrac{T}{2},T\right]$ to get a value of $\*z(T; \*z_0)$;
\item The unknown $\*z_0$ is tuned (through a Newton method) so that 
\begin{equation*}
\*R(\*z_0)=\*G(\*z(0, \*z_0),\*z(T, \*z_0))=0.
\end{equation*} 
\end{enumerate}

 \caption{Turnpike-enhanced shooting method.}
\end{algorithm}

\section{Hamilton-Jacobi-Bellman asymptotics} \label{sec: 12}

\subsection{Setting}

The turnpike property for linear, finite-dimensional systems can also be used to derive the asymptotics of solutions to the associated Hamilton-Jacobi-Bellman equations. 
Following the recent work \cite{esteve2020turnpikea}, let us make precise the specific setup and the exact asymptotic behavior.
We consider the finite dimensional system
\begin{equation} \label{eq: AB.HJE}
\begin{cases}
\dot{y}(t) = Ay(t) + Bu(t) &\text{ in }(0,T),\\
y(0) = x
\end{cases}
\end{equation}
where $A\in \mathbb{R}^{d\times d}(\mathbb{R})$, $B\in\mathbb{R}^{d\times m}(\mathbb{R})$ with $d, m\geqslant 1$ (typically, of course, $d>m$). The initial datum is denoted\footnote{\ldots as it will play the role of the spatial variable for the value function $V(T,x)$.} by $x\in\mathbb{R}^d$. 
We consider the following linear quadratic (LQ) optimal control problem
\begin{equation} \label{eq: hj.func}
\inf_{\substack{u \in L^2(0,T; \mathbb{R}^m)\\ y \text{ solves} \eqref{eq: AB.HJE}}} \underbrace{\phi(y(T))+ \frac12 \int_0^T \|y(t)-y_d\|^2\diff t + \frac12 \int_0^T \|u(t)\|^2 \diff t}_{:=\mathscr{J}_{T,x}(u)},
\end{equation}
where $y_d\in \mathbb{R}^d$ is a given target, and $\phi\in\text{Lip}_{\text{loc}}(\mathbb{R}^d; \mathbb{R})$ is a final pay-off which is bounded from below. 
More general scenarios can be considered, such as, for instance, replacing the state tracking term by $\|Cy(t)-y_d\|^2$ for some matrix $C\in \mathbb{R}^{d\times d}(\mathbb{R})$. We focus on the simpler case of $\mathscr{J}_{T,x}$ to avoid technical details and additional assumptions, and we refer to \cite{esteve2020turnpikea} for more details.

\subsection{Hamilton-Jacobi-Bellman equation}

Now, the main goal is to establish a connection between the turnpike property of the solution $(u_T,y_T)$ of \eqref{eq: hj.func}, 
and the asymptotic behavior, as $T\to+\infty$, of the value function $V(T,x)$ associated to \eqref{eq: hj.func}. The latter is defined as
\begin{equation*}
V(T, x) := \inf_{\substack{u\in L^2(0,T;\mathbb{R}^m)\\ y \text{ solves } \eqref{eq: AB.HJE}}} \mathscr{J}_{T,x}(u).
\end{equation*}
We recall that $V(T,x)$ is the unique viscosity solution\footnote{As noted in \cite{esteve2020turnpikea}, considering a final pay-off $\phi$ in the LQ problem consisting of minimizing $\mathscr{J}_{T,x}$ subject to \eqref{eq: AB.HJE} allows to study the associated Hamilton-Jacobi-Bellman equation with a general initial condition (equal to $\phi$). That being said, when $\phi$ is non-convex (even if smooth), the gradient of the solution to the HJB equation would cease to exist in the classical sense for $T\gg1$. Therefore one has to work in the setting of viscosity solutions.} to the Hamilton-Jacobi-Bellman (HJB) equation
\begin{equation} \label{eq: HJB}
\begin{cases}
\partial_T V + \frac12 \|B^*\nabla_x V\|^2 - Ax \cdot \nabla_x V = \frac12 \|x-y_d\|^2 &\text{ in } (0,+\infty)\times\mathbb{R}^d,\\
V_{|_{t=0}} = \phi &\text{ in } \mathbb{R}^d.
\end{cases}
\end{equation}

\subsection{Asymptotics of the value function}

There is a large literature which already deals with the asymptotics of solutions to HJB equations (\cite{barles2000large, fujita2006asymptotic, ishii2006asymptotic, ishii2008asymptotic, barles2019large}). What we present herein, namely the study of \cite{esteve2020turnpike}, is a characterization of the HJB asymptotics through the turnpike property in the finite-dimensional, LQ setting, which is a rather natural idea. 

When studying the long-time behavior of $V(T,x)$, one may be inclined to simply set $T=+\infty$ in \eqref{eq: hj.func} and characterize the resulting problem. This approach fails in general since there is no reason to guarantee that the running cost of $\mathscr{J}_{T,x}$ is integrable in $(0,+\infty)$ for an arbitrary $u\in L^2_{\mathrm{loc}}(0,+\infty;\mathbb{R}^m)$. At this point, one should use the insight that is provided by the turnpike property: when $T\gg1$, the running cost  of $\mathscr{J}_{T,x}$, evaluated along an optimal pair $(u_T, y_T)$ solving \eqref{eq: hj.func}, satisfies
\begin{equation*}
 \frac12 \|y_T(t)-y_d\|^2 + \frac12 \|u_T(t)\|^2 \sim V_s
\end{equation*}
for any $t\in (0,T)$ away from $t=0$ and $t=T$. Here, $V_s$ denotes the steady cost corresponding to $\mathscr{J}_{T,x}$, namely
\begin{equation*}
V_s := \inf_{\substack{(u_s,y_s)\in \mathbb{R}^m \times \mathbb{R}^d \\ Ay_s+Bu_s=0}}  \underbrace{\frac12 \|y_s-y_d\|^2 + \frac12 \|u_s\|^2}_{:=\mathscr{J}_s(u_s)}.
\end{equation*} 
Hence, the lack of integrability when considering the infinite time horizon problem can be handled by subtracting the constant $V_s$ from said running cost. Consequently, in view of this discussion, we consider the corrected, infinite time horizon functional 
\begin{equation*}
\mathscr{J}_{\infty,x}(u) := \int_0^\infty \left\{\frac12 \|y(t)-y_d\|^2 + \frac12 \|u(t)\|^2  - V_s \right\}\diff t,
\end{equation*}
defined over $u\in \mathscr{A}_x$, where 
\begin{align*}
\mathscr{A}_x:= \Bigg\{ u \in L^2_{\text{loc}}(0,+\infty; \mathbb{R}^m) \,\Biggm|\,\, &y \text{ solves } \eqref{eq: AB.HJE}, \\
&\frac12 \|y(\cdot)-y_d\|^2 + \frac12\|u(\cdot)\|^2  - V_s \in L^1(0,+\infty) \Bigg\}.
\end{align*}
The space $\mathscr{A}_x$ can be shown to be non-void for $x\in\mathbb{R}^d$ by assuming that the pair $(A,B)$ is stabilizable (\cite{esteve2020turnpikea}). 
We now denote by $V_\infty(x)$ the value function for $\mathscr{J}_{\infty,x}$, namely
\begin{equation*}
V_\infty(x) := \inf_{u\in \mathscr{A}_x} \mathscr{J}_{\infty, x} (u).
\end{equation*}
Note that $V_\infty$ is independent of the pay-off $\phi$. In fact, it can be shown (\cite{esteve2020turnpikea}) that $V_\infty\in C^1(\mathbb{R}^d)$ is, up to a constant, the unique viscosity solution to the stationary Hamilton-Jacobi equation
\begin{equation*}
V_s + \frac12 \|B^* \nabla_x V_\infty \|^2 - Ax \cdot \nabla_x V_\infty = \frac12 \|x-y_d\|^2 \hspace{1cm} \text{ in } \mathbb{R}^d.
\end{equation*}
The following result can then be shown to hold. 

\begin{theorem}[\cite{esteve2020turnpikea}]  Suppose $\phi\in\mathrm{Lip}_{\mathrm{loc}}(\mathbb{R}^d; \mathbb{R})$ is bounded from below, and let $y_d\in \mathbb{R}^d$. Suppose that the exponential turnpike property holds for \eqref{eq: hj.func} -- \eqref{eq: AB.HJE}.
Then, for any bounded set $\Omega\subset\mathbb{R}^d$, we have
\begin{equation*}
V(T,x)-V_sT \xrightarrow[T\to+\infty]{} V_\infty(x) + \lambda_\infty,
\end{equation*}
uniformly for $x\in\Omega$. Here, $\lambda_\infty\in\mathbb{R}$ is given by
\begin{equation*}
\lambda_\infty=\lim_{T\to+\infty}\Big(V(T,\overline{y}) - V_sT\Big),
\end{equation*}
where $\overline{y}\in \mathbb{R}^d$ is the optimal steady state associated to the global minimizer $\overline{u}\in \mathbb{R}^d$ of $\mathscr{J}_s$. 
\end{theorem}

Looking at the above theorem, one sees that, for any $x\in\Omega \subset\mathbb{R}^d$, the value function $V(T,x)$ has the following asymptotic decomposition 
\begin{equation*}
V(T,x) \sim V_\infty(x) + V_sT + \lambda_\infty
\end{equation*}
as $T\to+\infty$.
The authors in \cite{esteve2020turnpikea} provide the following interpretation of each of the three terms, in terms of what they signify regarding the turnpike property for \eqref{eq: hj.func}.
\begin{itemize}
\item The term $V_\infty(x)$ designates the cost of stabilizing the trajectory $y_T(t)$ from the initial state $x$ at time $t=0$ to the turnpike $\overline{y}$. 
The optimal strategy for the infinite time horizon problem (i.e. \eqref{eq: hj.func} with $T=+\infty$) would rather be to stabilize towards the turnpike $\overline{y}$ and remain in that configuration forever. 
Said stabilizing phase is not seen in the classical definition of the infinite time horizon problem (see e.g. \cite{fleming1995risk}), where the limit of the time averages only captures the transient arc during which the optima are close to the turnpike.
\smallskip
\item The term $V_sT$ corresponds to the running cost accumulated in the intermediate, transient arc, during which the time-evolution optima are close to the steady ones. The constant $V_s$ is commonly referred to as the \emph{ergodic constant}.
\smallskip
\item The constant $\lambda_\infty$ designates the cost of the final arc, namely the one during which the optimal trajectory $y_T$ leaves the turnpike $\overline{y}$ in order to minimize the final pay-off $\phi$. 
Although this final arc does not appear in the infinite time horizon problem, it does appear in the finite time horizon one, no matter how large $T$ is. Thus, it ought to be taken into account when analyzing the behavior of $V(T,x)$ as $T\to+\infty$. 
One can separate this final arc from the remainder of the trajectory by considering the finite time horizon problem with $\overline{y}$ as initial state. In this way, the cost of reaching the turnpike $\overline{y}$ is $0$. One may then subtract this cost during the transient arc $V_sT$. 
\end{itemize}

\section{Deep learning} \label{sec: 13}

As implied in the introduction, and using the theory developed in Section \ref{sec: 10}, the turnpike property may appear and be used as a guideline in the context of \emph{supervised learning} via \emph{residual neural networks}. 
For such neural networks, which can in essence be interpreted as time-discretized ODEs, with a specific scalar nonlinearity acting elementwise, and with controls entering the dynamics in some nonlinear way, found by minimizing some cost functional, any running target will designate the turnpike (namely, the unique optimal steady state solution). 
Such a property will be due to the fact that some of the controls are of a multiplicative nature -- consequently, any constant vector will be a steady state when the controls are null. 
In turn, this will lead to a turnpike property without a final arc near $t=T$, namely, roughly, $\|y_T(t)-\overline{y}\|+\|u_T(t)\|$ would be in $\mathcal{O}(e^{-\lambda t})$.
The latter is an exponential stability estimate, which ensures that the optimal controls $u_T(t)$ are exponentially small in every time $t$ (designating a layer), while the trajectories approximate the targets arbitrarily well in large time. 
Hence, in addition to providing a quantitative estimate for the number of layers needed to interpolate the data, the trained states would oscillate little over layers, which could be beneficial for generalization on unseen data. 

Let us give more details and structure regarding the above discussion.
We refer readers with a machine learning background directly to Section \ref{sec: dl.turnpike}.
	
	\subsection{ResNets and optimal control} 
	
	\textbf{Supervised learning} aims to approximate an unknown function 
	$$f:\mathcal{X}\to\mathcal{Y}$$ from  data 
	$$\left\{x^{(i)}, y^{(i)}=f\left(x^{(i)}\right)\right\}_{i\in[n]}\subset\mathcal{X}\times\mathcal{Y}.$$ Here and henceforth, $[n]:= \{1, \ldots, n\}$.  
	 Typically in practice, $\mathcal{X}\subset\mathbb{R}^d$, whereas either $\mathcal{Y}\subset\mathbb{R}^m$ (namely, we are solving a regression task) or $\mathcal{Y}\subset\mathbb{N}$ with $\#\mathcal{Y}=m$, $m\geqslant2$ (namely solving a classification task).
	Among the many possible classes of functions from which one can construct an approximation of $f$ (e.g., Fourier series, wavelets, and so on), neural networks have proven to be the most promising one for many computational tasks, in particular large scale ones such as image recognition \cite{krizhevsky2012imagenet}. 
	In particular, the large \emph{depth} (number of \emph{layers}) of the networks used in these experiments has been observed to play a key role in this computational supremacy. 
	\medskip
	
	\noindent
	\textbf{Residual neural networks.}
	A recent and very popular neural network architecture are the so-called \emph{residual neural networks} (ResNets, \cite{he2016deep}). 
	They may, in the simplest case (see Remark \ref{rem: settings} for extensions), be cast as discrete-time dynamical systems of the mould
	\begin{equation} \label{eq: 1.1}
	\begin{cases}
	\*x_i^{k+1} = \*x_i^k + \sigma\left(w^k \*x_i^k + b^k\right) &\text{ for }k \in \{0, \ldots, n_t-1\}\\
	\*x_i^0 = x^{(i)}
	\end{cases}
	\end{equation}
	for all $i \in [n]$. 
	The unknowns in system \eqref{eq: 1.1} are the states $\*x_i^k\in \mathbb{R}^d$ for any $i\in[n]$, while $\left\{w^k, b^k\right\}_{k=0}^{n_t-1}$ are the controls (referred to as \emph{parameters}), with $w^k \in \mathbb{R}^{d\times d}$, and $b^k \in \mathbb{R}^d$. 
	Furthermore, $\sigma\in\mathrm{Lip}(\mathbb{R})$ is a prescribed scalar, nonlinear function, defined component-wise in \eqref{eq: 1.1} (the commonly-used example being $x\mapsto \max\{x,0\}$, but also $x\mapsto\tanh(x)$), while $n_t\geqslant 1$ designates the number of layers of the network, referred to as the depth, as mentioned in the preceding paragraph. 
	
	The controls\footnote{Thus $u^k\in\mathbb{R}^{d^2+d}$ is a column vector containing the components of $w^k$ and $b^k$. For instance, and without loss of generality, this is done by vectorizing the matrix $w^k$ followed by appending the column vector $b^k$.} $u^k = \left(w^k, b^k\right)$ for all $k\geqslant0$ are found by solving the regularized empirical risk minimization problem\footnote{This is by no means a general formulation. One may consider other ways to regularize the controls and promote different patterns, such as  sparsity via $\ell^1$--regularization, and so on.}
	\begin{equation} \label{eq: ERM}
	\inf_{\left\{u^k\right\}_{k=0}^{n_t-1}=\left\{w^k, b^k\right\}_{k=0}^{n_t-1}} \underbrace{\frac{1}{n} \sum_{i=1}^n \text{loss}\left(P\*x_i^{n_t}, y^{(i)}\right)}_{\text{empirical risk } :=\*E(\*x^{n_t})} + \alpha \sum_{k=0}^{n_t-1}\left\|u^k\right\|^2,
	\end{equation}
	where $\alpha\geqslant0$ is fixed, while 
	\begin{equation*}
	\text{loss}(\cdot,\cdot)\in C^0(\mathbb{R}^m \times \mathcal{Y}; \mathbb{R})
	\end{equation*}
	is a given function which is bounded from below (say, for simplicity, by $0$) and which differs depending on the task in hand -- for instance, 
	\begin{equation*}
	\text{loss}(x,y) := \|x-y\|^p_{p}
	\end{equation*}
	for $p\in\{1,2\}$ is commonly used for regression tasks (here and henceforth the norm designates the entry-wise matrix norm), while the cross-entropy loss\footnote{More accurately, the cross-entropy loss with softmax activation, sometimes also referred to as cross-entropy with "logits" (with the "logits" designating the pre-softmax vectors, namely $x$ in this case).}
	\begin{equation*}
	\text{loss}(x,y) = -\log\left(\frac{e^{x_{y}}}{\sum_{j=1}^m e^{x_j}}\right) \hspace{1cm} (x,y)\in\mathbb{R}^m\times[m]
	\end{equation*}
	is commonly used for classification tasks. 
	Finally, $P:\mathbb{R}^d\to\mathbb{R}^m$ is an affine map (the \emph{output layer}), whose coefficients in practice are part of the optimizable parameters. More precisely,
	\begin{equation*}
	\mathbb{R}^d\ni x\mapsto Px := w^{n_t}x+b^{n_t}\in\mathbb{R}^m.
	\end{equation*}
	Herein, we shall assume that $P$ is given and fixed.
	
	\begin{remark}[Training]
	In practice, the minimization problem is solved by variants of stochastic gradient descent -- a procedure colloquially named as \emph{training}. We emphasize that here, we shall focus on deriving properties of the global minimizers of the regularized empirical risk rather than convergence of training algorithms. Our presentation is -- in principle -- algorithm independent.
	\end{remark}
	\medskip
	
	\noindent
	\textbf{Neural ODEs.}
	One may readily observe (as done by \cite{weinan2017proposal, haber2017stable}) that for any $i\in[n]$, and for $T>0$, \eqref{eq: 1.1} is a forward Euler scheme for
	\begin{equation} \label{eq: 1.2}
	\begin{cases}
	\dot{\*x}_i(t) = \sigma(w(t)\*x_i(t)+b(t)) & \text{ for }t \in (0, T) \\
	\*x_i(0) = x^{(i)}.
	\end{cases}
	\end{equation} 	
	Continuous-time ResNets such as \eqref{eq: 1.2} are commonly referred to as \emph{neural ordinary differential equations} (neural ODEs) in the computing literature\footnote{Albeit in the paper \cite{chen2018neural} where this denomination was originally introduced, the controls are time-independent.}. 
	The regularized empirical risk minimization problem reduces to the optimal control problem with final cost (defined in \eqref{eq: ERM})
	\begin{equation} \label{eq: non.regular.standard.sup.learn}
	\inf_{\substack{u:=[w,b]\in L^2(0,T; \mathbb{R}^{d_u})\\ \*x_i(\cdot) \text{ solves } \eqref{eq: 1.2}}} \*E(\*x(T))+ \alpha \int_0^T \left\|u(t)\right\|^2 \diff t.
	\end{equation} 
	Here, we used the notation $\*x:= \left[\*x_1^\top,\ldots,\*x_n^\top\right]^\top\in \mathbb{R}^{d_x}$, with $d_x:=d\cdot n$.
	We note that, written as such, training a ResNet is an optimal control problem for a nonlinear, discrete (or continuous)-time dynamical system. 
	There is, however, an important point to be made in addition. 
	In statistical learning, the concept of \emph{generalization}, namely ensuring reliable performance of the trained/controlled network on unseen data (namely, new points outside of $\{x^{(i)}, y^{(i)}\}_{i\in[n]}$), is of paramount importance. 
	And so, one sees that, contrary to the optimal control problems we considered in preceding sections, wherein we generally fix a single initial datum and find an optimal control (which may or may not be in feedback form), in \eqref{eq: 1.2} -- \eqref{eq: non.regular.standard.sup.learn} we find a single pair of time-dependent controls for $n$ different initial and target data\footnote{In control-theoretical terms, this is more in the spirit of notions such as \emph{simultaneous controllability} or \emph{ensemble controllability}, studied in various different contexts, see, for instance, \cite{lions1988exact, tucsnak2000simultaneous} for control systems with different control operators, or \cite{beauchard2010controllability, loheac2016averaged, agrachev2016ensemble} for control systems with parameter-dependent dynamics.}.
	
	\begin{remark}[Settings] \label{rem: settings}
	 \label{rem: remark.sobolev}
	\begin{itemize}
	\item One may consider variations of the nonlinear dynamics in \eqref{eq: 1.1} or \eqref{eq: 1.2}, as is typically done for canonical feed-forward neural networks. Among many possibilities, some simple examples could include
	\begin{equation} \label{eq: 5.7}
	\dot{\*x}_i(t) = w(t)\sigma(\*x_i(t)) + b(t) \hspace{1cm} \text{ for } t\in(0,T),
	\end{equation}
	and
	\begin{equation} \label{eq: 5.8}
	\dot{\*x}_i(t) = w_1(t) \sigma(w_2(t) \*x_i(t) + b_2(t)) + b_1(t) \hspace{1cm} \text{ for } t\in (0,T),
	\end{equation}
	for $i\in[n]$. 
	In \eqref{eq: 5.8}, one could also envisage having $w_2(t)\in \mathbb{R}^{d_{\text{hid}}\times d}$ and $w_1(t)\in \mathbb{R}^{d\times d_{\text{hid}}}$ (accordingly, $b_2\in \mathbb{R}^{d_{\text{hid}}}$ and $b_1\in \mathbb{R}^{d}$), where $d_{\text{hid}}\neq d$. 
	\smallskip
	
	\item We stress that considering solely an $L^2(0,T; \mathbb{R}^{d_u})$--regularization of the controls $(w,b)$ in \eqref{eq: non.regular.standard.sup.learn} may not be enough for guaranteeing the existence of minimizers when considering underlying networks with dynamics such as \eqref{eq: 1.2} or \eqref{eq: 5.8} (although it does suffice for \eqref{eq: 5.7}). 
	This is due to the way in which the nonlinearity $\sigma$ is applied upon the controls. 
	To our knowledge, it is not clear how one can ensure compactness of minimizing sequences in $L^1(0,T; \mathbb{R}^{d_u})$, needed for passing to the limit within the continuous-time neural network.
	In such cases, one should rather replace the $L^2(0,T; \mathbb{R}^{d_u})$ norm by either an $H^1(0,T; \mathbb{R}^{d_u})$ or $\text{BV}(0,T; \mathbb{R}^{d_u})$ norm regularization, both of which would ensure the desired compactness. 
	This issue is specific to the continuous-time setting, as in the discrete-time, finite dimensional setting, weak and strong convergences coincide.
	\smallskip
	\item First of all, we ought to suppose that $x^{(i)}\neq x^{(j)}$ for $i\neq j$. Now, due to the uniqueness of Lipschitz-nonlinear ODEs (in both directions of time), trajectories corresponding to different initial data cannot cross\footnote{In the discrete-time setting, this is not an issue on coarse grids (namely when the time-step $\bigtriangleup t$ is large), as is the case for the canonical ResNet, where it equals $1$.}. Hence, in the context of binary classification tasks for instance (namely, where $f$ is the characteristic function of some set), if the original dataset is not linearly separable, one cannot separate the dataset by a controlled neural ODE flow in a way that the underlying topology of the data (namely, the unknown function $f$) is captured and generalized. 
	As presented in \cite{dupont2019augmented}, a simple remedy\footnote{Another remedy consists in considering \emph{momentum} ResNets (\cite{sander2021momentum}), which consist in adding $\ddot{\*x}(t)$ to the ODE.} for this issue in this case is to embed $x^{(i)}\in \mathbb{R}^d$, for any $i\in[n]$, in a higher dimensional space, for instance in $\mathbb{R}^{d+1}$, by setting 
	\begin{equation} \label{eq: xi0}
	\*x_i^0=\begin{bmatrix}x^{(i)}\\0\end{bmatrix}
	\end{equation}
	for $i\in[n]$. 
	Unless stated otherwise, we shall henceforth consider initial data in such form, and use $d$ to denote the dimension of the augmented system.
	\end{itemize}
	\end{remark}
	
	\noindent
	The neural ODE formalism of deep learning has been used to great effect in several machine learning contexts. 
	To name a few, these include the use of adaptive ODE solvers (\cite{chen2018neural, dupont2019augmented, kidger2020neural}) and symplectic schemes (\cite{celledoni2020structure}) for efficient training, the use of indirect training algorithms based on the Pontryagin Maximum Principle (\cite{li2017maximum, benning2019deep}), image super-resolution (\cite{he2019ode}), as well as unsupervised learning and generative modeling (\cite{grathwohl2018ffjord, papamakarios2019normalizing}). 
		The origins of continuous-time supervised learning date back at least to \cite{lecun1988theoretical}, in which the back-propagation method is connected to the adjoint method. See also \cite{sontag1997complete, sontag1999further} and the references therein for earlier works in this direction.
	
	\subsection{The role of $T$} 
	In the ResNet \eqref{eq: 1.1}, the time-step $\bigtriangleup t=\sfrac{T}{n_t}$ is fixed (equal to $1$ in fact), and each time instance of a forward Euler discretization of \eqref{eq: 1.2} would represent a different layer of \eqref{eq: 1.1}. Hence, whenever the time-step $\bigtriangleup t$ is fixed, the time horizon $T$ in \eqref{eq: 1.2} serves as an indicator of the number of layers $n_t=\sfrac{T}{{\bigtriangleup}t}$ in the ResNet \eqref{eq: 1.1}. 
	
	In view of this, and the empirical success of residual neural networks with large depths, we are interested in studying the behavior of global minimizers $u(\cdot)$ (and corresponding states $\{\*x_i(\cdot)\}_{i=1}^n$) solutions to
	\begin{equation} \label{eq: 5.9}
	\inf_{\substack{u\in\mathscr{U}(0,T; \mathbb{R}^{d_u})\\ \*x_i(\cdot) \text{ solves } \eqref{eq: 5.10}}} \*E(\*x(T))+ \alpha \left\|u\right\|_{\mathscr{U}(0,T; \mathbb{R}^{d_u})}^2,
	\end{equation} 
	as $T\to+\infty$, where the constraint satisfied by the states $\{\*x_i(\cdot)\}_{i=1}^n$ is given by the nonlinear ODE
	\begin{equation} \label{eq: 5.10}
	\begin{cases}
	\dot{\*x}_i(t) = \mathfrak{f}(u(t), \*x_i(t)) &\text{ in }(0,T),\\
	\*x_i(0) = \*x_i^0,
	\end{cases}
	\end{equation}
	with $\mathfrak{f}$ as in \eqref{eq: 1.2} or \eqref{eq: 5.7} or \eqref{eq: 5.8}, with $\*x_i^0$ as in \eqref{eq: xi0}.
	(We henceforth drop the subscripts $T$ used in preceding discussions as to not overburden the notation.) 
	In \eqref{eq: 5.9}, the dimension $d_u$ of the tensor-valued function $u(t)$, and its definition, encompass the different dynamics for which there might be more than a pair of matrix-vector controls (i.e. \eqref{eq: 5.8}). 
	Here $\mathscr{U}$ is either $L^2$ or $H^1$.
	
	As discussed in the introduction, when one considers an optimal control problem with solely a final cost as in \eqref{eq: 5.9}, one cannot expect the appearance of the turnpike property. 
	But rather, whenever $\sigma$ is $1$--homogeneous, $\mathscr{U}$ is $L^2$ or $H^1$, and one omits the case of $\mathfrak{f}$ as in \eqref{eq: 5.8}, it can be shown that $\*E(\*x(T))$ decays to $0$ at most like $\mathcal{O}\left(\frac{1}{T}\right)$ (see \cite{esteve2020large}). 
	When $\*E(\cdot)$ attains its minimum, then the optimal controls $u$ also converge, when appropriately rescaled, to some solution $u^\star$ of 
	\begin{equation} \label{eq: 5.11}
	\inf_{\substack{u\in\mathscr{U}(0,1; \mathbb{R}^{d_u}) \\ \*x_i(\cdot) \text{ solves } \eqref{eq: 5.10} \text{in} (0,1)\\ \*E(\*x(1)) = 0}} \|u\|^2_{\mathscr{U}(0,1; \mathbb{R}^{d_u})}, 
	\end{equation} 
	along some subsequence as $T\to+\infty$ (\cite{esteve2020large}).
	Problem \eqref{eq: 5.11}, on the other hand, is seen as the minimal norm control ensuring controllability, similarly to what was discussed in the introductory sections for linear PDEs. Namely, it provides controls of least oscillations among those who interpolate the data, in the sense that  $\*E(\*x(1)) = 0$, which, when $\text{loss}$ is an $\ell^p$ distance, rewrites as $P\*x_i(1)=y^{(i)}$ for all $i\in[n]$. 
	
	This convergence result mainly makes use of the fact that, under the aforementioned homogeneity assumptions, one has $\mathfrak{f}(\zeta u, \*x)=\zeta \mathfrak{f}(u, \*x)$ for $\zeta>0$. 
	Hence, given $(u_T, \*x_T)$ solving $\dot{\*x}_T(t) = \mathfrak{f}(u_T(t),\*x_T(t))$ for $t\in(0,T)$, it follows that the rescaled map $\*x_1(t):=\*x_T(\frac{t}{T})$ solves $\dot{\*x}_1(t) = \mathfrak{f}(u_1(t), \*x_1(t))$ for $t\in(0,1)$, with $u_1(t) := \frac{1}{T} u_T(\frac{t}{T})$. 
	This rescaling relation yields the scaling of the norm of the controls
	\begin{equation} \label{eq: scaling.norm}
	\alpha \int_0^T \|u_T(t)\|^2 \diff t = \frac{\alpha}{T} \int_0^1 \|u_1(t)\|^2 \diff t.
	\end{equation}	
	In addition to the polynomial convergence of $\*E(\*x(T))$, one sees that \eqref{eq: scaling.norm} also indicates an equivalence of the limits $T\to+\infty$ (with $\alpha>0$ fixed) and $\alpha\searrow0$ (with $T$ fixed)\footnote{As a matter of fact, when $T$ is fixed, one can show that the desired result holds for general dynamics $\mathfrak{f}$, without any homogeneity assumptions -- the proof follows the same arguments as those presented in \cite{esteve2020large}. Homogeneity is needed solely to ensure the scaling in time, which is where one sees the appearance of a factor of $\sfrac{1}{T}$, and thus the pattern when $T\to+\infty$.}.  The latter is known as the \emph{regularization path} limit, and is well-studied in the statistical learning literature, mainly for linear problems (see \cite{rosset2004boosting}, for instance). 

	While such a result could or might be desirable in practice -- due to the possible generalization properties of the limiting controls--, it does not provide any specific stability estimates for optimal controls or the trajectories over each time instant $t$. This in turn would be desirable for choosing a sharper number of layers needed to interpolate the dataset, whilst still retaining controls of moderate amplitude.
	
	This lack of stability may also be seen numerically. 
	In Figures \ref{fig: no.turnpike.1} -- \ref{fig: no.turnpike.2}, we solve a toy binary classification task, by solving \eqref{eq: 5.9} with cross-entropy loss, $\mathfrak{f}$ as in \eqref{eq: 5.7} and $\sigma\equiv\text{tanh}$, making use of an explicit midpoint scheme with $T=4$, $n_t=16$ and thus $\bigtriangleup t=0.25$ (hence, a relative error of $6.25\%$, due to the quadratic convergence of the midpoint scheme).

	\begin{figure}[h!]
	\center
	\includegraphics[scale=0.5]{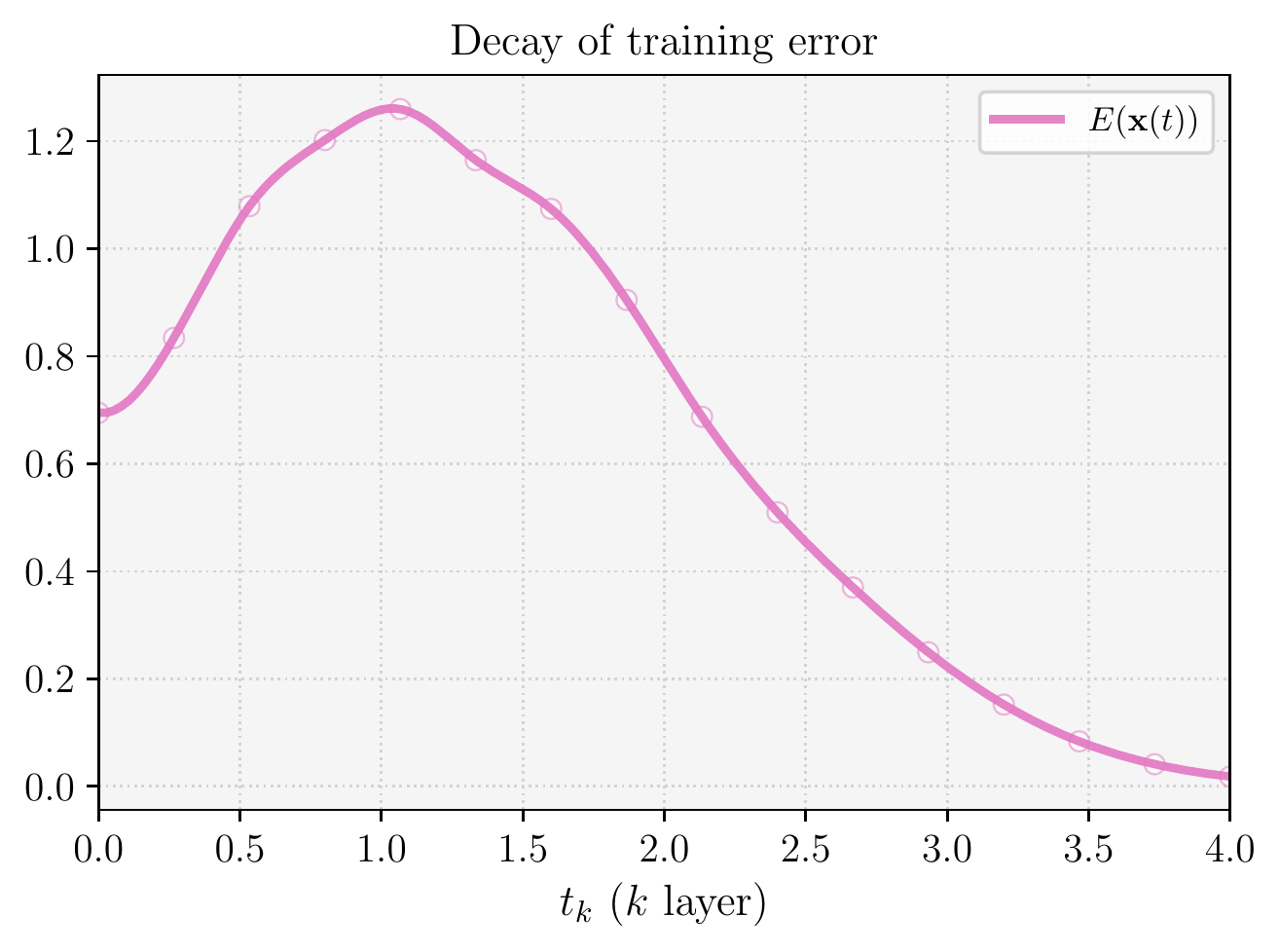}
	\includegraphics[scale=0.5]{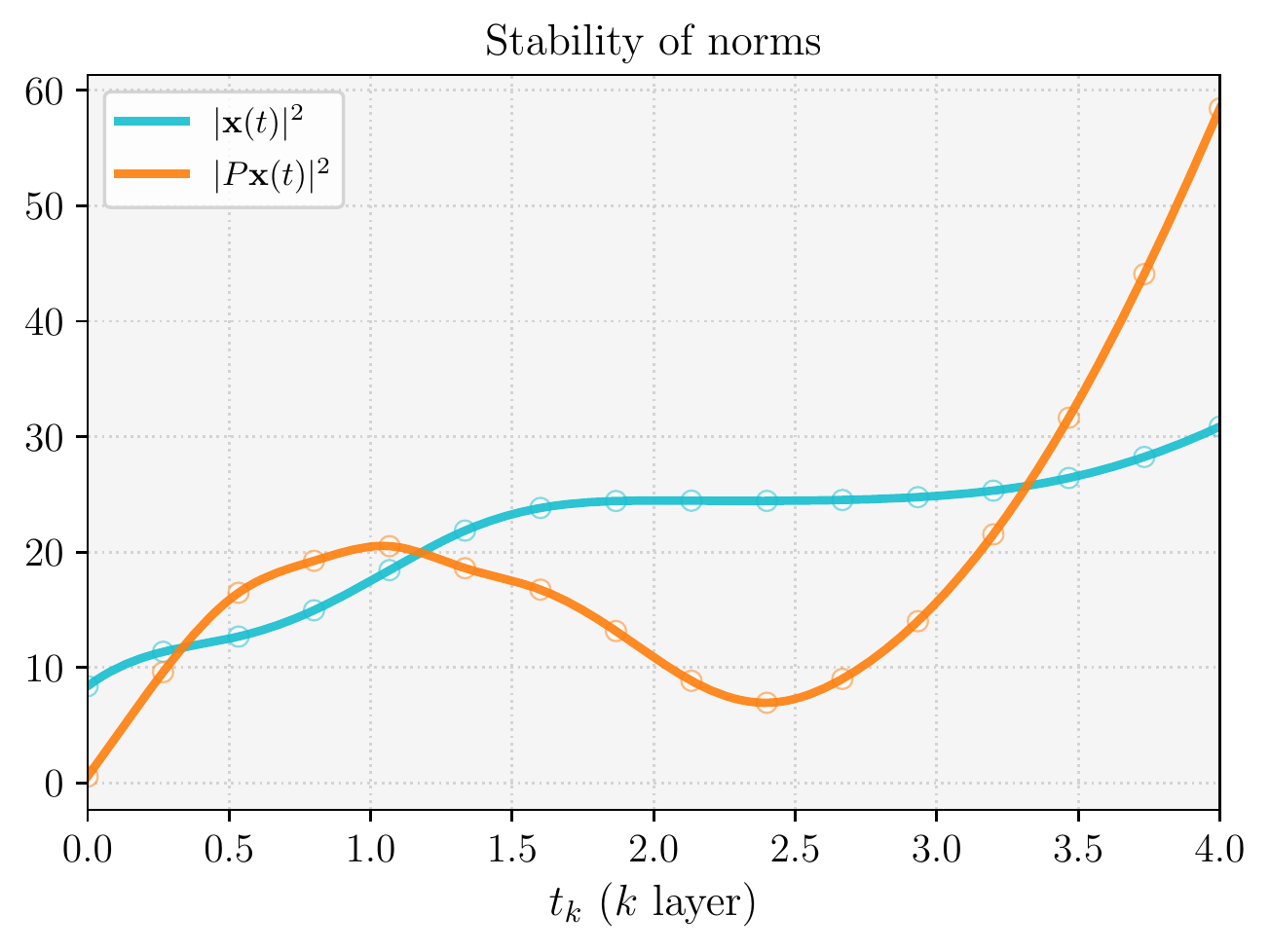}
	\caption{The training error $\*E(\*x(t))$ is only minimal at time $t=T=4$ (\emph{left}), and we do not see a turnpike-like stability for the trajectories (\emph{right}).}
	 \label{fig: no.turnpike.1}
	\end{figure}
	
	\begin{figure}[h!]
	\center
	\includegraphics[scale=0.6]{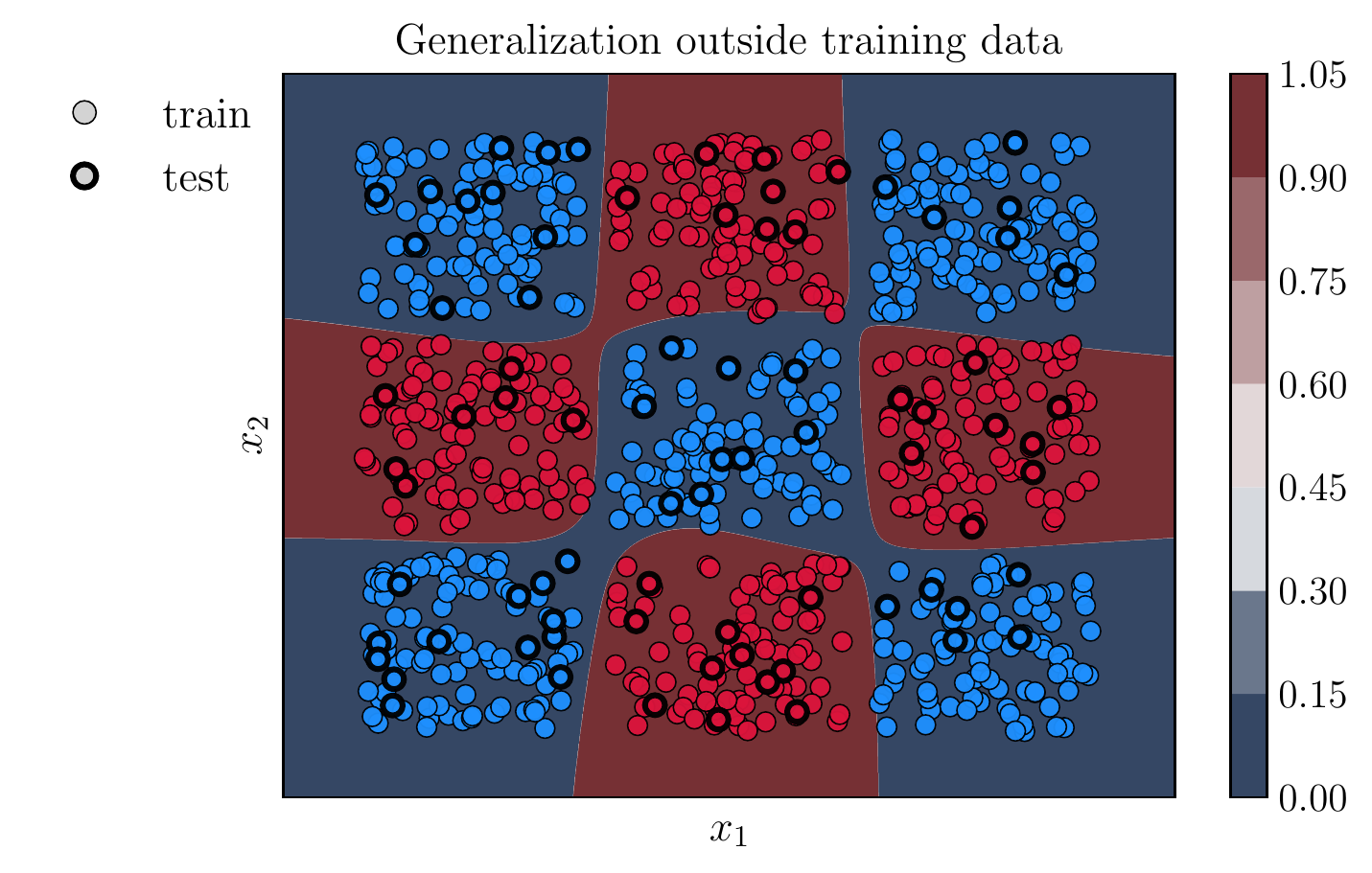}
	\caption{We plot the trained classifier $P\*x_x(T)$ for any initial datum $x\in[-1.1,1.1]^2$, along with the training data, as well as test data.
	The shape of the data is captured accurately, and thus the unknown function $f$ is approximated well, ensuring generalization, but only at $t=T$, as per Figure \ref{fig: no.turnpike.1}.}
	\label{fig3.ex1}
	\end{figure}
	
	\begin{figure}[h!] 
	\center
	\includegraphics[scale=0.6]{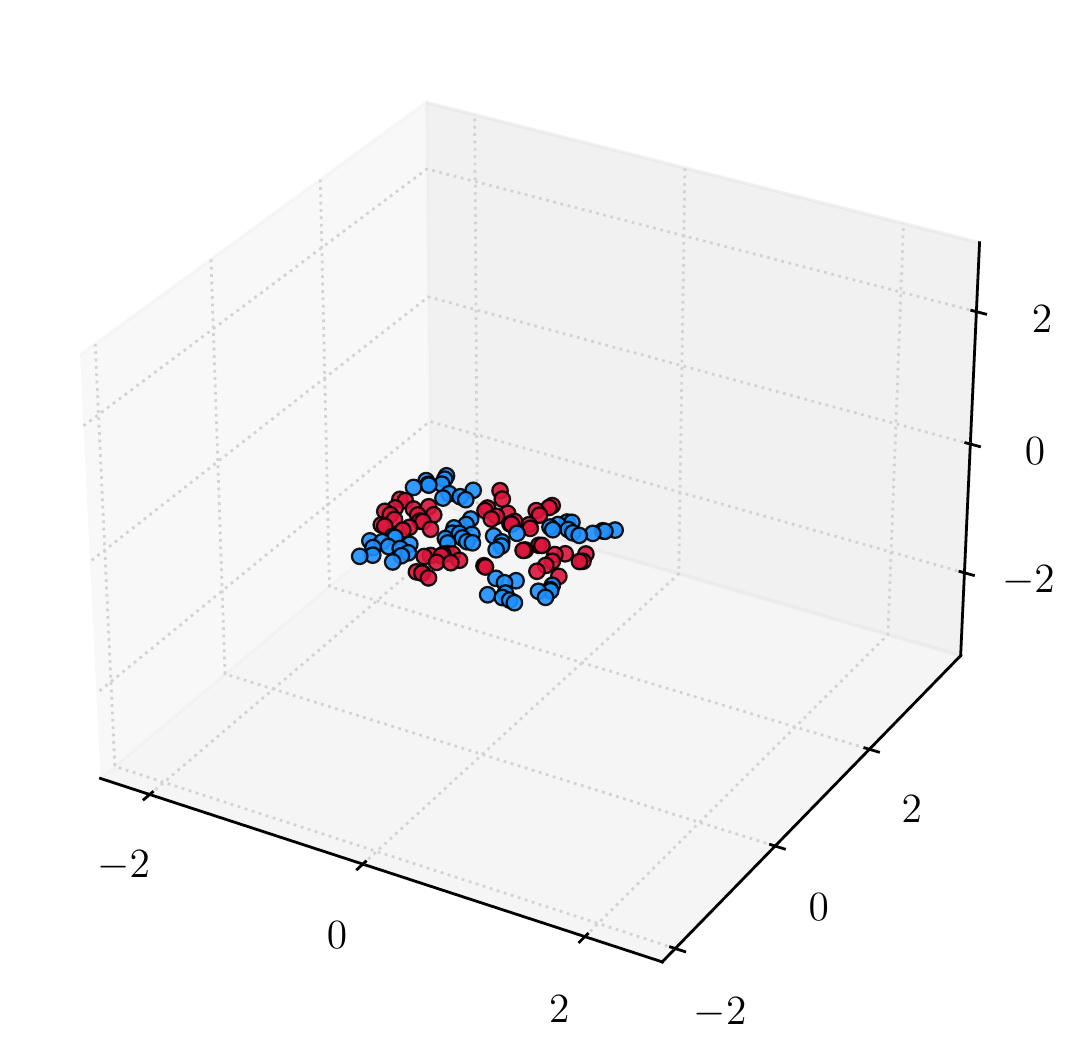}
	\includegraphics[scale=0.6]{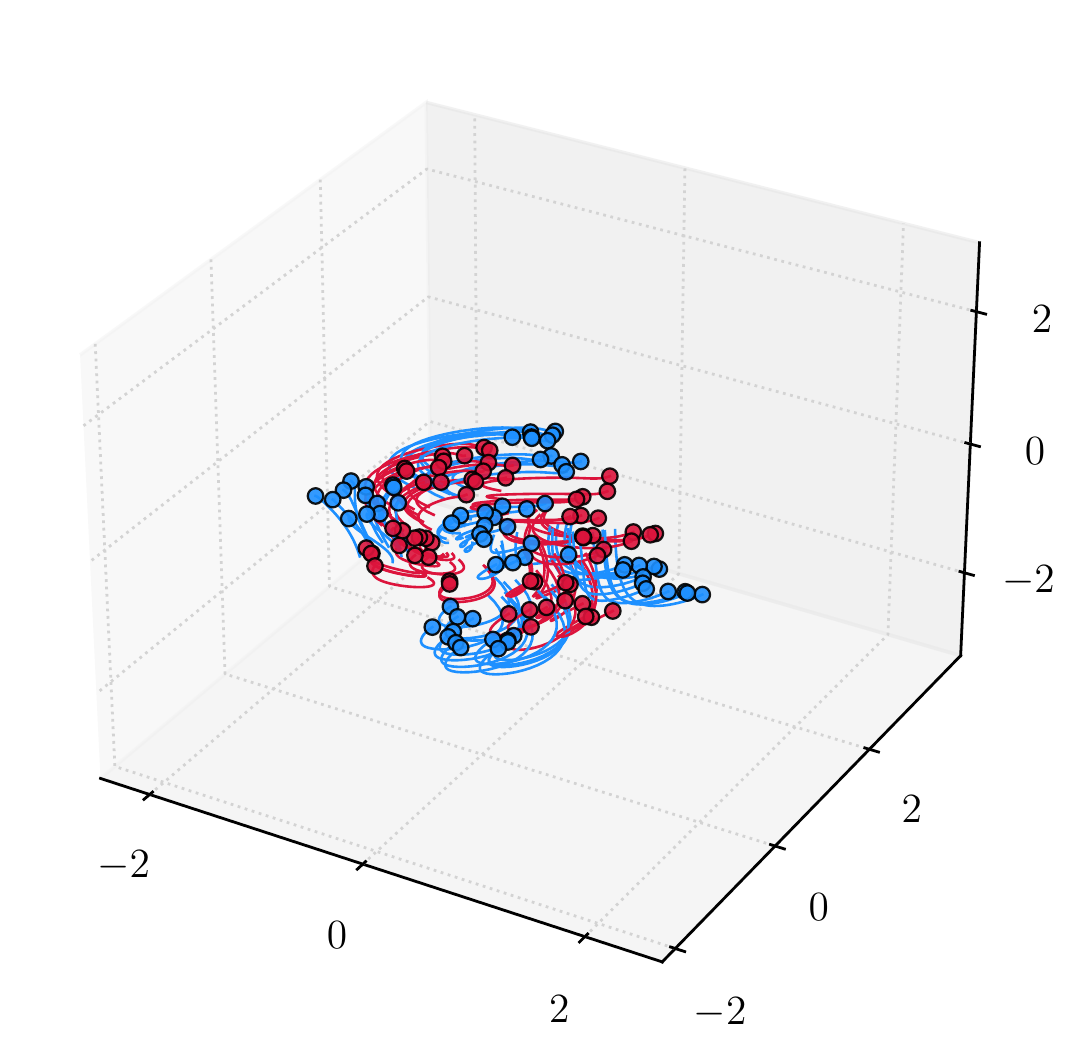}
	\includegraphics[scale=0.6]{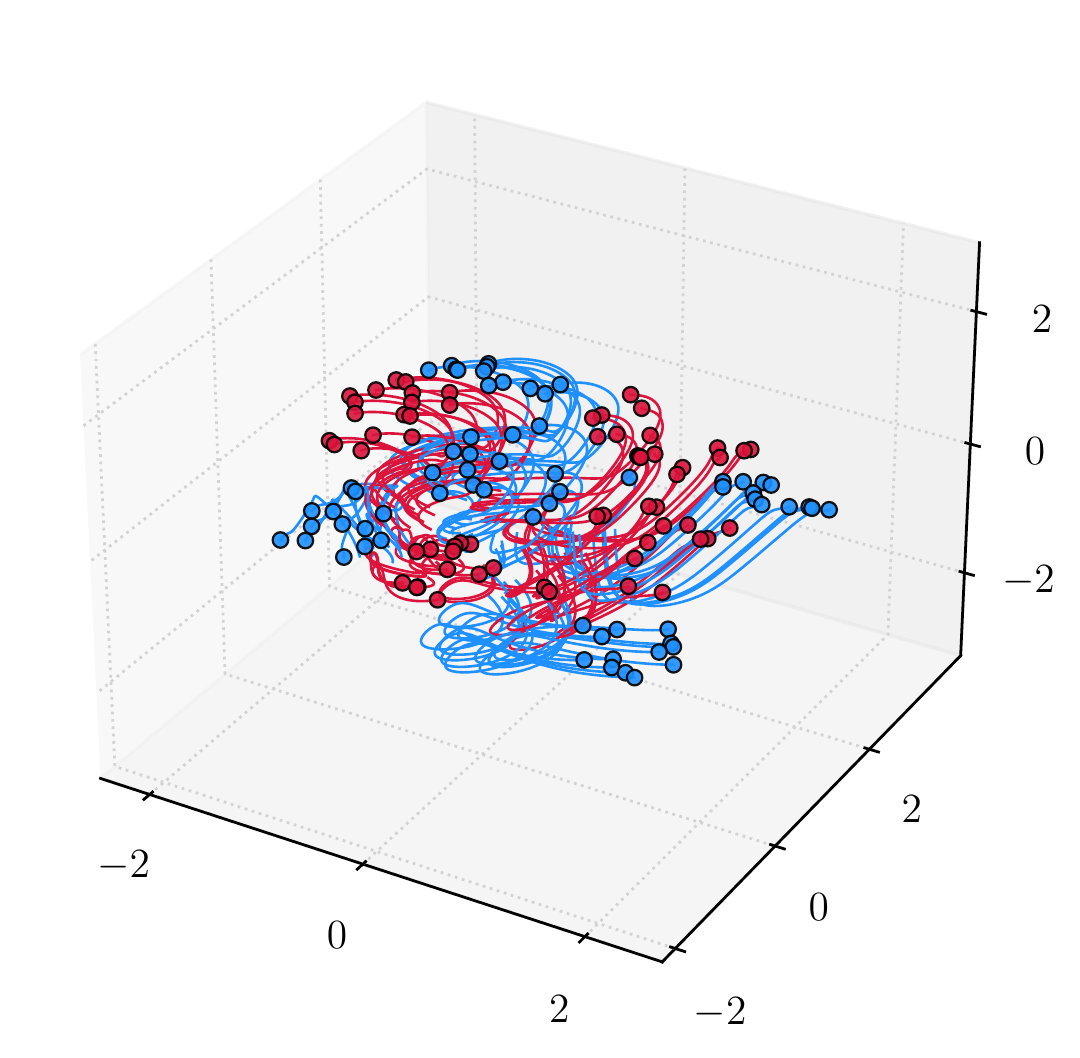}
	\includegraphics[scale=0.6]{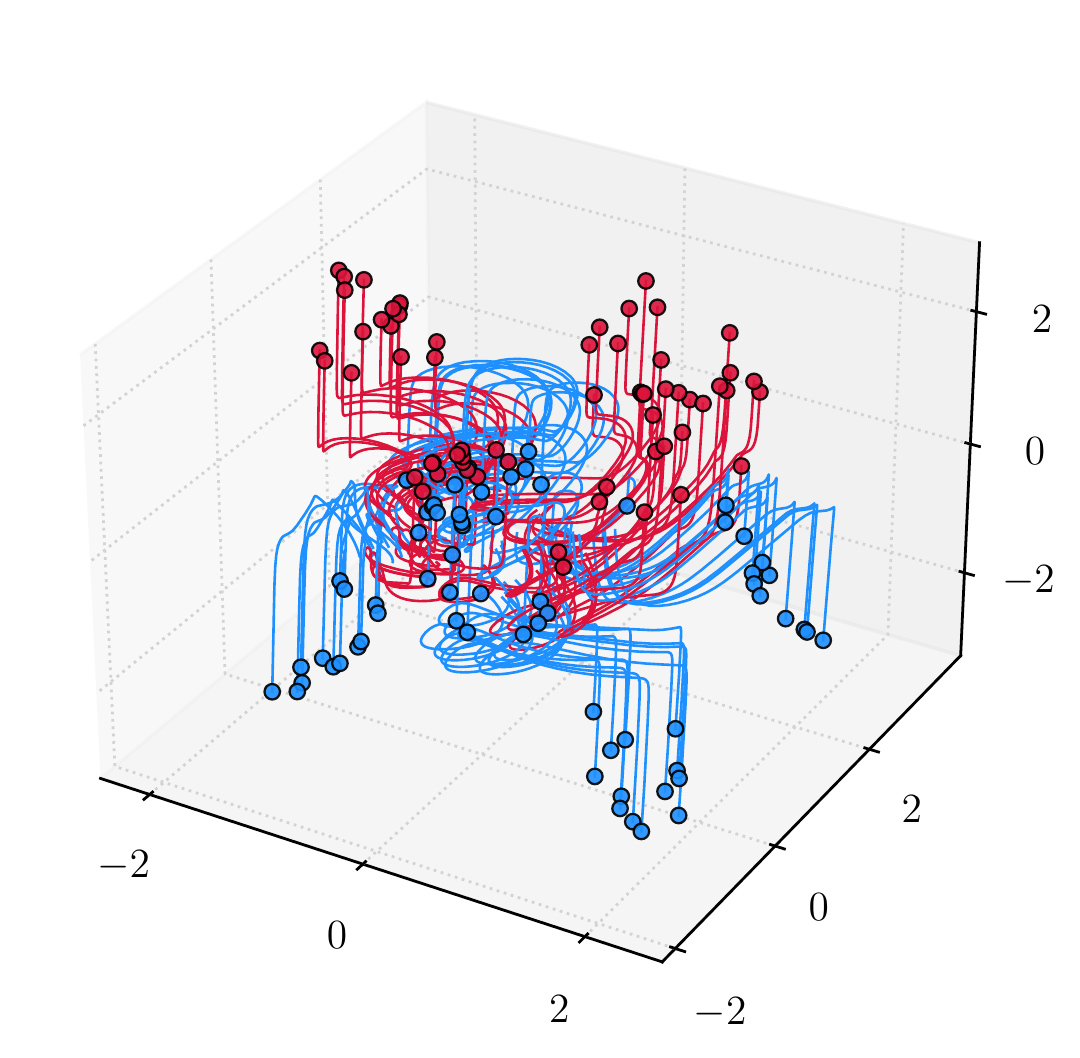}
	\caption{The evolution of the trajectories $\{\*x_i(t)\}_{i\in[n]}$, solutions to \eqref{eq: 5.7} with data $\{\*x_i^0\}_{i\in[n]}$ (\emph{top left}), shown for $t\leqslant1$ (\emph{top right}), $t\leqslant 2$ (\emph{bottom left}), $t\leqslant T=4$ (\emph{bottom right}). Separation of the dataset is only done towards $t=T$, as per Figure \ref{fig: no.turnpike.1}.}
	\label{fig: no.turnpike.2}
	\end{figure}
	
	\begin{remark}[Solutions to \eqref{eq: 5.11}] It should be noted that the study of existence of solutions to \eqref{eq: 5.11} (say, when $\text{loss}(\cdot,\cdot)$ is an $\ell^p$ distance), namely, the \emph{controllability} of $n\gg1$ trajectories $\{\*x_i(\cdot)\}_{i\in[n]}$ of \eqref{eq: 5.10} corresponding to different initial data $\*x_i^0$, by means of the same controls $u(\cdot)$ (living in a space which is of a higher dimension than that of each individual $\*x_i$), is an independent topic which has attracted considerable interest in recent years. 
	This particular controllability property has been referred to as \emph{ensemble controllability}, \emph{universal interpolation}, or \emph{simultaneous controllability}, among other denominations. 
	We refer to several recent works for various different techniques -- for instance, \cite{agrachev2020control} and \cite{cuchiero2019deep}, where the authors make use of geometric, Lie bracket techniques for dynamics such as \eqref{eq: 5.7} and \eqref{eq: 1.2}, for which such tools are quite natural (see \cite{coron2007control} for a detailed presentation on Lie algebra techniques for nonlinear control). For more compound neural ODE dynamics, such as \eqref{eq: 5.8}, we refer to \cite{ruiz2020universal} and \cite{li2019deep}, where the controls are built explicitly in a constructive way. In particular, in these works (see also \cite{ruiz2021interpolation}), the link with the closely related topic of universal approximation (see the seminal works \cite{cybenko1989approximation, pinkus1999approximation}, and also the recent survey \cite{devore2021neural}) is clearly established.
	The analog property for discrete-time neural networks has been thoroughly investigated in the statistical learning literature, under the name of \emph{finite sample expressivity} (albeit mainly for simplified models such as linear regression or shallow neural networks, see \cite{zhang2016understanding} for further details).
	The study of which optimal controls/parameters are the "best" (in the sense of generalization) in this interpolating regime remains an open problem. 
	\end{remark}

	\subsection{Turnpike and turnpike-like properties} \label{sec: dl.turnpike}
	For the turnpike property to hold, one would require more coercivity of the cost functional to be minimized with respect to the state over all time.  
	\medskip
	
	\noindent
	\textbf{Augmented functional.} 
	Looking at \eqref{eq: 5.9} (and \eqref{eq: ERM}), let us thus consider its natural extension, which is an augmented empirical risk minimization problem of the form
	\begin{equation} \label{eq: 5.12}
	\inf_{\substack{u\in \mathscr{U}(0,T; \mathbb{R}^{d_u})\\ \*x_i(\cdot) \text{ solves } \eqref{eq: 5.10}}} \int_0^T \*E(\*x(t))\diff t+ \alpha \int_0^T \|u(t)\|^2\diff t.
	\end{equation}
	When considered for the discrete-time ResNet case, by using Riemann sums, \eqref{eq: 5.12} would reformulate as 
	\begin{equation} \label{eq: 5.13}
	\inf_{\substack{\left\{u^k\right\}_{k=0}^{n_t-1} \subset \mathbb{R}^{d_u}\\ \*x_i^{k+1}=\*x_i^k+\mathfrak{f}(u^k, \*x_i^k)}}  \bigtriangleup t\left(\sum_{i=1}^n \sum_{k=1}^{n_t-1} \*E(\*x^k) + \alpha \sum_{k=0}^{n_t-1} \left\|u^k\right\|^2\right).
	\end{equation}
	One sees in \eqref{eq: 5.13} that the artificial tracking term introduces an additional regularization of the states over every layer $k\in\{1,\ldots,n_t-1\}$. 
	\medskip
	
	\noindent
	\textbf{The steady problem.}
	A turnpike property for \eqref{eq: 5.12} would entail proximity of solutions to \eqref{eq: 5.12} to the corresponding stationary optimal control problem. The latter thus needs to be properly characterized. Let us assume for now that $\*E(\cdot)$ attains its minimum (say, $0$, for simplicity). Then the stationary problem corresponding to \eqref{eq: 5.12} would read as
	\begin{equation} \label{eq: 5.15}
	\inf_{\substack{(u,\*x)\in \mathbb{R}^{d_u}\times\mathbb{R}^{dn}\\ \mathfrak{f}(u, \*x_i) =0,  \quad i\in[n]}}  \*E(\*x)+ \alpha\|u\|^2.
	\end{equation}
	But due to the specific form of $\mathfrak{f}$ and $u$, which are as in \eqref{eq: 1.2}, \eqref{eq: 5.7} or \eqref{eq: 5.8}, one sees that the unique optimal solution to \eqref{eq: 5.15} is $\overline{u}\equiv0$, with $\overline{\*x}\in \text{argmin }\*E$. 
	Hence, the turnpike is a couple $(\overline{u}, \overline{\*x})$ at which the running cost $(u, \*x)\mapsto \*E(\*x)+\|u\|^2$ is minimized, with $\*x=\{\*x_i\}_{i\in[n]}$ also being a steady state of the underlying ODE. 
	\medskip
	
	\noindent
	\textbf{Exponential turnpike/decay/stability.}
	As the state turnpike is a steady state of the neural ODE, and there is no final cost, one should expect the final arc near $t=T$ of the exponential turnpike estimate to vanish -- this is indeed seen in numerical experiments presented in Figures \ref{fig1.ex2} -- \ref{fig2.ex2}. 
	In other words, the turnpike property for the supervised learning problem \eqref{eq: 5.12} would be characterized by a decay/stability estimate of the form
	\begin{equation} \label{eq: turnpike.dl}
	 \*E(\*x(t)) + \|u(t)\|^2 \leqslant C e^{-\lambda t},
	\end{equation}
	for all $t\in[0,T]$ and for some constants $C>0$ and $\lambda>0$ independent of $T$. Furthermore, when $\text{argmin }\*E\neq\varnothing$, we would also have
	\begin{equation*}
	\inf_{\{\*z_i\}_{i\in[n]} \in \text{argmin } \*E } \sum_{i=1}^n \|\*x_i(t)-\*z_i\|^2\leqslant Ce^{-\lambda t}.
	\end{equation*}
	Such results are indeed shown, in specific settings ($\ell^2$ losses), in \cite{esteve2020large} -- this is the theory presented in Section \ref{sec: 10}. We also refer the reader to \cite{yague2021sparse} for results with general losses and $L^1(0,T;\mathbb{R}^{d_u})$ control penalties, albeit with polynomial decay rates. A proof of the turnpike property \eqref{eq: turnpike.dl} in the case of general losses with $L^2$ control penalties is an open problem.
	\medskip
	
	\noindent
	\textbf{What does the decay \eqref{eq: turnpike.dl} entail?}
	Taking into account the fact that the time-step $\bigtriangleup t = \sfrac{T}{n_t}$ is fixed, \eqref{eq: turnpike.dl} provides a quantitative estimate of the number of layers needed to fit the data, whilst keeping the controls small (thus possibly helping in generalization). In fact, these estimates ensure and indicate that the time horizon (or number of layers) ought not be large at all so that the error reaches $0$ with controls of small amplitude (in our toy experiments for instance, we use $T=4$, and stability occurs beyond a stopping time (layer) $T^*\sim 1$ or so).
	In other words, the exponential stability indicates that any layers beyond a certain stopping time $T^*$ can be dropped (in theory) from training.
	
	\begin{remark}[Stability trade-off] Let us briefly comment on the choice of $T$.
	\begin{itemize}
	\item The time $T$ ought to be large enough (namely, $T\geqslant T^*$ for some $T^*$ large enough) for \eqref{eq: turnpike.dl} to hold in the general case, in particular if one refers to the strategy presented in Section \ref{sec: 10}. 
	Therein, the minimal time $T^*>0$ for which the stability property holds, is seen to depend on the data $\left\{x^{(i)}, y^{(i)}\right\}_{i\in[n]}$ through \eqref{eq: 5.11} (the minimal norm control which interpolates the dataset), and typically this dependence can be exponential. A precise characterization of \eqref{eq: 5.11} in terms of the "complexity" of the data (or even the number of samples\footnote{Note that this may be an important caveat in the direct application of the techniques presented in Section \ref{sec: 10} to the neural ODE setting. Due to the multiplicative nature of the controls, when one applies the Gr\"onwall inequality in various instances, a factor of $n$ will appear, and ultimately, one could end up with constants in \eqref{eq: turnpike.dl} which depend on $n$ in an exponential way. We believe that this dependence could be sharpened by an appropriate scaling of the augmented functional. On another hand, estimate \eqref{eq: turnpike.dl} could also insinuate an interplay between $n$ and $T$, once the underlying constants have been sharpened. We believe that this presentation is solely a first step in obtaining a clearer picture.} $n$) is not known to our knowledge in this nonlinear setting. Partial results are provided in \cite{ruiz2020universal}, where a characterization in terms of the fractal dimension of the dataset is provided, but solely when referring to explicitly constructed controls/parameters which interpolate the data, and not those of minimal norm.
	\smallskip	
	\item
	In turn, the presence of a minimal time $T^*$ would mean that one still needs several layers -- namely a large enough $T$--, before entering in the stability regime, from which point on the training error can be ensured to be exponentially small. This insinuates a trade-off in how large $T$ should actually be. In our numerical simulations, we see that $T$ is generally rather small, but one should keep in mind that these are toy examples, and do not convey possible difficulties encountered for various real-life datasets, which may be significantly more complex and high-dimensional.
	\end{itemize}
	
	\noindent
	All in all, we believe that further clarification on the role and size of $T$ in the neural ODE context is an open problem.
	\end{remark}
	
	\begin{remark}[Non-coercive losses]
	Interestingly enough, the exponential stability stated in \eqref{eq: turnpike.dl} may also be observed for running costs which do not attain a minimum, as is the case for instance for cross-entropy losses. (See for example Figures \ref{fig1.ex2} -- \ref{fig2.ex2}.) 
	In this case, the cost functional is actually not coercive with respect to the state. As a matter of fact, $\*E(\*x(t))$ approaches $0$ only if every trajectory $\*x_i(t)$ for $i\in[n]$ grows to $+\infty$ in an appropriate direction in $\mathbb{R}^d$. Thus, in this non-coercive case, we do not interpret the numerical results below as a turnpike property for the state, since the turnpike would depend (and increase with) $T$. Rather, the trajectories $\*x(t)$ become almost stationary (due to the exponentially small error and controls) beyond time $t\geqslant T^*$ to some point $\overline{\*x}\in\mathbb{R}^{d_x}$, which is exponentially "sliding" to $+\infty$ as $T\to+\infty$.
	\end{remark}
			
\begin{figure}[h]
	\centering
	\includegraphics[scale=0.5]{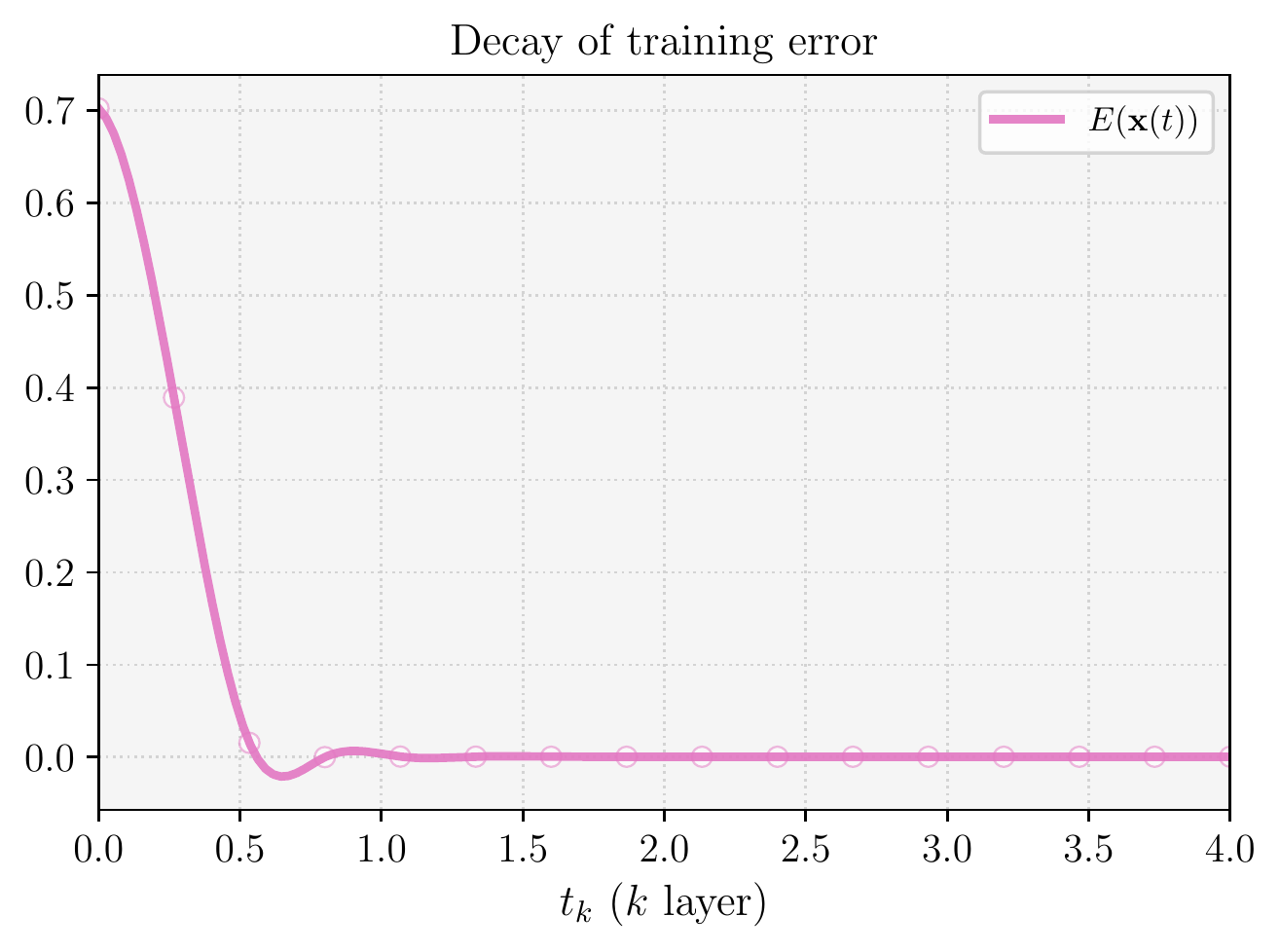}
\includegraphics[scale=0.5]{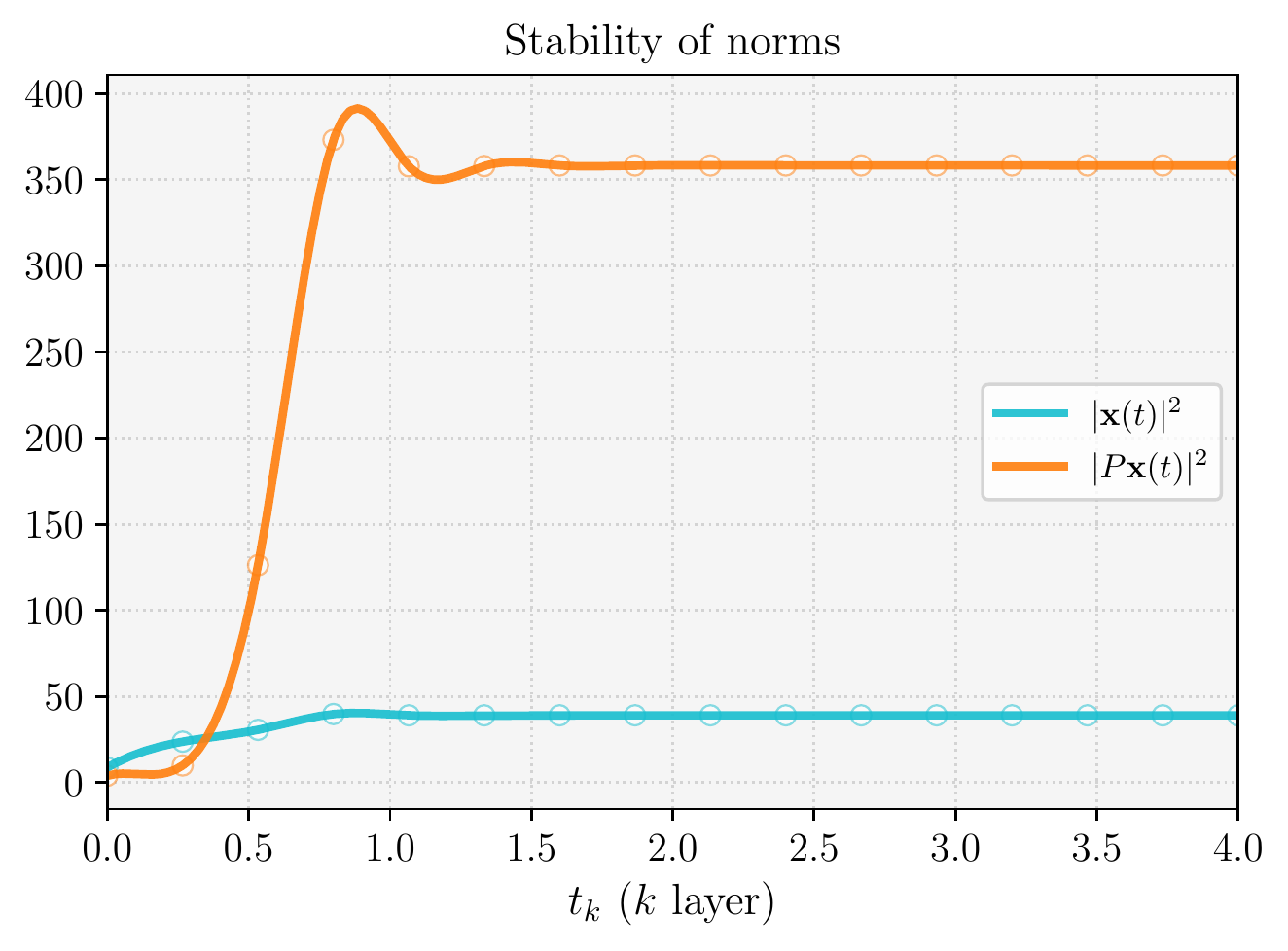}
\caption{Decay of the training error (\textit{left}) and "stabilization" of the trained trajectories $\{\*x_i(t)\}_{i\in[n]}$ (solutions to \eqref{eq: 5.8}) and $\{P\*x_i(t)\}_{i\in[n]}$ (\textit{right}) for $t\in[0, 4]$. We see that the error reaches $0$ and trajectories become almost stationary in time $\sim 1$, and since $\bigtriangleup t=0.25$, we solely need $4$ layers to train the network appropriately.}
\label{fig1.ex2}
\end{figure}

\begin{figure}[h]
\centering
\includegraphics[scale=0.6]{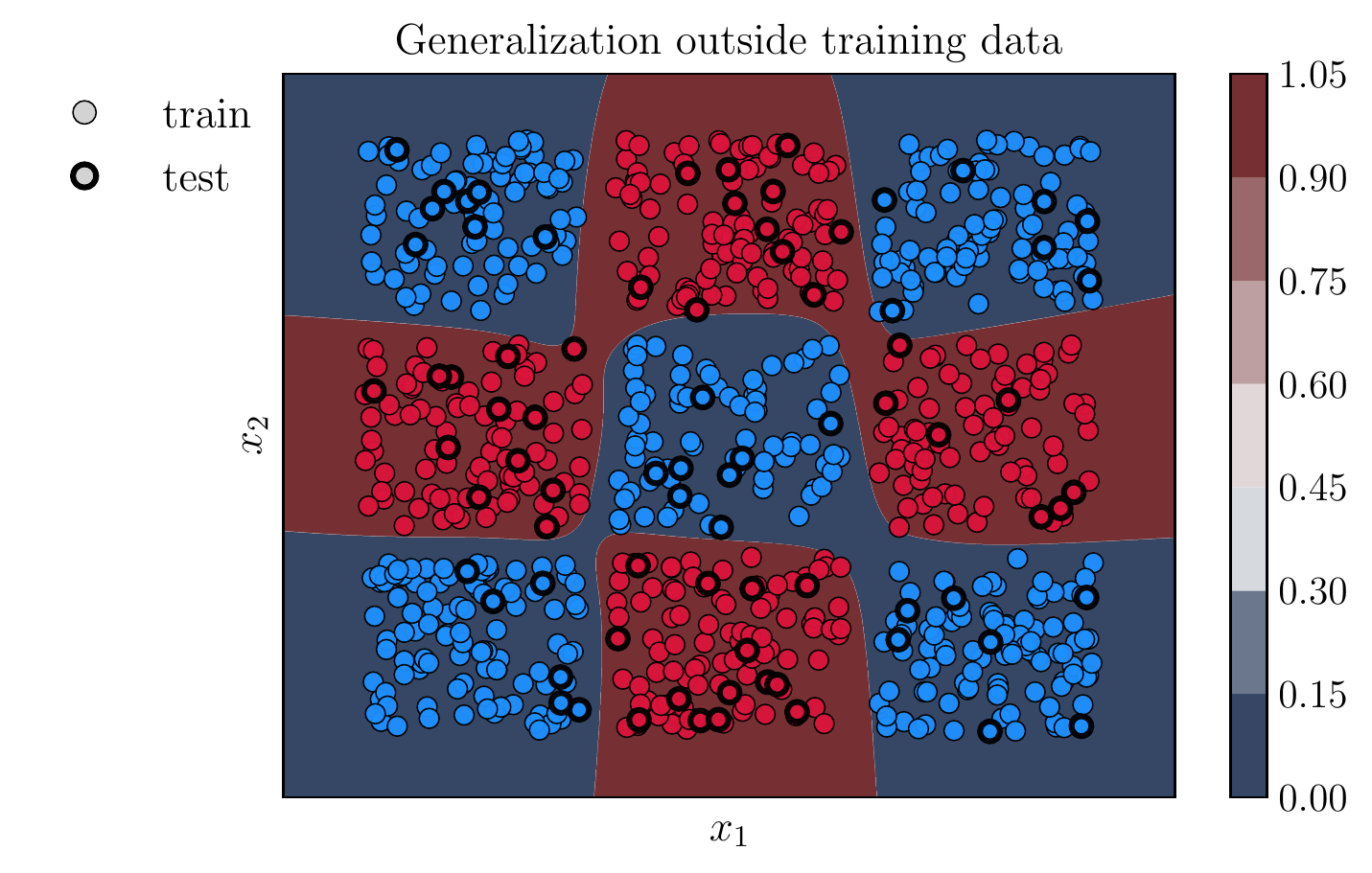}
\label{fig3.ex2}
\caption{We plot the trained classifier $P\*x_x(T)$ for any initial datum $x\in[-1.1,1.1]^2$, along with the training data, as well as test data. The shape of the dataset is captured accurately and thus the unknown function $f$ is approximated well, ensuring generalization, just as in Figure \ref{fig3.ex1}, but this regime is actually reached in time $t\sim 1$, unlike for Figure \ref{fig3.ex1} ($t=T=4$).}
\end{figure}

\begin{figure}[h]
\centering
\includegraphics[scale=0.6]{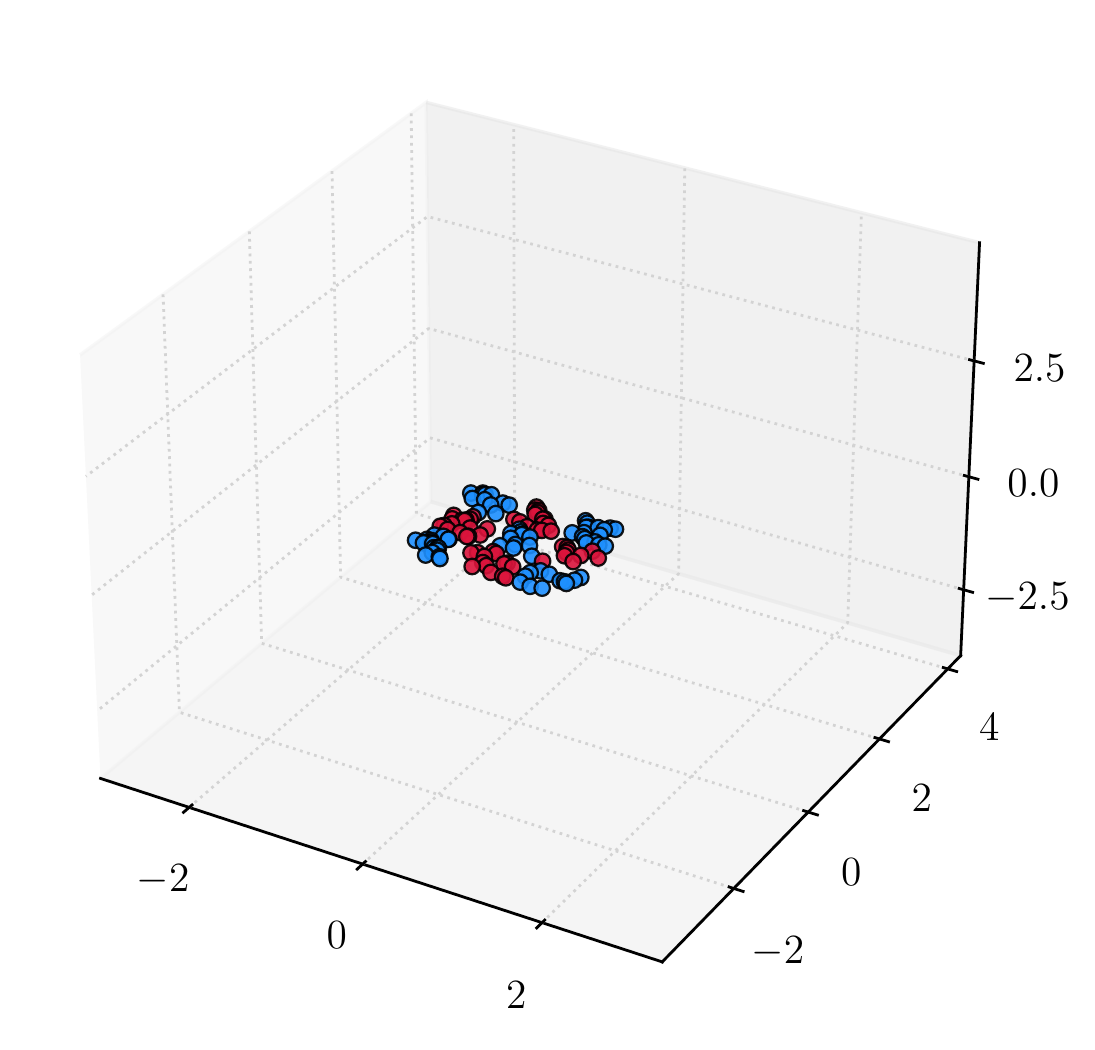}
\includegraphics[scale=0.6]{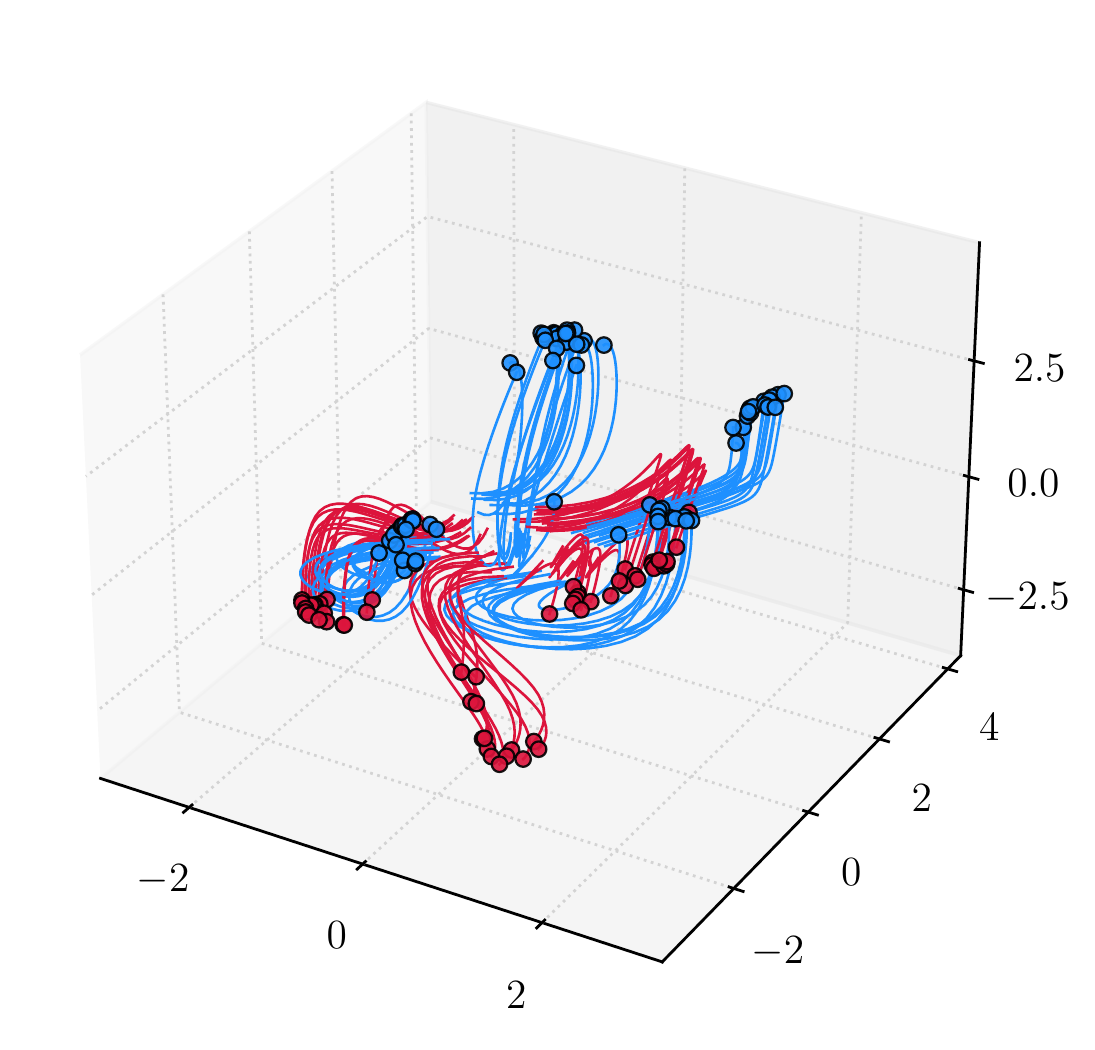}
\includegraphics[scale=0.6]{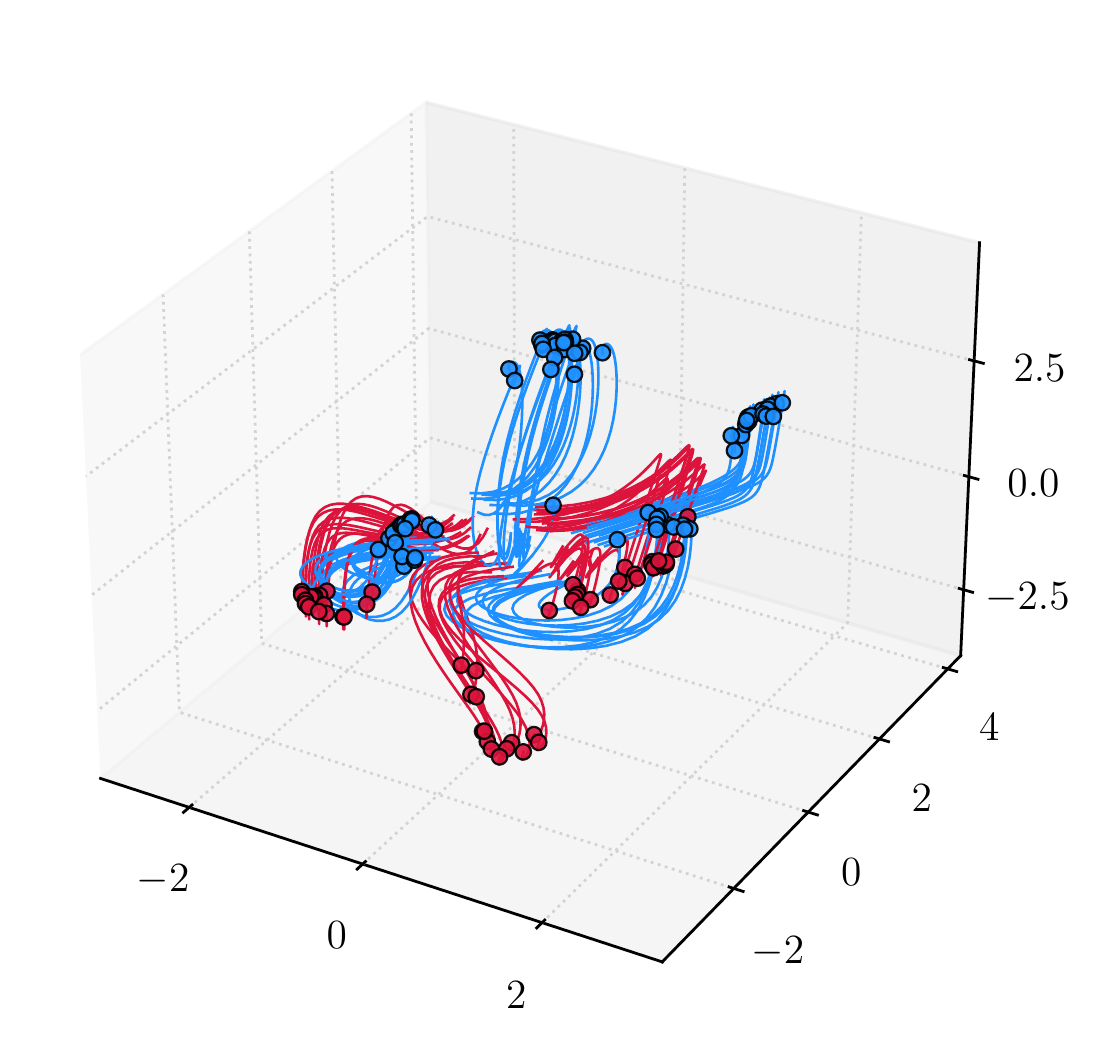}
\includegraphics[scale=0.6]{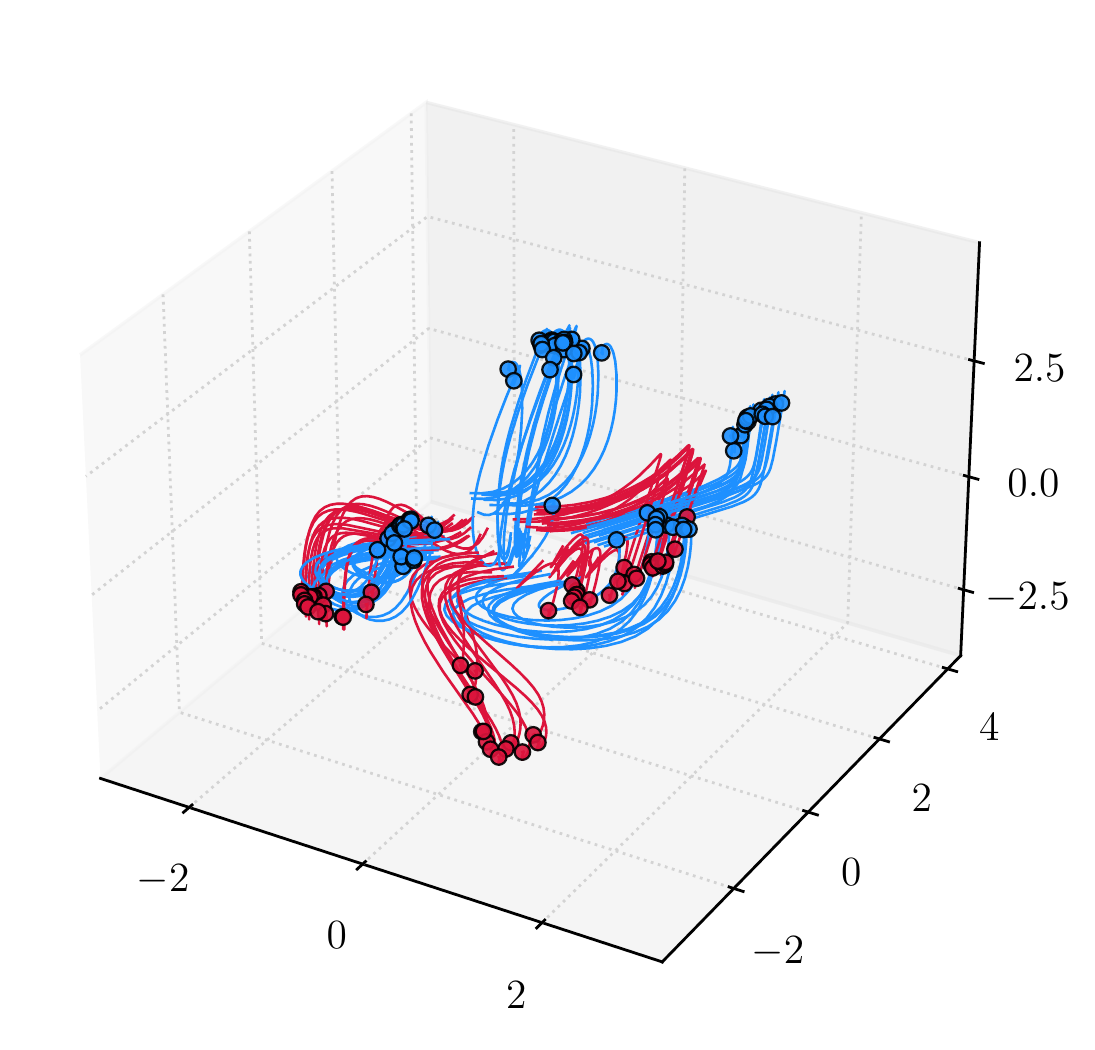}
\caption{The evolution of the trajectories $\{\*x_i(t)\}_{i\in[n]}$, solutions to \eqref{eq: 5.7} with data $\{\*x_i^0\}_{i\in[n]}$ (\emph{top left}), shown for $t\leqslant1$ (\emph{top right}), $t\leqslant 2$ (\emph{bottom left}), $t\leqslant T=4$ (\emph{bottom right}). The trajectories are stationary (in a separation regime) beyond time $t\geqslant1$.}
\label{fig2.ex2}
\end{figure}
	
	The decay may also be observed for ResNets (where the time-grid is coarse, with $\bigtriangleup t=1$), and also on datasets such as MNIST \cite{lecun2010mnist}. 
	In the latter, each input sample $x^{(i)}$ is a grayscale, $28\times 28$ image of a handwritten digit, and thus an element of $\mathbb{R}^{784}$; the dataset has $10$ labels: $y^{(i)} \in [10]$. 
	The appearance of the turnpike property for the corresponding ResNet, i.e. the discretized ODE on a coarse mesh (which has been shown to hold independently in \cite{faulwasser2021turnpikeML}, by making use of dissipativity arguments) may also be interpreted as a stability guarantee for the forward Euler scheme. 
	
	To justify these claims, we make use of \eqref{eq: 5.8} and a forward Euler scheme to obtain a corresponding ResNet with fixed time-step $\bigtriangleup t=1$, $T=20$, $d_{\text{hid}} = 32$ and $\sigma\equiv\tanh$. 
	We make use of fully connected layers only.
	The output layer is parametrized by $Px = p_1x+p_2$, where $p_1\in\mathbb{R}^{10\times784}$, $p_2\in\mathbb{R}^{10}$ are part of the optimization variables.
	We show the results of the experiments in Figures \ref{fig: figure.mnist.turnpike.1} -- \ref{fig2.ex4}. 
	Analog experiments for Fashion-MNIST are shown in Figures \ref{fig2.ex6} -- \ref{fig3.ex6}.

	\begin{figure}[h!] 
	\centering
	\includegraphics[scale=0.5]{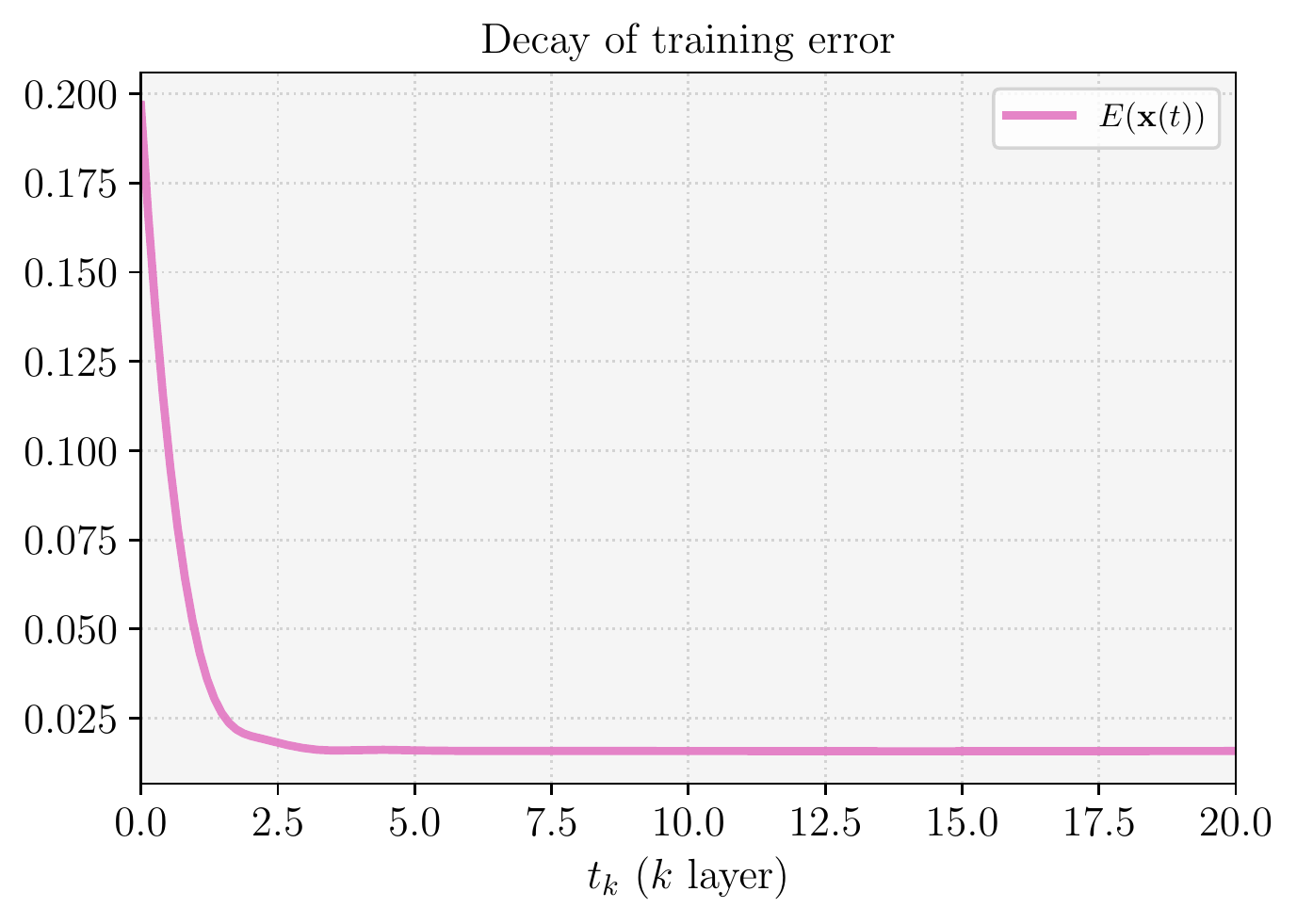}
	\includegraphics[scale=0.5]{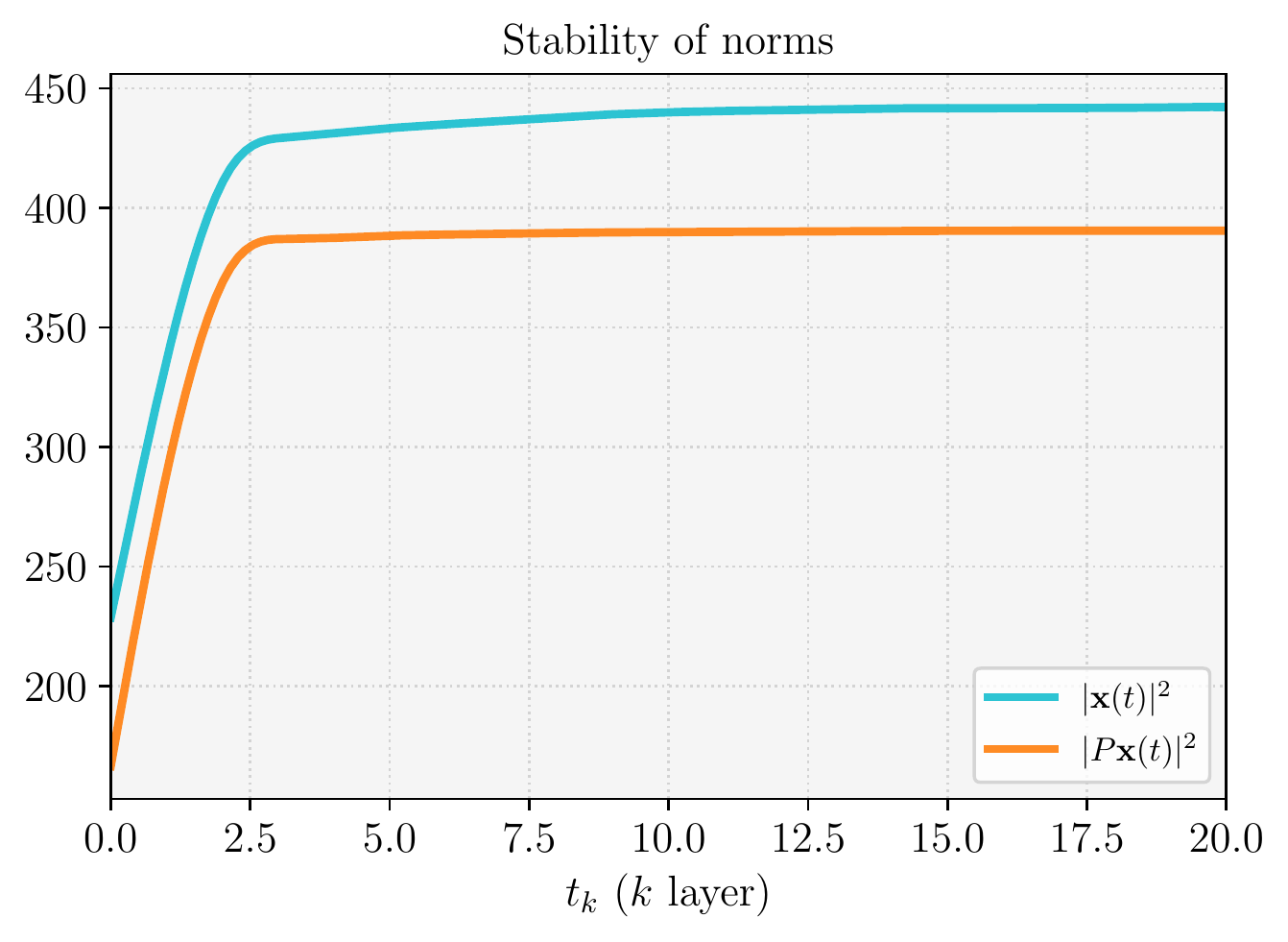}
	\caption{Decay of the training error (\textit{left}) and "stabilization" of the trained trajectories $\*x(t)$ and $\{P\*x_i(t)\}_{i\in[n]}$ (\textit{right}).}
	\label{fig: figure.mnist.turnpike.1}
	\end{figure}
	
	\begin{figure}
	\center
	\includegraphics[scale=1]{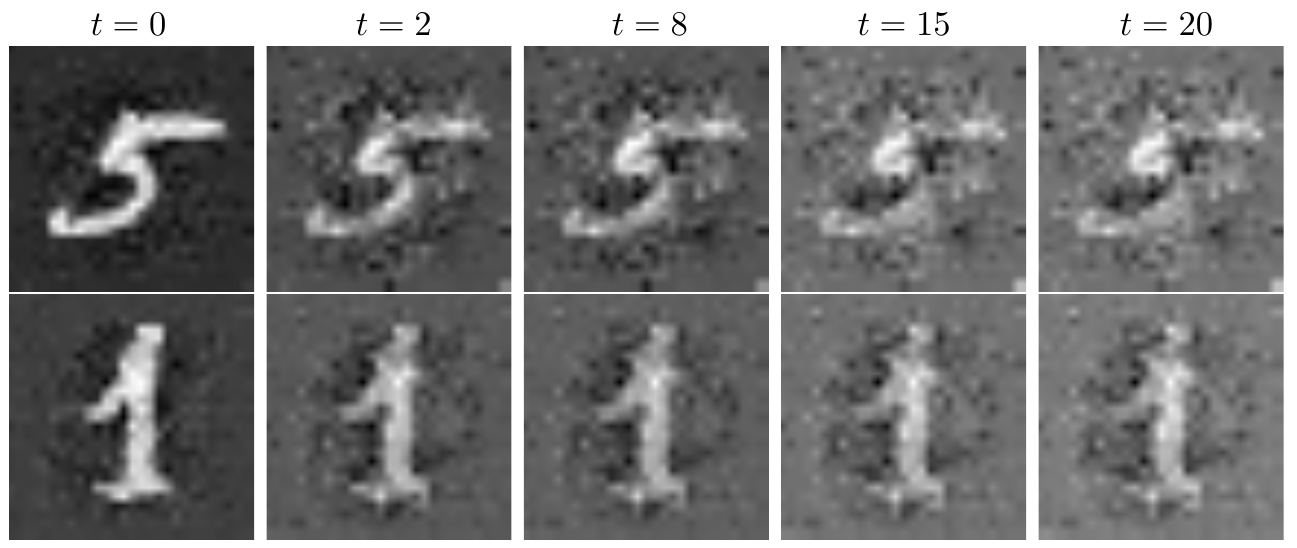}
	\caption{Evolution of two individual samples $\*x_i(t)\in\mathbb{R}^{784}$ mapped onto a $28\times28$ grid.  
	Each trajectory reaches some stationary configuration. 
	The trained model tends to "diffuse" the input signal ahead of classifying via the softmax applied to $P\*x_i(t) \in \mathbb{R}^{10}$.}
	\end{figure}
	
	\begin{figure}
	\centering
	\includegraphics[scale=0.5]{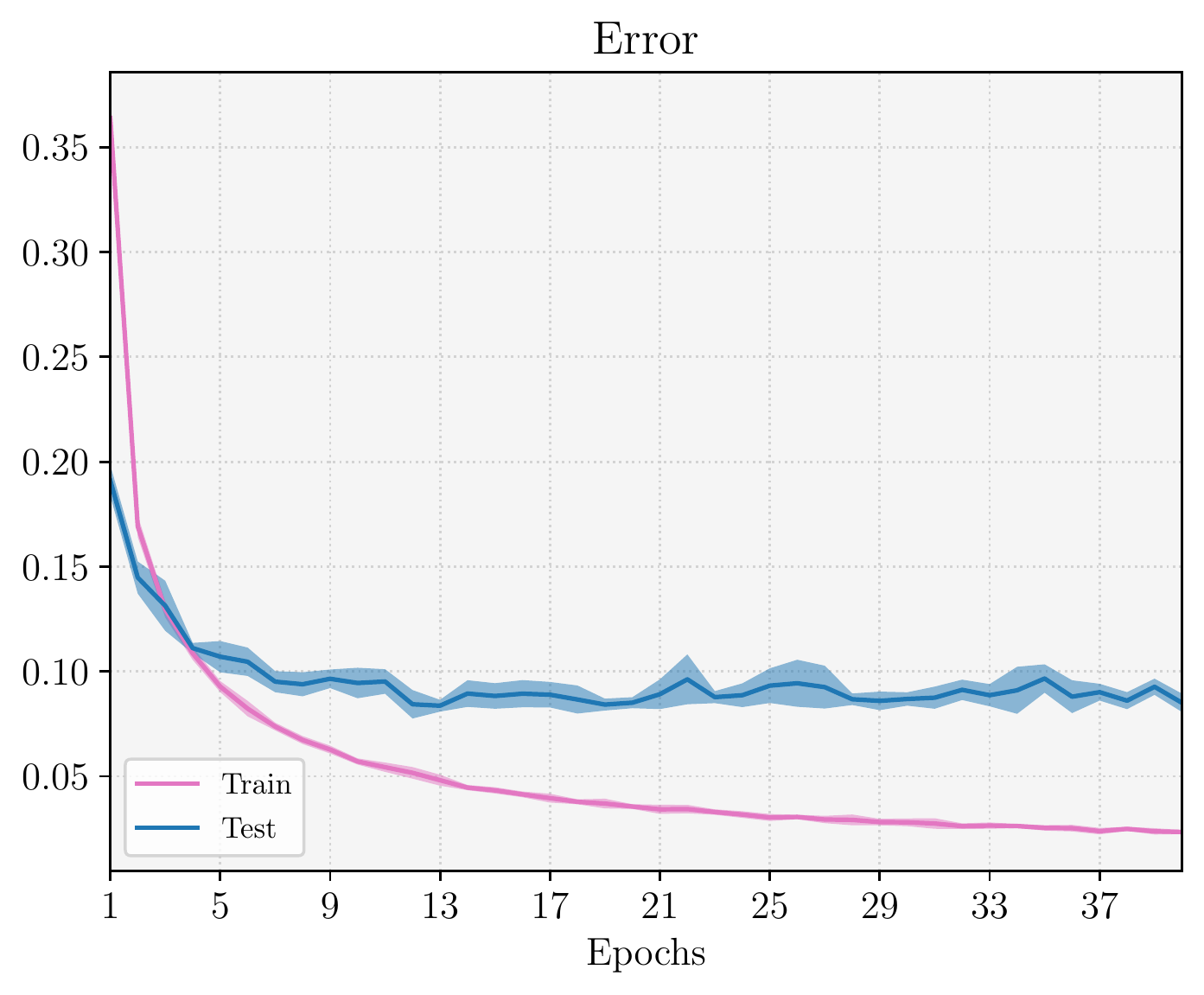}
	\includegraphics[scale=0.5]{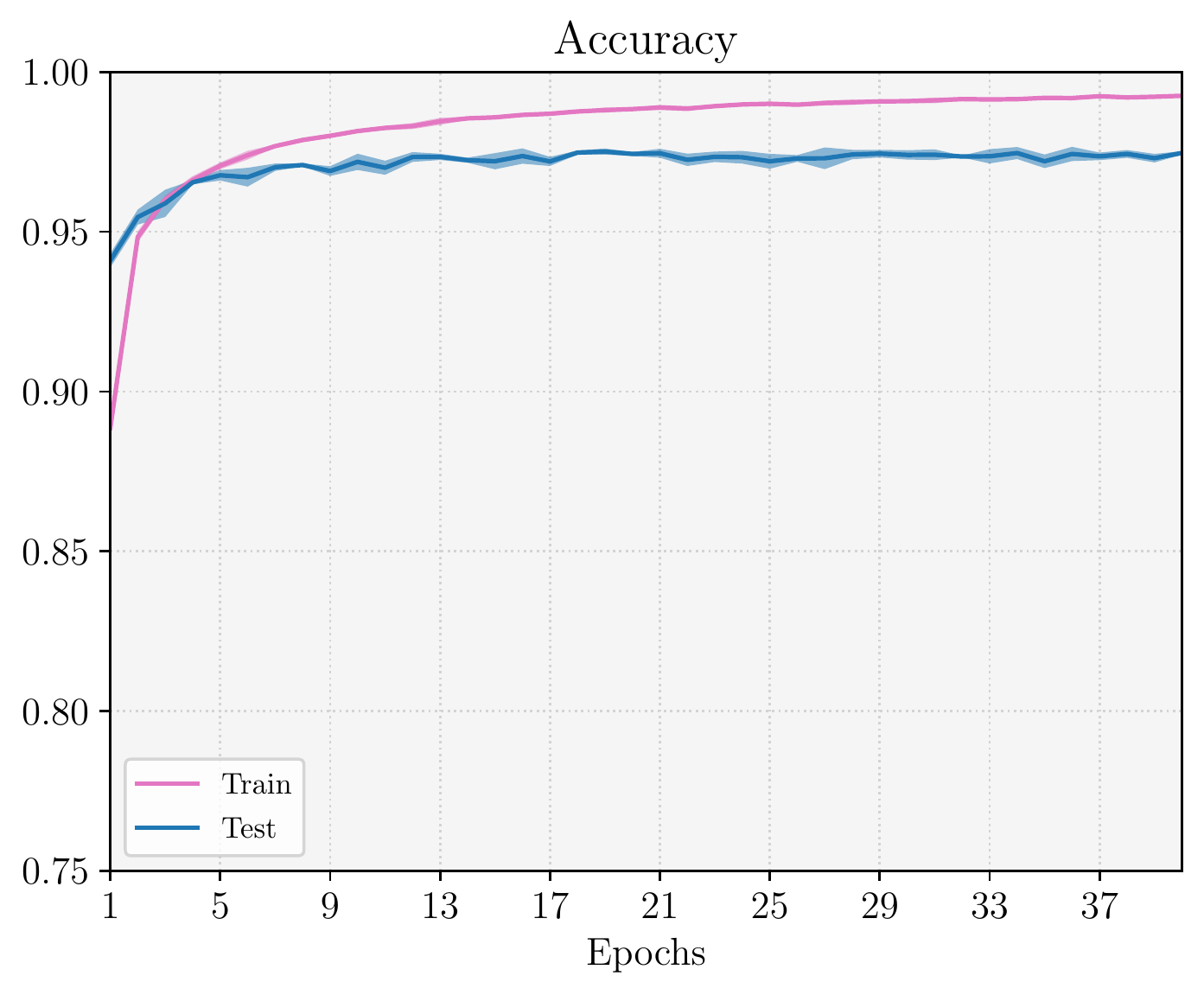}
	\caption{Validation error and accuracy over epochs (experiments repeated $10$ times). Generalization is not compromised due to the introduction of an integrated empirical risk.}
	\label{fig2.ex4}
	\end{figure}

	\begin{figure}
	\center
	\includegraphics[scale=1]{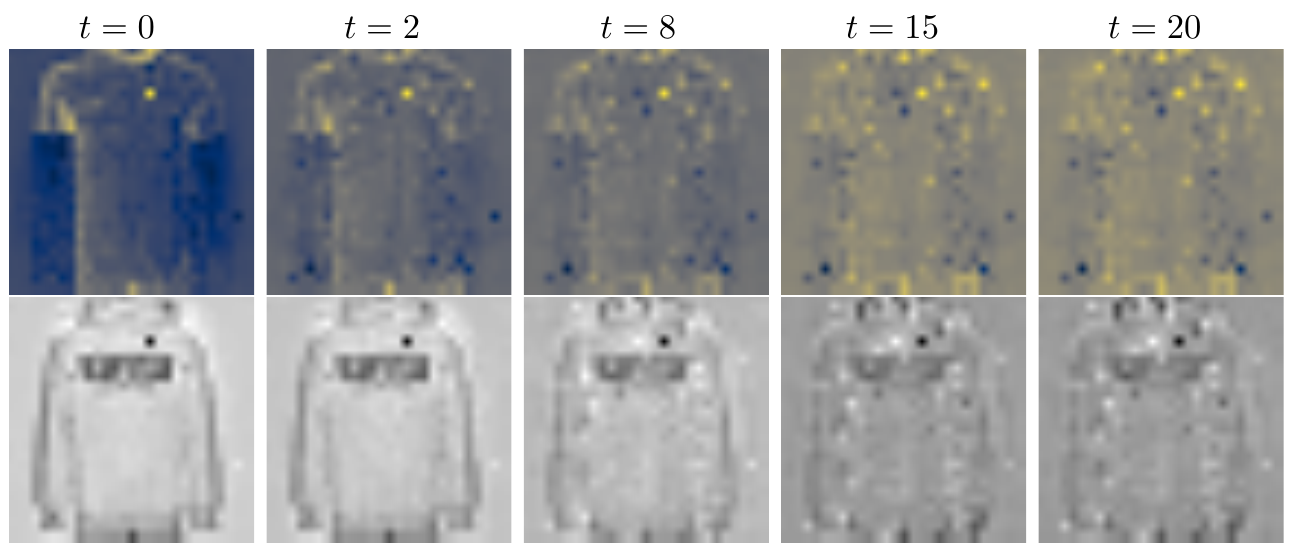}
	\caption{The evolution of two individual samples $\*x_i(t)\in\mathbb{R}^{784}$ mapped onto a $28\times28$ grid. Images are grayscale, but a different colormap is used to enhance visibility.}
	\label{fig2.ex6}
	\end{figure}
	
	\begin{figure}
	\centering
	\includegraphics[scale=0.5]{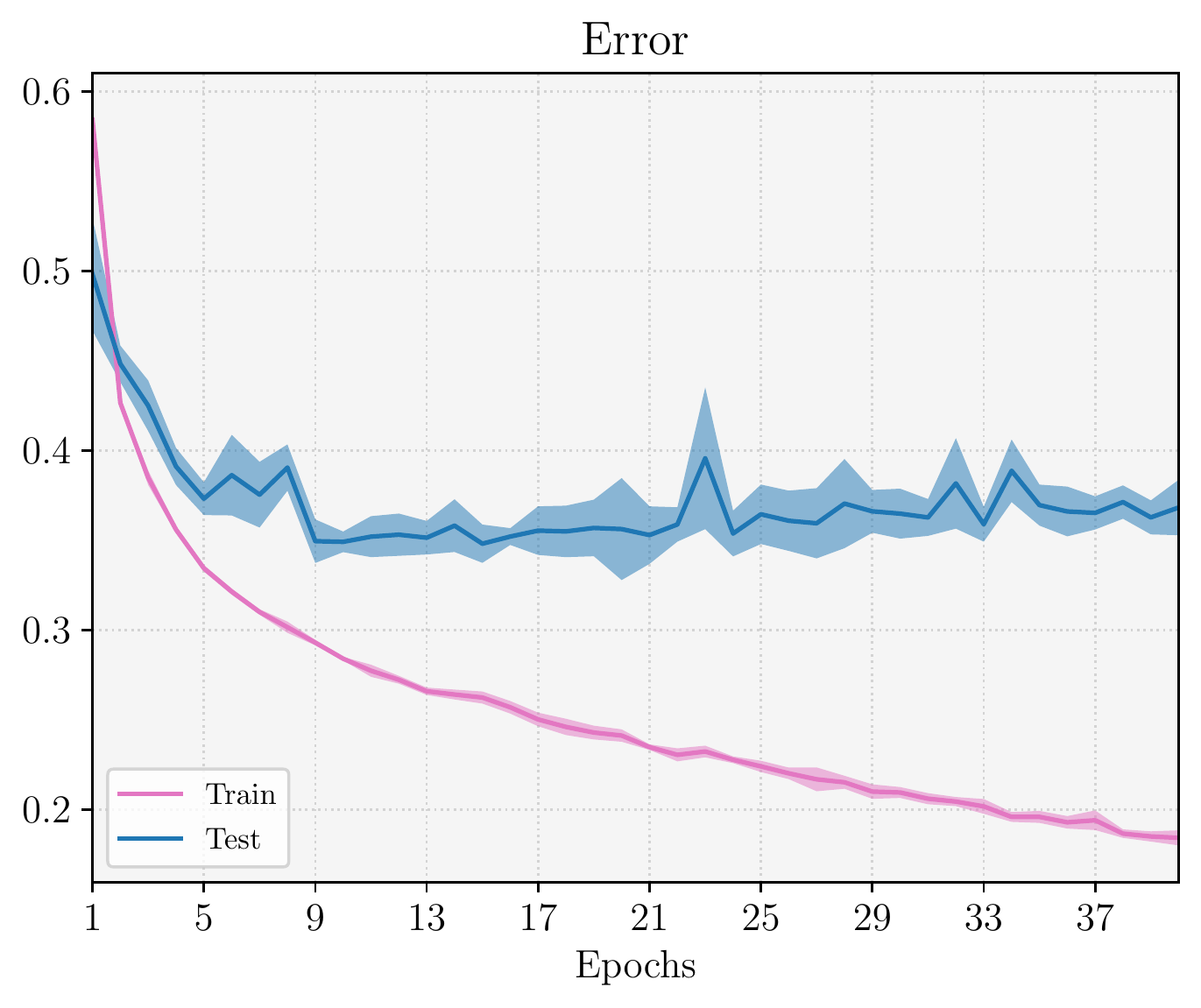}
	\includegraphics[scale=0.5]{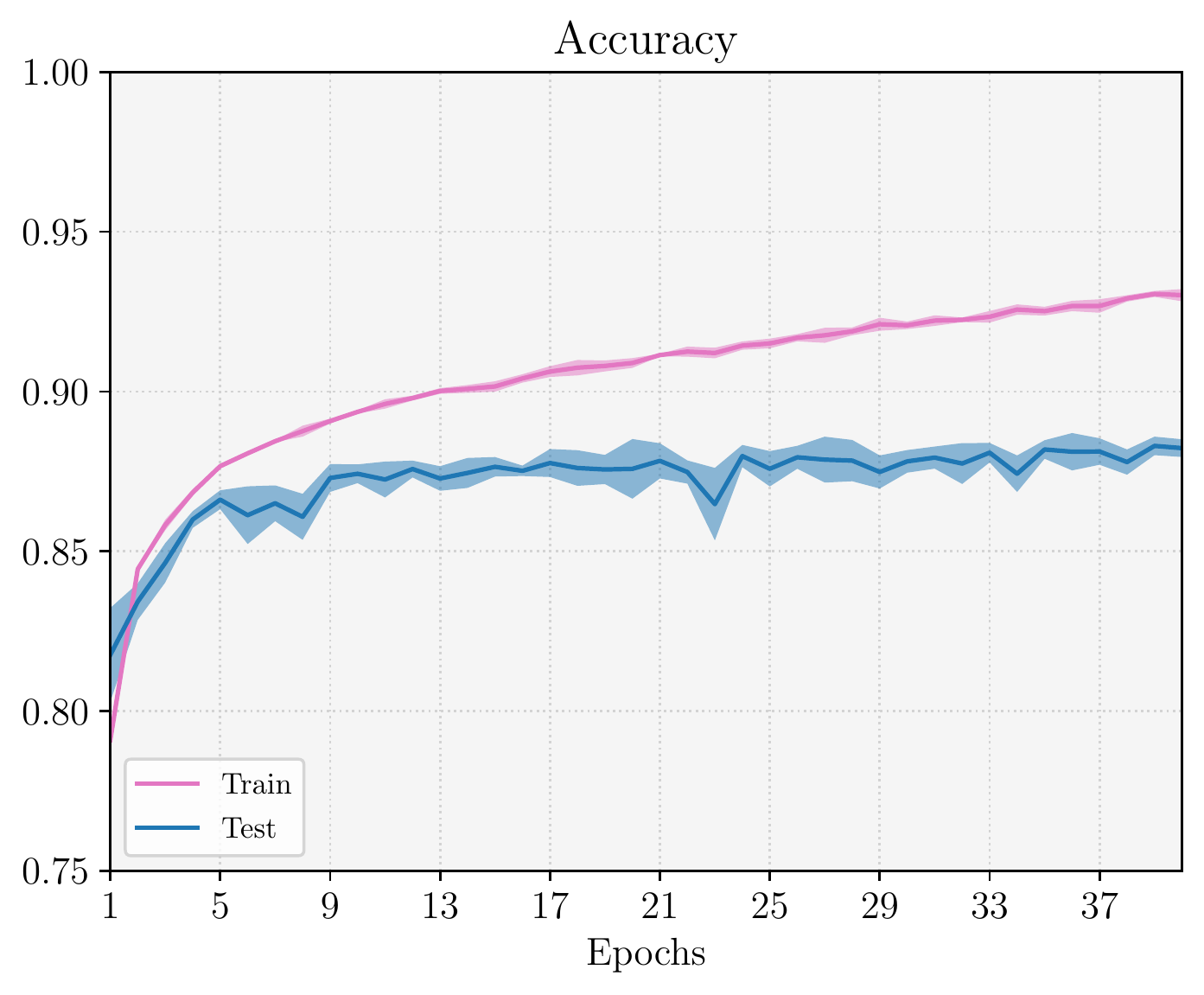}
	\caption{Validation error and accuracy over training epochs (experiments repeated $10$ times). The lower accuracy with respect to state of the art configurations is due to the simplified architecture.}
	\label{fig3.ex6}
	\end{figure}

\section{Further topics} \label{sec: 14}

\subsection{Model predictive control}

The turnpike property has also been used in the design of adequate temporal grids for numerical discretizations which are moulded to  \emph{model predictive control} (MPC) feedback loops. We henceforth follow \cite{grune2019sensitivity}. 
Model predictive control (\cite{garcia1989model, grune2017nonlinear}) is one of the most successful paradigms in contemporary control theory, with reliable performance in several practical applications, ranging from chemical to aerospace engineering. For a given, arbitrary, time-dependent optimal control problem set in a time horizon $[0,T]$, which we designate by OCP, in which one looks for an optimal control $u(t)$ and associated state $y(t)$ solving an ODE or PDE, the standard MPC algorithm can roughly be summarized in Algorithm 2. 

\begin{algorithm}[h!] \label{algo: mpc}
\SetAlgoLined
\noindent
Initialize $T>\tau>0$, $k=0$, $K>0$, and $y(0)$\;
\While{$k<K$}{
  Solve OCP on $[0,T]$ with initial data $y(k\tau)$, giving control $u^k$\;
  Solve ODE or PDE on $[0,T]$ with initial data $y(k\tau)$ and control $u^k$, giving state $y^k$\;
  Set $y(t):= y^k(t-k\tau)$ and $u(t):=u^k(t-k\tau)$ for $t\in[k\tau,(k+1)\tau]$\;
  $k\leftarrow k+1$
 }
\caption{Model predictive control (MPC).}
\end{algorithm}

Note that in Algorithm 2, 
 one would normally use a uniform mesh for discretizing the time interval $[0,T]$, and use one's favorite quadrature formula to integrate the underlying ODE (or spatially-discretized PDE). And since the MPC algorithm only implements the first part of the trajectory until a time $\tau$, one is particularly interested in a high accuracy of the computed control on $[0,\tau]$.
Yet, the turnpike property indicates the precise distribution of the times for which a resolution of the evolutionary problem is needed: except for the boundary layers near $t=0$ and $t=T$, the optimal pairs for the time-dependent problem are essentially constant. As observed in \cite{grune2019sensitivity}, instead of considering a conventional uniform grid with $n_t\geqslant1$ nodes, one can also construct a turnpike-adapted grid $\{t_j\}_{j=0}^{n_t-1}$ by solving 
\begin{equation*}
\int_{t_j}^{t_{j+1}} \left(e^{-\lambda t} + e^{-\lambda(T-t)}\right) \diff t = C, \hspace{1cm} \text{ for all } j\in\{0,\ldots,n_t-2\}.
\end{equation*}
Here the constants $\lambda>0$ and $C>0$ are externally tuned. The numerical experiments performed in \cite{grune2019sensitivity}, in the case where the target is a steady state, and thus only the term $e^{-\lambda t}$ is needed in the above construction, insinuate a significant reduction in the number $n_t$ of nodes needed to render the value of the cost functional in the OCP near $0$, when compared to a uniform grid. We also refer the reader \cite{grune2020efficient} for a recent improvement which combines spatial mesh refinement in the context of goal-oriented MPC. 
All in all, these ideas are quite broad and go beyond solely MPC design, being in the spirit of what was presented in the above discussion on deep learning.

\subsection{Greedy algorithms}

In many practical applications, robustness of the optimal controls with respect to various parameters for the underlying PDE (for instance, diffusivity or conductivity coefficients, Reynolds number, and so on) needs to be ensured to have a viable policy. 
This should in turn result in the consideration of a parameter dependent optimal control problem. 
For instance, consider 
\begin{equation} \label{eq: ocp.parameter}
\inf_{\substack{u\in L^2((0,T)\times\omega) \\ y \text{ solves } \eqref{eq: heat.parameter}}} \int_0^T \|y(t)-y_d\|^2_{L^2(\Omega)} + \int_0^T \|u(t)\|_{L^2(\omega)}^2 \diff t,
\end{equation}
where 
\begin{equation} \label{eq: heat.parameter}
\begin{cases}
\partial_t y - \nabla \cdot \left(a(x, \nu)\nabla y\right) + c(x, \nu) y = u1_\omega &\text{ in } (0,T)\times\Omega,\\
y= 0 &\text{ in } (0,T)\times\partial\Omega,\\
y_{|_{t=0}} = y^0 &\text{ in } \Omega.
\end{cases}
\end{equation}
Here the coefficients $a(\cdot, \nu) \in L^\infty(\Omega)$ and $c(\cdot, \nu) \in L^\infty(\Omega)$ are such that 
$$Ay=-\nabla \cdot \left(a(\cdot, \nu) \nabla y\right) + c(\cdot, \nu) y$$ 
is elliptic, and $\nu\in\mathbb{R}^d$ is a parameter. Of course, now, $y=y(t,x;\nu)$ will depend on the parameter $\nu$. But from the Pontryagin Maximum Principle, one sees that the optimal control $u_T$ will be given by $u_T\equiv p_T1_\omega$, where $p_T$ is the adjoint state, which will thus depend on the parameter $\nu$. Thus, one has to solve \eqref{eq: ocp.parameter} for each new choice of $\nu$. This, combined with the time-dependence, renders the above optimal control problem computationally unfeasible due to the exorbitant dimensionality it manifests. 

The turnpike property can however serve as a remedy to this dimensionality issue. One can aim to rather consider the corresponding elliptic problem,
namely
\begin{equation} \label{eq: static.ocp.parameter}
\inf_{\substack{u\in L^2(\omega)\\ y \text{ solves } \eqref{eq: poisson.parameter}}} \|y-y_d\|^2_{L^2(\Omega)} + \|u\|_{L^2(\omega)}^2,
\end{equation}
where 
\begin{equation} \label{eq: poisson.parameter}
\begin{cases}
-\nabla \cdot \left(a(x,\nu) y\right) + c(x, \nu) = u1_\omega &\text{ in }\Omega,\\
y=0 &\text{ on }\partial\Omega,
\end{cases}
\end{equation}
 and thus reduce the dimensionality by removing the time-dependence (this is, in some sense, a model reduction step). On another hand, the parameter dependence can be addressed in an optimal manner by making use of greedy algorithms, to determine the most relevant values of a parameter-space and provide the best possible approximation of the set of parameter dependent optimal controls. 

This is the approach taken in \cite{hernandez2019greedy}, inspired by \cite{lazar2016greedy}. 
Let us provide brief details.
Assume that the parameter $\nu$ ranges within a compact set $\mathcal{K}\subset\mathbb{R}^d$, and that the functions $\nu\mapsto a(\cdot, \nu)$ and $\nu\mapsto c(\cdot, \nu)$ are holomorphic\footnote{The holomorphy assumption is needed to ensure a polynomial decay of the so-called \emph{Kolmogorov $n$-widths}, used to quantify the best approximation via greedy algorithms. We refer to \cite{devore2013greedy} for details.}. 
Now consider the set of controls solving \eqref{eq: static.ocp.parameter} for each $\nu\in \mathcal{K}$, namely
\begin{equation*}
\mathscr{U}_s = \left\{ \overline{u}(\cdot; \nu)\,\Bigm|\, \nu \in \mathcal{K}\right\}.
\end{equation*}
By making use of the characterization of optimal controls via the adjoint state, and making use of some classic functional analysis arguments, it can be seen that $\mathscr{U}_s$ is a compact subset of $L^2(\omega)$. 
Now, given $\varepsilon>0$, one seeks to determine a family of parameters $\{\nu_1, \ldots, \nu_n\} \subset \mathcal{K}$, with $n=n(\varepsilon)\geqslant1$, so that the corresponding controls $u_{\nu_1}, \ldots, u_{\nu_n}$ are such that for every $\nu\in \mathcal{K}$, there exists $u_\nu^\star\in\text{span}\{u_{\nu_1},\ldots,u_{\nu_n}\}$ such that
\begin{equation*}
\|u_\nu^\star-\overline{u}_\nu\|_{L^2(\omega)}\leqslant\varepsilon.
\end{equation*}
This problem can subsequently be solved by making use of greedy algorithms (also known as matching pursuit) suggested in \cite{hernandez2019greedy} and established theory from \cite{barron2008approximation, devore2013greedy, cohen2015approximation}, for instance.
The greedy algorithm theory is mostly done for elliptic parameter dependent PDEs. In the optimal control setting, the turnpike property ensures that the elliptic theory can then be transferred to the parabolic case. 
In \cite{hernandez2019greedy}, the optimal control associated to parameters $\nu$ found by weak greedy algorithms is seen to corroborate this fact through numerical simulations.

\subsection{Mean field games}

To conclude, we briefly discuss some links and appearances of the turnpike properties in \emph{mean field games}. These are coupled systems modeling the joint interactions of multiple and many agents, consisting of a Hamilton-Jacobi-Bellman equation, evolving backward in time, governing the computation of an optimal path for each agent, and a Fokker-Planck equation, evolving forward in time, governing the evolution of the density of the agents. Namely, these are systems of the form
\begin{equation} \label{eq: mfg}
\begin{cases}
-\partial_t u - \nu\Delta v + \*H(x,Du) = f^1(x,m) &\text{ in }(0,T)\times\mathbb{T}^d,\\
\partial_t m - \nu\Delta m + \nabla \cdot (m\, \partial_p \*H(x,Du)) = 0 &\text{ in } (0,T)\times\mathbb{T}^d,\\
m_{|_{t=0}} = m_0 &\text{ on }\mathbb{T}^d,\\
u_{|_{t=T}} = \phi(x,m(T)) &\text{ on }\mathbb{T}^d.
\end{cases}
\end{equation}
Here $\mathbb{T}^d=\mathbb{R}^d/\mathbb{Z}^d$ denotes the $d$-dimensional flat torus --  this consideration of spatial domain is done for simplicity regarding boundary conditions. Mean field games models were introduced in \cite{lasry2007mean}.

Let us begin by motivating the meaning of the equations in \eqref{eq: mfg}. 
Suppose that an agent (player) is at location $X_0=x\in\mathbb{R}^d$ at time $\tau=0$, and wishes to "improve" its position $X_T$ at time $\tau=T$. An elementary approach in proceeding with the resolution of this problem would be to endow the agent with controls $\alpha_\tau$ at all time $\tau$, and solve
\begin{equation*}
\begin{cases}
\dot{X}_\tau= \alpha_\tau &\text{ for } \tau\in(0,T),\\
X_0=x.
\end{cases}
\end{equation*}
(We use the notation $X_\tau$ to stay in line with common notation in stochastic calculus.)
One finds these controls $\alpha_\tau$ by minimizing a cost, which in this theory also accounts for the density of all the agents $m(t,x)$ at time $t$ and position $x$. In the absence of noise in the dynamics, the value function will solve a hyperbolic Hamilton-Jacobi-Bellman equation, which is often not very desirable from an analytical point of view, and does not account for inherent uncertainties. Thus, one rather models the position of the agent at time $\tau$ by a stochastic differential equation (SDE) of the form
\begin{equation} \label{eq: sde}
\begin{cases}
\diff X_\tau= \alpha_\tau \diff \tau + \nu \diff B_\tau &\text{ for } \tau\in(0,T),\\
X_0=x.
\end{cases}
\end{equation}
Here, $\nu>0$, and $\{B_\tau\}_{\tau\geqslant0}$ is the standard $d$-dimensional Brownian motion. The equation \eqref{eq: sde} is interpreted in the sense of the Duhamel formula. The minimization of the cost functional then reads
\begin{equation*}
\inf_{\alpha} \mathbb{E}\left(\phi(X_T,m(T))+\int_t^T\left(f^0(X_\tau,\alpha_\tau) + f^1(X_\tau, m(\tau))\right)\diff \tau \right),
\end{equation*}
where $f^0, f^1$ and $\phi$ are given. Note that the evolution of the measure $m(t)$ enters as a parameter. One solves this minimization problem by taking a Hamilton-Jacobi approach, and defines the value function 
\begin{equation*}
u(t,x) = \inf_\alpha \mathbb{E}\left(\phi(X_T,m(T))+\int_t^T\left(f^0(X_\tau,\alpha_\tau) + f^1(X_\tau, m(\tau))\right)\diff \tau \right),
\end{equation*}
where $\alpha$ is an admissible control such that $X$ solves \eqref{eq: sde}. 
It can then be shown that $u(t,x)$ solves the first equation in \eqref{eq: mfg}, where
\begin{equation*}
\*H(x,p)=\sup_{q} \left(-q \cdot p - f^0(x,q)\right).
\end{equation*} 
Moreover, given the value function $u$, it is known that the agent plays in the optimal way by using the feedback control $\alpha^\star(t,x)=-\partial_p \*H(x,Du(t,x))$.
Furthermore, if all agents have independent associated noises and follow the same strategy as above, the law of large numbers (applied to the number of agents) leads one to deduce that the density $m$ of the agents satisfies the second equation in \eqref{eq: mfg}. This leads one to the prefix "mean-field".
The game-theoretical interpretation comes from seeing \eqref{eq: mfg} as a description of a Nash equilibrium (see \cite{lasry2007mean, cardaliaguet2010notes} for details). 

Given the above derivation of \eqref{eq: mfg}, one sees the backward HJB equation as representing the agents’ decisions based on where they want to be in the future, while the forward Fokker-Planck equation as representing where they actually end up, based on their initial distribution.
Solving this coupled system of equations, one evolving backwards in time, and one evolving forwards in time, is highly non-trivial, and in some cases existence or uniqueness, or both, break down.

In fact, one notes a striking similarity to optimality systems we encountered in previous discussions on optimal control for nonlinear PDEs.
It is this similarity that leads to a connection with the turnpike property, when studying the asymptotics of solutions to \eqref{eq: mfg}. The MFG theory can actually be seen as a catalyst in the turnpike one. Indeed, works on a double-arc exponential estimate for \eqref{eq: mfg} (e.g., \cite{cardaliaguet2012long, cardaliaguet2013long}, and also more recently \cite{cardaliaguet2019long, cardaliaguet2020introduction}) precede and have motivated those on the turnpike property (first appearing in \cite{porretta2013long}). 
Regarding this exponential estimate, in the recent paper \cite{cirant2021long} for instance, the authors roughly show under the assumption that 
the Hamiltonian $\*H(x,p)$ is $C^2$ and locally Lipschitz with respect to $p$, and locally convex with respect to $p$, with $f^0(x,\alpha), f^1(\alpha, m)$ also satisfying Lipschitz assumptions, and being locally bounded, then any classical solution $(u_T,m_T)$ to \eqref{eq: mfg} satisfies 
\begin{equation*}
\|m_T(t)-\overline{m}\|_{L^\infty(\mathbb{T}^d)} + \|Du_T(t)-D\overline{u}\|_{L^\infty(\mathbb{T}^d)} \leqslant C\left(e^{-\omega t} + e^{-\omega(T-t)}\right)
\end{equation*}
for all $t\in(1,T-1)$ and for some $C,\omega>0$ independent of $T$, where $(\overline{u}, \overline{m}, \overline{\lambda})$ is the unique solution to 
\begin{equation} \label{eq: mfg.static}
\begin{cases}
\overline{\lambda} - \nu\Delta\overline{u} +\*H(x,D\overline{u})=f^1(x,\overline{m}) &\text{ on } \mathbb{T}^d,\\
-\nu\Delta \overline{m} - \nabla \cdot(\overline{m}\,\partial_p\*H(x,D\overline{u}))=0 &\text{ on } \mathbb{T}^d,\\
\int_{\mathbb{T}^d} \overline{m}= 1, \, \int_{\mathbb{T}^d} \overline{u} = 0. 
\end{cases}
\end{equation}
This is an exponential turnpike-like property. Indeed, one sees system \eqref{eq: mfg.static} as similar to the steady optimality system in turnpike theory. Moreover, one can also readily connect the convergence of the value function $u$ to the corresponding steady problem to the behavior observed in the HJB asymptotics via turnpike.

\part{Epilogue}

Summarizing, the turnpike property occurs naturally and generically among a variety of optimization problems encountered in applications of different nature. These range from shape design in aerodynamics, to stability estimates for residual neural networks in machine learning. 
This being said, the full mathematical theory of turnpike is far from mature -- non LQ problems, or bilinear control systems, could give rise to all kinds of different turnpike-like patterns, among other open problems. As a matter of fact, even a precise mathematical definition of what \emph{the turnpike} may be, a priori, for a general optimal control problem, is still not completely clear.
Our goal, through this article, was to illustrate the cases where the picture is (relatively) clear, and those where further analysis is needed.

We saw that,  for LQ problems -- the staple of contemporary optimal control theory -- turnpike holds whenever the cost functional is sufficiently coercive with respect to the state and control, and under natural stabilizability assumptions on the underlying ODE or PDE dynamics. The turnpike property may then be characterized by a spectrum of different definitions, ranging from integral or cardinal turnpike, to measure turnpike, all the way to the exponential, double arc characterization, which was the major theme of this work. 

As is natural in many problems in analysis, a local theory can then be developed for nonlinear problems -- nonlinear here implies that the underlying set wherein one optimizes, is not a linear space. The latter could be due, for instance, typically, to nonlinear underlying dynamics, but also to presence of specific nonlinear constraints on the control and/or state.
In the context of nonlinear dynamics, the linearization strategy comes along with smallness assumptions, in particular on the running target the trajectory seeks to match over time. Said smallness assumptions, which appear to be of a technical nature, when removed, raise the critical issue of characterizing the actual turnpike. 
Similar issues are raised when one forays away from the setting of quadratic functionals, in which, while the turnpike property can be seen to hold numerically (as seen, for instance, in the context of deep learning), a full picture of the underlying arguments is lacking. These considerations lead us to the following open problems.

\section{Open problems} \label{sec: open.problems}

\setcounter{tocdepth}{1}

\subsection{Alternative proofs in the LQ setting}

We have discussed, in depth, two strategies (which rely on very similar ideas, namely using a "corrected Riccati" feedback to decouple the optimality system) for proving the exponential turnpike property for LQ problems in the infinite-dimensional setting. We believe however, that there ought to be different ways to prove this result. 

\begin{itemize}
\item A first direction could be to cleanse the picture regarding \emph{scaling} and \emph{singular perturbation} ideas, somewhat inspired from boundary layer theory in fluid mechanics. To be more specific, consider for instance the optimal control problem 
\begin{equation} \label{eq: 15.1.0}
\inf_{\substack{u\\ y \text{ solves } \eqref{eq: 15.1.1}}}\frac{1}{T}\int_0^T \|y(t)-y_d\|^2_{\mathscr{H}} \diff t + \frac{1}{T}\int_0^T \|u(t)\|^2_{\mathscr{U}}\diff t
\end{equation}
where
\begin{equation} \label{eq: 15.1.1}
\begin{cases}
\partial_t y = Ay + Bu &\text{ in }(0,T),\\
y_{|_{t=0}} = y^0.
\end{cases}
\end{equation}
Setting $s=\frac{t}{T}$ and $\varepsilon=\frac{1}{T}$, problem \eqref{eq: 15.1.0} readily rewrites as 
\begin{equation} \label{eq: 15.1.3}
\inf_{\substack{u\\ y \text{ solves } \eqref{eq: 15.1.2}}}\int_0^1 \|y(s)-y_d\|^2_{\mathscr{H}} \diff s + \int_0^1 \|u(s)\|^2_{\mathscr{U}}\diff s
\end{equation}
where
\begin{equation} \label{eq: 15.1.2}
\begin{cases}
\varepsilon\partial_s y =Ay + Bu &\text{ in }(0,1),\\
y_{|s=0} = y^0.
\end{cases}
\end{equation}
One sees that when $T\to+\infty$ then $\varepsilon\to0$, and can, heuristically, stipulate some convergence of \eqref{eq: 15.1.3} to the steady problem
\begin{equation*}
\inf_{\substack{(u,y)\\ Ay+Bu=0}}\|y-y_d\|^2_{\mathscr{H}} +\|u\|^2_{\mathscr{U}}.
\end{equation*}
We are not aware if this direction has been fully developed in the existing literature, and we believe that doing so would be of paramount importance.
\smallskip

\item
On another hand, we had also discussed the so-called dissipativity strategy in the sense of Willems, which in the PDE context is only known to guarantee the weak, measure-turnpike property. We believe that a transparent study of whether dissipativity theory (which, as said in what precedes, is an open loop extension of the Lyapunov method) can be used to recover the results obtained by the Riccati-inspired approaches, is warranted. 
\end{itemize}

\subsection{Non-uniqueness and the turnpike set}

A major theme in this paper was the possibility of non-uniqueness of turnpikes for quadratic optimal control problems with underlying nonlinear PDE dynamics, whenever the running targets are large. The canonical example of this artifact is the cubic heat equation.
We also saw that the definition of the turnpike property depends on whether the target is time-dependent or not (and not just on its smallness), as periodic turnpike may occur whenever the target is time-dependent and periodic. It is thus rather necessary to provide a more general yet tractable characterization of the turnpike property, which would account for such scenarios. 
Such a characterization might naturally and probably come by making use of the value function for the infinite time horizon problem. This idea was already raised in \cite{trelat2018integral}.

\begin{itemize}
\item
Focusing on the setting where non-uniqueness arises, further clarity and theoretical underpinning is needed to characterize which one among the global minimizers is the actual turnpike. We have a \emph{turnpike set} $$\mathfrak{T}\subset\mathscr{H}\times\mathscr{U};$$
as before, $\mathscr{H}$ is the state space, and $\mathscr{U}$ is the control space. 
Given some solution $(y_T,p_T)$ of the transient optimality system under consideration (since, naturally, we cannot guarantee or expect uniqueness for the transient system, if it breaks down in the steady one), one would then characterize the turnpike property by an estimate such as
\begin{equation*}
\text{dist}\Big(\Big(y_T(t),p_T(t)\Big), \mathfrak{T}\Big) \leqslant C\Big(e^{-\lambda t} + e^{-\lambda (T-t)}\Big)
\end{equation*}
for all $t\in[0,T]$.  This would then mean that there exists \emph{some} point $(\overline{y},\overline{p})\in\mathfrak{T}$ to which $(y_T(t),p_T(t))$ is "near" in the sense of the above estimate. But characterizing said point remains an open problem.
\smallskip

\item
In \cite{trelat2020linear}, numerical examples show the competition of two global turnpikes. The author mentions that the turnpike is determined by measuring its proximity to the terminal conditions. 
If one looks at Figure \ref{fig: dario}, where two global minimizers for the cubic Poisson optimal control problem are shown, then one would look to see which of the two wells is the basin of attraction, and could do so by (numerically) computing the spectrum of the Hamiltonian matrix for the optimality system linearized around the minimizer within said well. 
A full picture of this artifact remains an open problem.
\end{itemize}

\subsection{Large targets for the semilinear heat equation}

Much in line with the above subject, new ideas and techniques are also required for proving a turnpike property for semilinear heat equations whenever the target $y_d$ is arbitrarily large. This is due to a lack of complete understanding of the linearized optimality system. 

To fix ideas, we consider the case of the cubic nonlinearity: $f(y)=y^3$. This is clearly a  dissipative system. 
The optimality system, when linearized and considered in perturbation variables, reads
\begin{equation*}
\begin{cases}
\partial_t \zeta - \Delta\zeta + 3\overline{y}^2\zeta = \varphi 1_\omega &\text{ in }(0,T)\times\Omega,\\
\partial_t \varphi +\Delta \varphi - 3\overline{y}^2\varphi = (1-6\overline{y}\overline{p}) \zeta &\text{ in }(0,T)\times\Omega.
\end{cases}
\end{equation*}
As pointed out in Section \ref{sec: 8}, a key point is to check the validity of the turnpike property for the linearized optimality system just above. This is complicated because of the term $1-6\overline{y}(x)\overline{p}(x)$, whose sign is difficult to determine for general large targets. Furthermore, due to non-uniqueness of steady minima, it is not evident which one among the multiple global minimizers would designate the turnpike. 

\subsection{Non-quadratic functionals} 
Another recurrent theme in this work was the setting of quadratic cost functionals. While this setting is quite flexible and covers many problems arising in applications, they come along with a certain smoothness, seen notably on the level of the optimality system. But in fact, the turnpike property (in some form or another) has been shown to hold for other cost functionals as well. For instance, \cite{gugat2019turnpike} penalize the $\text{TV}$--norm (in time) of the control and obtain an integral turnpike property for linear, first-order hyperbolic systems. In \cite{yague2021sparse}, a polynomial turnpike property is obtained for the optimal controls for finite-dimensional driftless nonlinear systems, when the $L^1$--norm (in time) of the control is penalized, and a rather general cost is used for the state. And more specifically, when the $L^1$--norm (in time) of the discrepancy of the state to the running target is penalized, the authors in \cite{gugat2021finite} show, for finite-dimensional systems, that a finite-time turnpike property occurs for the state, namely, the $L^1$ norm is saturated and the state reaches the turnpike exactly in finite time. 
Even more surprisingly, an integral turnpike property is obtained in \cite{mazari2020quantitative} for a functional without a tracking term in the state, but, albeit, with mass and pointwise positivity and boundedness constraints on the control (implying those for the state, by the parabolic maximum principle).
The authors work with a specific bilinear control problem for the heat equation, and show that the turnpike limit (for optimal controls) is a design which minimizes the first eigenvalue of the Laplacian. 
These works illustrate a "universality" of the turnpike property, but a general theory encompassing these cases, much akin to the quadratic case, is not present in the literature to our knowledge. 

\subsection{Turnpike with constraints}

We had considered optimal control problems without any constraints on the admissible pair as to be in line with the setting of the exponential turnpike property, for which, to our knowledge, results are only known in the unconstrained case. 
Already in the setting of linear systems $\dot{y} = Ay+Bu$, the presence of constraints renders the optimality system derived from the Pontryagin maximum principle significantly more compound, and detecting hyperbolicity patterns of the optimality system may be very challenging. 
For instance, a constraint of the form $u(t)\geqslant 0$ on the control may promote \emph{chattering} phenomena --  an infinite number of control switchings over a compact time interval. This is somewhat transparent when looking at the form of the optimal control $u_T$:
\begin{equation*}
u_T(t)\equiv\max\{0, B^*p_T(t)\}.
\end{equation*} 
It is not obvious whether the Riccati-inspired strategies can be readily extended to this context; now, the control is not given linearly in terms of the adjoint state $p_T$. 

Although the controllability theory under constraints on the state and/or control for linear and semilinear PDEs is now rather well established (\cite{loheac2017minimal, pighin2018controllability, hegoburu2018controllability, pighin2019controllability, pouchol2019phase, maity2019controllability, le2020local, ruiz2020control, lissy2020state, mazari2020constrained}), the validity and proof of the exponential turnpike property in such cases is a challenging open problem.

\subsection{Using HJB asymptotics}

We saw, following \cite{esteve2020turnpike}, that the turnpike property for LQ problems provides a rather clear picture of the asymptotics of the value function solving the associated HJB equation. What we are rather asking here, is to see as to how different properties of the HJB equation, under more general assumptions on the Hamiltonian, would translate to turnpike properties for different optimal control problems. To our knowledge, the literature on this issue is rather scarce.
The HJB interpretation of optimal control is clear in the setting of finite-dimensional problems. But it is not clear, to our knowledge, how one provides a transparent formulation of the Hamilton-Jacobi-Bellman equations for optimal control problems governed by PDEs. The derivation of the \emph{master equation} in mean field games could lead to some pointers regarding this issue.
Related to this, a stronger link between turnpike and the notion of ergodicity in optimal control and differential game theory, as studied in \cite{quincampoix2011existence, buckdahn2015representation, renault2017long}, also warrants further development. We refer to \cite{backhoff2020mean} for a related study in this direction.

\subsection{Decay rates for nonlinear problems} 
When one proceeds in proving turnpike by linearizing the optimality system, the explicit decay rate $\lambda$, given by the spectral abscissa of the operator $-A+BB^*\mathscr{E}_\infty$, may be lost. It would be of interest to have a clear understanding of the interplay between the linear turnpike decay rate, the size of the initial data (if any), and the running target, in the nonlinear turnpike context. It could be said that the current results are not completely transparent regarding this issue, which is much clearer in the context of (feedback) stabilization of nonlinear systems. Such considerations could first be addressed for systems which have a convenient variational structure (e.g., power-type nonlinearities), where the decay rates are more transparent in the stabilization context. 

\subsection{Turnpike and optimal shape design}
The proof of an exponential turnpike property for optimal shape design problems such as those discussed in Section \ref{sec: 7} remains completely open (even in the case of the heat equation, let alone that of the Navier-Stokes equations). 
Looking beyond, there are many other problems in the interface of shape optimization and turnpike that warrant further study.
For instance, it would be of interest to have a complete picture of the asymptotic behavior of optimal shapes for actuators, minimizing the controllability cost for partial differential equations. The problem of characterizing such optimal shapes for linear PDEs has been partially resolved, namely by making use of a randomization procedure (see  \cite{privat2015optimal, privat2016optimal} and the references therein). In the finite dimensional context, the optimal actuator shape may happen to be time-independent (\cite{geshkovski2021optimal}). 
But a full picture in the PDE setting is lacking.

\subsection{Beyond supervised learning}

The turnpike property, and insights stemming from the theory surrounding it, could have further applications beyond those discussed in preceding sections. The field of reinforcement learning in particular is known to have strong connections and  to classical optimal control and HJB equations (\cite{recht2019tour, bertsekas2019reinforcement, bertsekas2021lessons}). In reinforcement learning, one sometimes (but not always) works in a \emph{model-free} scenario, namely, the model on its own is not known, and is typically replaced by a Markov decision process. In a simpler setting, one could assume an LQ structure for a canonical system form $\dot{y}=Ay+Bu$, and assume that $A$ and $B$ are unknown, but can be estimated from data with high probability using some contemporary method (see e.g., \cite{dean2020sample}). In such cases, the turnpike property would apply to the estimated system. But what happens in the general case is not clear. 
Perhaps the turnpike property can again be used as a blueprint in view of avoiding discovering using full time-series data for learning the unknown dynamics. At any rate, an in-depth study in this direction is warranted.

\begin{acks}[Acknowledgments]
A major part of this work was completed while B.G. was affiliated with the Chair of Computational Mathematics, Fundación Deusto.

We warmly thank all the people who have contributed to the improvement of this manuscript through careful reading, comments, discussions and so forth. We in particular thank Emmanuel Trélat, Martin Lazar, Daniel Veldman, Sergi Andreu, Manuel Schaller, Carlos Esteve-Yag\"ue, and Charlotte Rodriguez for a careful reading of the manuscript and judicious remarks and suggestions. We also thank Lars Gr\"une and Dario Pighin for insightful comments.
\end{acks}

\begin{funding}
We gratefully acknowledge funding received from the European Union’s Horizon 2020 research and innovation programme under the Marie Sklodowska-Curie grant agreement No.765579-ConFlex, the Alexander von Humboldt-Professorship program, 
the European Research Council (ERC) under the European Union’s Horizon 2020 research and innovation programme (grant agreement NO. 694126-DyCon), 
the Transregio 154 Project “Mathematical Modeling, Simulation and Optimization Using the Example of Gas Networks” of the German DFG, grant MTM2017-92996-C2-1-R COSNET of MINECO (Spain), by the Elkartek grant KK-2020/00091 CONVADP of the Basque government and by the Air Force Office of Scientific Research (AFOSR) under Award NO: FA9550-18-1-0242.
\end{funding}

\bibliographystyle{imsart-number} 
\bibliography{refs}       

\end{document}